\documentclass[11pt,reqno]{amsart}
\usepackage[letterpaper,margin=1.1in]{geometry}
\usepackage{amsmath,amssymb,amsthm,mathtools}
\usepackage{enumitem}
\usepackage{booktabs}
\usepackage{graphicx}
\usepackage{array}
\usepackage[authoryear,round]{natbib}
\usepackage{microtype}
\usepackage{hyperref}
\usepackage{placeins}
\usepackage{subcaption}
\usepackage{longtable}
\usepackage{booktabs}
\usepackage{rotating}
\usepackage{multirow}
\usepackage{caption}
\usepackage{subcaption}
\usepackage{longtable}
\usepackage{capt-of}
\usepackage{float}
\usepackage{multirow}
\usepackage{xcolor}
\makeatletter
\AtBeginDocument{\let\@biblabel\NAT@biblabelnum\let\@bibsetup\NAT@bibsetnum}
\makeatother

\hypersetup{
  colorlinks=true,
  linkcolor=blue,
  citecolor=blue,
  urlcolor=blue
}
\newtheorem{theorem}{Theorem}[section]
\newtheorem{proposition}[theorem]{Proposition}
\newtheorem{lemma}[theorem]{Lemma}
\newtheorem{corollary}[theorem]{Corollary}
\theoremstyle{definition}

\newtheorem{assumption}[theorem]{Assumption}

\theoremstyle{remark}
\newtheorem{remark}[theorem]{Remark}
\DeclareMathOperator*{\argmin}{arg\,min}

\title[Scale Analysis and Shape Selection for the GGM]
{Scale Analysis and Shape Selection for the Generalized Gaussian Mechanism under Approximate Differential Privacy}
\author{Xiang Zhang}
\author{Mohamedou Ould Haye}
\author{Yiqiang Q. Zhao}
\address{Department of Mathematics and Statistics, Carleton University, Ottawa, Canada}

\makeatletter
\def\@settitle{\begin{center}%
  \baselineskip14\p@\relax
  \bfseries\Large
  \@title
  \end{center}%
}
\def\@setauthors{%
  \begingroup
  \trivlist
  \centering\footnotesize \@topsep30\p@\relax
  \advance\@topsep by -\baselineskip
  \item\relax
  \author@andify\authors
  \def\\{\protect\linebreak}%
  \authors
  \endtrivlist
  \endgroup
}
\makeatother
\begin{document}
\begin{abstract}
Differential privacy provides a rigorous framework for protecting private information, typically achieved by adding random noise to query results. The generalized Gaussian family is a flexible class of additive noise distributions indexed by the shape parameter \(p\) and includes the Laplace and Gaussian distributions as special cases \(p=1\) and \(p=2\), respectively. This paper studies the privacy-feasible scale estimation and the shape parameter selection of the generalized Gaussian mechanism (GGM) under \((\varepsilon,\delta)\)-differential privacy. For a given sensitivity vector $\Delta$ and $p\in[1,\infty]$, let \(b(p)\) denote the smallest value of the scale parameter for which the mechanism satisfies this privacy requirement. In the one-dimensional case, \(b(p)\) can be implicitly characterized by a system of equations. For vector-valued queries, we construct a computable upper approximation \(\widehat b(p)\) of \(b(p)\) that preserves the privacy guarantee. Shapes are compared under a scale-homogeneous utility criterion, with the \(m\)-th absolute moment as the main example. We develop an interval-wise shape search algorithm with an approximation guarantee that can be made arbitrarily precise. We also establish the invariance of the optimal shape under rescaling of the sensitivity vector and characterize its limiting behaviour under high privacy limits. Computational experiments show that optimizing shape parameters can improve utility by reducing the variance of each coordinate by \(5\%\)--\(20\%\) across a variety of cases, with some cases showing even greater reductions, while maintaining the same level of privacy protection. Task-specific experiments further show that shape optimization can improve task-level utility, reduce attacker success, or achieve both.
\end{abstract}
\maketitle
\section{Introduction}
The Laplace and Gaussian mechanisms are two classical additive noise methods in differential privacy \citep{Dwork2006,DworkRoth2014}. They correspond to shape parameters \(p=1\) and \(p=2\) in the generalized Gaussian family, respectively. This motivates considering other shapes in the generalized Gaussian family that may provide better utility at the same privacy level. Several questions then arise. For a fixed shape, how much noise is required for the generalized Gaussian mechanism (GGM) to satisfy \((\varepsilon,\delta)\)-differential privacy? After determining the scale \(b(p)\) for each shape, which \(p\) value minimizes the utility loss? To what extent can shape optimization improve utility compared to classical choices such as Laplace or Gaussian mechanisms? What factors influence the optimal shape parameter values? 

The generalized Gaussian mechanism has been studied before. For example, \citet{Liu2019} introduced the generalized Gaussian mechanism with a positive integer shape parameter and adopted the \(\ell_p\) global sensitivity formula, comparing its utility with Laplace and Gaussian mechanisms. \citet{GaneshZhao2021} studied generalized Gaussian mechanisms for counting queries with bounded \(\ell_\infty\) sensitivity under approximate differential privacy and derived sufficient privacy bounds and asymptotic \(\ell_\infty\)-error guarantees. Their analysis is presented mainly for even integer values of \(p\), although they note that this restriction can be removed. Recently, \citet{Rinberg2025} extended the Privacy Random Variable (PRV) accountant to generalized Gaussian mechanisms and evaluated them in high-composition private learning procedures, including Private Aggregation of Teacher Ensembles (PATE) and Differentially Private Stochastic Gradient Descent (DP-SGD), finding that the Gaussian choice \(\beta=2\) performed as well as or better than the other shape parameters tested within the computationally tractable ranges considered.  For the classical special cases, \citet{BalleWang2018} developed the analytic Gaussian mechanism by calibrating the Gaussian noise directly through the Gaussian distribution function, while \citet{HolohanLeithMason2015} gave a sufficient scale condition for the Laplace mechanism under approximate differential privacy. 

Related work on optimal additive noise includes \citet{GengViswanath2016}, who studied optimal mechanisms under pure differential privacy, and \citet{GengViswanathApprox2016}, who studied near-optimal noise-adding mechanisms for single integer-valued queries and vector-valued histogram-like queries under approximate differential privacy. For the special case of \((0,\delta)\)-differential privacy, \citet{GengDingGuoKumar2019} studied the optimal query-output-independent additive noise mechanism for a scalar query under a general cost-minimization framework. \citet{geng2020tight} derived tight upper and lower bounds for the privacy--utility trade-off for single real-valued queries under approximate differential privacy and analyzed truncated Laplacian mechanisms. More recently, \citet{joseph2025approximate} analyzed the \(\ell_2\) mechanism for vector-valued queries under approximate differential privacy and found that it can achieve lower error than both the Laplace and Gaussian mechanisms across a range of dimensions. In contrast, we treat the shape parameter as continuous over \(p\in[1,\infty]\), study scalar and vector-valued queries, and solve the privacy-preserving scale determination and utility-based shape selection problems within the generalized Gaussian family. Our analysis applies to general scale-homogeneous utility criteria and also considers task-specific criteria.

For each shape \(p\in[1,\infty]\), the differential privacy constraint determines a privacy-feasible scale \(b(p)\). The scale is studied in both one and multiple dimensions, and its regularity properties are used in the certified interval-wise search for the shape parameter. We also consider how the optimal shape depends on the sensitivity vector and what happens as the privacy parameters approach zero. Numerical examples show how the selected shape and the utility gain vary with the privacy parameters, dimension, and sensitivity vector.

The main results are:
\begin{enumerate}
\item For each fixed shape, we study the privacy-feasible scale of the generalized Gaussian mechanism. In the one-dimensional case, \(b(p)\) is determined implicitly as the unique solution to a system of equations. For vector-valued queries, we construct a computable upper approximation \(\widehat b(p)\) of \(b(p)\).

\item We prove the regularity and boundary properties of the scale map \(p \mapsto b(p)\). Based on this, under the general scale-homogeneous utility criterion, we develop a certified interval-wise shape search algorithm on the interval \([1,P_{\max}]\) with an approximation guarantee that can be made arbitrarily precise.

\item We establish properties of the optimal shape, including its invariance under rescaling of the sensitivity vector, and study its limiting behaviour as the privacy parameters approach zero.

\item Through computational experiments, we examine how the selected shape and utility gain vary with the privacy parameters, dimension, and sensitivity vector. For the second absolute moment, numerical results show variance reductions of approximately \(5\%\)--\(20\%\) for a range of the considered privacy parameters, with reductions exceeding \(20\%\) in some cases. We also consider a task-specific example in which shape selection is evaluated through the utility of the released result and the attacker’s success probability, showing that shape optimization can provide benefits beyond variance reduction.
\end{enumerate}

The remainder of the paper is organized as follows. Section~\ref{sec:setup} introduces the privacy definitions, the generalized Gaussian mechanism, and the utility criteria. Section~\ref{sec:calibration} studies fixed-shape scale selection in one and multiple dimensions. Section~\ref{sec:optimization} analyzes the scale map and presents the certified interval-wise shape search algorithm. Section~\ref{sec:optimal-shape-properties} studies further properties of the optimal shape, including rescaling invariance and limiting behaviour. Section~\ref{sec:Computational Experiments} gives numerical illustrations, and Section~\ref{sec:Conclusion} concludes.
 Finally, for convenience, Table~\ref{tab:main-notation} summarizes the main notation used throughout the paper.
\begingroup
\small
\renewcommand{\arraystretch}{1.12}
\setlength{\tabcolsep}{5pt}

\begin{longtable}{
>{\raggedright\arraybackslash}p{0.24\textwidth}
>{\raggedright\arraybackslash}p{0.68\textwidth}}

\caption{Main notation}
\label{tab:main-notation}\\
\toprule
Notation & Meaning \\
\midrule
\endfirsthead

\caption[]{Main notation (continued)}\\
\toprule
Notation & Meaning \\
\midrule
\endhead

\endfoot

\bottomrule
\endlastfoot

\multicolumn{2}{l}{\textit{Queries, privacy, and sensitivity}} \\*[2pt]

\(\mathcal D\), \(\mathcal D^n\)
& Data space and the corresponding space of datasets with \(n\) records. \\

\(D\sim D'\)
& Neighbouring datasets \(D,D'\in\mathcal D^n\) differing in at most one record. \\

\(q\), \(\mathcal H_q\)
& Query \(q:\mathcal D^n\to\mathbb R^d\) and the set of neighbouring query differences \(\mathcal H_q=\{q(D)-q(D'):D\sim D'\}\). \\

\(h\), \(\Delta\)
& Query difference \(h\in\mathcal H_q\) and sensitivity vector \(\Delta=(\Delta_1,\ldots,\Delta_d)\). \\

\(\mathcal S_\Delta\)
& Coordinate-wise sensitivity set \(\{h\in\mathbb R^d:|h_i|\le\Delta_i,\ i=1,\ldots,d\}\). \\

\((\varepsilon,\delta)\)
& Privacy parameters in approximate differential privacy. \\

\(D_\varepsilon(P\|Q)\)
& The \(\varepsilon\)-hockey-stick divergence between probability measures \(P\) and \(Q\). \\

\midrule
\multicolumn{2}{l}{\textit{Generalized Gaussian mechanism}} \\*[2pt]

\(d\), \(p\), \(b\)
& Query dimension, shape parameter, and scale parameter. \\

\(Z_{p,b}\), \(M_{p,b}\)
& Generalized Gaussian noise vector and the additive mechanism \(M_{p,b}(D)=q(D)+Z_{p,b}\). \\

\(\delta_{p,b}(h)\)
& Hockey-stick divergence associated with shape \(p\), scale \(b\), and shift \(h\). \\

$b(p)$ & Minimum privacy-feasible scale for shape $p$ under sensitivity vector $\Delta$. \\
\midrule
\multicolumn{2}{l}{\textit{Utility and shape selection}} \\*[2pt]

\(\mathcal L(p,b)=b^r\nu(p)\)
& Scale-homogeneous utility criterion, where \(r>0\) and \(\nu(p)>0\). \\

\(M_m(p)\)
& Shape factor in the \(m\)-th absolute moment: \(M_m(p)=\Gamma((m+1)/p)/\Gamma(1/p)\) for \(p<\infty\), with \(M_m(\infty)=1/(m+1)\). \\

\(L_m(p,b)\)
& Per-coordinate \(m\)-th absolute moment, \(L_m(p,b)=b^mM_m(p)\). \\

\(F(p)\)
& Logarithmic objective \(F(p)=\log L_m(p,b(p))\). \\

\(p^\star\), \(\hat p\)
& An optimal shape and the shape returned by the numerical procedure, respectively. \\

\([1,P_{\max}]\), \(\eta\)
& Search interval and optimization tolerance; \(e^\eta\) is the corresponding multiplicative accuracy factor. \\

\(\widehat b(p)\)
& Computed privacy-feasible upper approximation of \(b(p)\). \\

\(\tau_b,\alpha\)
& Scale tolerance and Monte Carlo confidence budget. \\
\end{longtable}
\endgroup

\FloatBarrier
\section{Problem Setup and the Generalized Gaussian Mechanism}
\label{sec:setup}
This section introduces the main definitions used throughout the paper. The notation is summarized in Table~\ref{tab:main-notation}. We recall differential privacy, define the generalized Gaussian mechanism and the coordinate-wise sensitivity set used in the multidimensional analysis, and formulate the utility-based shape-selection problem.

\subsection{Differential privacy}
\label{subsec:dp-definition}

Let \(D,D'\in\mathcal D^n\) be neighbouring datasets, written \(D\sim D'\), if they differ in at most one record. A randomized mechanism \(M\) is \((\varepsilon,\delta)\)-differentially private if, for all \(D\sim D'\) and every measurable set \(A\subseteq\mathbb R^d\),
\begin{equation}
\label{eq:dp-definition-paper}
\mathbb P(M(D)\in A)
\le
e^\varepsilon\mathbb P(M(D')\in A)+\delta.
\end{equation}
For probability measures \(P\) and \(Q\), define the \(\varepsilon\)-hockey-stick divergence by
\begin{equation}
\label{eq:hockey-stick-paper}
D_\varepsilon(P\|Q)
:=
\sup_A\{P(A)-e^\varepsilon Q(A)\},
\end{equation}
where the supremum is over all measurable sets; see \citet{barthe2013beyond}. Hence \eqref{eq:dp-definition-paper} is equivalent to
\begin{equation}
\label{eq:dp-hockey-equivalence-paper}
D_\varepsilon(P_D\|P_{D'})
\le
\delta,
\qquad
D\sim D',
\end{equation}
where \(P_D\) denotes the distribution of \(M(D)\). A density representation of \(D_\varepsilon\) is given in Appendix~\ref{app:setup-details}.

\subsection{The generalized Gaussian mechanism}
\label{subsec:ggm-definition}

For \(p\in[1,\infty)\) and \(b>0\), let \(GGD(0,b,p)\) denote the one-dimensional generalized Gaussian distribution with density
\begin{equation}
\label{eq:gg-one-density-paper}
g_{p,b}(x)
=
\frac{p}{2b\Gamma(1/p)}
\exp\!\left(
-\left|\frac{x}{b}\right|^p
\right),
\qquad
x\in\mathbb R.
\end{equation}
For \(p=\infty\), the limiting distribution is uniform on \([-b,b]\), with density
\begin{equation}
\label{eq:gg-infty-density-paper}
g_{\infty,b}(x)
=
\frac{1}{2b}\mathbf 1_{[-b,b]}(x).
\end{equation}
The cases \(p=1\), \(p=2\), and \(p=\infty\) give the Laplace, Gaussian \(N(0,b^2/2)\), and uniform distributions, respectively.

For \(z=(z_1,\ldots,z_d)\in\mathbb R^d\), define
\begin{equation}
\label{eq:gg-d-density-paper}
f^{(d)}_{p,b}(z)
=
\prod_{i=1}^d g_{p,b}(z_i),
\qquad
p\in[1,\infty].
\end{equation}
The generalized Gaussian mechanism is
\begin{equation}
\label{eq:gg-mechanism-paper}
M_{p,b}(D)
=
q(D)+Z_{p,b},
\end{equation}
where \(Z_{p,b}\) has density \(f^{(d)}_{p,b}\).

\subsection{Coordinate-wise sensitivity and privacy-feasible scale}
\label{subsec:coordinate-sensitivity}

Let \(q:\mathcal D^n\to\mathbb R^d\) be a query, and define the set of neighbouring query differences
\begin{equation}
\label{eq:shift-set}
\mathcal H_q
:=
\{q(D)-q(D'):D,D'\in\mathcal D^n,\ D\sim D'\}.
\end{equation}
A sensitivity vector is any \(\Delta=(\Delta_1,\ldots,\Delta_d)\in[0,\infty)^d\) such that
\[
|h_i|
\le
\Delta_i,
\qquad
h\in\mathcal H_q,
\qquad
i=1,\ldots,d.
\]
The corresponding coordinate-wise sensitivity set is
\begin{equation}
\label{eq:coordinate-sensitivity-set-paper}
\mathcal S_\Delta
:=
\{h\in\mathbb R^d:|h_i|\le\Delta_i,\ i=1,\ldots,d\}.
\end{equation}
Thus
\[
\mathcal H_q\subseteq\mathcal S_\Delta.
\]
The set \(\mathcal S_\Delta\) retains the coordinate-level information in the sensitivity vector.

For \(p\in[1,\infty]\), \(b>0\), and \(h\in\mathbb R^d\), let \(P_{p,b}\) be the distribution with density \(f^{(d)}_{p,b}\), and let \(P_{p,b,h}\) have density \(f^{(d)}_{p,b}(z+h)\). Define
\begin{equation}
\label{eq:delta-pbh-paper}
\delta_{p,b}(h)
:=
D_\varepsilon(P_{p,b}\|P_{p,b,h}).
\end{equation}
Then
\begin{equation}
\label{eq:dp-sup-shift-paper}
M_{p,b}\text{ is }(\varepsilon,\delta)\text{-DP}
\quad\Longleftrightarrow\quad
\sup_{h\in\mathcal H_q}\delta_{p,b}(h)
\le
\delta.
\end{equation}
Since \(\mathcal H_q\subseteq\mathcal S_\Delta\),
\[
\sup_{h\in\mathcal H_q}\delta_{p,b}(h)
\le
\sup_{h\in\mathcal S_\Delta}\delta_{p,b}(h).
\]
We define the privacy-feasible scale determined by the coordinate-wise sensitivity vector \(\Delta\) as
\begin{equation}
\label{eq:bstar-paper}
b(p)
:=
\inf\left\{
b>0:
\sup_{h\in\mathcal S_\Delta}\delta_{p,b}(h)
\le
\delta
\right\}.
\end{equation}
Thus, any \(b\) satisfying the constraint in \eqref{eq:bstar-paper} is privacy-feasible. Additional details are given in Appendix~\ref{app:setup-details}.

\subsection{Utility criteria and shape selection}
\label{subsec:utility-shape-selection}

Once a privacy-feasible scale \(b(p)\) has been assigned to each shape \(p\), the shape-selection problem compares a scale-homogeneous utility criterion
\begin{equation}
\label{eq:general-objective-paper}
\mathcal L(p,b)
=
b^r\nu(p),
\qquad
r>0,
\qquad
\nu(p)>0.
\end{equation}
The optimal shape satisfies
\begin{equation}
\label{eq:shape-selection-paper}
p^\star
\in
\arg\min_{p\in[1,\infty]}
\mathcal L\bigl(p,b(p)\bigr).
\end{equation}

The principal example used in this paper is the per-coordinate \(m\)-th absolute moment. Let \(Y_p\sim GGD(0,1,p)\) and define
\begin{equation}
\label{eq:Mm-paper}
M_m(p)
:=
\begin{cases}
\displaystyle
\frac{\Gamma((m+1)/p)}{\Gamma(1/p)},
& 1\le p<\infty,\\[1.2em]
\displaystyle
\frac{1}{m+1},
& p=\infty.
\end{cases}
\end{equation}
Then \(\mathbb E|Y_p|^m=M_m(p)\), and for \(Z_1\sim GGD(0,b,p)\),
\begin{equation}
\label{eq:Lm-paper}
L_m(p,b)
:=
\mathbb E|Z_1|^m
=
b^mM_m(p).
\end{equation}
Thus \(L_m\) is a special case of \eqref{eq:general-objective-paper} with \(r=m\) and \(\nu(p)=M_m(p)\). In particular,
\begin{equation}
\label{eq:variance-paper}
L_2(p,b)
=
\operatorname{Var}(Z_1)
=
b^2\frac{\Gamma(3/p)}{\Gamma(1/p)},
\qquad
1\le p<\infty,
\end{equation}
with \(L_2(\infty,b)=b^2/3\). Other scale-homogeneous utility criteria are recorded in Appendix~\ref{app:setup-details}.

\section{Fixed-Shape Scale under $(\varepsilon,\delta)$-Differential Privacy}
\label{sec:calibration}
For a fixed shape \(p\), we consider the minimum privacy-feasible scale \(b(p)\). In one dimension, \(b(p)\) is characterized exactly. In multiple dimensions, a computable upper approximation \(\widehat b(p)\) is constructed using the coordinate-wise sensitivity vector.
\subsection{Exact scale characterization in one dimension}

Consider a scalar query with sensitivity \(\Delta>0\). Let
\[
M_{p,b}(D)=q(D)+Z,
\qquad
Z\sim \mathrm{GGD}(0,b,p),
\]
where, for \(1\le p<\infty\),
\[
f_{p,b}(z)
=
\frac{p}{2b\Gamma(1/p)}
\exp\left\{-\left(\frac{|z|}{b}\right)^p\right\},
\qquad z\in\mathbb R .
\]
For a given sensitivity \(\Delta\), define
\[
\delta_{p,\Delta}(b)
:=
\int_{\mathbb R}
\left(f_{p,b}(z)-e^\varepsilon f_{p,b}(z+\Delta)\right)_+\,dz .
\]
The one-dimensional mechanism $M$ is \((\varepsilon,\delta)\)-differentially private under sensitivity \(\Delta\) if and only if
\[
\delta_{p,\Delta}(b)\le \delta .
\]
For \(p>1\), set
\[
\phi_p(z):=|z+\Delta|^p-|z|^p .
\]
The following elementary monotonicity property identifies the boundary of the hockey-stick integral.

\begin{lemma}[Monotonicity of the one-dimensional privacy-loss boundary]
\label{lem:phi-monotone}
Let \(p>1\) and \(\Delta>0\). Then \(\phi_p\) is continuous and increasing on \(\mathbb R\). Moreover,
\[
\phi_p(-\Delta/2)=0,
\qquad
\lim_{z\to-\infty}\phi_p(z)=-\infty,
\qquad
\lim_{z\to+\infty}\phi_p(z)=+\infty.
\]
Consequently, for every \(b>0\) and \(\varepsilon>0\), the equation
\[
|z+\Delta|^p-|z|^p=\varepsilon b^p
\]
has a unique solution \(z^\star(b)\in\mathbb R\).
\end{lemma}

\begin{proof}
See Appendix~\ref{lem:unique-zstar}.
\end{proof}
By Lemma~\ref{lem:phi-monotone}, for each \(b>0\) the positive part in the hockey-stick divergence is a right tail. More precisely, if \(z^\star=z^\star(b)\) is the unique solution of
\[
\phi_p(z^\star)=\varepsilon b^p,
\]
then
\[
\left\{z:
f_{p,b}(z)>e^\varepsilon f_{p,b}(z+\Delta)
\right\}
=
[z^\star,\infty).
\]
Thus
\[
\delta_{p,\Delta}(b)
=
\int_{z^\star}^{\infty}
\left\{
f_{p,b}(z)-e^\varepsilon f_{p,b}(z+\Delta)
\right\}\,dz.
\]
This tail representation yields the exact scale characterization in the next theorem.
\begin{theorem}[One-dimensional exact scale characterization]
\label{thm:1d-tight-calibration}
Fix \(p>1\), \(\Delta>0\), \(\varepsilon>0\), and \(0<\delta<1\). Define
\[
b_p(\varepsilon,\delta;\Delta)
:=
\inf\{b>0:\delta_{p,\Delta}(b)\le \delta\}.
\]
Then there is a unique pair
\[
(z^\star,b^\star)\in\mathbb R\times(0,\infty)
\]
such that, with
\[
s^\star:=\mathbf 1_{\{z^\star<0\}},
\]
one has
\begin{equation}
\label{eq:1d-tight-system-paper}
\begin{cases}
|z^\star+\Delta|^p-|z^\star|^p=\varepsilon (b^\star)^p,\\[4pt]
\displaystyle
\delta
=
s^\star
+
\frac{1-2s^\star}{2\Gamma(1/p)}
\Gamma\left(
\frac1p,
\left(\frac{|z^\star|}{b^\star}\right)^p
\right)
-
\frac{e^\varepsilon}{2\Gamma(1/p)}
\Gamma\left(
\frac1p,
\left(\frac{|z^\star+\Delta|}{b^\star}\right)^p
\right).
\end{cases}
\end{equation}
Moreover,
\[
b_p(\varepsilon,\delta;\Delta)=b^\star,
\]
and the one-dimensional generalized Gaussian mechanism is \((\varepsilon,\delta)\)-differentially private under sensitivity \(\Delta\) if and only if
\[
b\ge b^\star .
\]
\end{theorem}

\begin{proof}
See Appendix~\ref{thm:1d-ggm-tight}.
\end{proof}

\subsection{Classical one-dimensional cases}

The Gaussian case \(p=2\) gives the standard analytic Gaussian expression. Writing \(\sigma=b/\sqrt2\), the hockey-stick divergence is
\[
\delta_{2,\Delta}(b)
=
\Phi\left(
\frac{\Delta}{2\sigma}-\frac{\varepsilon\sigma}{\Delta}
\right)
-
e^\varepsilon
\Phi\left(
-\frac{\Delta}{2\sigma}-\frac{\varepsilon\sigma}{\Delta}
\right),
\qquad
\sigma=\frac{b}{\sqrt2}.
\]
This is the analytic Gaussian form of \citet{BalleWang2018}.

The endpoint \(p=1\) must be treated separately, because the privacy loss is no longer increasing on the whole real line. The next proposition gives the exact one-dimensional Laplace scale.

\begin{proposition}[Exact one-dimensional Laplace scale under \((\varepsilon,\delta)\)-differential privacy]
\label{prop:laplace-calibration}
Let \(\Delta>0\), \(\varepsilon>0\), and \(0\le\delta<1\). For the one-dimensional Laplace mechanism with scale \(b\),
\[
\delta_{1,\Delta}(b)
=
\left[
1-\exp\left(\frac{\varepsilon b-\Delta}{2b}\right)
\right]_+ .
\]
Consequently, the mechanism is \((\varepsilon,\delta)\)-differentially private if and only if
\[
b
\ge
\frac{\Delta}{\varepsilon-2\log(1-\delta)}.
\]
When \(\delta=0\), this reduces to the classical pure-DP Laplace scale
\[
b\ge \frac{\Delta}{\varepsilon}.
\]
\end{proposition}

\begin{proof}
The proof is given in Appendix~\ref{app:proof-laplace-scale}.
\end{proof}

For completeness, the limiting case \(p=\infty\) corresponds to uniform noise on \([-b,b]\). A direct overlap calculation gives
\[
\delta_{\infty,\Delta}(b)
=
\min\left\{1,\frac{\Delta}{2b}\right\},
\]
so the uniform mechanism is \((\varepsilon,\delta)\)-differentially private if and only if
\[
b\ge \frac{\Delta}{2\delta}.
\]

\subsection{Dependence on the sensitivity vector and the special role of \(p=2\)}
\label{subsec:scalar-coordinate-profile}
In multidimensional cases, the coordinate structure of a sensitivity vector is often not captured by a scalar norm. The following proposition states that the Gaussian distribution is an exception to the family of generalized Gaussian distributions.
For a shift vector \(\Delta\in\mathbb R^d\), write
\[
\delta_{p,b}(\varepsilon;\Delta)
:=
\int_{\mathbb R^d}
\left(
 f^{(d)}_{p,b}(z)
 -e^\varepsilon f^{(d)}_{p,b}(z+\Delta)
\right)_+\,dz,
\qquad \varepsilon\ge0.
\]
The corresponding minimum feasible scale for a fixed shift vector is
\[
b_p(\varepsilon,\delta;\Delta)
:=
\inf\left\{
 b>0:
 \delta_{p,b}(\varepsilon;\Delta)\le\delta
\right\}.
\]
\begin{proposition}[Gaussian case and Euclidean sensitivity]
\label{prop:gaussian-scalar-coordinate-profile}
Assume \(d\ge2\). Within the generalized Gaussian family, the Gaussian shape \(p=2\) is the only shape with the following property: for every \(b>0\) and \(\varepsilon\ge0\), the value
\[
\delta_{p,b}(\varepsilon;\Delta)
\]
depends on the shift vector \(\Delta\) only through its Euclidean norm \(\|\Delta\|_2\). In particular, for \(p=2\),
\[
\delta_{2,b}(\varepsilon;\Delta)
\]
depends on \(\Delta\) only through \(\|\Delta\|_2/b\), and therefore
\[
b_2(\varepsilon,\delta;\Delta)
=
\|\Delta\|_2\, b_2(\varepsilon,\delta;e_1).
\]
where \(e_1=(1,0,\ldots,0)\) represents a unit shift in one coordinate.
For \(p\ne2\), this reduction generally fails. The way the sensitivity is spread across coordinates can change the value of \(\delta_{p,b}(\varepsilon;\Delta)\). More precisely, for any \(q\in[1,\infty]\), there can be two sensitivity vectors with the same \(\ell_q\)-norm but different values of \(\delta_{p,b}(\varepsilon;\Delta)\). Thus, outside the Gaussian case, knowing only one scalar norm of \(\Delta\) is generally not enough.
\end{proposition}

\begin{proof}
See Appendix~\ref{app:sensitivity-geometry}.
\end{proof}
This proposition explains why the coordinate-wise sensitivity vector \(\Delta\) is used when computing privacy-feasible scales for vector-valued queries. In the Gaussian case, the dependence on the shift is fully captured by \(\|\Delta\|_2\). For other shapes, the allocation of sensitivity across coordinates can affect the privacy constraint.
\subsection{Worst-case shift over a coordinate-wise sensitivity box}
\label{subsec:worstcase-box}

Recall that \(\delta_{p,b}(h)\) denotes the hockey-stick divergence associated with a shift \(h\), as defined in Section~\ref{sec:setup}. The following lemma reduces the supremum over the coordinate-wise sensitivity box \(\mathcal S_\Delta\) to the corner point \(\Delta=(\Delta_1,\ldots,\Delta_d)\).

\begin{lemma}[Worst-case shift over a coordinate-wise sensitivity box]
\label{lem:worstcase-box-paper}
Fix \(1\le p<\infty\), \(b>0\), and \(\varepsilon\ge0\). Then:
\begin{enumerate}[label=(\roman*), leftmargin=2em]
\item \(\delta_{p,b}(h)\) is invariant under coordinate-wise sign changes of \(h\).
\item If \(0\le h_i\le t_i\) for all \(i=1,\ldots,d\), then
$
\delta_{p,b}(h)\le \delta_{p,b}(t).
$
\item For the coordinate-wise sensitivity set \(\mathcal S_\Delta\),
$
\sup_{h\in\mathcal S_\Delta}\delta_{p,b}(h)
=
\delta_{p,b}(\Delta).
$
\end{enumerate}
\end{lemma}
\begin{proof}
See Appendix~\ref{app:worstcase-box-proof}.
\end{proof}
Consequently, by part~(iii) and \eqref{eq:bstar-paper},
$
b(p)
=
\inf\left\{
b>0:
\delta_{p,b}(\Delta)\le\delta
\right\}.
$
\subsection{Empirical Bernstein upper bound for a fixed scale}

For \(1\le p<\infty\), the integral defining \(\delta_{p,b}(\Delta)\) is generally not available in closed form. It can be written as an expectation. Let \(Z\) have density \(f^{(d)}_{p,b}\) and define
\[
w_{p,b}(z)
:=
\left(
1
-
e^\varepsilon
\frac{f^{(d)}_{p,b}(z+\Delta)}{f^{(d)}_{p,b}(z)}
\right)_+ .
\]
Equivalently,
\[
w_{p,b}(z)
=
\left(
1-
\exp\left\{
\varepsilon
-
\frac{\|z+\Delta\|_p^p-\|z\|_p^p}{b^p}
\right\}
\right)_+ .
\]
Then
\[
0\le w_{p,b}(z)\le1,
\qquad
\delta_{p,b}(\Delta)
=
\mathbb E\bigl[w_{p,b}(Z)\bigr].
\]
Let
\[
Z^{(1)},\ldots,Z^{(N)}
\stackrel{\mathrm{i.i.d.}}{\sim}
f^{(d)}_{p,b},
\qquad
W_i:=w_{p,b}(Z^{(i)}).
\]
Define
\[
\widehat\delta_N(b)
:=
\frac1N\sum_{i=1}^N W_i,
\qquad
\widehat V_N(b)
:=
\frac{1}{N-1}
\sum_{i=1}^N
\left(W_i-\widehat\delta_N(b)\right)^2.
\]
Following the empirical Bernstein bound of \citet{maurer2009empirical}, for \(\alpha\in(0,1)\) set
\[
R_N(\alpha;b)
:=
\sqrt{
\frac{2\widehat V_N(b)\log(2/\alpha)}{N}
}
+
\frac{7\log(2/\alpha)}{3(N-1)}.
\]
Define
\[
U_N(\alpha;b)
:=
\widehat\delta_N(b)+R_N(\alpha;b),
\qquad
L_N(\alpha;b)
:=
\max\{\widehat\delta_N(b)-R_N(\alpha;b),0\}.
\]

\begin{proposition}[Empirical Bernstein upper bound]
\label{prop:certified-fixed-scale}
Fix \(1\le p<\infty\), \(b>0\), \(N\ge2\), and \(\alpha\in(0,1)\). With probability at least \(1-\alpha\),
\[
L_N(\alpha;b)
\le
\delta_{p,b}(\Delta)
\le
U_N(\alpha;b).
\]
Consequently, on the same event, if
\[
U_N(\alpha;b)\le\delta,
\]
then the mechanism \(M_{p,b}\) is \((\varepsilon,\delta)\)-differentially private under the coordinate-wise sensitivity \(\Delta\).
\end{proposition}

\begin{proof}
The variables \(W_1,\ldots,W_N\) are i.i.d. in \([0,1]\), and
\[
\mathbb E W_i=\delta_{p,b}(\Delta).
\]
The empirical Bernstein inequality of \citet{maurer2009empirical} gives the stated confidence interval. On this event, \(U_N(\alpha;b)\le\delta\) implies
\[
\delta_{p,b}(\Delta)\le\delta.
\]
The privacy statement follows from Lemma~\ref{lem:worstcase-box-paper}.
\end{proof}

\subsection{Upper approximation of the fixed-shape scale}
\label{subsec:fixed-shape-upper-approximation}
Fix \(1\le p<\infty\), \(\varepsilon>0\), \(0<\delta<1\), and a coordinate-wise sensitivity vector \(\Delta\in[0,\infty)^d\setminus\{0\}\). The goal is to compute an upper approximation \(\widehat b(p)\) of \(b(p)\) that is privacy-feasible with high probability.
We will also use the following property:
\begin{equation}
\label{eq:fixed-shape-sign}
b<b(p)\Longrightarrow\delta_{p,b}(\Delta)>\delta,
\qquad
b>b(p)\Longrightarrow\delta_{p,b}(\Delta)<\delta.
\end{equation}
A proof is given in Appendix~\ref{app:initial-upper-scale}.
Since \(b\mapsto\delta_{p,b}(\Delta)\) is nonincreasing, we maintain a lower and upper bracket for \(b(p)\). Set
\[
b_{\mathrm{low}}^{(0)}=0.
\]
For \(p=1\), take
\[
b_{\mathrm{high}}^{(0)}
=
\frac{\|\Delta\|_1}{\varepsilon}.
\]
For \(p>1\), let
\[
\kappa
=
\frac{d}{p}
+
\sqrt{
\frac{2d}{p}\log\frac1\delta
}
+
\log\frac1\delta,
\]
and set
\[
b_{\mathrm{high}}^{(0)}
=
\max\left\{
\frac{p\,2^p}{\varepsilon}
\|\Delta\|_p\,
\kappa^{(p-1)/p},
\left(
\frac{p\,2^p}{\varepsilon}
\right)^{1/p}
\|\Delta\|_p
\right\}.
\]
This choice satisfies
\(\delta_{p,b_{\mathrm{high}}^{(0)}}(\Delta)\le\delta\); the proof is given in Appendix~\ref{app:initial-upper-scale}.
If
$
b_{\mathrm{high}}^{(0)}-b_{\mathrm{low}}^{(0)}\le\tau_b,
$
set \(\widehat b(p)=b_{\mathrm{high}}^{(0)}\) and stop. Otherwise, define
\[
K
=
\left\lceil
\log_{4/3}
\left(
\frac{
b_{\mathrm{high}}^{(0)}-b_{\mathrm{low}}^{(0)}
}{
\tau_b
}
\right)
\right\rceil.
\]
At step \(t\), let
\[
w_t
=
b_{\mathrm{high}}^{(t)}
-
b_{\mathrm{low}}^{(t)},
\]
and consider
\[
b_{1/2}^{(t)}
=
b_{\mathrm{low}}^{(t)}
+\frac12 w_t,
\qquad
b_{3/4}^{(t)}
=
b_{\mathrm{low}}^{(t)}
+\frac34 w_t.
\]
The midpoint is tested first using the empirical Bernstein bounds in Proposition~\ref{prop:certified-fixed-scale}. If it is not resolved, the \(3/4\) point is also tested, and the sample sizes are increased until one of the two comparisons is resolved.

For \(s\in\{1/2,3/4\}\), repeated attempts use confidence budgets
\[
\alpha_{t,s,r}
=
\frac{3\alpha}{\pi^2K(r+1)^2},
\qquad
t=0,\ldots,K-1,\quad r=0,1,2,\ldots,
\]
so that
\[
\sum_{t=0}^{K-1}
\sum_{s\in\{1/2,3/4\}}
\sum_{r=0}^{\infty}
\alpha_{t,s,r}
\le\alpha.
\]

For a tested point \(b_s^{(t)}\), if
\[
U_N(\alpha_{t,s,r};b_s^{(t)})\le\delta,
\]
it is certified feasible and becomes the new upper endpoint. If
\[
L_N(\alpha_{t,s,r};b_s^{(t)})>\delta,
\]
it is certified infeasible and becomes the new lower endpoint. 
After \(K\) refinement steps, the procedure returns
\[
\widehat b(p)=b_{\mathrm{high}}^{(K)}.
\]
\begin{theorem}[Finite-sample upper approximation of \(b(p)\)]
\label{thm:fixed-shape-upper-approximation}
The procedure above terminates after at most \(K\) refinement steps and uses finitely many samples. Moreover, with probability at least \(1-\alpha\),
\[
\delta_{p,\widehat b(p)}(\Delta)\le\delta,
\qquad
0\le\widehat b(p)-b(p)\le\tau_b.
\]
\end{theorem}

\begin{proof}
Let \(E\) be the event on which all empirical Bernstein bounds used by the procedure are valid. By Proposition~\ref{prop:certified-fixed-scale} and the choice of the confidence budgets,
\(\mathbb P(E)\ge1-\alpha\). Work on \(E\).

At each step,
\(b^{(t)}_{1/2}\ne b^{(t)}_{3/4}\), so at most one of them equals
\(b(p)\). By~\eqref{eq:fixed-shape-sign}, at least one satisfies
\(\delta_{p,b}(\Delta)\ne\delta\). For such a strict comparison, the empirical Bernstein bounds resolve the comparison after a sufficiently large but finite sample.
For a resolved test point \(b_s^{(t)}\), if
\(U_N(\alpha_{t,s,r};b_s^{(t)})\le\delta\), then \(b_s^{(t)}\) becomes the new upper endpoint. If
\(L_N(\alpha_{t,s,r};b_s^{(t)})>\delta\), then \(b_s^{(t)}\) becomes the new lower endpoint. Therefore,
$
b_{\mathrm{low}}^{(t)}
\le
b(p)
\le
b_{\mathrm{high}}^{(t)}
$
is preserved.

If \(b_{1/2}^{(t)}\) is used, then \(w_{t+1}\le\frac12w_t\); if
\(b_{3/4}^{(t)}\) is used, then \(w_{t+1}\le\frac34w_t\). Therefore
\[
w_K
\le
\left(\frac34\right)^K w_0
\le
\tau_b.
\]
Since \(\widehat b(p)=b_{\mathrm{high}}^{(K)}\),
\[
0\le\widehat b(p)-b(p)\le w_K\le\tau_b.
\]
The upper endpoint remains privacy-feasible, so
$
\delta_{p,\widehat b(p)}(\Delta)\le\delta.
$
\end{proof}

\begin{corollary}[Unconditional privacy from a high-probability scale guarantee]
\label{cor:mc-scale-delta-alpha}
Let \(\widehat b(p)\) be a random scale satisfying
\[
\mathbb P\left(\delta_{p,\widehat b(p)}(\Delta)\le\delta\right)\ge1-\alpha.
\]
Conditionally on \(\widehat b\), let \(Z\) have density \(f^{(d)}_{p,\widehat b(p)}\), and define
\[
\widetilde M(D)=q(D)+Z .
\]
Then \(\widetilde M\) is \((\varepsilon,\delta+\alpha)\)-differentially private.
\end{corollary}

\begin{proof}
Let
\[
\mathcal G
=
\left\{
\delta_{p,\widehat b(p)}(\Delta)\le\delta
\right\}.
\]
Then \(\mathbb P(\mathcal G^c)\le\alpha\). For neighbouring datasets \(D\sim D'\) and any measurable set \(A\subseteq\mathbb R^d\),
\[
\begin{aligned}
\mathbb P\left(\widetilde M(D)\in A\right)
&=
\mathbb E\left[
\mathbb P\left(\widetilde M(D)\in A\mid \widehat b\right)\mathbf 1_{\mathcal G}
\right]
+
\mathbb E\left[
\mathbb P\left(\widetilde M(D)\in A\mid \widehat b\right)\mathbf 1_{\mathcal G^c}
\right] \\
&\le
\mathbb E\left[
\left(
e^\varepsilon
\mathbb P\left(\widetilde M(D')\in A\mid \widehat b\right)
+
\delta
\right)
\mathbf 1_{\mathcal G}
\right]
+
\mathbb P(\mathcal G^c) \\
&\le
e^\varepsilon
\mathbb P\left(\widetilde M(D')\in A\right)
+
\delta
+
\alpha .
\end{aligned}
\]
Therefore \(\widetilde M\) satisfies \((\varepsilon,\delta+\alpha)\)-differential privacy.
\end{proof}

\begin{remark}[Interpretation of the Monte Carlo guarantee]
Theorem~\ref{thm:fixed-shape-upper-approximation} gives a high-probability guarantee for the computed scale:
$
\mathbb P\left(
\delta_{p,\widehat b(p)}(\Delta)\le\delta,
\quad
\widehat b(p)\ge b(p)
\right)
\ge
1-\alpha .
$
Corollary~\ref{cor:mc-scale-delta-alpha} converts the high-probability guarantee for the scale into an unconditional privacy guarantee: the resulting randomized mechanism satisfies \((\varepsilon,\delta+\alpha)\)-differential privacy. Thus, the uncertainty introduced by the Monte Carlo method is absorbed into the failure probability through the additional term \(\alpha\).
\end{remark}
\section{The Scale Map and Shape Selection}
\label{sec:optimization}
This section treats the shape parameter \(p\) as a variable. For each fixed \(p\), let \(b(p)\) denote the minimum privacy-feasible scale determined by the \((\varepsilon,\delta)\)-DP constraint and the coordinate-wise sensitivity vector \(\Delta\). For the \(m\)-th absolute moment,
\[
L_m(p,b)=b^mM_m(p),\qquad
M_m(p)=\frac{\Gamma((m+1)/p)}{\Gamma(1/p)},\qquad 1\le p<\infty,
\]
with \(M_m(\infty)=1/(m+1)\). Define
\begin{equation}
\label{eq:F-objective-section4}
F(p):=\log L_m(p,b(p))=\log M_m(p)+m\log b(p).
\end{equation}
On \([1,P_{\max}]\), the finite-shape problem is
\begin{equation}
\label{eq:truncated-shape-problem}
F^\star:=\inf_{p\in[1,P_{\max}]}F(p),\qquad
p^\star\in\argmin_{p\in[1,P_{\max}]}F(p).
\end{equation}
The case \(p=\infty\), corresponding to uniform noise, is evaluated separately. In multidimensional computations, \(b(p)\) is evaluated using the upper approximation \(\widehat b(p)\) from Section~\ref{subsec:fixed-shape-upper-approximation}. The goal is to construct \(\hat p\in[1,P_{\max}]\) such that
\begin{equation}
\label{eq:section4-target-guarantee}
F(\hat p)\le F^\star+\eta,
\end{equation}
or equivalently,
\[
L_m(\hat p,b(\hat p))\le e^\eta\inf_{p\in[1,P_{\max}]}L_m(p,b(p)).
\]
\subsection{The scale map}
Let \(\Delta\in[0,\infty)^d\setminus\{0\}\), and let \(\delta_\Delta(p,b)\) denote the hockey-stick divergence at the worst-case shift \(\Delta\). By Lemma~\ref{lem:worstcase-box-paper}, for \(p>1\),
\begin{equation}
\label{eq:b-p-section4}
b(p):=\inf\{b>0:\delta_\Delta(p,b)\le\delta\}.
\end{equation}
\begin{theorem}[Existence, uniqueness, differentiability, and endpoint behaviour
of the scale function]
\label{thm:scale-function-properties}
Assume \(d\ge 1\), \(\varepsilon\ge 0\), \(0<\delta<1\), and
\(\Delta\in[0,\infty)^d\setminus\{0\}\).
For every \(p\in(1,\infty)\), the map
\(b\mapsto\delta_\Delta(p,b)\) is continuous and strictly decreasing on
\((0,\infty)\), with
\[
\lim_{b\downarrow0}\delta_\Delta(p,b)=1,
\qquad
\lim_{b\to\infty}\delta_\Delta(p,b)=0.
\]
Hence there is a unique \(b(p)\in(0,\infty)\) such that
\[
\delta_\Delta(p,b(p))=\delta.
\]
Moreover, \(p\mapsto b(p)\) is continuously differentiable on
\((1,\infty)\), and
\[
\lim_{p\downarrow1}b(p)=b(1),
\qquad
\lim_{p\to\infty}b(p)=b(\infty).
\]
If \(u(1)=1/b(1)\) does not belong to the exceptional set
\[
\mathcal E
:=
\left\{
u>0:
u\langle s,\Delta\rangle=\varepsilon
\text{ for some }s\in\{-1,1\}^d
\right\},
\]
then the right derivative \(b'_+(1)\) exists.
More precisely, with \(u(p)=1/b(p)\) and
\[
g_p(z)
:=
\left(\frac{p}{2\Gamma(1/p)}\right)^d
\exp\!\left(-\|z\|_p^p\right),
\]
define
\[
\Phi(p,u)
:=
\int_{\mathbb R^d}
\bigl(g_p(z)-e^\varepsilon g_p(z+u\Delta)\bigr)_+\,dz.
\]
Then
\[
\delta_\Delta(p,b)=\Phi(p,1/b),
\qquad
\Phi(p,u(p))=\delta,
\]
and, for \(p\in(1,\infty)\),
\[
u'(p)
=
-\frac{\partial_p\Phi(p,u)}
        {\partial_u\Phi(p,u)}
\bigg|_{u=u(p)},
\qquad
b'(p)
=
-\frac{u'(p)}{u(p)^2}.
\]
If \(u(1)\notin\mathcal E\), then
\[
u'_+(1)
=
-\frac{\partial_{p+}\Phi(1,u(1))}
        {\partial_u\Phi(1,u(1))},
\qquad
b'_+(1)
=
-\frac{u'_+(1)}{u(1)^2}.
\]
\end{theorem}

\begin{proof}
See Appendix~\ref{app:proof-scale-function-properties}.
\end{proof}
\begin{remark}[Scale-map regularity and shape selection]
For any scale-homogeneous utility criterion,
\[
\mathcal L(p,b)=b^r\nu(p),\qquad r>0,\qquad \nu(p)>0,\qquad
\mathcal L(p,b(p))=b(p)^r\nu(p).
\]
The only implicit term in the feasible objective is \(b(p)\). Once the regularity of \(b(p)\) is established, the regularity of \(p\mapsto \mathcal L(p,b(p))\) follows from that of \(\nu\).

The \(m\)-th absolute moment corresponds to \(r=m\) and \(\nu(p)=M_m(p)\). More generally, if \(\nu\) is continuous and positive on \([1,P_{\max}]\), then \(p\mapsto b(p)^r\nu(p)\) is continuous on \([1,P_{\max}]\), and
\[
\arg\min_{p\in[1,P_{\max}]}\mathcal L(p,b(p))\ne\varnothing.
\]
\end{remark}
\subsection{Certified interval bounds}
For an interval \(I=[a,c]\subset[1,P_{\max}]\), let
\[
p_I:=\frac{a+c}{2},\qquad h_I:=\sup_{p\in I}|p-p_I|=\frac{c-a}{2}.
\]
The search uses interval lower bounds and midpoint upper bounds,
\[
L_I\le \inf_{p\in I}F(p),\qquad F(p_I)\le U_I.
\]
Define
\begin{equation}
\label{eq:M-lower-section4}
D_M(I):=\sup_{p\in I}\left|\frac{d}{dp}\log M_m(p)\right|,
\qquad
\underline M_I:=\log M_m(p_I)-h_ID_M(I).
\end{equation}
Then
\[
\underline M_I\le \inf_{p\in I}\log M_m(p).
\]
If \(\underline b_I\le b(p)\) for all \(p\in I\), then
\begin{equation}
\label{eq:interval-lower-bound-F}
L_I:=\underline M_I+m\log\underline b_I
\end{equation}
satisfies
\[
L_I\le \inf_{p\in I}F(p).
\]
If \(\overline b_I\ge b(p_I)\), then
\begin{equation}
\label{eq:midpoint-upper-bound-F}
U_I:=\log M_m(p_I)+m\log\overline b_I
\end{equation}
satisfies
\[
F(p_I)\le U_I.
\]
When both bounds are available,
\begin{equation}
\label{eq:local-gap-decomposition-section4}
U_I-L_I=A(I)+B(I),
\end{equation}
where
\begin{equation}
\label{eq:A-B-section4}
A(I):=\log M_m(p_I)-\underline M_I,\qquad
B(I):=m\log\frac{\overline b_I}{\underline b_I}.
\end{equation}
Consequently,
\[
A(I)\le h_ID_M(I),
\qquad
B(I)=m\log\frac{\overline b_I}{\underline b_I}.
\]
Decreasing the length of the sub-interval \(|I|\) will decrease \(A(I)\); tightening \([\underline b_I,\overline b_I]\) will decrease \(B(I)\).
It remains to certify lower and upper bounds for the scale. Let
\[
H(p,b):=\delta_\Delta(p,b)-\delta,
\]
where \(\delta_\Delta(p,b)\) is the hockey-stick divergence for the sensitivity vector \(\Delta\), and \(\delta\) is the prescribed threshold in the \((\varepsilon,\delta)\)-DP condition. For fixed \(p\), \(b\mapsto H(p,b)\) is strictly decreasing, and
\[
H(p,b)=0
\quad\Longleftrightarrow\quad
b=b(p).
\]
Consequently,
\[
H(p,b)\ge0\quad\Longrightarrow\quad b\le b(p),
\qquad
H(p,b)\le0\quad\Longrightarrow\quad b\ge b(p).
\]
\begin{lemma}[Interval bounds for the scale function]
\label{lem:interval-sign-propagation}
Let
\[
I=[a,c]\subset[1,P_{\max}],
\qquad
p_0\in I\cap(1,\infty),
\qquad
b>0,
\]
and define
\[
h_I:=\sup_{p\in I}|p-p_0|.
\]
Let \(S(I,b)<\infty\) satisfy
\[
S(I,b)
\ge
\sup_{p\in I\cap(1,\infty)}
|\partial_p H(p,b)|.
\]
Then
\[
H(p_0,b)\ge h_I S(I,b)
\quad\Longrightarrow\quad
b\le b(p),
\qquad p\in I,
\]
and
\[
H(p_0,b)\le -h_I S(I,b)
\quad\Longrightarrow\quad
b\ge b(p),
\qquad p\in I.
\]
\end{lemma}

\begin{proof}
For \(p\in I\cap(1,\infty)\),
\[
|H(p,b)-H(p_0,b)|
\le
\int_{\min\{p,p_0\}}^{\max\{p,p_0\}}
|\partial_q H(q,b)|\,dq
\le
S(I,b)|p-p_0|
\le
h_I S(I,b).
\]
If \(1\in I\), letting \(p\downarrow1\) and using the continuity of
\(H(\cdot,b)\) at \(p=1\) gives
\[
|H(1,b)-H(p_0,b)|
\le
h_I S(I,b).
\]
Thus, for every \(p\in I\),
\[
|H(p,b)-H(p_0,b)|
\le
h_I S(I,b).
\]
Hence
\[
H(p_0,b)\ge h_I S(I,b)
\quad\Longrightarrow\quad
H(p,b)\ge0,
\qquad p\in I,
\]
and
\[
H(p_0,b)\le-h_I S(I,b)
\quad\Longrightarrow\quad
H(p,b)\le0,
\qquad p\in I.
\]
The conclusions follow from the monotonicity of
\(b\mapsto H(p,b)\) and \(H(p,b(p))=0\).
\end{proof}
By Theorem~\ref{thm:scale-function-properties}, for every fixed \(b>0\), the map
\(p\mapsto H(p,b)\) is continuous at \(p=1\) and continuously
differentiable on \((1,\infty)\).
Sections~\ref{app:one-dimensional-tight-bound} and~\ref{app:proof-derivative-bound}  give constructions of \(S(I,b)\) satisfying
$
S(I,b)
\ge
\sup_{p\in I\cap(1,\infty)}
|\partial_p H(p,b)|
$
for compact intervals \(I\subset[1,\infty)\).

When \(\delta_\Delta(p,b)\) is not available in closed form, the sign of \(H(p,b)\) is tested using Proposition~\ref{prop:certified-fixed-scale}. With the shape parameter displayed, write the corresponding empirical Bernstein bounds as
\[
L_N(\alpha_i;p,b)\le \delta_\Delta(p,b)\le U_N(\alpha_i;p,b)
\]
on an event \(E_i\) satisfying \(\mathbb P(E_i)\ge1-\alpha_i\). Hence, on \(E_i\), for any \(\gamma\ge0\),
\begin{equation}
\label{eq:mc-positive-sign-section4}
L_N(\alpha_i;p,b)\ge\delta+\gamma
\quad\Longrightarrow\quad
H(p,b)\ge\gamma,
\end{equation}
and
\begin{equation}
\label{eq:mc-negative-sign-section4}
U_N(\alpha_i;p,b)\le\delta-\gamma
\quad\Longrightarrow\quad
H(p,b)\le-\gamma.
\end{equation}
For an adaptive sequence of tests with budgets \((\alpha_i)_{i\ge1}\),
\[
\sum_{i\ge1}\alpha_i\le\alpha
\quad\Longrightarrow\quad
\mathbb P\left(\bigcap_{i\ge1}E_i\right)\ge1-\alpha.
\]
Define the simultaneous validity event
\[
E
:=
\bigcap_{i\ge1}E_i.
\]
Then \(\mathbb P(E)\ge1-\alpha\), and all sign decisions below are made on \(E\).

\subsection{Adaptive shape selection}

The adaptive search maintains a finite active family \(\mathcal A\) of intervals. Each \(I\in\mathcal A\) carries a lower bound \(L_I\). Some intervals also carry a midpoint upper bound \(U_I\). The global bounds are
\begin{equation}
\label{eq:global-bounds-section4}
L^*:=\min_{I\in\mathcal A}L_I,\qquad
U^*:=\min\{U_I:U_I\ \text{has been computed}\}.
\end{equation}
A midpoint attaining \(U^*\) is denoted by \(\hat p\). The stopping condition is
\begin{equation}
\label{eq:stopping-rule-section4}
U^*-L^*\le\eta.
\end{equation}
The search is initialized by the deterministic upper bound
\begin{equation}
\label{eq:initial-upper-section4}
U^{(0)}(p):=\log M_m(p)+m\log b_{\mathrm{high}}^{(0)}(p),
\end{equation}
where \(b_{\mathrm{high}}^{(0)}(p)\) is the initial fixed-shape upper scale from Section~\ref{sec:calibration}. Since \(b(p)\le b_{\mathrm{high}}^{(0)}(p)\),
\[
F(p)\le U^{(0)}(p).
\]
An initial finite set \(\mathcal P_0\subset[1,P_{\max}]\) provides
\[
U_0^\star:=\min_{p\in\mathcal P_0}U^{(0)}(p),
\qquad
\hat p_0\in\arg\min_{p\in\mathcal P_0}U^{(0)}(p),
\]
so that
\[
F(\hat p_0)\le U_0^\star.
\]
The interval family is then initialized on \([1,P_{\max}]\), and intervals satisfying
\[
L_I>U_0^\star
\]
are discarded.

At each refinement step, the interval with the smallest lower bound is selected. The maintained invariants are
\begin{equation}
\label{eq:invariants-section4}
L_I\le\inf_{p\in I}F(p),\qquad
F(p_I)\le U_I\quad\text{whenever }U_I\text{ is defined}.
\end{equation}
Refinement acts on the decomposition \eqref{eq:local-gap-decomposition-section4}:
\[
U_I-L_I=A(I)+B(I).
\]
If \(A(I)\) is large, the interval is bisected. If \(B(I)\) is large, the scale bracket \([\underline b_I,\overline b_I]\) is tightened by testing intermediate scale values through \eqref{eq:mc-positive-sign-section4}--\eqref{eq:mc-negative-sign-section4} and Lemma~\ref{lem:interval-sign-propagation}. Appendix~\ref{app:derivative-bounds-section4} gives the constructive details. 
\subsection{Validity of the selection rule}

\begin{proposition}[Safe deletion]
\label{prop:safe-deletion}
On \(E\), if an active interval \(J\) satisfies
\[
L_J>U^*,
\]
then \(J\) contains no minimizer of \(F\) over \([1,P_{\max}]\).
\end{proposition}

\begin{proof}
For every \(p\in J\),
\[
F(p)\ge\inf_{q\in J}F(q)\ge L_J>U^*.
\]
Since \(U^*\) is attained by an evaluated midpoint \(\hat p\) satisfying \(F(\hat p)\le U^*\), no point in \(J\) can be globally optimal.
\end{proof}

\begin{theorem}[Near-optimality of the returned shape]
\label{thm:certified-near-optimality}
Assume that the construction stops with
\[
U^*-L^*\le\eta
\]
and returns a midpoint \(\hat p\) attaining \(U^*\). Then, on \(E\),
\[
F(\hat p)\le\inf_{p\in[1,P_{\max}]}F(p)+\eta.
\]
Consequently,
\[
L_m(\hat p, b(\hat p))\le e^\eta\inf_{p\in[1,P_{\max}]}L_m(p,b(p)).
\]
In particular,
\[
\mathbb P\left(F(\hat p)\le\inf_{p\in[1,P_{\max}]}F(p)+\eta\right)\ge1-\alpha.
\]
\end{theorem}
\begin{proof}
On \(E\), \eqref{eq:invariants-section4} gives
$
F(\hat p)\le U^*.
$
By Proposition~\ref{prop:safe-deletion}, there exists \(I^\star\in\mathcal A\) such that
\[
I^\star\cap\arg\min_{p\in[1,P_{\max}]}F(p)\ne\varnothing .
\]
Hence
\[
L^*
=
\min_{I\in\mathcal A}L_I
\le
L_{I^\star}
\le
\inf_{p\in I^\star}F(p)
=
\inf_{p\in[1,P_{\max}]}F(p).
\]
Together with \(U^*-L^*\le\eta\),
\[
F(\hat p)
\le
U^*
\le
L^*+\eta
\le
\inf_{p\in[1,P_{\max}]}F(p)+\eta .
\]
Since \(F(p)=\log L_m(p,b(p))\),
\[
L_m(\hat p,b(\hat p))
=
e^{F(\hat p)}
\le
e^\eta
\inf_{p\in[1,P_{\max}]}L_m(p,b(p)).
\]
The Monte Carlo certificates are simultaneous with probability at least \(1-\alpha\). The unconditional privacy statement for the random upper scale follows from Corollary~\ref{cor:mc-scale-delta-alpha}.
\end{proof}

\subsection{Arbitrary accuracy with finite partitions}
Theorem~\ref{thm:certified-near-optimality} gives an \(\eta\)-accurate objective value relative to the optimum. 
The following technical assumptions concern the upper and lower scale bounds. Under these assumptions, the adaptive search reaches
\(U^*-L^*\le\eta\) after finitely many refinements on \(E\). Therefore, the optimization error can be arbitrarily small without imposing any additional privacy assumption.
\begin{assumption}[Lower-scale approximability]
\label{ass:lower-scale-approximability}
Let \(J_0\supset J_1\supset J_2\supset\cdots\) be any nested chain generated by repeated bisection in the adaptive search, with \(\bigcap_{n\ge0}J_n=\{p_\infty\}\). For each \(n\), let \(\mathcal C_{J_n}\subset(0,\infty)\) be the finite set of candidate scales \(c\) tested by the refinement procedure to certify \(c\le b(p)\) for all \(p\in J_n\), and hence to update \(\underline b_{J_n}\). For every \(u<b(p_\infty)\), there exists \(N_u\) such that, for all \(n\ge N_u\),
\[
\underline b_{J_n}\ge u
\qquad\text{or}\qquad
\mathcal C_{J_n}\cap\left[u,\frac{u+b(p_\infty)}2\right]\ne\varnothing.
\]
\end{assumption}

\begin{assumption}[Non-degeneracy of geometric scale midpoints]
\label{ass:geometric-midpoint-nondegeneracy}
Let \(I_n\) be any infinite sequence of intervals selected by the adaptive search for geometric scale refinement, and define \(b_{\rm geo}(I_n):=\sqrt{\underline b_{I_n}\overline b_{I_n}}\). Then
\[
b_{\rm geo}(I_n)=b(p_{I_n})
\]
occurs only finitely many times.
\end{assumption}
\begin{theorem}[Finite attainment of arbitrary accuracy]
\label{thm:finite-attainability}
Assume Assumptions~\ref{ass:lower-scale-approximability} and~\ref{ass:geometric-midpoint-nondegeneracy}. Then, on \(E\), for every \(\eta>0\), the adaptive search reaches
\[
U^*-L^*\le\eta
\]
after finitely many refinements. Consequently, the returned \(\hat p\) satisfies
\[
\mathbb P\left(
F(\hat p)\le\inf_{p\in[1,P_{\max}]}F(p)+\eta
\right)\ge1-\alpha.
\]
\end{theorem}

\begin{proof}
Let
\[
J_0\supset J_1\supset J_2\supset\cdots,
\qquad
\bigcap_{n\ge0}J_n=\{p_\infty\},
\]
be any infinite nested chain generated by repeated bisection, and set
\[
b_\infty:=b(p_\infty).
\]
We first show that
\[
U_{J_n}-L_{J_n}\to0.
\]
Since
\[
U_{J_n}-L_{J_n}
=
A(J_n)+B(J_n),
\qquad
0\le A(J_n)\le h_{J_n}D_M(J_n),
\]
and \(D_M(J_n)\) is bounded for all sufficiently large \(n\), while \(|J_n|\to0\), it follows that
\[
A(J_n)\to0.
\]
Let
\[
\underline b_n:=\underline b_{J_n},
\qquad
\overline b_n:=\overline b_{J_n}.
\]
By construction,
\[
\underline b_n\le \inf_{p\in J_n}b(p),
\qquad
\overline b_n\ge b(p_{J_n}).
\]
Thus, by continuity of \(b\),
\[
\limsup_{n\to\infty}\underline b_n\le b_\infty,
\qquad
\liminf_{n\to\infty}\overline b_n\ge b_\infty.
\]
Fix \(u<b_\infty\). By Assumption~\ref{ass:lower-scale-approximability}, for all sufficiently large \(n\), either
\[
\underline b_n\ge u,
\]
or there exists \(c_n\in\mathcal C_{J_n}\) such that
\[
u\le c_n\le\frac{u+b_\infty}{2}<b_\infty.
\]
In the second case, continuity of \(b\) gives
\[
c_n<\inf_{p\in J_n}b(p)
\]
for all sufficiently large \(n\). Hence
\[
\inf_{p\in J_n}H(p,c_n)>0.
\]
Since the inequality is strict, increasing the pointwise sample size eventually enables the empirical Bernstein certificate to verify the corresponding low-scale inequality. The lower scale bound can therefore be updated to at least \(c_n\), and hence to at least \(u\). Consequently,
\[
\liminf_{n\to\infty}\underline b_n\ge u.
\]
Letting \(u\uparrow b_\infty\) gives
\[
\underline b_n\to b_\infty.
\]
Now fix \(\xi>0\). If
\[
B(J_n)
=
m\log\frac{\overline b_n}{\underline b_n}
>
\xi,
\]
then
\[
b_{\rm geo}(J_n)
:=
\sqrt{\underline b_n\,\overline b_n}
=
\underline b_n
\exp\left(\frac{B(J_n)}{2m}\right).
\]
Since \(\underline b_n\to b_\infty\) and \(b(p_{J_n})\to b_\infty\), for all sufficiently large \(n\),
\[
B(J_n)>\xi
\quad\Longrightarrow\quad
b_{\rm geo}(J_n)>b(p_{J_n}).
\]
Therefore,
\[
H\bigl(p_{J_n},b_{\rm geo}(J_n)\bigr)<0.
\]
Again, the inequality is strict, so increasing the pointwise sample size eventually allows the empirical Bernstein certificate to verify the corresponding upper-scale inequality. The geometric update is then valid and gives
\[
B(J_n)\longmapsto \frac12 B(J_n).
\]
Since this argument applies for every \(\xi>0\),
\[
B(J_n)\to0.
\]
Therefore,
\[
U_{J_n}-L_{J_n}
=
A(J_n)+B(J_n)
\to0.
\]
The next step is to connect this local convergence with the global stopping rule. Suppose, to the contrary, that for some \(\eta>0\), the adaptive search never reaches
\[
U^*-L^*\le\eta.
\]
By Assumption~\ref{ass:geometric-midpoint-nondegeneracy}, the search cannot remain at the same interval indefinitely. Hence, nontermination would imply infinitely many interval bisections.

Because each bisection creates two child intervals, infinitely many bisections give an infinite nested chain
\[
J_0\supset J_1\supset J_2\supset\cdots
\]
of intervals selected for refinement. For this chain, the result above gives
\[
U_{J_n}-L_{J_n}\to0.
\]
Thus, for all sufficiently large \(n\),
\[
U_{J_n}-L_{J_n}\le\eta.
\]
When \(J_n\) is selected for refinement, it has the smallest lower bound among the active intervals, so
\[
L_{J_n}=L^*.
\]
Moreover, since \(U^*\) is the smallest available upper bound,
\[
U^*\le U_{J_n}.
\]
It follows that
\[
U^*-L^*
\le
U_{J_n}-L_{J_n}
\le\eta,
\]
contradicting the assumption that the stopping condition is never reached. Therefore, on \(E\), the adaptive search reaches
\[
U^*-L^*\le\eta
\]
after finitely many refinements.

Since \(\mathbb P(E)\ge1-\alpha\), the probability statement follows from Theorem~\ref{thm:certified-near-optimality}.
\end{proof}
\begin{remark}[Non-scale homogeneous objectives]
\label{rem:non-scale-homogeneous-search}
The interval-wise argument above does not require the objective to be
scale-homogeneous. In particular, let \(J(p)\) be a task-specific objective
to be maximized. If, for every active interval $I$,
\[
\underline J_I\le J(p)\le\overline J_I,\qquad p\in I,
\]
then any interval satisfying
\[
\overline J_I<J(\widehat p)
\]
can be discarded. Moreover, if
\[
\max_I \overline J_I-J(\widehat p)\le\eta,
\]
then
\[
J(\widehat p)
\ge
\sup_{p\in[1,P_{\max}]}J(p)-\eta.
\]
Section \ref{sec:task-specific-example} applies this conclusion to the one-dimensional example.
\end{remark}

\section{Properties of the Optimal Shape}
\label{sec:optimal-shape-properties}
\subsection{Sensitivity scaling  and active coordinates}
\label{subsec:sensitivity-normalization-active}

Fix \(\varepsilon\ge0\), \(\delta\in[0,1)\), and a scale-homogeneous objective
\[
\mathcal L(p,b)=b^r\nu(p),
\qquad r>0,
\qquad \nu(p)>0 .
\]
For \(\Delta\in[0,\infty)^d\), \(\Delta\ne0\), write
\[
b^{(d)}(p;\Delta)
:=
\inf\left\{
b>0:\delta_{p,b}^{(d)}(\Delta)\le\delta
\right\},
\]
where \(\delta_{p,b}^{(d)}(\Delta)\) is the hockey-stick divergence for the
\(d\)-dimensional generalized Gaussian mechanism with sensitivity vector
\(\Delta\). For later comparison with the classical choices \(p\in\{1,2,\infty\}\), define
the relative improvement of a shape \(\hat p\) by
\[
1-
\frac{
\mathcal L\bigl(\hat p,b^{(d)}(\hat p;\Delta)\bigr)
}{
\min_{p\in\{1,2,\infty\}}
\mathcal L\bigl(p,b^{(d)}(p;\Delta)\bigr)
}.
\]
This quantity is positive when the shape \(\hat p\) is better than the Laplace, Gaussian, and uniform mechanisms under the same privacy constraints.
\begin{proposition}[Sensitivity scaling and active-coordinate reduction]
\label{prop:sensitivity-scaling-active-coordinates}
Let \(\Delta\in[0,\infty)^d\), \(\Delta\ne0\), and 
\[
\mathcal L(p,b)=b^r\nu(p),
\qquad r>0,\qquad \nu(p)>0 .
\]

\begin{enumerate}[label=(\roman*), leftmargin=2em]
\item For every \(c>0\) and every \(p\in[1,\infty]\),
\[
b^{(d)}(p;c\Delta)=c\,b^{(d)}(p;\Delta).
\]
Hence
\[
\mathcal L\bigl(p,b^{(d)}(p;c\Delta)\bigr)
=
c^r\mathcal L\bigl(p,b^{(d)}(p;\Delta)\bigr).
\]
Consequently, for every \(I\subseteq[1,\infty]\),
\[
\operatorname*{arg\,min}_{p\in I}
\mathcal L\bigl(p,b^{(d)}(p;c\Delta)\bigr)
=
\operatorname*{arg\,min}_{p\in I}
\mathcal L\bigl(p,b^{(d)}(p;\Delta)\bigr).
\]

\item 
Let \(A:=\{i:\Delta_i>0\}\), \(d':=|A|\) be its cardinality
and \(\Delta_A\in(0,\infty)^{d'}\) be the vector of nonzero coordinates.
Then, for every \(p\in[1,\infty]\) and \(b>0\),
\[
\delta_{p,b}^{(d)}(\Delta)
=
\delta_{p,b}^{(d')}(\Delta_A).
\]
Therefore,
\[
b^{(d)}(p;\Delta)=b^{(d')}(p;\Delta_A),
\]
and, for every \(I\subseteq[1,\infty]\),
\[
\operatorname*{arg\,min}_{p\in I}
\mathcal L\bigl(p,b^{(d)}(p;\Delta)\bigr)
=
\operatorname*{arg\,min}_{p\in I}
\mathcal L\bigl(p,b^{(d')}(p;\Delta_A)\bigr).
\]
\end{enumerate}
\end{proposition}

\begin{proof}
See Appendix~\ref{app:sensitivity-scaling-active-proof}.
\end{proof}
\begin{remark}[Rescaling and comparison of sensitivity vectors]
\label{rem:rescaling-comparison-sensitivity-vectors}
Proposition ~\ref{prop:sensitivity-scaling-active-coordinates} shows that multiplying the sensitivity vector by a positive constant and removing coordinates with zero sensitivity do not change the optimal shape. For two nonzero vectors \(\Delta\) and \(\widetilde\Delta\), choosing \(c=\|\Delta\|_q/\|\widetilde\Delta\|_q\), \(1\le q\le\infty\), gives \(\|c\widetilde\Delta\|_q=\|\Delta\|_q\), so the corresponding scales and objective values are compared at fixed \(\ell_q\)-sensitivity. In sensitivity vectors with the same \(\ell_q\) norm, the differences in privacy-feasible scale, objective value, and selected shape stem from the coordinate structure of the sensitivity vectors. According to Proposition ~\ref{prop:gaussian-scalar-coordinate-profile}, only when \(p=2\) are \(\delta_{p,b}^{(d)}(\Delta)\) and \(b^{(d)}(p;\Delta)\) completely determined by the Euclidean norm \(\|\Delta\|_2\).

A \(d'\)-dimensional problem can be embedded into a \(d\)-dimensional problem by adding \(d-d'\) zero coordinates. If only \(d'\) coordinates in \(\Delta\in[0,\infty)^d\) are non-zero, then Proposition ~\ref{prop:sensitivity-scaling-active-coordinates} indicates that retaining only \(d'\) non-zero coordinates does not change the privacy-feasible scale and the selected shape. In particular, a sensitivity vector of the form \((\Delta_1,0,\ldots,0)\), with \(\Delta_1>0\), gives exactly the corresponding one-dimensional problem.

The relative improvement defined above is unchanged under the transformation
\(\Delta\mapsto c\Delta\) and after removing zero-sensitivity coordinates.
\end{remark}
\subsection{Limits as the privacy parameters approach zero}
\label{subsec:high-privacy-limits-optimal-shape}
The following results are stated for the \(m\)-th absolute moment criterion \(L_m(p,b)=M_m(p)b^m\), where \(m>0\), \(M_m(p)=\Gamma((m+1)/p)/\Gamma(1/p)\) for \(1\le p<\infty\), and \(M_m(\infty)=1/(m+1)\).
\begin{theorem}[Convergence to the Laplace shape as \(\delta\downarrow0\)]
\label{thm:small-delta-p-to-one}
Fix \(m>0\), \(\varepsilon>0\), and \(\Delta\ne0\). Let \(b_\delta(p;\Delta)\) be the minimum privacy-feasible scale under \((\varepsilon,\delta)\)-DP. Assume that, for every \(\delta>0\), a minimizer \(p^\star(\delta)\in\operatorname*{arg\,min}_{p\in[1,\infty]} L_m\bigl(p,b_\delta(p;\Delta)\bigr)\) exists. Then \(p^\star(\delta)\to1\) as \(\delta\downarrow0\). Equivalently,
\[
\forall \rho>0,\ \exists \delta_\rho>0\ \text{such that}\quad 0<\delta<\delta_\rho\Longrightarrow p^\star(\delta)<1+\rho .
\]
\end{theorem}

\begin{proof}
See Appendix~\ref{app:small-delta-proof}.
\end{proof}
Thus, for a fixed \(\varepsilon>0\), when \(\delta\downarrow0\), the optimal shape converges to the Laplace case \(p=1\).

\begin{theorem}[One-dimensional small-\(\varepsilon\) limit]
\label{thm:small-eps-p-to-infty}
Fix \(m>0\), \(d=1\), \(\Delta>0\), and
$
\delta\in\left(0,\frac{m}{m+1}\right].
$
Let \(b_\varepsilon(p;\Delta)\) be the minimum privacy-feasible scale under \((\varepsilon,\delta)\)-DP. Assume that, for every \(\varepsilon>0\), a minimizer
\[
p^\star(\varepsilon)
\in
\operatorname*{arg\,min}_{p\in[1,\infty]}
L_m\bigl(p,b_\varepsilon(p;\Delta)\bigr)
\]
exists. Then
\[
p^\star(\varepsilon)\to\infty
\qquad
(\varepsilon\downarrow0).
\]
Equivalently, for every finite \(P>1\), there exists \(\varepsilon_P>0\) such that
\[
0<\varepsilon<\varepsilon_P
\quad\Longrightarrow\quad
p^\star(\varepsilon)>P .
\]
\end{theorem}

\begin{proof}
See Appendix~\ref{app:small-eps-proof}.
\end{proof}
Theorem~\ref{thm:small-eps-p-to-infty} is specific to one dimension. Appendix~\ref{app:multidim-uniform-counterexample} gives a two-dimensional counterexample.
\subsection{Joint high-privacy asymptotics}
\label{subsec:joint-high-privacy}
Let \(\varepsilon_n\downarrow0\) and \(\delta_n\downarrow0\). For fixed \(\Delta\), let \(b_{\varepsilon,\delta}(p;\Delta)\) be the minimum privacy-feasible scale under \((\varepsilon,\delta)\)-DP. We consider the criterion \(L_m\bigl(p,b_{\varepsilon,\delta}(p;\Delta)\bigr)\), where \(L_m(p,b)=M_m(p)b^m\).
\begin{theorem}[Joint high-privacy asymptotics]
\label{thm:joint-high-privacy}
Let \(\varepsilon_n\downarrow0\) and \(\delta_n\downarrow0\).

\begin{enumerate}[label=(\roman*), leftmargin=2em]
\item Suppose \(\delta_n/\varepsilon_n\to0\). For any \(d\ge1\) and any \(\Delta\ne0\), if
\[
p^\star(\varepsilon_n,\delta_n)
\in
\operatorname*{arg\,min}_{p\in[1,\infty]}
L_m\bigl(p,b_{\varepsilon_n,\delta_n}^{(d)}(p;\Delta)\bigr),
\]
then \(p^\star(\varepsilon_n,\delta_n)\to1\).
\item Suppose \(d=1\), \(\Delta>0\), and \(\delta_n/\varepsilon_n\to\lambda\in(0,\infty)\). If
\[
p^\star(\varepsilon_n,\delta_n)
\in
\operatorname*{arg\,min}_{p\in[1,\infty]}
L_m\bigl(p,b_{\varepsilon_n,\delta_n}^{(1)}(p;\Delta)\bigr),
\]
then any subsequential limit of \(p^\star(\varepsilon_n,\delta_n)\) belongs to \(\operatorname*{arg\,min}_{p\in[1,\infty]}K_{m,\lambda}(p)\). Here
\[
K_{m,\lambda}(1)=\frac{\Gamma(m+1)}{(1+2\lambda)^m},
\qquad
K_{m,\lambda}(\infty)=\frac{1}{(m+1)(2\lambda)^m},
\]
and, for \(1<p<\infty\),
\[
K_{m,\lambda}(p)=M_m(p)\rho_{p,\lambda}^m,
\]
where \(\rho_{p,\lambda}\) is the unique solution of \(A_p(\rho)=\lambda\rho\), with
\[
A_p(\rho)
=
\int_{(\rho/p)^{1/(p-1)}}^\infty
\left(px^{p-1}-\rho\right)
\frac{p}{2\Gamma(1/p)}e^{-x^p}\,dx .
\]
If \(K_{m,\lambda}\) has a unique minimizer \(p_\lambda\), then \(p^\star(\varepsilon_n,\delta_n)\to p_\lambda\).

\item Suppose \(d=1\), \(\Delta>0\), and \(\delta_n/\varepsilon_n\to\infty\). If
\[
p^\star(\varepsilon_n,\delta_n)
\in
\operatorname*{arg\,min}_{p\in[1,\infty]}
L_m\bigl(p,b_{\varepsilon_n,\delta_n}^{(1)}(p;\Delta)\bigr),
\]
then \(p^\star(\varepsilon_n,\delta_n)\to\infty\).
\end{enumerate}
\end{theorem}

\begin{proof}
See Appendix~\ref{app:joint-high-privacy-proof}.
\end{proof}
Theorem~\ref{thm:joint-high-privacy} describes the joint high-privacy limit in terms of the ratio \(\delta/\varepsilon\). If \(\delta/\varepsilon\to0\), the optimal shape converges to \(p=1\) in any dimension. In the one-dimensional case, if \(\delta/\varepsilon\to\lambda\in(0,\infty)\), then the limiting behaviour is determined by \(K_{m,\lambda}\); if \(\delta/\varepsilon\to\infty\), then the optimal shape converges to \(p=\infty\). The conclusion in part~(iii) is specific to one dimension; see Appendix~\ref{app:multidim-uniform-counterexample}.
\section{Computational Experiments}
\label{sec:Computational Experiments}
Sections~\ref{subsec:numerics-1d}--\ref{subsec:five-dimensional-certified-search} use the second absolute moment criterion (variance minimization), while Section~\ref{sec:task-specific-example} considers a one-dimensional threshold-decision example with a task-specific criterion. Three numerical methods are used for the variance-based experiments. The first is the certified interval-wise search method proposed in Section ~\ref{sec:optimization}. The second is a full-grid search for \(p=1, 1.01, 1.02, \ldots\). At each grid point, the fixed-\(p\) algorithm described in Section~\ref{subsec:fixed-shape-upper-approximation} is applied, and the grid point with the smallest variance is selected. To reduce the computation, the calculation at a grid point is terminated when the lower bound of the objective function value at that point exceeds the best upper bound found so far. The third method is a fast two-stage grid search algorithm used to generate heatmaps that illustrate how the selected shape and relative improvement vary across different privacy levels. This method first performs a search using a coarse grid with a step size of 0.1, and then conducts a finer-grid search using a step size of 0.01 around the optimal point identified in the first stage.

In Section~6.1, all results are obtained using the certified interval-wise search. At each evaluated value of \(p\), the privacy-feasible scale is computed by solving the system in \eqref{eq:1d-tight-system-paper}, so no grid search or Monte Carlo approximation is used.  In the multidimensional case, the specific examples use the full-grid search over \(p=1,1.01,1.02,\ldots,11\), followed by larger sample calculations for the selected shapes. The heatmaps instead use the fast two-stage grid search and smaller samples because they cover \(5000\) privacy-parameter settings. The heatmaps use smaller sample sizes, and the two-stage search may miss the best point on the full grid. More refined calculations may therefore give larger improvements.

Throughout Sections~\ref{subsec:numerics-1d}--\ref{subsec:numerics-improvement-exceedance}, the selected shape is compared with the three benchmark mechanisms: the Laplace, Gaussian, and uniform mechanisms, corresponding to \(p=1,2,\infty\), respectively. The objective function is \(L_m\!\left(p,\widehat b(p)\right)\). Define
\[
R_m
=
100\%
\frac{
L_m\bigl(\widehat p,\widehat b(\widehat p)\bigr)
}{
\min\left\{
L_m\bigl(1,\widehat b(1)\bigr),
L_m\bigl(2,\widehat b(2)\bigr),
L_m\bigl(\infty,\widehat b(\infty)\bigr)
\right\}
},
\qquad
\mathrm{Improvement}_m=100\%-R_m.
\]
Thus, \(\operatorname{Improvement}_m\) is the percentage reduction in the \(m\)-th absolute moment relative to the best of the three benchmark mechanisms. 
\subsection{One-dimensional scale determination and shape selection}
\label{subsec:numerics-1d}
Consider \(d=1\), \(\Delta=\Delta_1\), and \(m=2\). The search is performed on \(p\in[1,1001]\cup\{\infty\}\). With \(F(p)=\log L_2(p,b(p))\), the stopping rule gives
\[
L_2(\hat p,b(\hat p))
\le
e^\eta
\inf_{p\in[1,1001]\cup\{\infty\}}
L_2(p,b(p)).
\]
The interval search uses \(\eta=10^{-4}\), so the selected variance is at most \(100(e^{10^{-4}}-1)\%\approx0.01\%\) larger than the minimum over \(p\in[1,1001]\cup\{\infty\}\). Table~\ref{tab:numerics-1d-fixed-p-comparison} compares the selected shape with the choices \(p=1,2,\infty\) for a representative case.
\begin{table}[!htbp]
\centering
\small
\setlength{\tabcolsep}{6pt}
\renewcommand{\arraystretch}{1.08}
\caption{One-dimensional shape comparison for \((\varepsilon,\delta,\Delta_1)=(1,0.01,3)\).}
\label{tab:numerics-1d-fixed-p-comparison}
\begin{tabular}{lccc}
\toprule
Shape & \(p\) & \(b\) & \(L_2(p,b)\) \\
\midrule
Selected & \(1.0649\) & \(3.19\) & \(16.71\) \\
Laplace & \(1\) & \(2.94\) & \(17.30\) \\
Gaussian & \(2\) & \(7.97\) & \(31.74\) \\
Uniform & \(\infty\) & \(150.00\) & \(7500.00\) \\
\bottomrule
\end{tabular}
\end{table}
For \(p=\infty\), the scale is independent of \(\varepsilon\); hence the reported uniform noise already satisfies \((0,0.01)\)-DP. In this representative case, among the three fixed shapes \(p=1,2,\infty\), the Laplace mechanism \(p=1\) gives the smallest value of \(L_2\). Relative to this choice, the selected shape \(\hat p=1.0649\) reduces the second moment by
\[
100\left(1-\frac{L_2(\hat p,b(\hat p))}{L_2(1,b(1))}\right)\%
=
100\left(1-\frac{16.71}{17.30}\right)\%
=
3.40\%.
\]
For the twelve cases in Table~\ref{tab:numerics-1d-all-results}, set
\[
\varepsilon\in\{0.25,0.5,1,2\},
\qquad
\delta\in\{0.1,0.01,0.001\},
\qquad
\Delta_1=3.
\]
\begin{table}[!htbp]
\centering
\small
\setlength{\tabcolsep}{5pt}
\renewcommand{\arraystretch}{1.08}
\caption{One-dimensional shape comparison across twelve \((\varepsilon,\delta)\) pairs, with \(\Delta_1=3\) and \(\eta=10^{-4}\).}
\label{tab:numerics-1d-all-results}
\begin{tabular}{cccccc}
\toprule
\(\varepsilon\) & \(\delta\) & \(\hat p\) & \(L_2(\hat p)\) & Best benchmark \(L_2\) & Improvement \\
\midrule
\(0.25\) & \(0.1\) & \(2.7790\) & \(38.10\) & \(40.19\) & \(5.21\%\) \\
\(0.25\) & \(0.01\) & \(1.3038\) & \(200.05\) & \(241.92\) & \(17.31\%\) \\
\(0.25\) & \(0.001\) & \(1.0384\) & \(275.92\) & \(283.44\) & \(2.65\%\) \\
\(0.5\) & \(0.1\) & \(1.9852\) & \(21.80\) & \(21.80\) & \(0.00\%\) \\
\(0.5\) & \(0.01\) & \(1.1520\) & \(60.32\) & \(66.54\) & \(9.35\%\) \\
\(0.5\) & \(0.001\) & \(1.0174\) & \(70.63\) & \(71.43\) & \(1.12\%\) \\
\(1\) & \(0.1\) & \(1.4884\) & \(9.86\) & \(10.61\) & \(7.09\%\) \\
\(1\) & \(0.01\) & \(1.0649\) & \(16.71\) & \(17.30\) & \(3.40\%\) \\
\(1\) & \(0.001\) & \(1.0069\) & \(17.86\) & \(17.93\) & \(0.37\%\) \\
\(2\) & \(0.1\) & \(1.1847\) & \(3.55\) & \(3.68\) & \(3.62\%\) \\
\(2\) & \(0.01\) & \(1.0209\) & \(4.39\) & \(4.41\) & \(0.52\%\) \\
\(2\) & \(0.001\) & \(1.0021\) & \(4.49\) & \(4.49\) & \(0.05\%\) \\
\bottomrule
\end{tabular}
\end{table}

\FloatBarrier

The selected shape moves closer to \(p=1\) as \(\delta\) decreases or \(\varepsilon\) increases. The largest improvement in Table~\ref{tab:numerics-1d-all-results} occurs at \((\varepsilon,\delta)=(0.25,0.01)\), where the second moment is reduced by \(17.31\%\) relative to the best fixed choice.

For visualization, the certified interval-wise search is applied with
\[
d=1,\qquad \Delta_1=1,\qquad m=2,
\]
over
\[
\varepsilon\in\{0.1,0.2,\ldots,5.0\},\qquad \delta\in\{0.001,0.002,\ldots,0.1\}.
\]
By Proposition ~\ref{prop:sensitivity-scaling-active-coordinates}, choosing \(\Delta_1=1\) does not affect the selected shape and the relative improvement. Different values of \(\Delta_1\) only change the privacy-feasible scale.

\begin{figure}[!htbp]
\centering

\begin{minipage}{0.49\textwidth}
\centering
\includegraphics[
  width=\linewidth,
  height=0.30\textheight,
  keepaspectratio
]{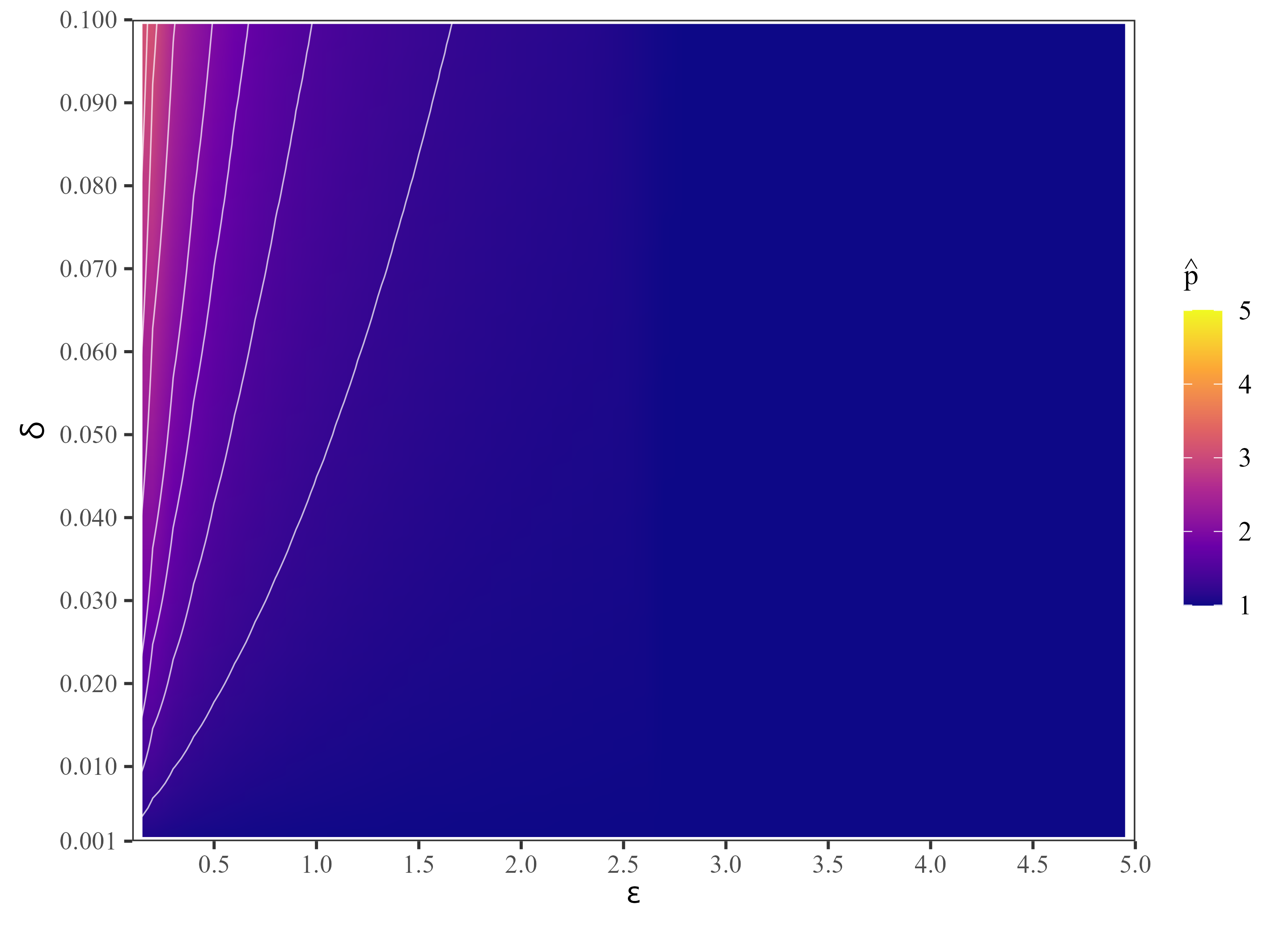}

\vspace{0.1em}
\((a)\) Selected shape \(\hat p\)
\end{minipage}
\hfill
\begin{minipage}{0.49\textwidth}
\centering
\includegraphics[
  width=\linewidth,
  height=0.30\textheight,
  keepaspectratio
]{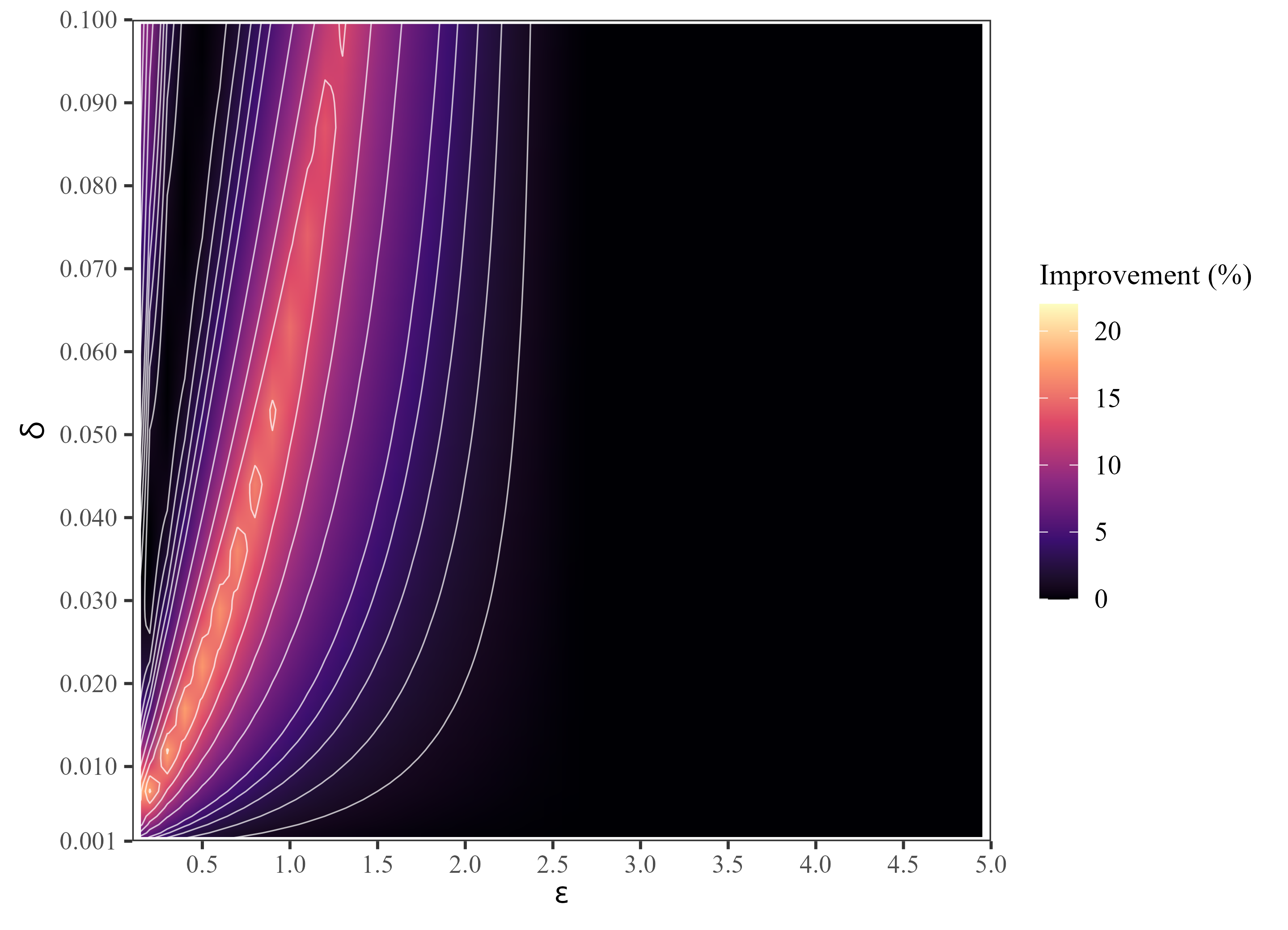}

\vspace{0.1em}
\((b)\) Improvement
\end{minipage}

\caption{One-dimensional certified interval-wise search results for \(5000\) privacy-parameter pairs. The left panel shows the selected shape \(\hat p\), and the right panel shows the improvement relative to the best fixed choice among \(p=1,2,\infty\).}
\label{fig:numerics-1d-dense-grid}
\end{figure}

\FloatBarrier
Figure ~\ref{fig:numerics-1d-dense-grid} shows that significant improvements are mainly concentrated in a relatively narrow region where \(\varepsilon\) is small and \(\delta\) is relatively large. Out of \(5000\) grid points, \(2337\) selected \(p=1\). At \(1036\) grid points, the improvement exceeded \(5\%\); at \(450\) grid points, the improvement exceeded \(10\%\). The largest shape and the greatest improvement both occurred at \((\varepsilon,\delta)=(0.1,0.1)\), where \(\hat p=4.6316\) and the variance was reduced by \(20.53\%\).
\subsection{Multidimensional setup and sensitivity vectors}
\label{subsec:numerics-imbalance-design}
Throughout the multidimensional calculations, \(\delta\) denotes the privacy failure probability used in the scale computation, and \(\alpha\) is the Monte Carlo failure budget. We set \(\alpha\) to \(1\%\) of the overall privacy failure probability \(\delta+\alpha\), so that
\[
\alpha=0.01(\delta+\alpha),
\qquad
\delta=0.99(\delta+\alpha).
\]
The computed mechanism satisfies \((\varepsilon,\delta)\)-DP with probability at least \(1-\alpha\). By Corollary~\ref{cor:mc-scale-delta-alpha}, the resulting mechanism unconditionally satisfies \((\varepsilon,\delta+\alpha)\)-differential privacy.

To compare the degree of imbalance among sensitivity vectors, we use the normalized Herfindahl--Hirschman index of  \cite{cracau2016normalized}. For
\[
w_i=\frac{\Delta_i}{\|\Delta\|_1},
\qquad
\operatorname{HHI}(\Delta)
=
\sum_{i=1}^d w_i^2
=
\frac{\|\Delta\|_2^2}{\|\Delta\|_1^2},
\]
the normalized index is
\[
\operatorname{HHI}_{\mathrm{norm}}(\Delta)
=
\frac{\operatorname{HHI}(\Delta)-1/d}{1-1/d}.
\]
When all coordinates of \(\Delta\) are equal, this index equals \(0\); such a sensitivity vector is called balanced. As the sensitivity becomes increasingly concentrated on a single coordinate, the index approaches \(1\).

For the multidimensional numerical study, we consider the following sensitivity vectors, grouped by their normalized HHI values:
\[
\begin{array}{c|c|c}
\operatorname{HHI}_{\mathrm{norm}} & d=2 & d=5\\
\hline
0 & (1,1) & (1,1,1,1,1)\\
1/9 & (1,2) &
\begin{array}{c}
(3.5,1,1,1,1)\\
(12,9,5,3,1)
\end{array}\\
16/25 & (1,9) & (21,1,1,1,1)
\end{array}
\]
These two five-dimensional vectors, both with a normalized HHI of \(1/9\), have different coordinate distributions, allowing us to compare whether vectors with the same imbalance index produce similar numerical results.

This design distinguishes between two types of comparisons: increasing the degree of imbalance within a fixed dimension, and changing the dimension while keeping the normalized HHI constant. This index is used solely to select comparable examples; it does not imply that the chosen shapes or improvements are determined solely by the HHI.

All heatmaps below use the same \((\varepsilon,\delta)\) grid, color scale, color bar values, and contour levels. Therefore, in each subfigure, the same colors and contour lines represent the same numerical values.
\subsection{Two-dimensional two-stage grid-search heatmaps}
\label{subsec:numerics-2d-comparisons}

Next consider
\[
d=2,\qquad
\Delta=(1,1),\ (1,2),\ \text{and}\ (1,9),
\qquad
m=2.
\]
The two-stage grid search is performed over
\[
\varepsilon\in\{0.1,0.2,\ldots,5.0\},
\qquad
\delta\in\{0.001,0.002,\ldots,0.1\}.
\]
This gives \(50\times100=5000\) privacy-parameter pairs.
\begin{figure}[!htbp]
\centering
\begin{minipage}{0.48\textwidth}
\centering
\includegraphics[width=\linewidth,height=0.24\textheight,keepaspectratio]{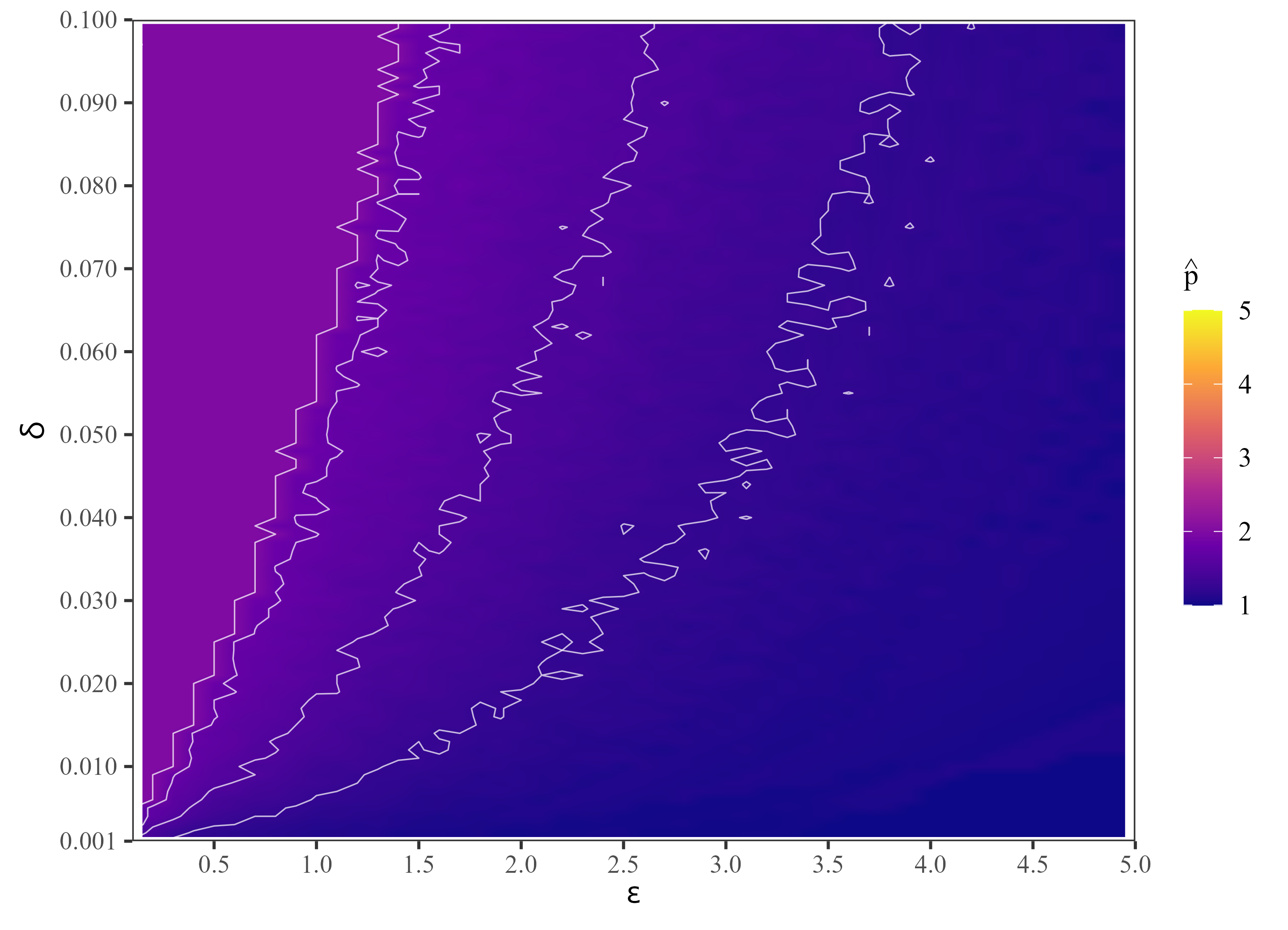}

\smallskip
\((a)\) \(\Delta=(1,1)\): selected shape
\end{minipage}
\hfill
\begin{minipage}{0.48\textwidth}
\centering
\includegraphics[width=\linewidth,height=0.24\textheight,keepaspectratio]{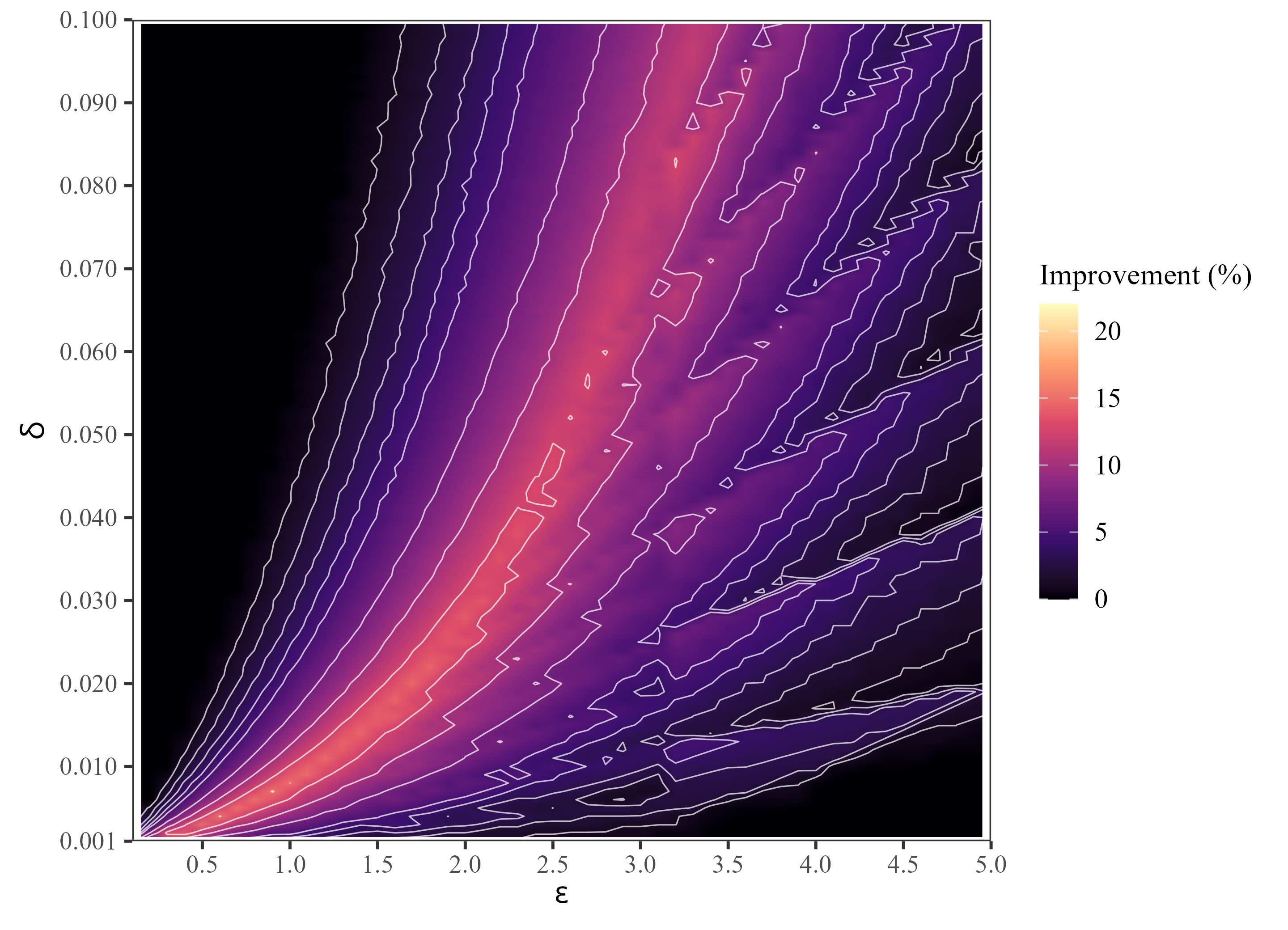}

\smallskip
\((b)\) \(\Delta=(1,1)\): improvement
\end{minipage}
\end{figure}

\begin{figure}[!htbp]
\centering
\begin{minipage}{0.48\textwidth}
\centering
\includegraphics[width=\linewidth,height=0.24\textheight,keepaspectratio]{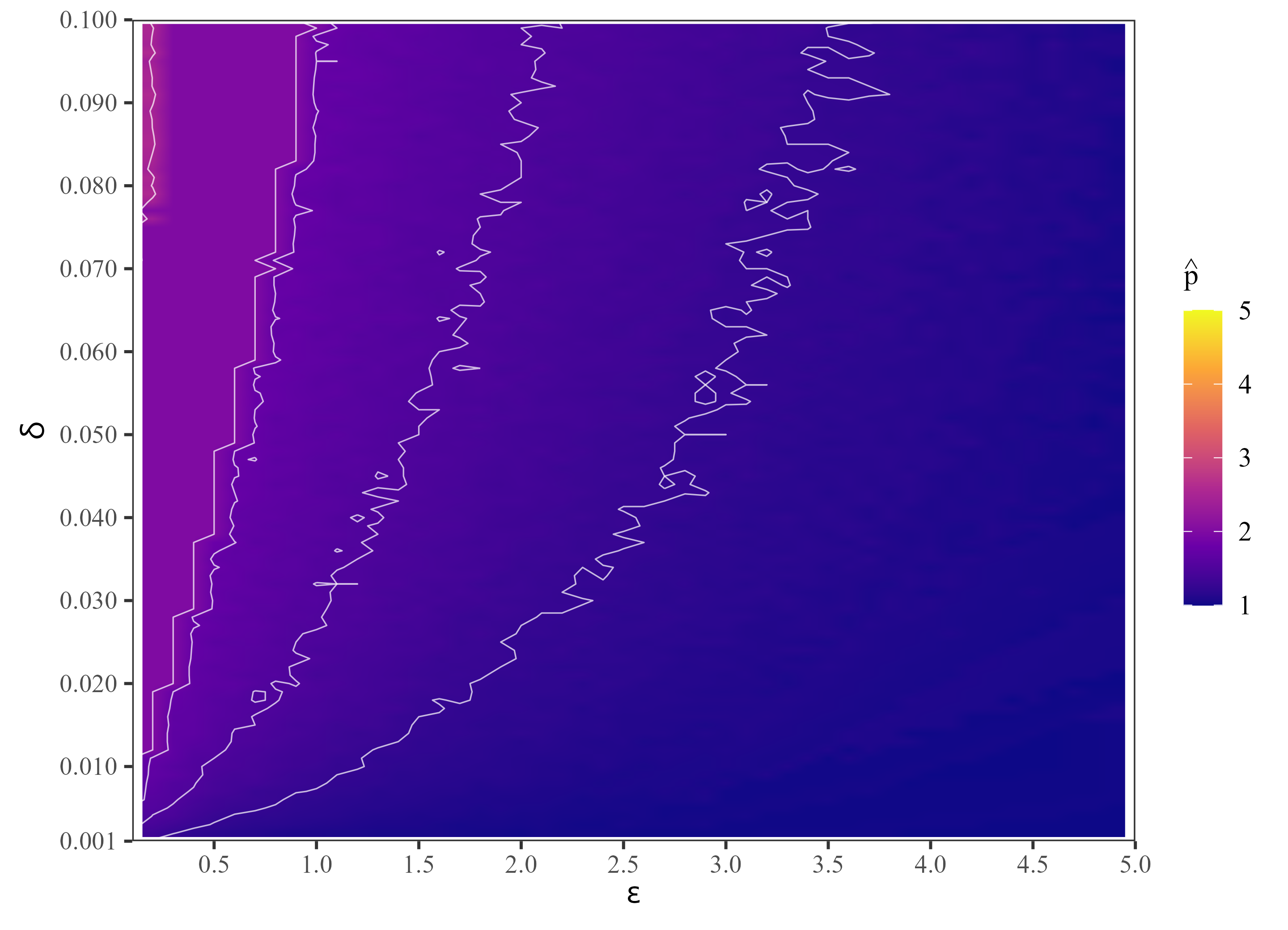}

\smallskip
\((c)\) \(\Delta=(1,2)\): selected shape
\end{minipage}
\hfill
\begin{minipage}{0.48\textwidth}
\centering
\includegraphics[width=\linewidth,height=0.24\textheight,keepaspectratio]{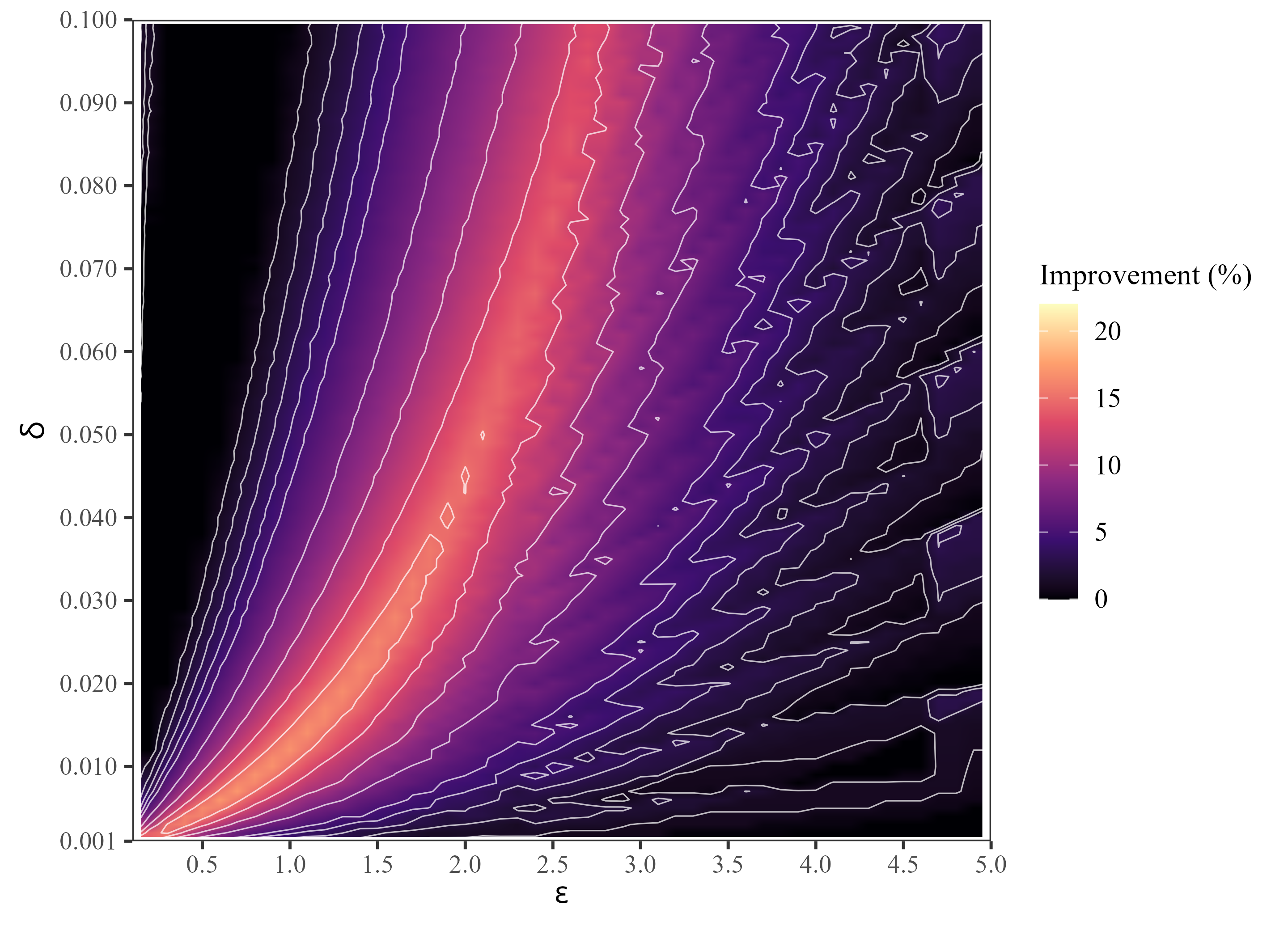}

\smallskip
\((d)\) \(\Delta=(1,2)\): improvement
\end{minipage}
\end{figure}

\begin{figure}[!htbp]
\centering
\begin{minipage}{0.48\textwidth}
\centering
\includegraphics[width=\linewidth,height=0.24\textheight,keepaspectratio]{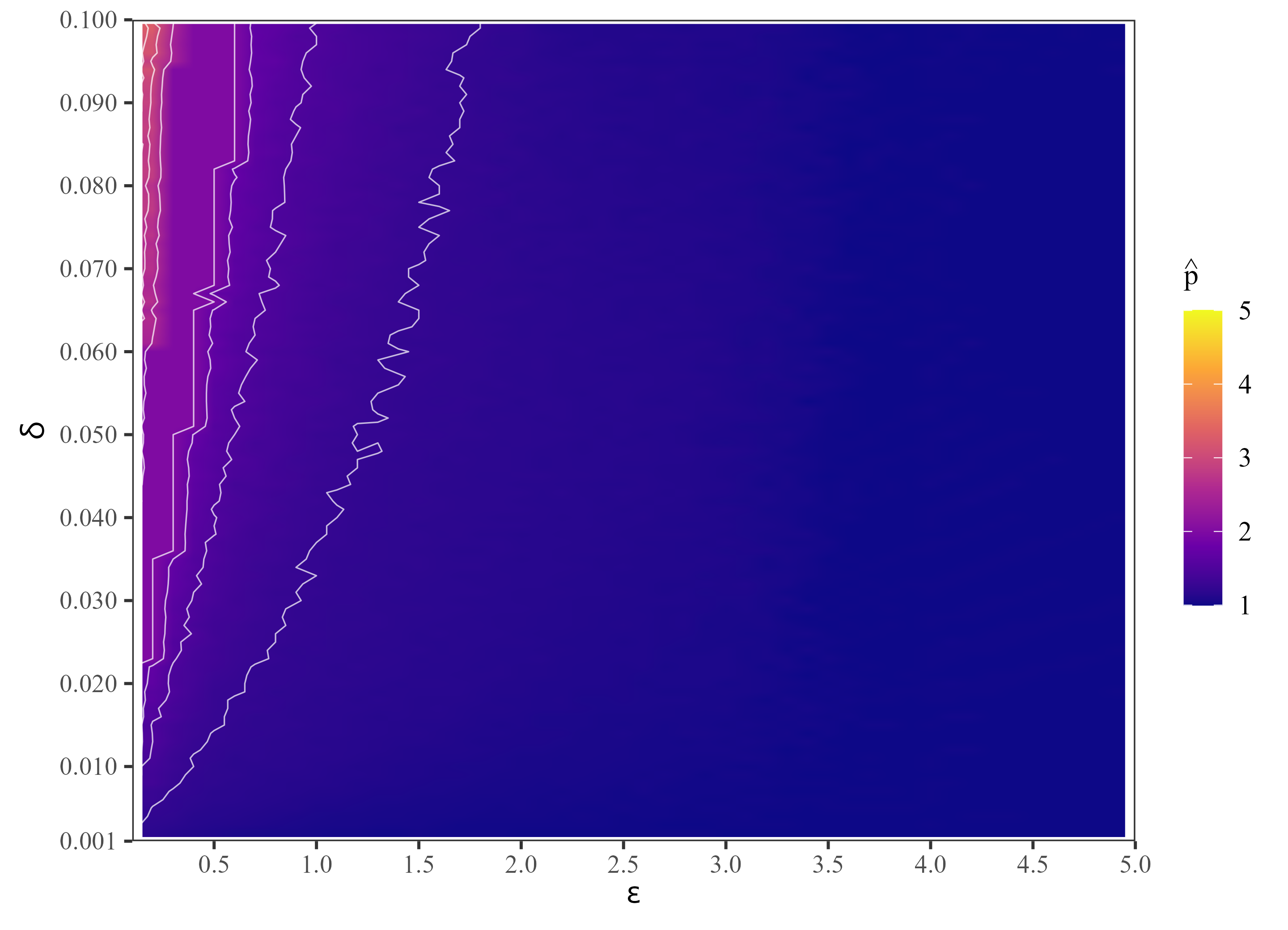}

\smallskip
\((e)\) \(\Delta=(1,9)\): selected shape
\end{minipage}
\hfill
\begin{minipage}{0.48\textwidth}
\centering
\includegraphics[width=\linewidth,height=0.24\textheight,keepaspectratio]{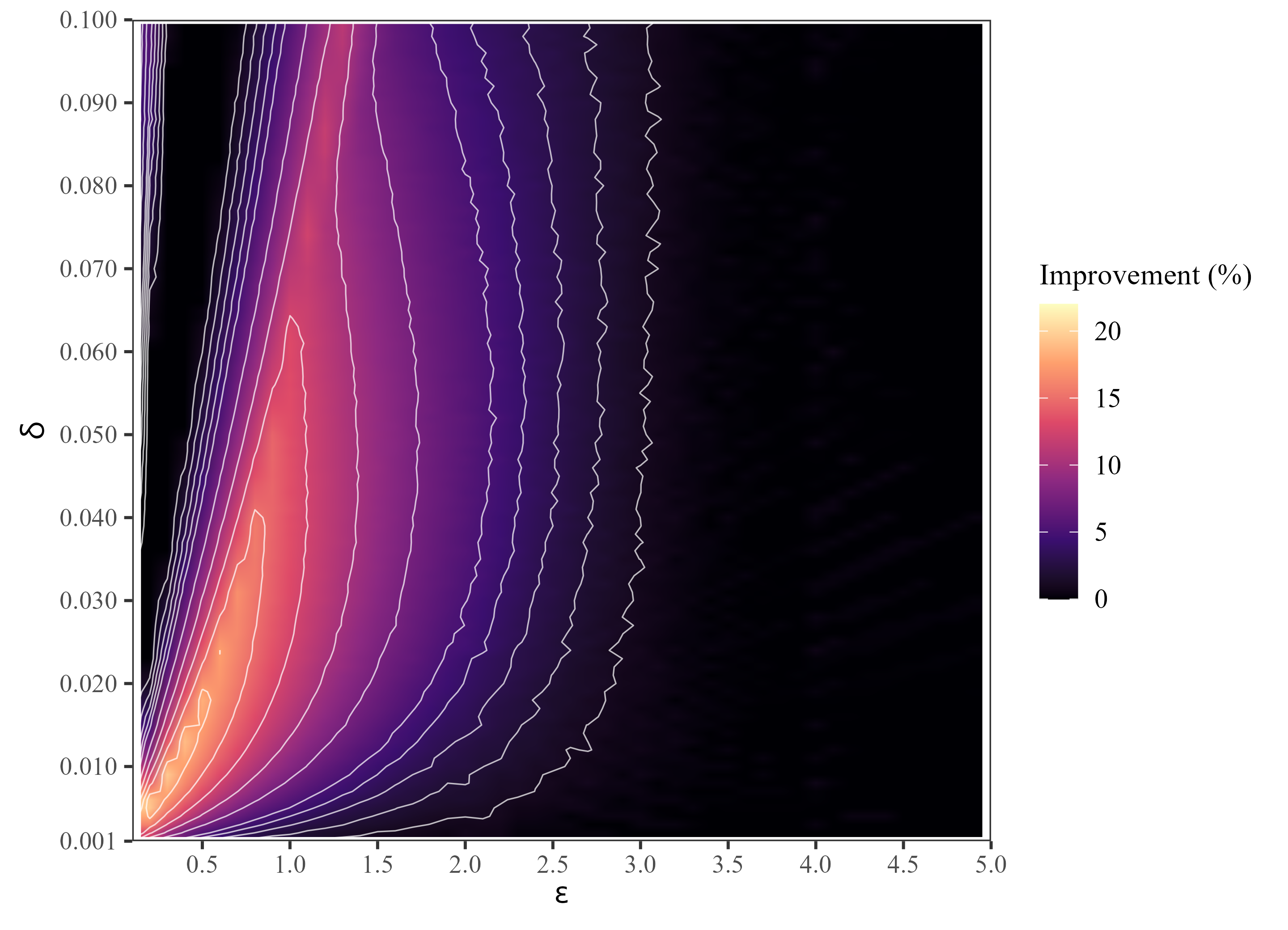}

\smallskip
\((f)\) \(\Delta=(1,9)\): improvement
\end{minipage}

\caption{Two-stage grid-search results for \(d=2\) over \(5000\) \((\varepsilon,\delta)\) pairs. The rows correspond to \(\Delta=(1,1)\), \((1,2)\), and \((1,9)\). The left column shows the selected shape \(\hat p\), and the right column shows the improvement.}
\label{fig:numerics-2d-grid-search}
\end{figure}
\FloatBarrier
In two-dimensional queries, the selected shape and the degree of improvement vary significantly as the sensitivity vector changes. For the balanced sensitivity vector \(\Delta=(1,1)\), the selected shape is close to \(p=2\) over a relatively large range of the grid. The vector \(\Delta=(1,2)\) produces the largest region of significant improvement, while the highly unbalanced sensitivity vector \(\Delta=(1,9)\) more frequently selects \(p=1\). Among the \(5000\) grid points, \(824\), \(458\), and \(219\) grid points selected \(p=2\) for \(\Delta=(1,1)\), \(\Delta=(1,2)\), and \(\Delta=(1,9)\), respectively. The corresponding numbers of grid points selecting \(p=1\) were \(129\), \(82\), and \(1181\), respectively. At \(2011\), \(2143\), and \(1400\) grid points, the improvement was at least \(5\%\); at \(592\), \(934\), and \(541\) grid points, the improvement was at least \(10\%\). The maximum improvement rates are \(15.13\%\), \(17.14\%\), and \(19.52\%\), respectively. Therefore, \(\Delta=(1,9)\) yields the greatest improvement, while \(\Delta=(1,2)\) yields significant improvement across the widest portions of the grid. Consequently, the degree of improvement does not vary monotonically with the normalized HHI. The jagged contours reflect Monte Carlo variability and the discretization used in the two-stage grid search.

\subsection{Two-dimensional certified interval-wise search}
\label{subsec:numerics-2d}

Consider the two-dimensional setting
\[
d=2,\qquad
\Delta=(1,2),
\qquad
m=2.
\]
For the representative case \((\varepsilon,\delta+\alpha)=(1,0.01)\), the search is run on \([1,11]\) with tolerance \(\eta=0.05\). The Monte Carlo failure budget is \(\alpha=10^{-4}\), so the effective target is \(\delta=0.0099\). Table~\ref{tab:numerics-2d-representative-summary} summarizes the output.
\begin{table}[!htbp]
\centering
\small
\setlength{\tabcolsep}{6pt}
\renewcommand{\arraystretch}{1.08}
\caption{Input settings and output of the certified interval-wise search.}
\label{tab:numerics-2d-representative-summary}
\begin{tabular}{lc}
\toprule
Quantity & Value \\
\midrule
Query output dimension \(d\) & \(2\) \\
Sensitivity vector \(\Delta\) & \((1,2)\) \\
Overall privacy failure probability \(\delta+\alpha\) & \(0.01\) \\
Monte Carlo failure budget \(\alpha\) & \(10^{-4}\) \\
Privacy failure probability \(\delta\) & \(0.0099\) \\
Absolute-moment order \(m\) & \(2\) \\
Search interval for \(p\) & \([1,11]\) \\
Search tolerance \(\eta\) & \(0.05\) \\
Selected shape \(\hat p\) & \(1.2947\) \\
Scale upper bound \(\widehat b(\hat p)\) & \(3.81\) \\
Second moment upper bound \(L_2\!\left(\hat p,\widehat b(\hat p)\right)\) & \(14.12\) \\
Final certified gap & \(0.05\) \\
\bottomrule
\end{tabular}
\end{table}

The scale and second absolute moment in Table~\ref{tab:numerics-2d-representative-summary} are the upper bounds obtained during the interval-wise search. After \(\hat p\) is selected, the selected shape and the three benchmark shapes \(p=1,2,\infty\) are evaluated separately. The selected shape and the Laplace mechanism are evaluated using Monte Carlo calculations with empirical Bernstein bounds (MC--EB), whereas the Gaussian and uniform scales are obtained by solving their corresponding privacy equations. Table~\ref{tab:numerics-2d-fixed-p-comparison} reports the resulting scales and second absolute moments for the selected shape and the three benchmark mechanisms.
\begin{table}[!htbp]
\centering
\small
\setlength{\tabcolsep}{6pt}
\renewcommand{\arraystretch}{1.08}
\caption{Shape comparison for \(\Delta=(1,2)\) at \((\varepsilon,\delta+\alpha)=(1,0.01)\).}
\label{tab:numerics-2d-fixed-p-comparison}
\begin{tabular}{lcccc}
\toprule
Shape & \(p\) & Method & \(\widehat b(p)\) & \(L_2\!\left(p,\widehat b(p)\right)\) \\
\midrule
Selected & \(1.2947\) & MC--EB & \(3.81\) & \(14.12\) \\
Laplace & \(1\) & MC--EB & \(2.89\) & \(16.66\) \\
Gaussian & \(2\) & Equation solving & \(5.94\) & \(17.63\) \\
Uniform & \(\infty\) & Equation solving & \(149.67\) & \(7466.63\) \\
\bottomrule
\end{tabular}
\end{table}

In this case, the Laplace mechanism has the smallest second  moment among the three benchmark mechanisms, and
\[
\operatorname{Improvement}_2(1,0.01)
=
100\%
\left(
1-
\frac{L_2\bigl(\widehat p,\widehat b(\widehat p)\bigr)}
{\min\!\left\{
L_2\bigl(1,\widehat b(1)\bigr),
L_2\bigl(2,b(2)\bigr),
L_2\bigl(\infty,b(\infty)\bigr)
\right\}}
\right)
=
15.25\%.
\]
Thus, the selected shape reduces the second moment by \(15.25\%\) relative to the Laplace mechanism and by \(19.91\%\) relative to the Gaussian mechanism.

The certified interval-wise search is also run for
\[
\varepsilon\in\{0.25,0.5,1,2\},
\qquad
\delta\in\{0.1,0.01,0.001\}.
\]

\begin{table}[!htbp]
\centering
\small
\setlength{\tabcolsep}{5pt}
\renewcommand{\arraystretch}{1.08}
\caption{Shape selection and improvement across twelve privacy settings for \(\Delta=(1,2)\).}
\label{tab:numerics-2d-all-results}
\begin{tabular}{cccccc}
\toprule
\(\varepsilon\) & \(\delta+\alpha\) & \(\hat p\) & \(L_2(\hat p)\) & Best benchmark \(L_2\) & Improvement \\
\midrule
\(0.25\) & \(0.1\) & \(2.4590\) & \(22.03\) & \(22.33\) & \(1.32\%\) \\
\(0.25\) & \(0.01\) & \(1.5969\) & \(129.49\) & \(134.40\) & \(3.66\%\) \\
\(0.25\) & \(0.001\) & \(1.1903\) & \(247.58\) & \(279.10\) & \(11.29\%\) \\
\(0.5\) & \(0.1\) & \(2.0000\) & \(12.11\) & \(12.11\) & \(0.00\%\) \\
\(0.5\) & \(0.01\) & \(1.5081\) & \(44.50\) & \(49.52\) & \(10.13\%\) \\
\(0.5\) & \(0.001\) & \(1.0727\) & \(67.06\) & \(70.88\) & \(5.38\%\) \\
\(1\) & \(0.1\) & \(1.7572\) & \(5.85\) & \(5.90\) & \(0.70\%\) \\
\(1\) & \(0.01\) & \(1.2947\) & \(14.12\) & \(16.66\) & \(15.25\%\) \\
\(1\) & \(0.001\) & \(1.0458\) & \(17.45\) & \(17.86\) & \(2.29\%\) \\
\(2\) & \(0.1\) & \(1.5163\) & \(2.47\) & \(2.68\) & \(7.89\%\) \\
\(2\) & \(0.01\) & \(1.1294\) & \(4.10\) & \(4.33\) & \(5.41\%\) \\
\(2\) & \(0.001\) & \(1.0125\) & \(4.46\) & \(4.48\) & \(0.59\%\) \\
\bottomrule
\end{tabular}
\end{table}
\FloatBarrier
The one-dimensional experiment uses \(\Delta_1=3\), while the two-dimensional sensitivity vector is \(\Delta=(1,2)\); the \(\ell_1\) norm of both sensitivity vectors is \(3\). At the privacy level of \((1,0.01)\), the selected shape changes from \(1.0649\) to \(1.2947\), while the second absolute moment decreases from \(16.71\) to \(14.12\). The corresponding improvement rate increases from \(3.40\%\) to \(15.25\%\). This comparison demonstrates that sensitivity vectors with the same \(\ell_1\) norm may yield different selected shapes and second absolute moments when the distribution of sensitivity across the coordinates differs. The evolution of the certified gap \(U^*-L^*\) and the best shape found so far during the iterations for the twelve two-dimensional cases is shown in Appendix~\ref{app:additional-numerical-results}.
\subsection{Five-dimensional grid-search results}
\label{subsec:numerics-5d-grid-search}
Consider four five-dimensional sensitivity vectors,
\[
d=5,\qquad
\Delta=(1,1,1,1,1),\quad
(12,9,5,3,1),\quad
(3.5,1,1,1,1),\quad
(21,1,1,1,1),
\qquad
m=2.
\]
Their normalized HHI values are \(0\), \(1/9\), \(1/9\), and \(16/25\), respectively. In particular, \((12,9,5,3,1)\) and \((3.5,1,1,1,1)\) have the same normalized HHI but different coordinate structures. This allows us to examine whether the imbalance index alone is sufficient to describe the effect of the sensitivity vector. By Proposition~\ref{prop:sensitivity-scaling-active-coordinates}, rescaling \(\Delta\) does not change the selected shape or the relative improvement, so the comparison focuses on how sensitivity is distributed across coordinates.

The two-stage grid search used for the heatmaps is performed over
\[
\varepsilon\in\{0.1,0.2,\ldots,5.0\},
\qquad
\delta+\alpha\in\{0.001,0.002,\ldots,0.1\}.
\]
The resulting heatmaps are shown below:
\FloatBarrier

\begin{center}

\noindent
\begin{minipage}{0.48\textwidth}
\centering
\includegraphics[width=\linewidth]{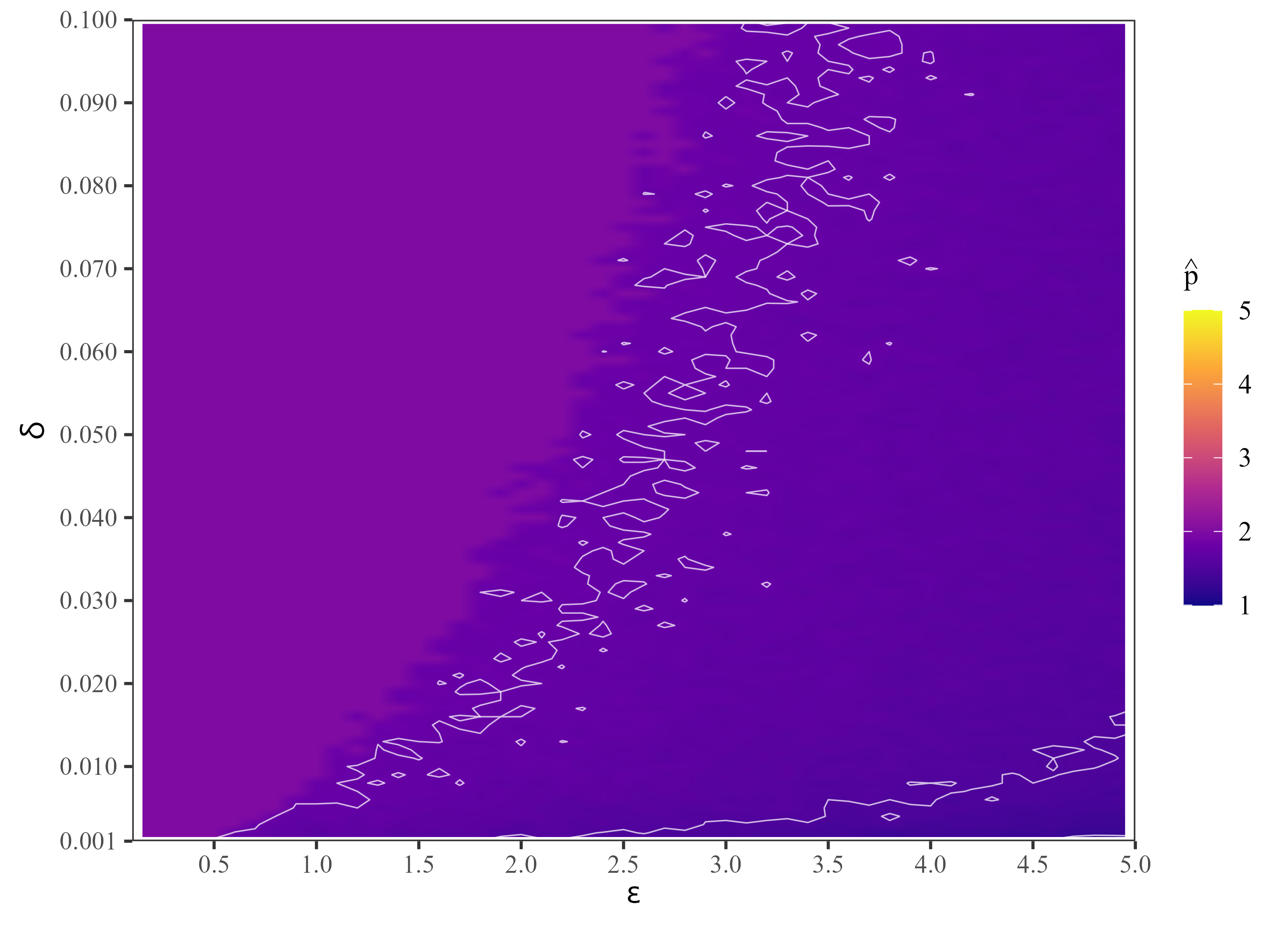}

\smallskip
\((a)\) \(\Delta=(1,1,1,1,1)\): selected shape
\end{minipage}
\hfill
\begin{minipage}{0.48\textwidth}
\centering
\includegraphics[width=\linewidth]{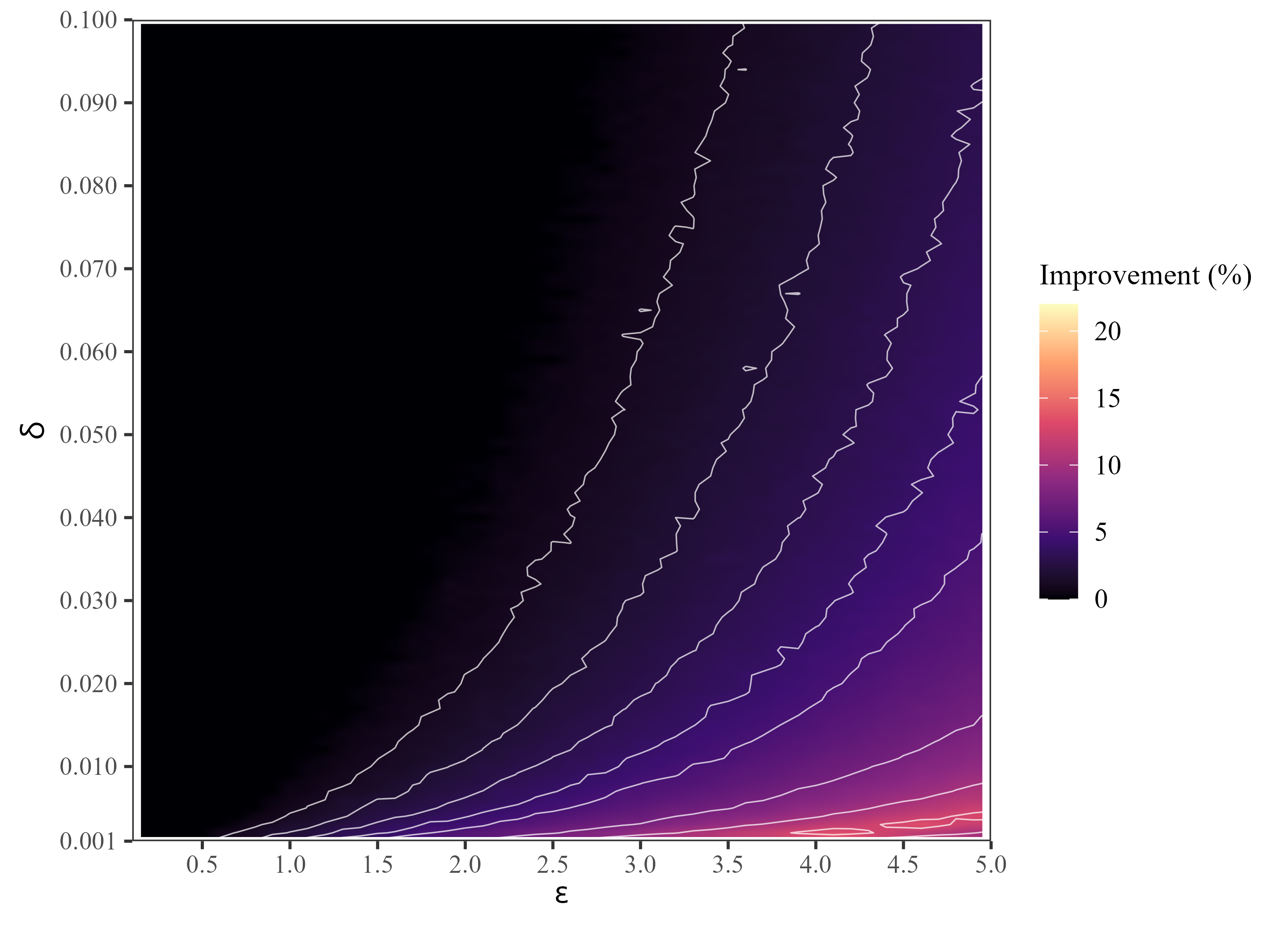}

\smallskip
\((b)\) \(\Delta=(1,1,1,1,1)\): improvement
\end{minipage}

\par\medskip

\noindent
\begin{minipage}{0.48\textwidth}
\centering
\includegraphics[width=\linewidth]{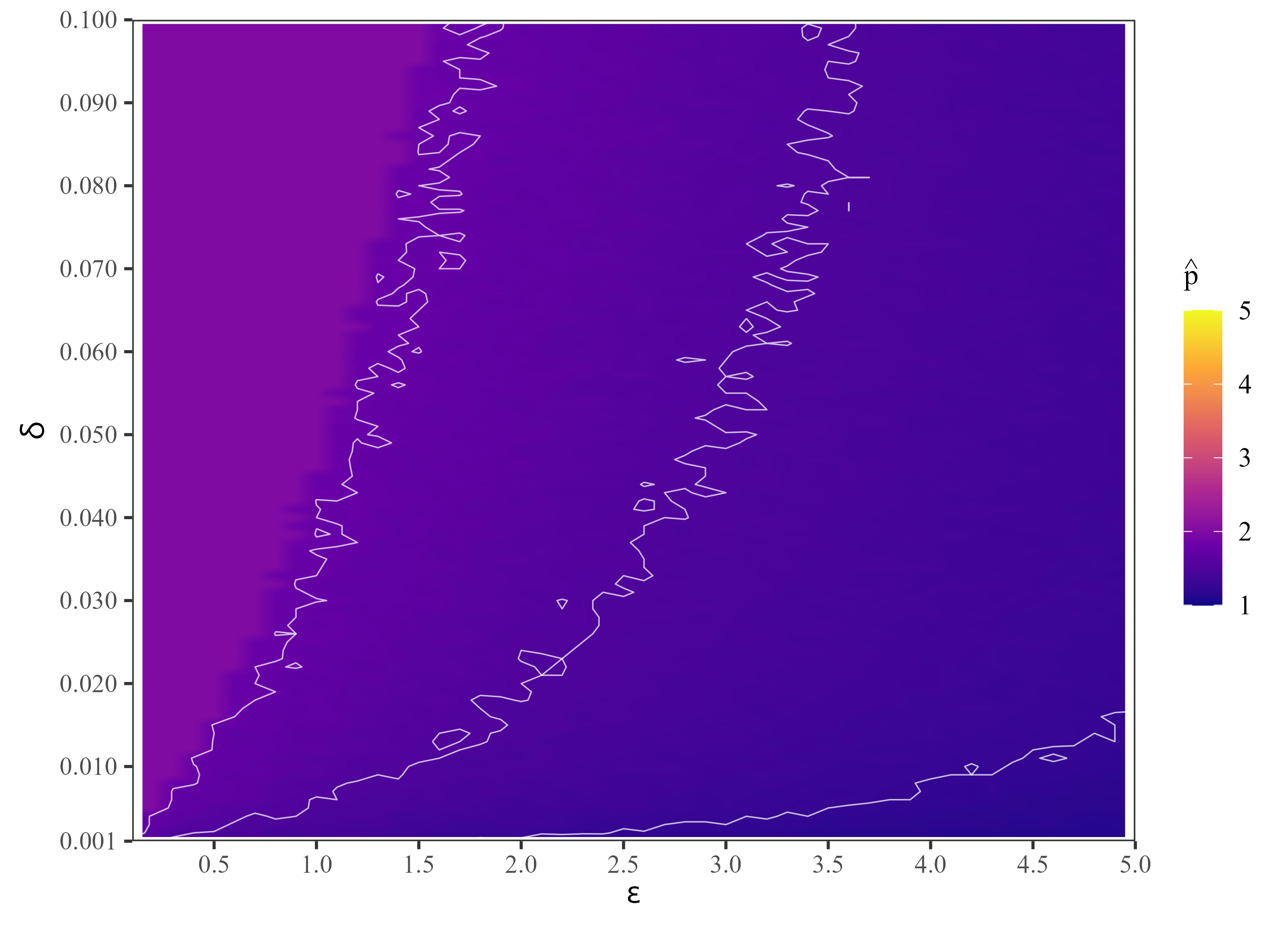}

\smallskip
\((c)\) \(\Delta=(12,9,5,3,1)\): selected shape
\end{minipage}
\hfill
\begin{minipage}{0.48\textwidth}
\centering
\includegraphics[width=\linewidth]{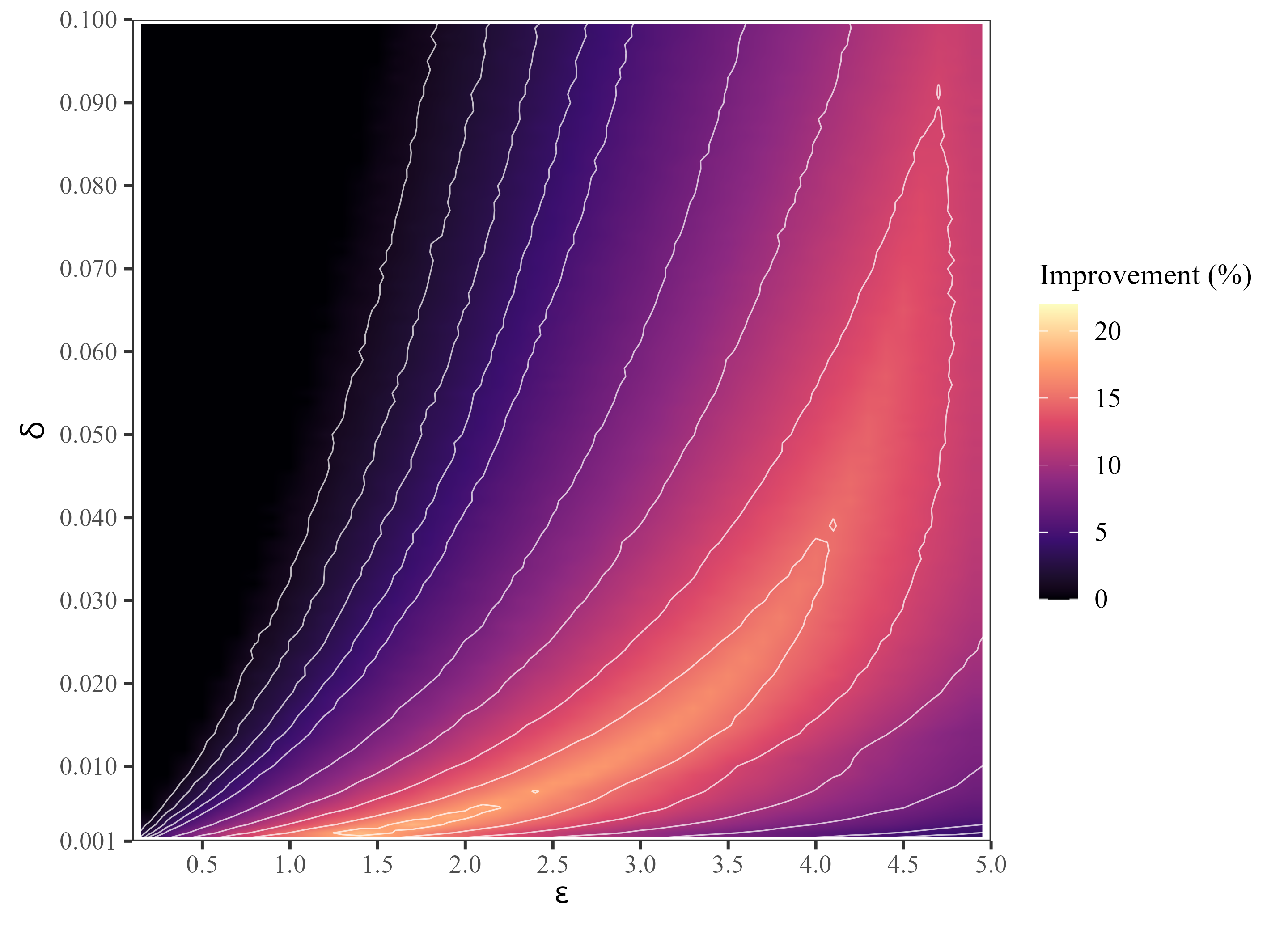}

\smallskip
\((d)\) \(\Delta=(12,9,5,3,1)\): improvement
\end{minipage}

\par\medskip

\noindent
\begin{minipage}{0.48\textwidth}
\centering
\includegraphics[width=\linewidth]{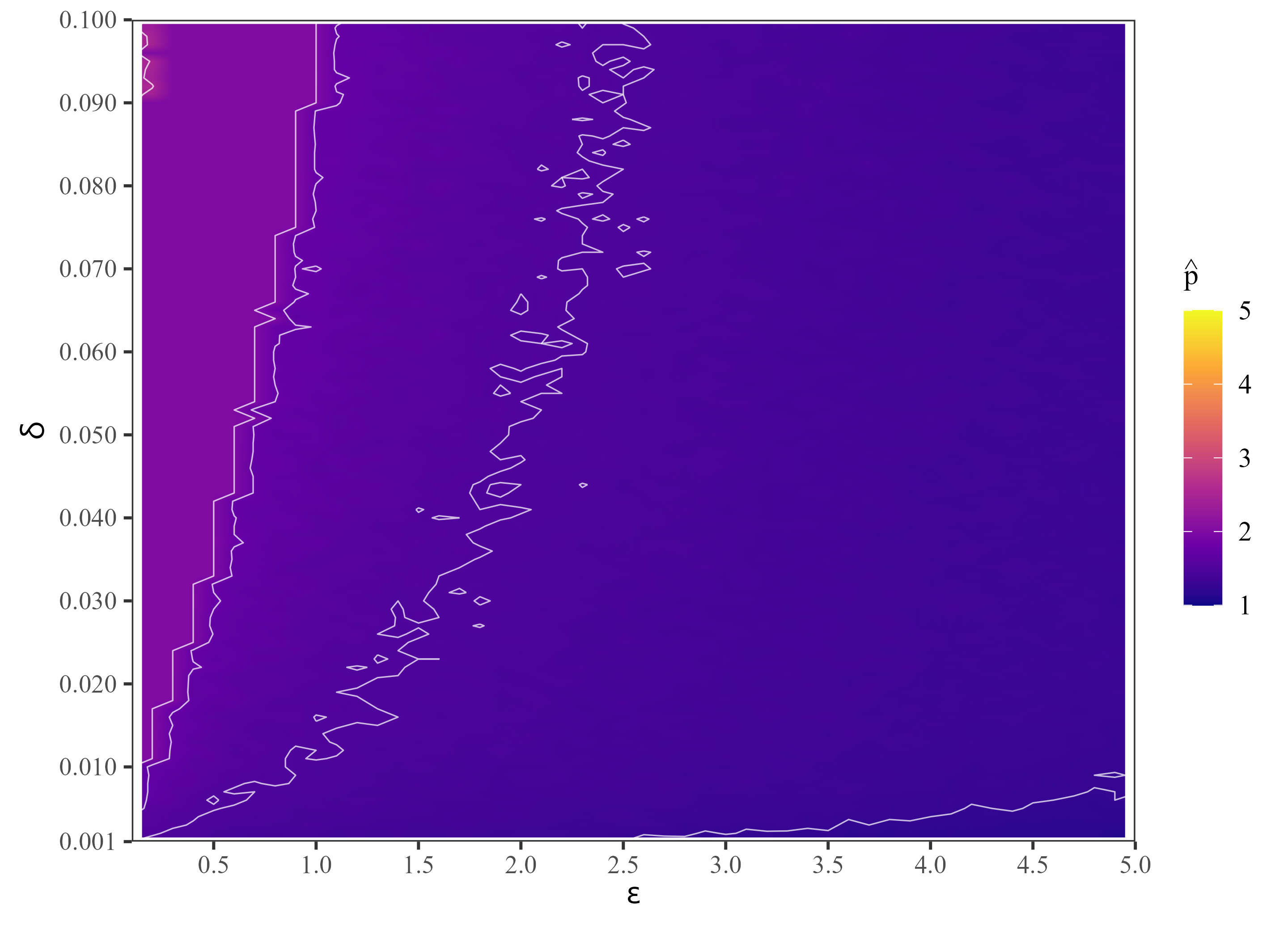}

\smallskip
\((e)\) \(\Delta=(3.5,1,1,1,1)\): selected shape
\end{minipage}
\hfill
\begin{minipage}{0.48\textwidth}
\centering
\includegraphics[width=\linewidth]{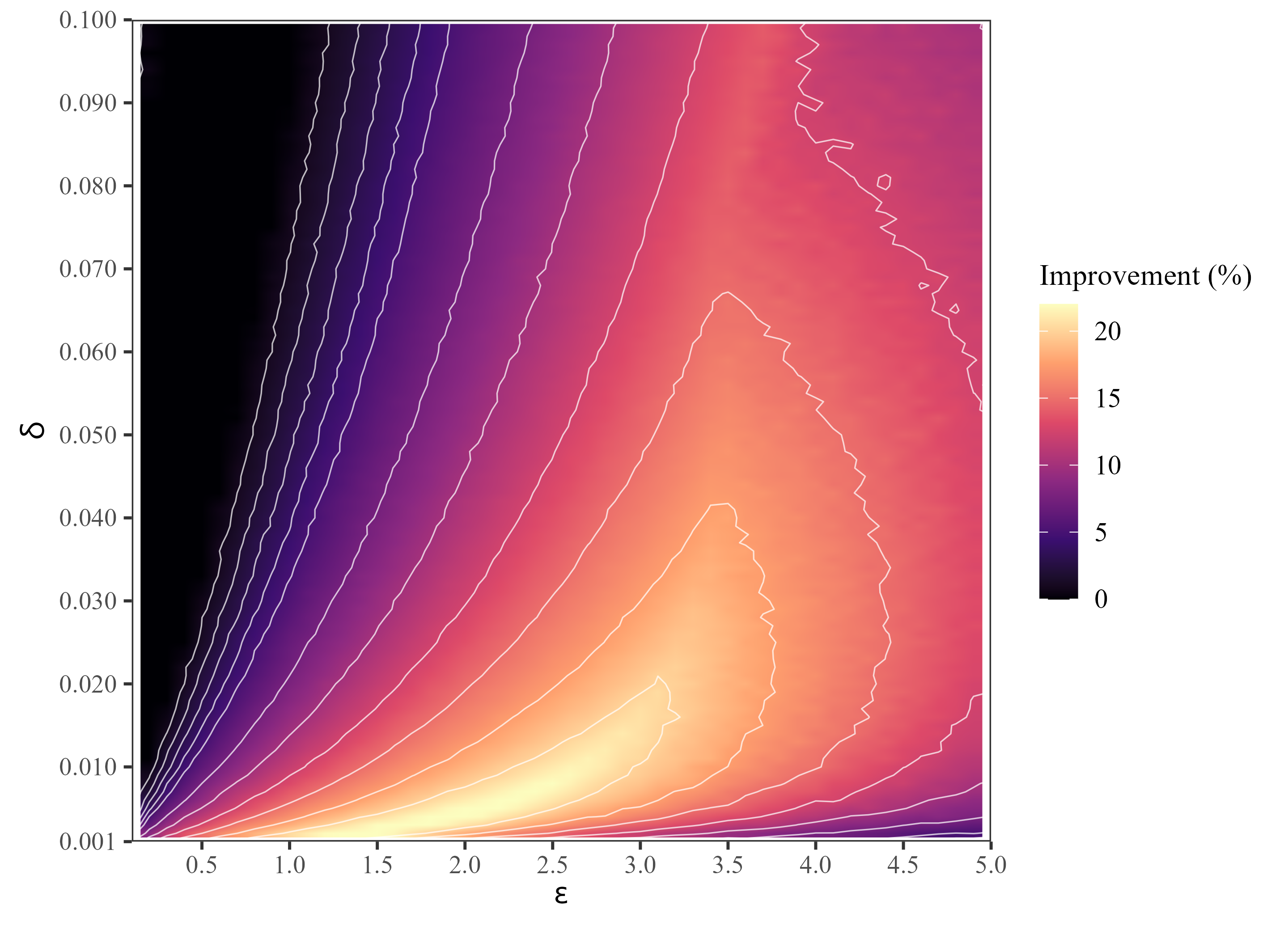}

\smallskip
\((f)\) \(\Delta=(3.5,1,1,1,1)\): improvement
\end{minipage}

\par\medskip

\noindent
\begin{minipage}{0.48\textwidth}
\centering
\includegraphics[width=\linewidth]{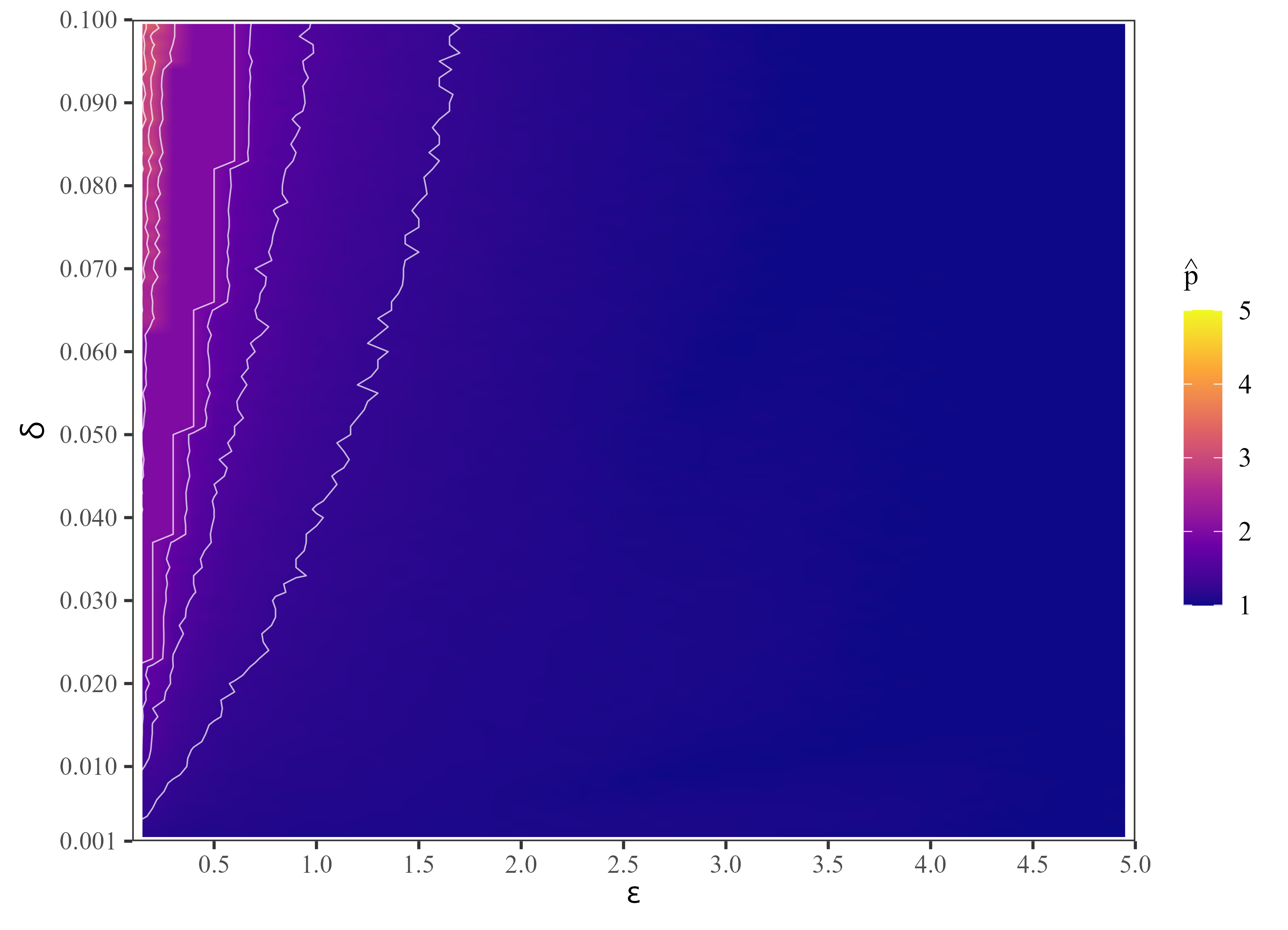}

\smallskip
\((g)\) \(\Delta=(21,1,1,1,1)\): selected shape
\end{minipage}
\hfill
\begin{minipage}{0.48\textwidth}
\centering
\includegraphics[width=\linewidth]{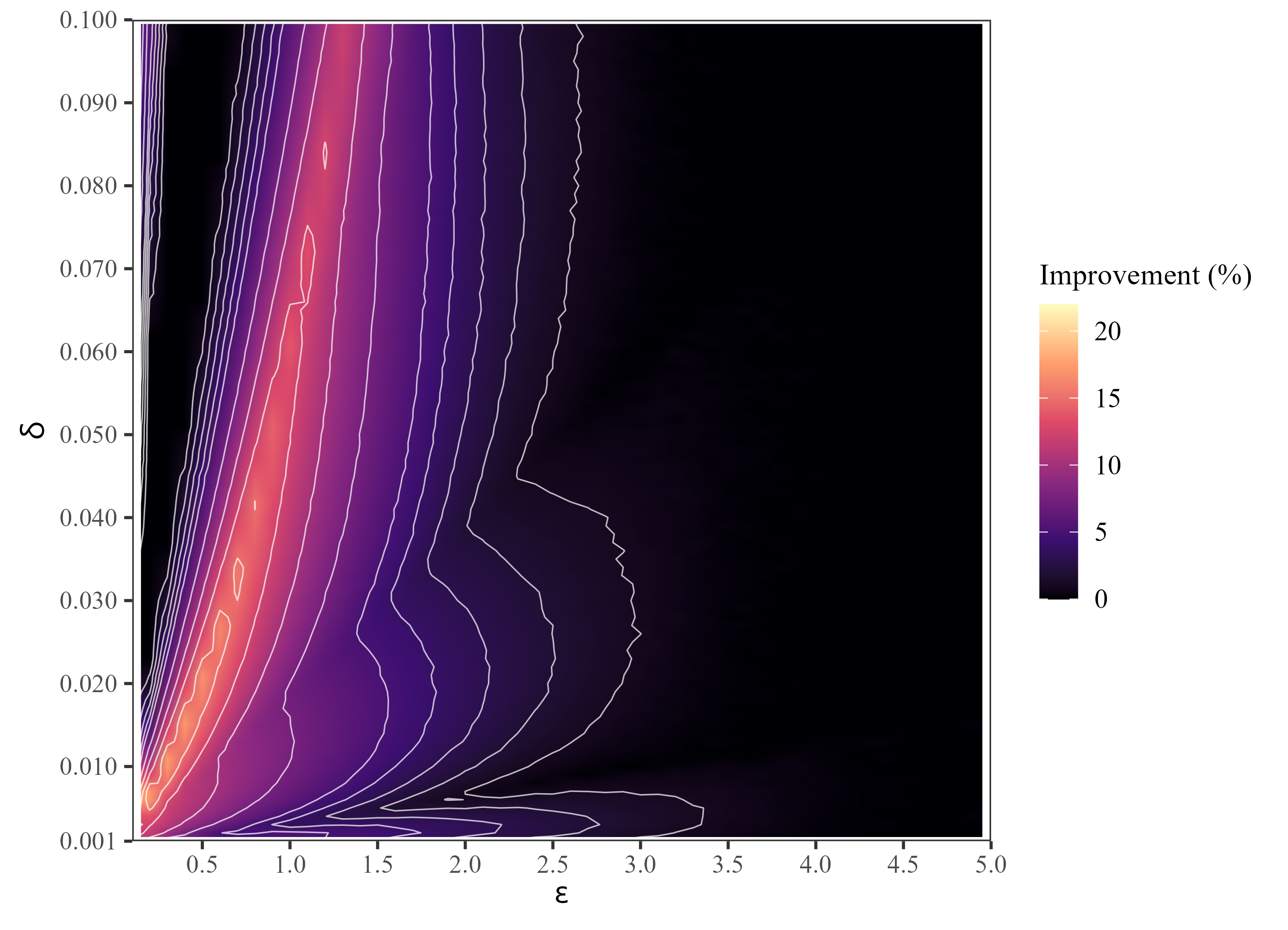}

\smallskip
\((h)\) \(\Delta=(21,1,1,1,1)\): improvement
\end{minipage}

\captionof{figure}{Five-dimensional heatmaps from the two-stage grid search. From top to bottom, the rows correspond to \(\Delta=(1,1,1,1,1)\), \((12,9,5,3,1)\), \((3.5,1,1,1,1)\), and \((21,1,1,1,1)\). The left column shows the selected shape \(\hat p\), and the right column shows the percentage improvement relative to the best of the three benchmark mechanisms \(p=1,2,\infty\).}
\label{fig:numerics-5d-grid-search}

\end{center}

\FloatBarrier

The heatmap shows significant differences in the results for the four sensitivity vectors. For the balanced vector \(\Delta=(1,1,1,1,1)\), \(\hat p=2\) is selected over most of the upper-left region, and the area of significant improvement is small. The other two moderately imbalanced cases, \(\Delta=(12,9,5,3,1)\) and \(\Delta=(3.5,1,1,1,1)\), have much wider areas of improvement. Although the normalized HHI is the same for both vectors, the latter performs best over most of the grid. For the highly imbalanced vector \(\Delta=(21,1,1,1,1)\), \(p=1\) is selected over most of the lower right region, and the significant improvement is concentrated in a relatively narrow band.

Among the 5000 privacy levels, for the four cases of \(\Delta=(1,1,1,1,1)\), \(\Delta=(12,9,5,3,1)\), \(\Delta=(3.5,1,1,1,1)\), and \(\Delta=(21,1,1,1,1)\), \(p=2\) was selected at \(1991\), \(930\), \(529\), and \(220\) grid points, respectively. The corresponding numbers of grid points selecting \(p=1\) were \(0\), \(0\), \(0\), and \(1401\), respectively. The corresponding numbers of grid points with an improvement of at least \(5\%\)  were \(455\), \(2984\), \(3820\), and \(1127\). The corresponding numbers of grid points with an improvement of at least \(10\%\) were \(71\), \(1783\), \(3033\), and \(420\), respectively. The maximum improvement rates are \(13.51\%\), \(18.83\%\), \(23.13\%\), and \(18.27\%\), respectively. Although the normalized Herfindahl-Hirschman Index (HHI) for \(\Delta=(12,9,5,3,1)\) and \(\Delta=(3.5,1,1,1,1)\) is the same, the shapes and improvement regions they select are significantly different. Therefore, in the multivariate case, the optimal shape and degree of improvement depend on the structure of the sensitivity vector.

\begin{remark}
In the one-dimensional, two-dimensional, and five-dimensional examples, as the privacy parameter grid moves from the top-left corner (smaller \(\varepsilon\) and larger \(\delta\)) to the bottom-right corner (larger \(\varepsilon\) and smaller \(\delta\)), the selected shapes generally become smaller. The \(\hat p\) values in the upper-left region tend to be large, typically close to or greater than \(2\), while the \(\hat p\) values in the lower-right region tend to be close to or equal to \(1\). The sizes of these regions vary significantly as the sensitivity vector changes. Balanced vectors tend to produce larger regions where \(\hat p\) equals \(2\), while highly unbalanced vectors tend to produce larger regions where \(\hat p\) equals \(1\).
\end{remark}
The heatmaps illustrate the overall patterns but do not provide exact values for each privacy setting. Therefore, Table ~\ref{tab:numerics-5d-full-grid-comparison} lists the selected shapes and their improvement rates corresponding to twelve representative  privacy-parameter pairs. These results were obtained using the full-grid search over \(p=1, 1.01, 1.02, \ldots, 11\).
\begin{table}[!htbp]
\centering
\small
\setlength{\tabcolsep}{2.7pt}
\renewcommand{\arraystretch}{1.08}
\caption{Shape selection and improvement for twelve privacy settings.}
\label{tab:numerics-5d-full-grid-comparison}

\begin{tabular}{cc|cc|cc|cc|cc}
\toprule
& &
\multicolumn{2}{c|}{\(\Delta=(1,1,1,1,1)\)} &
\multicolumn{2}{c|}{\(\Delta=(12,9,5,3,1)\)} &
\multicolumn{2}{c|}{\(\Delta=(3.5,1,1,1,1)\)} &
\multicolumn{2}{c}{\(\Delta=(21,1,1,1,1)\)} \\
\(\varepsilon\) & \(\delta+\alpha\) &
\(\hat p\) & Improvement &
\(\hat p\) & Improvement &
\(\hat p\) & Improvement &
\(\hat p\) & Improvement \\
\midrule

\multirow{3}{*}{\(0.25\)}
& \(0.1\)
& \(2.00\) & \(0.00\%\)
& \(2.00\) & \(0.00\%\)
& \(2.44\) & \(0.72\%\)
& \(2.77\) & \(3.89\%\)\\

&
\(0.01\)
& \(2.00\) & \(0.00\%\)
& \(1.83\) & \(0.31\%\)
& \(1.70\) & \(2.63\%\)
& \(1.33\) & \(15.82\%\) \\

&
\(0.001\)
& \(1.75\) & \(1.80\%\)
& \(1.50\) & \(9.56\%\)
& \(1.46\) & \(15.63\%\)
& \(1.11\) & \(7.73\%\) \\

\midrule

\multirow{3}{*}{\(0.5\)}
& \(0.1\)
& \(2.00\) & \(0.00\%\)
& \(2.00\) & \(0.00\%\)
& \(2.00\) & \(0.00\%\)
& \(2.00\) & \(0.00\%\) \\

&
\(0.01\)
& \(1.84\) & \(0.05\%\)
& \(1.72\) & \(2.56\%\)
& \(1.57\) & \(7.08\%\)
& \(1.20\) & \(11.23\%\) \\

&
\(0.001\)
& \(1.70\) & \(3.29\%\)
& \(1.42\) & \(14.58\%\)
& \(1.40\) & \(20.10\%\)
& \(1.09\) & \(6.47\%\) \\

\midrule

\multirow{3}{*}{\(1\)}
& \(0.1\)
& \(2.00\) & \(0.00\%\)
& \(2.00\) & \(0.00\%\)
& \(1.78\) & \(0.34\%\)
& \(1.50\) & \(6.10\%\) \\

&
\(0.01\)
& \(1.79\) & \(1.08\%\)
& \(1.59\) & \(6.64\%\)
& \(1.49\) & \(12.54\%\)
& \(1.13\) & \(7.46\%\) \\

&
\(0.001\)
& \(1.62\) & \(5.56\%\)
& \(1.34\) & \(19.44\%\)
& \(1.37\) & \(23.66\%\)
& \(1.07\) & \(5.78\%\) \\

\midrule

\multirow{3}{*}{\(2\)}
& \(0.1\)
& \(2.00\) & \(0.00\%\)
& \(1.73\) & \(2.00\%\)
& \(1.54\) & \(5.88\%\)
& \(1.18\) & \(3.56\%\) \\

&
\(0.01\)
& \(1.68\) & \(3.01\%\)
& \(1.44\) & \(13.42\%\)
& \(1.42\) & \(19.11\%\)
& \(1.07\) & \(2.29\%\) \\

&
\(0.001\)
& \(1.52\) & \(9.46\%\)
& \(1.22\) & \(11.60\%\)
& \(1.28\) & \(14.22\%\)
& \(1.06\) & \(2.71\%\) \\

\bottomrule
\end{tabular}
\end{table}
\FloatBarrier
In these examples, the greatest improvements were observed for the two moderately imbalanced vectors, particularly \(\Delta=(3.5,1,1,1,1)\). The balanced vector yielded the smallest gains, while the highly imbalanced vector achieved significant improvements only at a few privacy levels.

\subsection{Improvement exceedance curves}
\label{subsec:numerics-improvement-exceedance}

Using the two-stage grid-search results, Figure~\ref{fig:numerics-improvement-exceedance} reports, for each threshold on the horizontal axis, the percentage of the \(5000\) privacy-parameter pairs whose improvement is at least that threshold.

\begin{figure}[!htbp]
\centering
\includegraphics[width=0.98\linewidth]{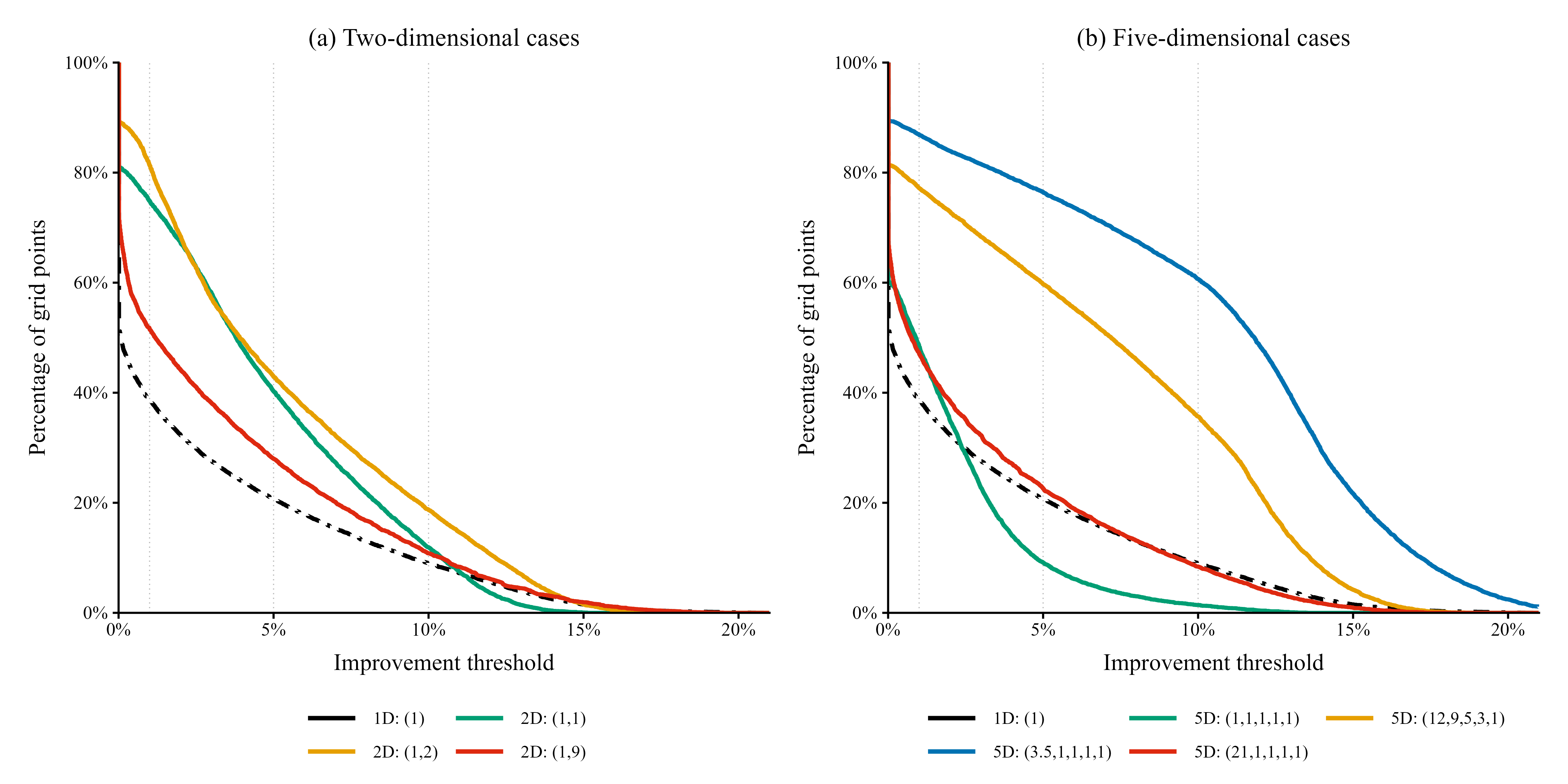}
\caption{Percentage of the \(5000\) privacy-parameter pairs for which the improvement is at least the value shown on the horizontal axis. Panel (a) compares the one-dimensional reference with the three two-dimensional vectors. Panel (b) compares the same reference with the four five-dimensional vectors.}
\label{fig:numerics-improvement-exceedance}
\end{figure}

\FloatBarrier

Figure ~\ref{fig:numerics-improvement-exceedance} summarizes the results of these heatmaps. In the two-dimensional case, \(\Delta=(1,2)\) yields the highest curve for most thresholds, while \(\Delta=(1,9)\) is close to the one-dimensional curve. In the five-dimensional case, the differences are more pronounced: \(\Delta=(3.5,1,1,1,1)\) yields the highest curve, followed by \(\Delta=(12,9,5,3,1)\), while the curve for the balanced case declines much more rapidly. The curves for \(\Delta=(1,9)\) and \(\Delta=(21,1,1,1,1)\) are close to the one-dimensional curve, which is consistent with the fact that most of their sensitivity is concentrated on a single coordinate. Overall, the performance improvement does not increase monotonically with increasing imbalance.
\subsection{Representative five-dimensional certified interval-wise shape search}
\label{subsec:five-dimensional-certified-search}

We also apply the certified interval-wise search to the five-dimensional
sensitivity vector
\[
\Delta=(21,1,1,1,1),
\]
with
\[
d=5,
\qquad
m=2,
\qquad
(\varepsilon,\delta+\alpha)=(1,0.01),
\qquad
\alpha=10^{-4}.
\]
Thus, the effective target is \(\delta=0.0099\), the search interval is
\([1,2]\), and the target gap is \(\eta=0.05\).

In this five-dimensional example, the upper bound for \(S(I,b)\) is looser
than in the lower-dimensional examples, and the certified gap decreases more
slowly. The search reaches \(U^*-L^*\leq\eta\) after 5195 search steps, with
\[
\widehat p=1.129824,
\qquad
U^*-L^*=0.049998.
\]

Figure~\ref{fig:five-dimensional-search-diagnostics} shows the evolution of
the certified gap and the best shape parameter during the search.

\begin{figure}[H]
\centering

\begin{minipage}[t]{0.48\textwidth}
    \centering
    \includegraphics[width=\linewidth]
    {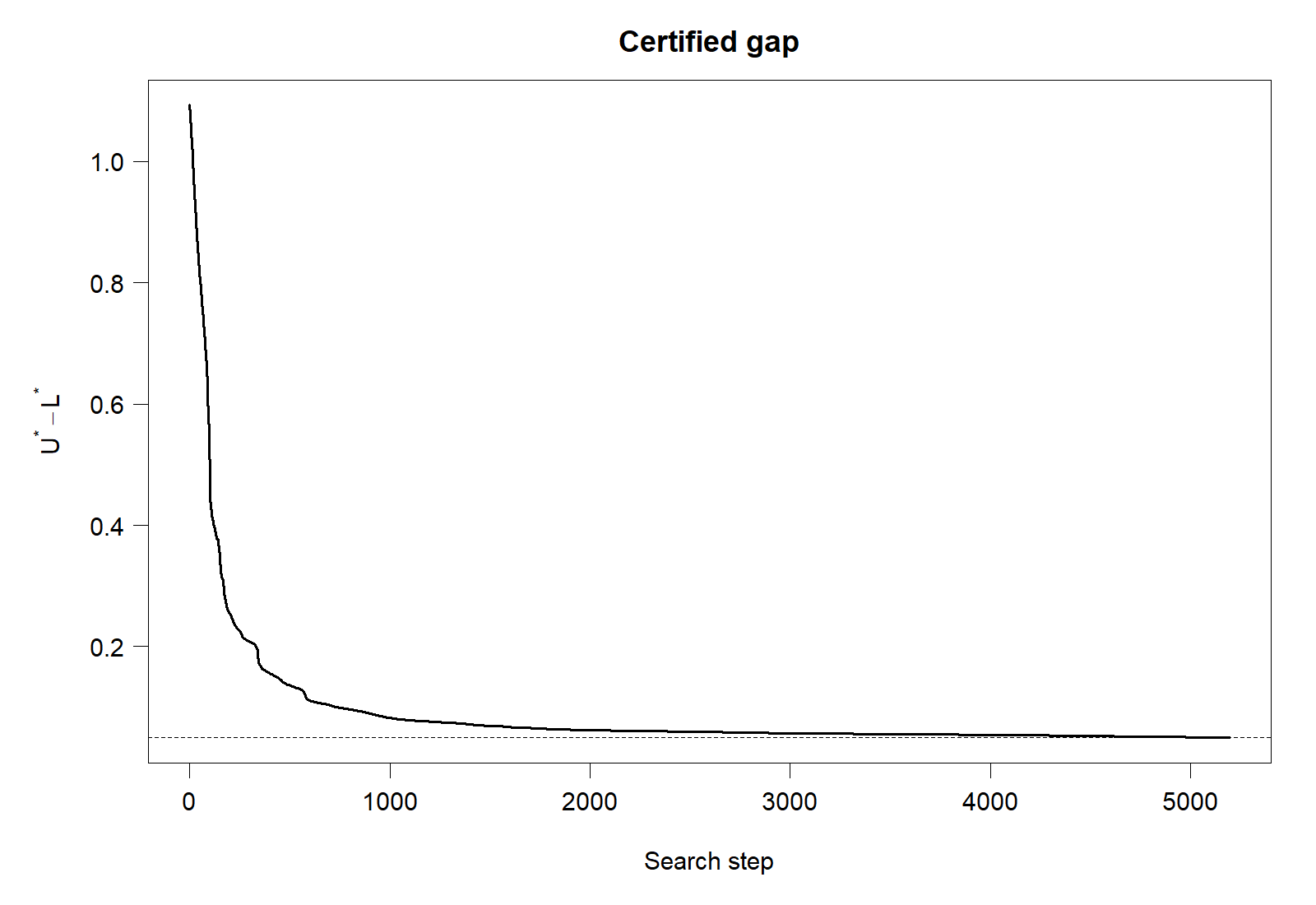}

    \vspace{2pt}
    \small
    (a) Certified gap \(U^*-L^*\).
\end{minipage}
\hfill
\begin{minipage}[t]{0.48\textwidth}
    \centering
    \includegraphics[width=\linewidth]
    {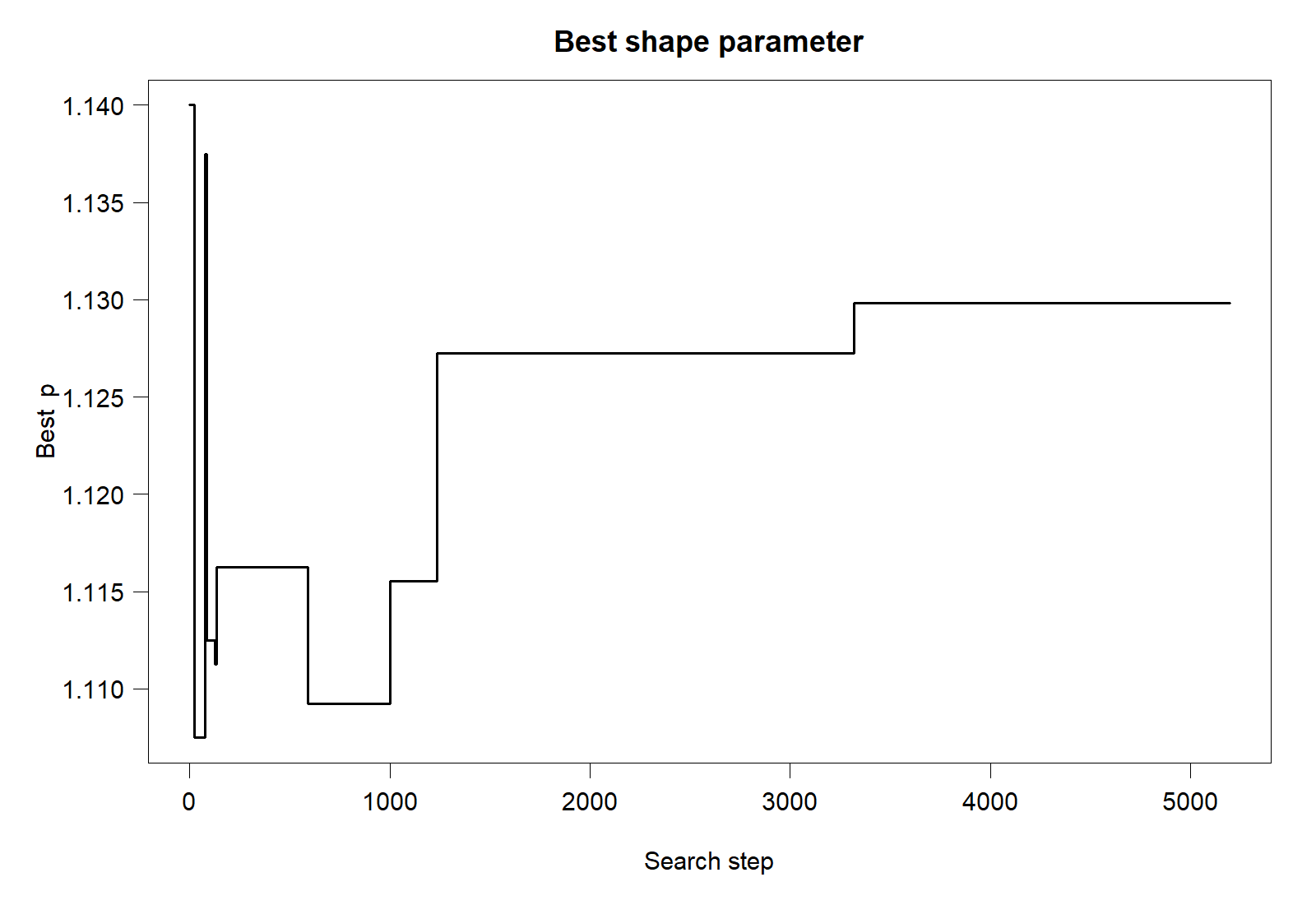}

    \vspace{2pt}
    \small
    (b) Best shape parameter.
\end{minipage}

\caption{Diagnostics for the five-dimensional certified interval-wise search
with \(\Delta=(21,1,1,1,1)\).
(a) Certified gap \(U^*-L^*\), with the dashed line marking the target
\(\eta=0.05\).
(b) Best shape parameter found during the search.}
\label{fig:five-dimensional-search-diagnostics}

\end{figure}

\FloatBarrier
\subsection{Shape selection under task-specific criteria}
\label{sec:task-specific-example}

The criteria considered in this subsection are not scale-homogeneous.
For \(s>0\),
write
\[
H_s(p):=\mathbb P\{|Z_{p,b(p)}|\le s\}.
\]
Let \(I=[p_L,p_R]\subset(1,P_{\max}]\), with midpoint
\(p_I=(p_L+p_R)/2\) and half-width \(h_I=(p_R-p_L)/2\), and suppose
\(\underline b_I\le b(p)\le\overline b_I\) for \(p\in I\).
The endpoints \(p=1\) and \(p=\infty\) are evaluated separately.
For finite \(p\),
\[
H_s(p)
=
1-
Q\left(
\frac1p,
\left(\frac{s}{b(p)}\right)^p
\right),
\]
where \(Q\) is the regularized upper incomplete gamma function.
Appendix~\ref{app:task-specific-bound} shows that there exists a quantity
\(V(I)\), depending only on \(I\), such that
\[
\sup_{p\in I}
\left|
\frac{d}{dp}
Q\left(\frac1p,c^p\right)
\right|
\le V(I),
\qquad c>0.
\]
Since \(\mathbb P\{|Z_{p,b}|\le s\}\) is decreasing in \(b\),
\[
\begin{aligned}
\underline H_s(I)
&:=
\max\left\{
0,\,
\mathbb P\{|Z_{p_I,\overline b_I}|\le s\}-h_I V(I)
\right\}
\\
&\le H_s(p)
\\
&\le
\min\left\{
1,\,
\mathbb P\{|Z_{p_I,\underline b_I}|\le s\}+h_I V(I)
\right\}
=:\overline H_s(I),
\qquad p\in I.
\end{aligned}
\]

We consider a counting query. Let
\(D=(X_1,\ldots,X_n)\in\{0,1\}^n\) and
\(K=q(D)=\sum_{i=1}^n X_i\), so the sensitivity is \(\Delta=1\).
For fixed \((\varepsilon,\delta)\), release
\[
Y=K+Z_{p,b(p)}.
\]
For a tolerance \(t>0\),
\[
H_t(p)
=
\mathbb P\{|Y-K|\le t\}
=
\mathbb P\{|Z_{p,b(p)}|\le t\},
\]
which is the probability that the released count differs from the true count
by at most \(t\).

The first criterion maximizes tolerance accuracy. Since the classical
benchmark is fixed for each \(t\), maximizing the utility gain relative to
that benchmark is equivalent to maximizing \(H_t(p)\). Thus,
\[
p_U^\star(t)
\in
\arg\max_{p\in[1,\infty]} H_t(p),
\qquad
p_{B,U}(t)
\in
\arg\max_{p\in\{1,2,\infty\}} H_t(p).
\]

We also consider disclosure between two neighbouring datasets.
Let \(D\sim D'\) with \(q(D')=q(D)+1\), and assume that \(D\) and \(D'\)
are equally likely under the prior. Given the release \(Y\), an attacker
attempts to identify which dataset generated it. Under \(D\) and \(D'\),
the release distributions are
\(q(D)+Z_{p,b(p)}\) and \(q(D)+1+Z_{p,b(p)}\), respectively.

By symmetry, the Bayes decision threshold is \(q(D)+1/2\). Let
\(F_{p,b}\) denote the CDF of \(GGD(0,b,p)\). Then the
attack-success probability is
\[
S(p)
=
F_{p,b(p)}\left(\frac12\right)
=
\frac{1+H_{1/2}(p)}{2},
\qquad
\frac12\le S(p)\le1.
\]
Thus, \(S(p)=1/2\) corresponds to random guessing, while larger values
indicate greater distinguishability between the neighbouring datasets.
For any shape \(p\) and benchmark \(p_B\), define the utility gain and
attacker reduction as
\[
U(p;p_B)
:=
H_t(p)-H_t(p_B),
\qquad
A(p;p_B)
:=
S(p_B)-S(p).
\]
The total gain is then
\[
U(p;p_B)+A(p;p_B)
=
[H_t(p)-H_t(p_B)]
+
[S(p_B)-S(p)].
\]
Since the benchmark terms are constant in \(p\), maximizing total gain is
equivalent to maximizing \(H_t(p)-S(p)\). The second criterion therefore
selects
\[
\begin{aligned}
p_T^\star(t)
&\in
\arg\max_{p\in[1,\infty]}
\{H_t(p)-S(p)\}
\\
&=
\arg\max_{p\in[1,\infty]}
\left\{
H_t(p)-\frac12H_{1/2}(p)
\right\},
\end{aligned}
\]
with classical benchmark
\[
p_{B,T}(t)
\in
\arg\max_{p\in\{1,2,\infty\}}
\left\{
H_t(p)-\frac12H_{1/2}(p)
\right\}.
\]
For \(p\in I\),
\[
\underline H_t(I)
-\frac12\overline H_{1/2}(I)
\le
H_t(p)-\frac12H_{1/2}(p)
\le
\overline H_t(I)
-\frac12\underline H_{1/2}(I).
\]
Hence, both \(p_U^\star(t)\) and \(p_T^\star(t)\) can be obtained by the
certified interval-wise search rather than by a grid search.

For each criterion, with its corresponding benchmark \(p_B(t)\), we report
\[
U(t)=H_t(p^\star(t))-H_t(p_B(t)),
\qquad
A(t)=S(p_B(t))-S(p^\star(t)),
\]
together with the total gain \(U(t)+A(t)\).

We use
\[
\varepsilon\in\{0.1,0.5,1,2\},
\qquad
\delta\in\{0.001,0.01,0.05,0.10\},
\]
and
\[
t\in\{0.50,0.51,\ldots,10.00\},
\]
giving \(15{,}216\) fixed-\(t\) problems for each criterion.

Figures~\ref{fig:task-shape-utility} and \ref{fig:task-shape-total} compare the best classical benchmark with the selected generalized Gaussian shape. The selected shape changes with the optimization criterion. Under the utility criterion, shapes close to Laplace are common. Under the total-gain criterion, larger finite shapes are selected more often, and the uniform case \(p=\infty\) is selected for some tolerance ranges. The best classical benchmark also changes with \(t\) and the privacy parameters.

\par\medskip
\noindent
\begin{minipage}{\linewidth}
\centering
\includegraphics[width=0.99\linewidth]
{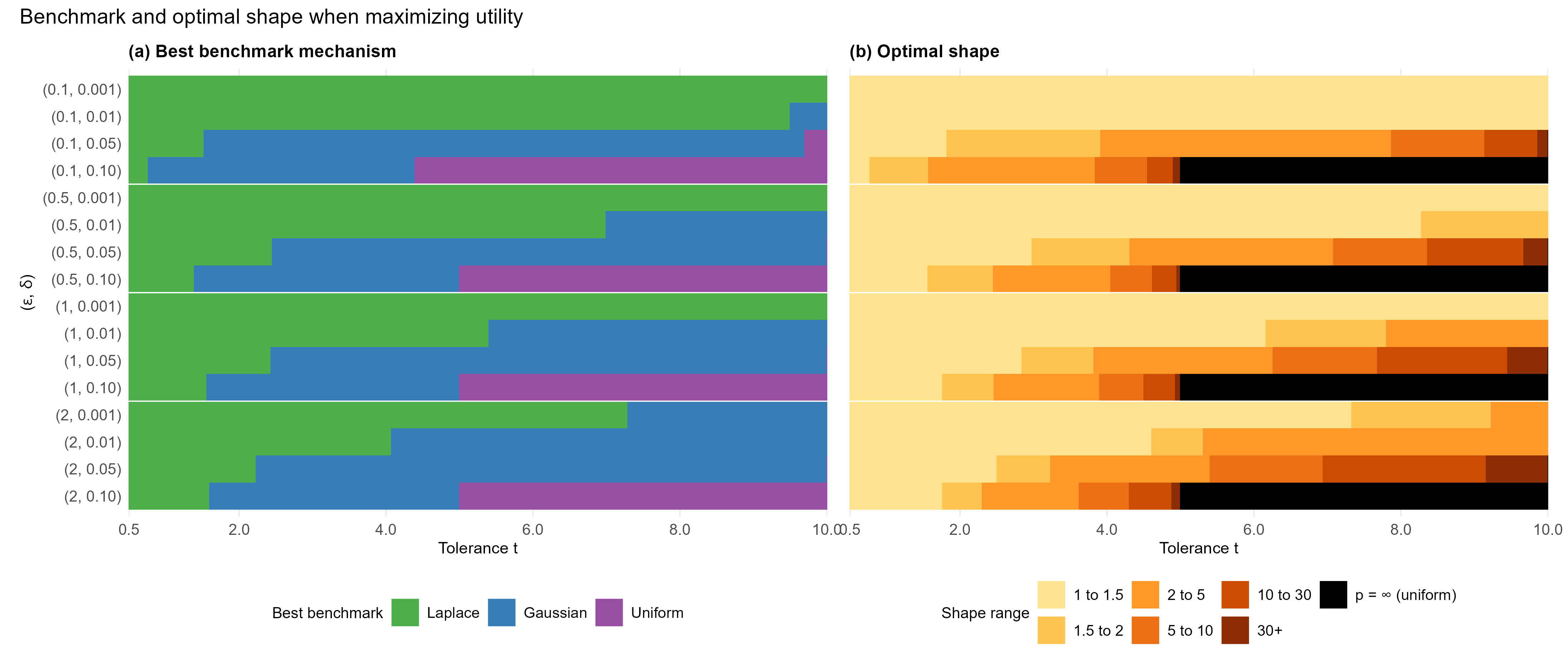}
\captionof{figure}{Best classical benchmark and selected generalized Gaussian
shape under the utility criterion.}
\label{fig:task-shape-utility}
\end{minipage}
\par\medskip

\noindent
\begin{minipage}{\linewidth}
\centering
\includegraphics[width=0.99\linewidth]
{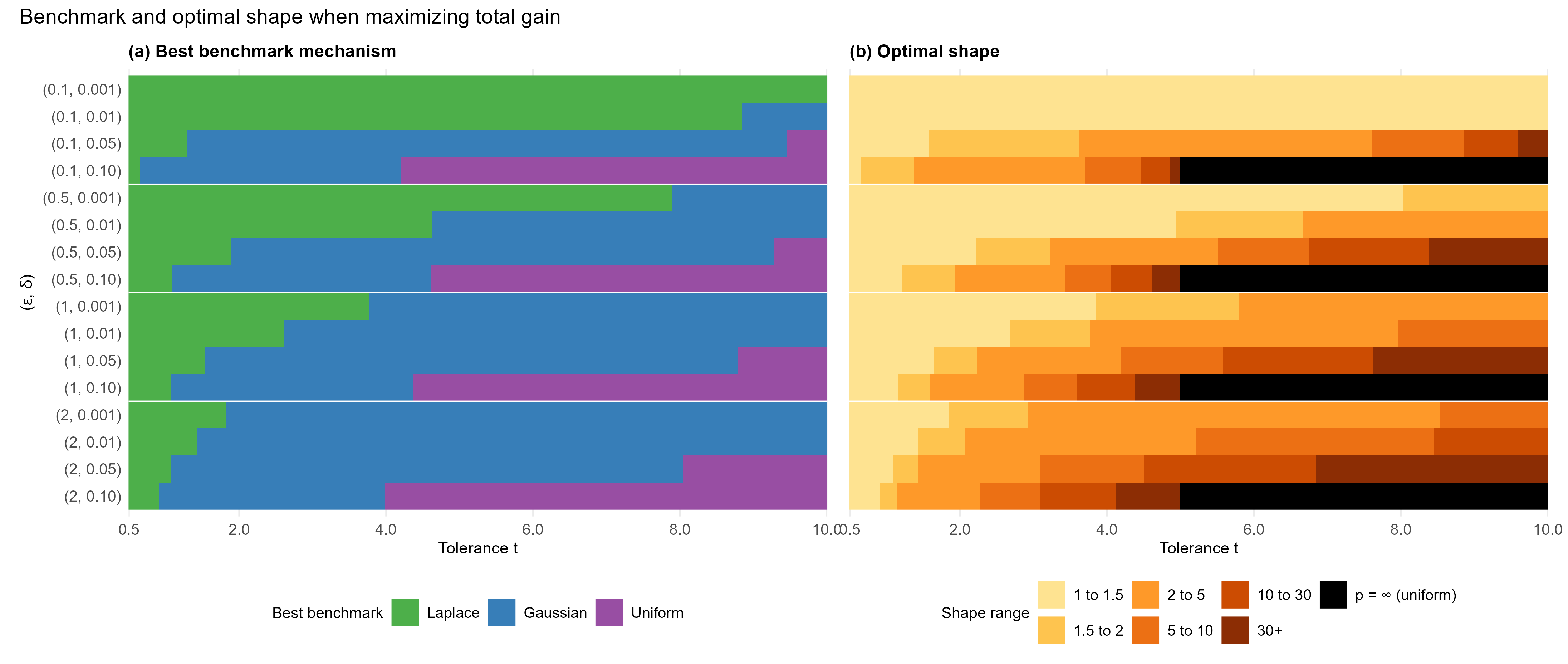}
\captionof{figure}{Best classical benchmark and selected generalized Gaussian
shape under the total-gain criterion.}
\label{fig:task-shape-total}
\end{minipage}
\par\medskip

Figures~\ref{fig:task-gain-utility} and
\ref{fig:task-gain-total} show the corresponding utility gain, attacker
reduction, and total gain. All three quantities are measured relative to the best classical benchmark among the Laplace, Gaussian, and uniform mechanisms under the same criterion.
The utility criterion produces substantial tolerance-accuracy improvements
for several privacy settings while often reducing attacker success at the
same time. The largest utility increase is \(7.01\%\), attained at
$
(\varepsilon,\delta,t)=(0.1,0.10,4.38),
$
where \(H_t\) increases from \(87.61\%\) to \(94.62\%\) and attacker success
decreases from \(56.97\%\) to \(55.76\%\).

\par\medskip
\noindent
\begin{minipage}{\linewidth}
\centering
\includegraphics[width=0.99\linewidth]
{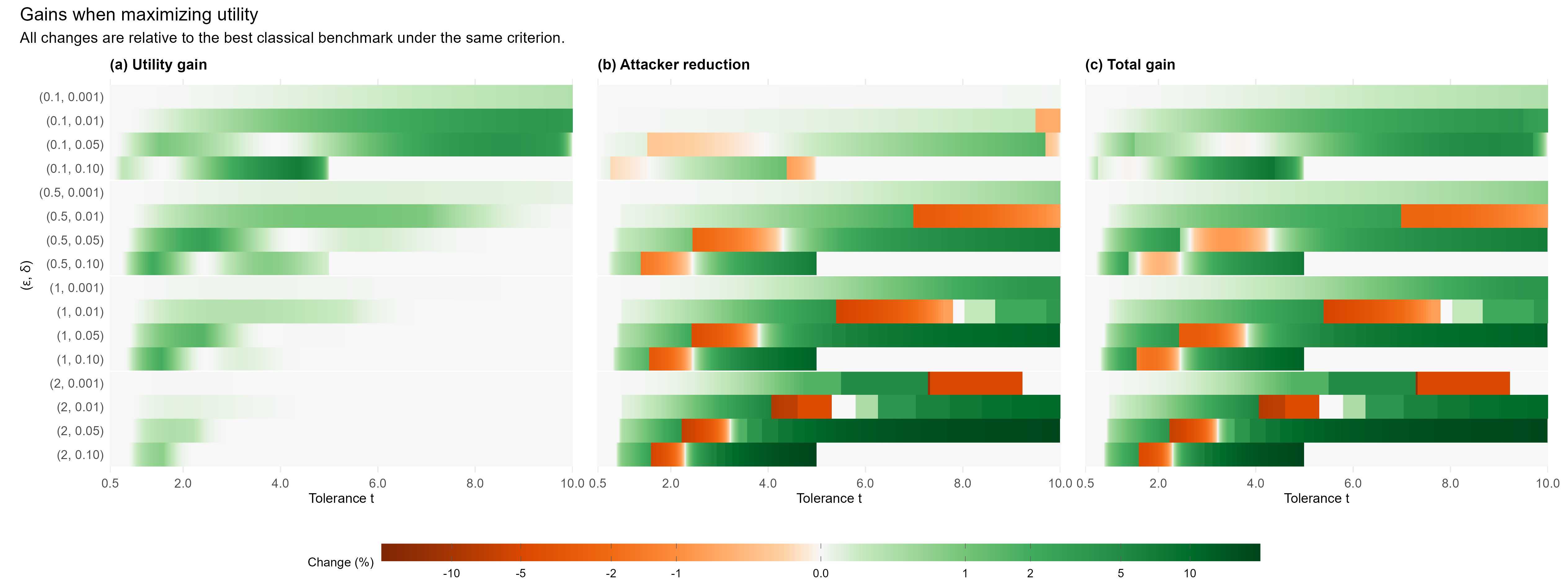}
\captionof{figure}{Utility gain, attacker reduction, and total gain under the
utility criterion. All changes are measured relative to the best classical
benchmark among \(p=1\) (Laplace), \(p=2\) (Gaussian), and \(p=\infty\)
(uniform), and are expressed as percentages.}
\label{fig:task-gain-utility}
\end{minipage}
\par\medskip

Under the total-gain criterion, larger improvements mainly come from reducing
attacker success. For example, at
\((\varepsilon,\delta)=(1,0.05)\) and \((2,0.05)\), the largest total gains
are \(11.69\%\) and \(18.85\%\), while the utility changes are only
\(-0.01\%\) and \(-0.02\%\). The selected shapes are
\(p^\star=69.71\) and \(58.18\), showing that large finite shapes can
substantially improve total gain beyond the classical Laplace, Gaussian, and
uniform choices.

\par\medskip
\noindent
\begin{minipage}{\linewidth}
\centering
\includegraphics[width=0.99\linewidth]
{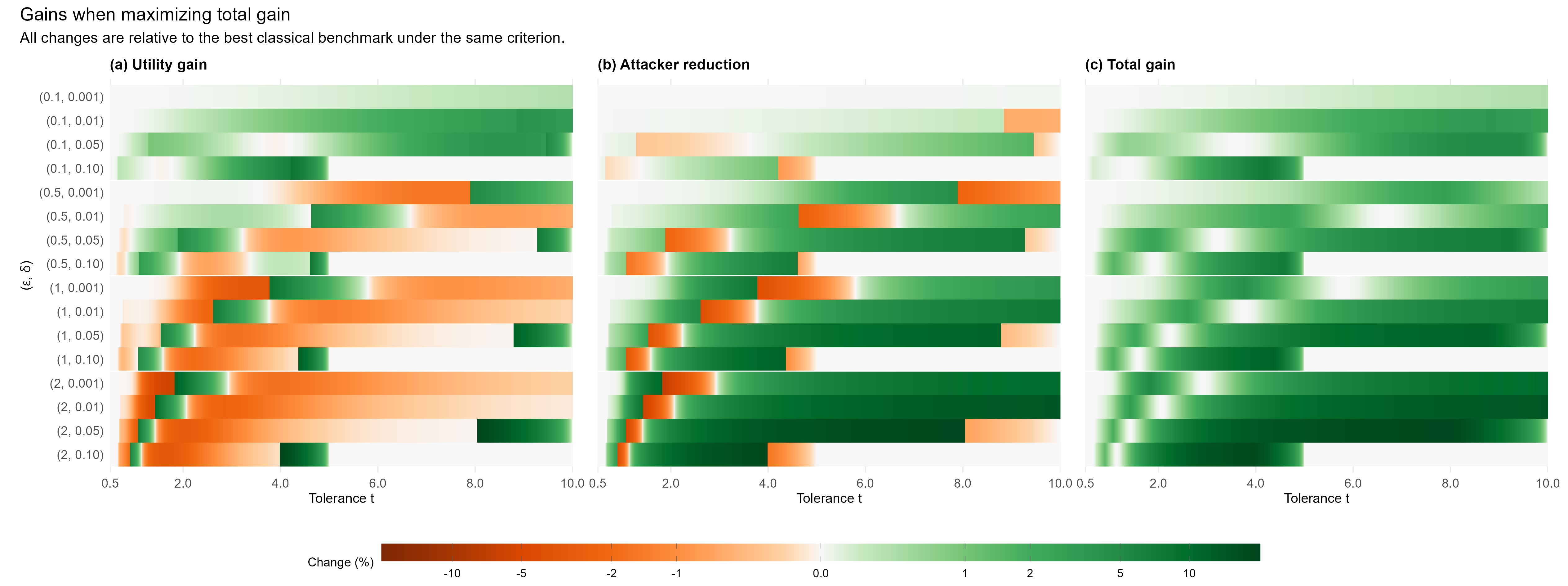}
\captionof{figure}{Utility gain, attacker reduction, and total gain under the total-gain criterion.}
\label{fig:task-gain-total}
\end{minipage}
\par\medskip
We classify each case as win--win, trade-off, or no improvement. A case is
win--win if both the utility gain and attacker reduction are positive, and it
is a trade-off if they have opposite signs. Across all 15216
privacy--tolerance combinations, the utility criterion gives \(64.43\%\)
win--win points, \(13.22\%\) trade-offs, and \(22.34\%\) no-improvement
points. Under the total-gain criterion, the corresponding proportions are
\(22.26\%\), \(62.49\%\), and \(15.25\%\). The larger proportion of
trade-offs under the total-gain criterion mainly reflects cases where a small
loss in utility is offset by a much larger reduction in attacker success.

\par\medskip
\noindent
\begin{minipage}{\linewidth}
\centering
\includegraphics[width=0.98\linewidth]
{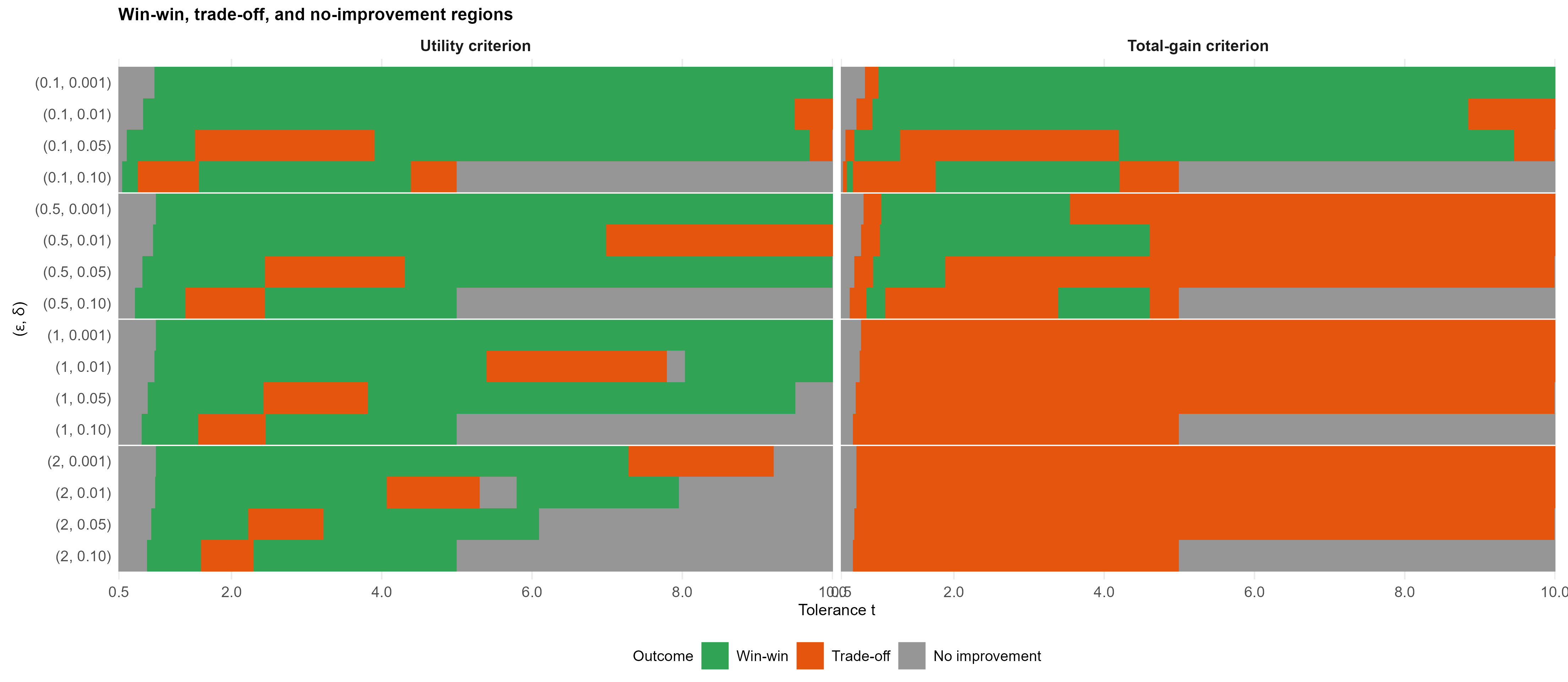}
\captionof{figure}{Win--win, trade-off, and no-improvement regions under the
two criteria. Win--win denotes \(U>0\) and \(A>0\), while trade-off denotes
opposite signs of \(U\) and \(A\).}
\label{fig:task-outcome}
\end{minipage}
\par\medskip
Table~\ref{tab:task-representative} reports representative cases for both
criteria across the four values of \(\varepsilon\), together with the selected
shape and the resulting gains.
\par\medskip
\noindent
\begin{minipage}{\linewidth}
\centering
\small
\setlength{\tabcolsep}{3.4pt}
\renewcommand{\arraystretch}{1.08}
\captionof{table}{Representative cases for the two criteria across the four
values of \(\varepsilon\). For each \(\varepsilon\), the utility rows maximize
\(U\), while the total-gain rows maximize \(U+A\), over the evaluated
\((\delta,t)\) values. The \(H_t\) and \(S\) columns report the percentages for the selected shape and the corresponding classical benchmark.}
\label{tab:task-representative}

\small
\setlength{\tabcolsep}{2.5pt}
\renewcommand{\arraystretch}{1.10}

\begin{tabular}{lcccccccccc}
\toprule
Criterion
& \(\varepsilon\)
& \(\delta\)
& \(t\)
& \(p^\star\)
& \(p_B\)
& \(H_t(p^\star)/H_t(p_B)\) (\%)
& \(S(p^\star)/S(p_B)\) (\%)
& \(U\)
& \(A\)
& \(U+A\) \\
\midrule
Utility
& 0.1 & 0.10 & 4.38 & 7.99 & 2
& 94.62/87.61 & 55.76/56.97
& 7.01 & 1.21 & 8.22 \\
Utility
& 0.5 & 0.05 & 2.44 & 1.37 & 1
& 79.83/77.01 & 62.02/63.01
& 2.81 & 0.99 & 3.80 \\
Utility
& 1 & 0.10 & 1.55 & 1.40 & 1
& 86.64/84.69 & 71.08/72.71
& 1.95 & 1.63 & 3.58 \\
Utility
& 2 & 0.10 & 1.59 & 1.39 & 1
& 97.73/97.03 & 80.73/83.45
& 0.70 & 2.71 & 3.41 \\
\midrule
Total gain
& 0.1 & 0.10 & 4.20 & 7.37 & 2
& 92.34/85.99 & 55.82/56.97
& 6.35 & 1.15 & 7.51 \\
Total gain
& 0.5 & 0.10 & 4.60 & 28.37 & 2
& 99.87/99.69 & 55.70/62.60
& 0.18 & 6.90 & 7.08 \\
Total gain
& 1 & 0.05 & 8.78 & 69.71 & 2
& 99.99/100.00 & 52.92/64.62
& -0.01 & 11.70 & 11.69 \\
Total gain
& 2 & 0.05 & 8.04 & 58.18 & 2
& 99.98/100.00 & 53.20/72.07
& -0.02 & 18.87 & 18.85 \\
\bottomrule
\end{tabular}

\end{minipage}
\par\medskip

Figure~\ref{fig:task-representative} shows representative cases under the two
criteria. Under the utility criterion, both gains increase as \(t\) approaches
\(4.38\); for \(t>4.38\), the attacker reduction becomes negative and the
total gain decreases. Under the total-gain criterion, the gain for
\(t\leq 8.04\) is driven mainly by attacker reduction, whereas for
\(t>8.04\) it is driven mainly by utility gain.

\par\medskip
\noindent
\begin{minipage}{\linewidth}
\centering
\includegraphics[width=0.98\linewidth]
{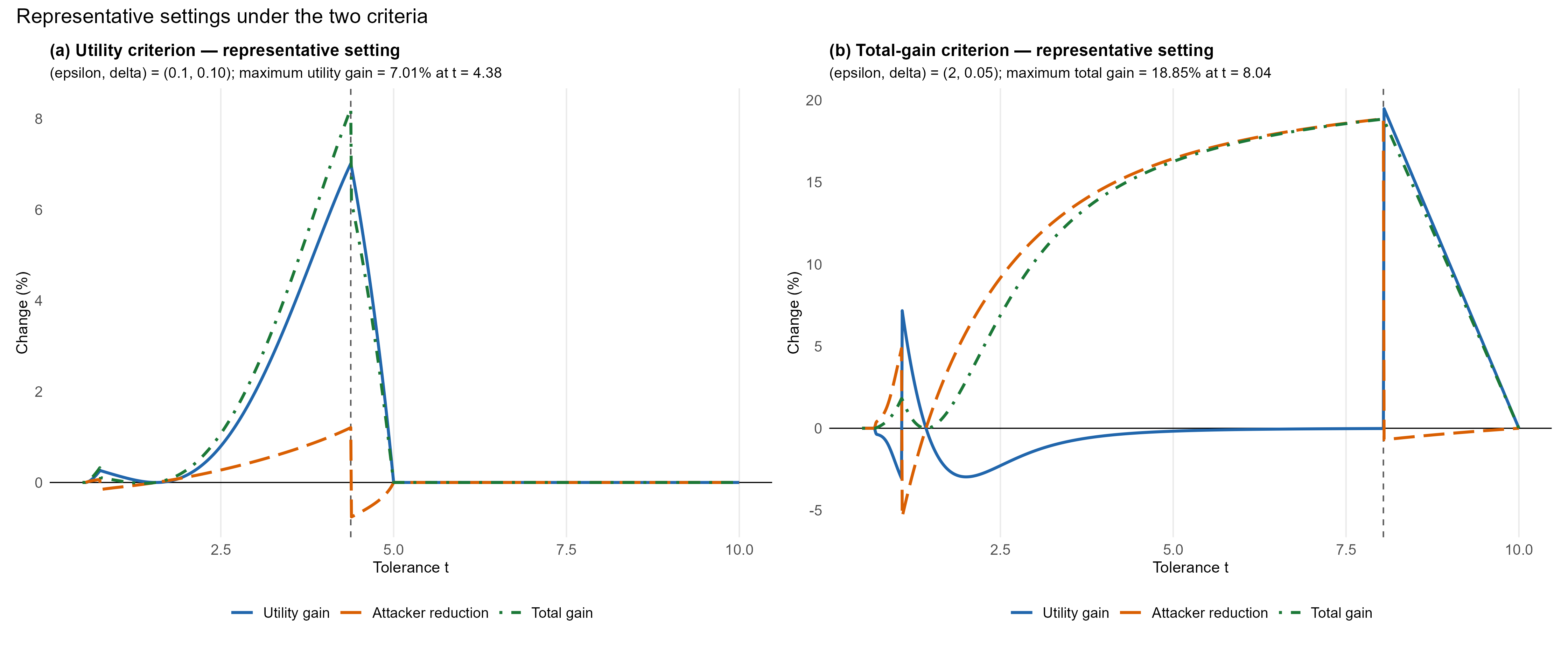}
\captionof{figure}{Representative utility, attacker-reduction, and total-gain
curves under the two criteria. The dashed vertical lines mark the tolerance
values maximizing utility gain in panel (a) and total gain in panel (b).}
\label{fig:task-representative}
\end{minipage}
\par\medskip
Overall, continuous shape optimization can outperform the fixed choices
$p\in\{1,2,\infty\}$ under the same privacy guarantee. The preferred shape
depends on the objective and the privacy parameters, and shape optimization
can be used for task-level criteria involving utility, attacker success, or both.
\section{Conclusion}
\label{sec:Conclusion}
The properties of \(p \mapsto b(p)\) provide the theoretical basis for a certified interval-wise shape search algorithm. We also prove several properties of the optimal shape, including its invariance under rescaling of the sensitivity vector and its limit behaviour as one or both privacy parameters approach zero.

The main conclusion is that, under appropriate utility criteria, the same differential privacy guarantees can be achieved with lower noise by selecting the shape together with its corresponding scale. For the second absolute moment, numerical results show that for some of the \((\varepsilon,\delta)\) pairs, the variance is reduced by approximately 5\% to 20\%, and in some cases, the reduction exceeds 20\%. Experiments on task-specific objectives further show that the value of shape optimization is not limited to variance reduction. The same approach can be used for broader utility and privacy objectives, including objectives that directly reflect the statistical task or the risk of information leakage. Together, these results show that the choice of noise shape is itself part of the privacy-utility design problem, rather than a fixed choice made in advance.

Future research could focus on reducing the computational cost of the certified interval-wise shape search in high-dimensional cases, where the certified gap may decrease more slowly. Possible extensions include broader classes of additive noise, such as truncated noise distributions. Another direction is to consider more general definitions of differential privacy, such as Rényi differential privacy and \(f\)-differential privacy. Applying generalized Gaussian mechanisms to differentially private machine learning is another potential direction for future research.

\clearpage
\bibliographystyle{plainnat}
\bibliography{references}

\makeatletter
\ifx\@empty\addresses\else
  \@setaddresses
  \global\let\addresses\@empty
\fi
\makeatother

\clearpage
\appendix

\section{Additional Details for the Problem Setup}
\label{app:setup-details}

This appendix records the supplementary information omitted from section~\ref{sec:setup}.
\subsection{Density representation of the hockey-stick divergence}
\label{app:hockey-density}

If \(P\) and \(Q\) have densities \(p(x)\) and \(q(x)\), then
\begin{equation}
\label{eq:hockey-density-paper}
D_\varepsilon(P\|Q)
=
\int_{\mathbb R^d}
\bigl(p(x)-e^\varepsilon q(x)\bigr)_+
\,dx.
\end{equation}

\subsection{Privacy-loss representation}
\label{app:privacy-loss-representation}

For \(1\le p<\infty\), define the privacy-loss function
\begin{equation}
\label{eq:Lambda-paper}
\Lambda_{p,b,h}(z)
:=
\frac{1}{b^p}
\sum_{i=1}^d
\left(
|z_i+h_i|^p-|z_i|^p
\right)
=
\log
\frac{f^{(d)}_{p,b}(z)}
     {f^{(d)}_{p,b}(z+h)}.
\end{equation}
Using \eqref{eq:hockey-density-paper},
\begin{equation}
\label{eq:delta-pbh-integral-paper}
\delta_{p,b}(h)
=
\int_{\mathbb R^d}
\left(
1-\exp\{\varepsilon-\Lambda_{p,b,h}(z)\}
\right)_+
f^{(d)}_{p,b}(z)
\,dz.
\end{equation}

\subsection{Coordinate-wise and scalar sensitivity sets}
\label{app:sensitivity-set-comparison}

For \(p\in[1,\infty]\), the \(\ell_p\)-sensitivity set associated with \(\Delta\) is
\begin{equation}
\label{eq:lp-sensitivity-set-paper}
\mathcal S_p(\Delta)
:=
\{h\in\mathbb R^d:\|h\|_p\le\|\Delta\|_p\}.
\end{equation}
Since
\[
\mathcal S_\Delta
\subseteq
\mathcal S_p(\Delta),
\]
we have
\[
\sup_{h\in\mathcal S_\Delta}
\delta_{p,b}(h)
\le
\sup_{h\in\mathcal S_p(\Delta)}
\delta_{p,b}(h).
\]
If
\[
b_{\mathcal S_p(\Delta)}(p)
:=
\inf\left\{
b>0:
\sup_{h\in\mathcal S_p(\Delta)}
\delta_{p,b}(h)
\le
\delta
\right\},
\]
then
\begin{equation}
\label{eq:coordinate-vs-scalar-scale-paper}
b_{\mathcal S_\Delta}(p)
\le
b_{\mathcal S_p(\Delta)}(p).
\end{equation}
Thus the coordinate-wise set gives a privacy-feasible scale no larger than the scale obtained from the \(\ell_p\)-sensitivity set. In this sense, it gives a tighter conservative bound whenever coordinate-wise sensitivity information is available.

\subsection{Expected-cost formulation}
\label{app:expected-cost-formulation}

For a nonnegative cost function \(\ell:\mathbb R\to[0,\infty)\), define
\begin{equation}
\label{eq:general-expected-cost-paper}
\mathcal C_\ell(p,b)
:=
\mathbb E\ell(Z_1),
\qquad
Z_1\sim GGD(0,b,p).
\end{equation}
The corresponding shape-selection problem is
\begin{equation}
\label{eq:general-cost-shape-selection-paper}
p^\star
\in
\arg\min_{p\in[1,\infty]}
\mathcal C_\ell\bigl(p,b(p)\bigr).
\end{equation}
The scale-homogeneous criteria in Section~\ref{subsec:utility-shape-selection} form the class used in the main theoretical results.

\subsection{Other scale-homogeneous utility criteria}
\label{app:other-utility-criteria}

Let \(Y_p\sim GGD(0,1,p)\).

\noindent\textbf{Absolute-error quantiles.}
For \(\tau\in(0,1)\), let \(q_\tau(p)\) denote the \(\tau\)-quantile of \(|Y_p|\), defined by
\[
\mathbb P\bigl(|Y_p|\le q_\tau(p)\bigr)
=
\tau.
\]
Since \(Z_1=bY_p\), the \(\tau\)-quantile of \(|Z_1|\) is
$
bq_\tau(p).
$
Thus, the \(\tau\)-quantile absolute error is scale-homogeneous with \(r=1\). For \(1\le p<\infty\),
\begin{equation}
\label{eq:q-tau-formula-paper}
q_\tau(p)
=
\left[
P^{-1}\left(1/p,\tau\right)
\right]^{1/p},
\end{equation}
where \(P(a,x)\) is the regularized lower incomplete gamma function and \(P^{-1}(a,\cdot)\) denotes its inverse with respect to the second argument. At \(p=\infty\),
\[
q_\tau(\infty)
=
\tau.
\]
\noindent\textbf{Expected tail error.}
Define
\[
\operatorname{TE}_\tau(p,b)
:=
\mathbb E\bigl[|Z_1|\mid |Z_1|>bq_\tau(p)\bigr].
\]
Then
\[
\operatorname{TE}_\tau(p,b)
=
bs_\tau(p),
\]
where, for \(1\le p<\infty\),
\begin{equation}
\label{eq:s-tau-formula-paper}
s_\tau(p)
=
\frac{
\Gamma\left(2/p,q_\tau(p)^p\right)
}{
(1-\tau)\Gamma(1/p)
}.
\end{equation}
At \(p=\infty\),
\[
s_\tau(\infty)
=
\frac{1+\tau}{2}.
\]
Hence the expected tail error also has the form
\[
\mathcal L(p,b)
=
b\nu(p).
\]

\section{Proofs for Section~\ref{sec:calibration}}
\label{ch:scale-bounds}

\subsection{Proof of Lemma~\ref{lem:phi-monotone}}
\label{lem:unique-zstar}

\begin{proof}
Let
\[
\phi_p(z)=|z+\Delta|^p-|z|^p .
\]
The function is continuous on \(\mathbb R\). If \(z\ge0\), then
\[
\phi_p'(z)=p\bigl[(z+\Delta)^{p-1}-z^{p-1}\bigr]>0.
\]
If \(-\Delta<z<0\), then
\[
\phi_p'(z)=p\bigl[(z+\Delta)^{p-1}+(-z)^{p-1}\bigr]>0.
\]
If \(z\le-\Delta\), write \(s=-z\ge\Delta\). Then
\[
\phi_p(z)=(s-\Delta)^p-s^p,
\]
and
\[
\phi_p'(z)=-p\bigl[(s-\Delta)^{p-1}-s^{p-1}\bigr]>0.
\]
Thus \(\phi_p\) is increasing on \(\mathbb R\). Moreover,
\[
\phi_p(-\Delta/2)=0,
\qquad
\lim_{z\to-\infty}\phi_p(z)=-\infty,
\qquad
\lim_{z\to+\infty}\phi_p(z)=+\infty .
\]
Therefore, for every \(b>0\), the equation
\[
\phi_p(z)=\varepsilon b^p
\]
has a unique solution \(z^\star(b)\in\mathbb R\).
\end{proof}

\subsection{Proof of Theorem~\ref{thm:1d-tight-calibration}}
\label{thm:1d-ggm-tight}

\begin{proof}
For \(p>1\), Lemma~\ref{lem:phi-monotone} implies that, for every \(b>0\), the equation
\[
|z+\Delta|^p-|z|^p=\varepsilon b^p
\]
has a unique solution \(z^\star(b)\). Since
\[
\log\frac{f_{p,b}(z)}{f_{p,b}(z+\Delta)}
=
\frac{|z+\Delta|^p-|z|^p}{b^p},
\]
the positive part in the hockey-stick divergence is supported on \([z^\star(b),\infty)\). Hence
\[
\delta_{p,\Delta}(b)
=
\int_{z^\star(b)}^\infty
\bigl[
f_{p,b}(z)-e^\varepsilon f_{p,b}(z+\Delta)
\bigr]\,dz .
\]
Since \(z^\star(b)>-\Delta/2\), one has \(z^\star(b)+\Delta>0\). Therefore
\[
\int_{z^\star(b)}^\infty f_{p,b}(z+\Delta)\,dz
=
\frac{1}{2\Gamma(1/p)}
\Gamma\left(
\frac1p,
\left(\frac{|z^\star(b)+\Delta|}{b}\right)^p
\right).
\]
Also,
\[
\int_{z^\star(b)}^\infty f_{p,b}(z)\,dz
=
\mathbf 1_{\{z^\star(b)<0\}}
+
\frac{1-2\mathbf 1_{\{z^\star(b)<0\}}}{2\Gamma(1/p)}
\Gamma\left(
\frac1p,
\left(\frac{|z^\star(b)|}{b}\right)^p
\right).
\]
Combining the two tail expressions gives the second equation in \eqref{eq:1d-tight-system-paper} once \(b=b^\star\), \(z^\star=z^\star(b^\star)\), and \(s^\star=\mathbf 1_{\{z^\star<0\}}\).

It remains to prove uniqueness. By the implicit function theorem, \(b\mapsto z^\star(b)\) is continuously differentiable. Since the integrand vanishes at \(z=z^\star(b)\), differentiation under the integral sign gives
\[
\frac{d}{db}\delta_{p,\Delta}(b)
=
\int_{z^\star(b)}^\infty
\partial_b
\bigl[
f_{p,b}(z)-e^\varepsilon f_{p,b}(z+\Delta)
\bigr]\,dz .
\]
For the one-dimensional generalized Gaussian density,
\[
\partial_b f_{p,b}(t)
=
-\frac1b\frac{d}{dt}\bigl[t f_{p,b}(t)\bigr].
\]
Hence
\[
\frac{d}{db}\delta_{p,\Delta}(b)
=
\frac1b
\left[
e^\varepsilon (z+\Delta)f_{p,b}(z+\Delta)-z f_{p,b}(z)
\right]_{z=z^\star(b)}^\infty .
\]
The upper boundary is zero. At \(z=z^\star(b)\), the boundary equation gives
\[
f_{p,b}(z^\star(b))
=
e^\varepsilon f_{p,b}(z^\star(b)+\Delta),
\]
and therefore
\[
\frac{d}{db}\delta_{p,\Delta}(b)
=
-\frac{\Delta}{b}f_{p,b}(z^\star(b))<0 .
\]
Thus \(b\mapsto\delta_{p,\Delta}(b)\) is strictly decreasing and continuous on \((0,\infty)\). Moreover,
\[
\lim_{b\downarrow0}\delta_{p,\Delta}(b)=1,
\qquad
\lim_{b\to\infty}\delta_{p,\Delta}(b)=0 .
\]
For every \(0<\delta<1\), there is a unique \(b^\star>0\) satisfying
\[
\delta_{p,\Delta}(b^\star)=\delta .
\]
Let \(z^\star=z^\star(b^\star)\). Then \((z^\star,b^\star)\) is the unique solution of \eqref{eq:1d-tight-system-paper}. Since \(\delta_{p,\Delta}(b)\) is strictly decreasing, the mechanism is \((\varepsilon,\delta)\)-differentially private if and only if \(b\ge b^\star\).
\end{proof}

\subsection{Proof of Proposition~\ref{prop:laplace-calibration}}
\label{app:proof-laplace-scale}

\begin{proof}
For \(p=1\),
\[
f_{1,b}(z)=\frac{1}{2b}e^{-|z|/b},
\qquad
\log\frac{f_{1,b}(z)}{f_{1,b}(z+\Delta)}
=
\frac{|z+\Delta|-|z|}{b}.
\]
If \(b\ge\Delta/\varepsilon\), then the privacy loss is at most \(\varepsilon\) for all \(z\), so \(\delta_{1,\Delta}(b)=0\). If \(b<\Delta/\varepsilon\), the boundary equation
\[
|z+\Delta|-|z|=\varepsilon b
\]
has the solution
\[
z^\star=\frac{\varepsilon b-\Delta}{2}\in(-\Delta,0).
\]
Thus
\[
\delta_{1,\Delta}(b)
=
\mathbb P(Z\ge z^\star)
-
e^\varepsilon \mathbb P(Z\ge z^\star+\Delta).
\]
Since \(z^\star<0<z^\star+\Delta\),
\[
\mathbb P(Z\ge z^\star)=1-\frac12 e^{z^\star/b},
\qquad
\mathbb P(Z\ge z^\star+\Delta)=\frac12 e^{-(z^\star+\Delta)/b}.
\]
Using
\[
e^\varepsilon e^{-(z^\star+\Delta)/b}=e^{z^\star/b},
\]
we obtain
\[
\delta_{1,\Delta}(b)
=
1-\exp\left(\frac{\varepsilon b-\Delta}{2b}\right).
\]
Combining the two cases,
\[
\delta_{1,\Delta}(b)
=
\left[
1-\exp\left(\frac{\varepsilon b-\Delta}{2b}\right)
\right]_+ .
\]
Solving \(\delta_{1,\Delta}(b)\le\delta\) gives
\[
b
\ge
\frac{\Delta}{\varepsilon-2\log(1-\delta)}.
\]
For \(\delta=0\), this reduces to \(b\ge\Delta/\varepsilon\).
\end{proof}

\subsection{Proof of Proposition~\ref{prop:gaussian-scalar-coordinate-profile}}
\label{app:sensitivity-geometry}
\begin{proof}
For \(p=2\),
\[
f^{(d)}_{2,b}(z)
=
C_{2,b,d}\exp\left(-\frac{\|z\|_2^2}{b^2}\right).
\]
This density is invariant under orthogonal transformations. Choose an orthogonal matrix \(R\) such that
\[
R\Delta=\|\Delta\|_2e_1 .
\]
The change of variables \(z=R^{-1}y\) gives
\[
\delta^{(d)}_{2,b}(\varepsilon;\Delta)
=
\delta^{(d)}_{2,b}(\varepsilon;\|\Delta\|_2e_1).
\]
Let \(g_{2,b}=f^{(1)}_{2,b}\) denote the one-dimensional marginal density. Since the shift \(\|\Delta\|_2e_1\) affects only the first coordinate, the remaining coordinates integrate to one. Hence
\[
\delta^{(d)}_{2,b}(\varepsilon;\Delta)
=
\int_{\mathbb R}
\left(
g_{2,b}(z)-e^\varepsilon g_{2,b}(z+\|\Delta\|_2)
\right)_+\,dz .
\]
Therefore the Gaussian expression depends on \(\Delta\) only through \(\|\Delta\|_2/b\). Solving the privacy constraint in \(b\) gives
\[
b_2(\varepsilon,\delta;\Delta)
=
\|\Delta\|_2 b_2(\varepsilon,\delta;e_1),
\]
where \(e_1=(1,0,\ldots,0)\) represents a unit shift in one coordinate.

We now show that this reduction is special to \(p=2\). Let \(f\) be a positive density on \(\mathbb R^d\) that is even in each coordinate, and define
\[
\ell_\Delta(z)=\frac{f(z)}{f(z+\Delta)} .
\]
Suppose that two shifts \(\Delta\) and \(\Delta'\) give the same hockey-stick divergence for every \(r\ge0\):
\[
\int_{\mathbb R^d}
\bigl(f(z)-e^r f(z+\Delta)\bigr)_+\,dz
=
\int_{\mathbb R^d}
\bigl(f(z)-e^r f(z+\Delta')\bigr)_+\,dz .
\]
Let \(\gamma=e^r\ge1\). Under the density \(f(z+\Delta)\),
\[
\int_{\mathbb R^d}
\bigl(f(z)-e^r f(z+\Delta)\bigr)_+\,dz
=
\mathbb E_\Delta\bigl[(\ell_\Delta-\gamma)_+\bigr],
\]
where \(\mathbb E_\Delta\) denotes expectation with respect to \(f(z+\Delta)\). Hence equality of the hockey-stick divergences for all \(r\ge0\) determines the upper tail of \(\ell_\Delta\) above \(1\).

The lower tail is also determined. For \(0<a\le1\), coordinate-wise evenness of \(f\) and the change of variables \(y=-z-\Delta\) give
\[
\mathbb E_\Delta\bigl[(a-\ell_\Delta)_+\bigr]
=
a\int_{\mathbb R^d}
\bigl(f(y)-a^{-1}f(y+\Delta)\bigr)_+\,dy .
\]
Therefore \(\ell_\Delta\) and \(\ell_{\Delta'}\) have the same distribution under their respective shifted densities. In particular,
\[
\mathbb E_\Delta\bigl[\ell_\Delta\log\ell_\Delta\bigr]
=
\mathbb E_{\Delta'}\bigl[\ell_{\Delta'}\log\ell_{\Delta'}\bigr].
\]
Now fix \(1\le p<\infty\), \(p\ne2\), and take \(f=f^{(d)}_{p,1}\). Since \(\mathbb E_\Delta\) is taken under the density \(f(z+\Delta)\), we have
\[
\mathbb E_\Delta\bigl[\ell_\Delta\log\ell_\Delta\bigr]
=
\int_{\mathbb R^d}
\ell_\Delta(z)\log\ell_\Delta(z) f(z+\Delta)\,dz
=
\int_{\mathbb R^d}
f(z)\log\ell_\Delta(z)\,dz .
\]
Thus equality of the hockey-stick divergences for all \(r\ge0\) implies equality of
\[
\mathbb E\bigl[\log\ell_\Delta(Z)\bigr],
\qquad
Z\sim f^{(d)}_{p,1}.
\]
For this density,
\[
\log\ell_\Delta(Z)
=
\sum_{i=1}^d
\bigl(|Z_i+\Delta_i|^p-|Z_i|^p\bigr).
\]
Therefore equality of the hockey-stick divergences implies equality of
\[
\sum_{i=1}^d h_p(|\Delta_i|),
\]
where
\[
h_p(u)=\mathbb E|X+u|^p-\mathbb E|X|^p,
\qquad u\ge0,
\]
and \(X\) has density proportional to \(\exp(-|x|^p)\).
By symmetry,
\[
h_p(u)=c_pu^2+o(u^2),
\qquad u\downarrow0,
\]
for some \(c_p>0\). Also,
\[
h_p(u)=u^p(1+o(1)),
\qquad u\to\infty .
\]
Fix \(q\in[1,\infty)\). For \(t>0\), let
\[
\Delta_t=te_1,
\qquad
\Delta_t'=t\,2^{-1/q}(e_1+e_2).
\]
Then
\[
\|\Delta_t\|_q=\|\Delta_t'\|_q=t .
\]
If the hockey-stick divergence were determined only by the scalar \(\ell_q\)-norm of the shift, then
\[
h_p(t)=2h_p(t\,2^{-1/q})
\qquad
\text{for all }t>0.
\]
Letting \(t\downarrow0\) gives
\[
1=2^{1-2/q},
\]
so \(q=2\). Letting \(t\to\infty\) gives
\[
1=2^{1-p/q},
\]
so \(q=p\). These two identities are compatible only when \(p=2\), contradicting \(p\ne2\).
For \(q=\infty\), take
\[
\Delta_t=te_1,
\qquad
\Delta_t'=t(e_1+e_2).
\]
Then
\[
\|\Delta_t\|_\infty=\|\Delta_t'\|_\infty=t,
\]
but equality of the hockey-stick divergences would imply
\[
h_p(t)=2h_p(t),
\]
which is impossible for \(t>0\).

It remains to consider \(p=\infty\). In this case the density is uniform on \([-b,b]^d\). A direct overlap calculation gives
\[
\delta^{(d)}_{\infty,b}(\varepsilon;\Delta)
=
1-
\prod_{i=1}^d
\left(1-\frac{|\Delta_i|}{2b}\right)_+ .
\]
This expression depends on how the coordinates of \(\Delta\) are arranged. For finite \(q\), the shifts
\[
te_1
\quad\text{and}\quad
t\,2^{-1/q}(e_1+e_2)
\]
have the same \(\ell_q\)-norm but give different overlap volumes for \(0<t<2b\). Similarly, \(te_1\) and \(t(e_1+e_2)\) have the same \(\ell_\infty\)-norm but give different overlap volumes.

Therefore, among generalized Gaussian shapes, \(p=2\) is the only case in which the dependence on \(\Delta\) can always be summarized by \(\|\Delta\|_2\). For \(p\ne2\), the coordinate-wise allocation of sensitivity generally affects the hockey-stick divergence and the corresponding privacy-feasible scale.
\end{proof}
\subsection{Proof of Lemma~\ref{lem:worstcase-box-paper}}
\label{app:worstcase-box-proof}

\begin{proof}
Fix \(p\), \(b\), and \(\varepsilon\). For \(\xi=(\xi_1,\ldots,\xi_d)\in\{\pm1\}^d\), set \(\tilde h=(\xi_1h_1,\ldots,\xi_dh_d)\). Since \(f^{(d)}_{p,b}\) is even in each coordinate, the change of variables \(z_i=\xi_i u_i\) gives
\[
\delta_{p,b}(h)=\delta_{p,b}(\tilde h).
\]
Taking the signs so that \(\tilde h_i=|h_i|\), we obtain
\[
\delta_{p,b}(h)=\delta_{p,b}(|h_1|,\ldots,|h_d|).
\]
It remains to prove coordinate-wise monotonicity on \([0,\infty)^d\). Suppose \(h,t\in[0,\infty)^d\) differ only in coordinate \(j\), with \(0\le h_j\le t_j\) and \(h_i=t_i\) for \(i\ne j\). Write \(z=(x,y)\), where \(x=z_j\) and \(y=(z_i)_{i\ne j}\). For fixed \(y\), the inner integral in coordinate \(j\) has the form
\[
J(u)=\int_{\mathbb R}\left(a f_{p,b}(x)-c f_{p,b}(x+u)\right)_+\,dx,
\qquad u\ge0,
\]
where \(a=\prod_{i\ne j}f_{p,b}(y_i)\) and \(c=e^\varepsilon\prod_{i\ne j}f_{p,b}(y_i+h_i)\).
Since \((r-s)_+=r-\min\{r,s\}\), we have
\[
J(u)=a-\int_{\mathbb R}\min\{a f_{p,b}(x),c f_{p,b}(x+u)\}\,dx.
\]
The density \(f_{p,b}\) is even and unimodal. By the layer-cake representation, the integral of the minimum can be written in terms of overlaps of centered intervals. These overlaps are nonincreasing as the separation \(u\) increases on \([0,\infty)\). Therefore
\[
u\mapsto \int_{\mathbb R}\min\{a f_{p,b}(x),c f_{p,b}(x+u)\}\,dx
\]
is nonincreasing, and hence \(J(u)\) is nondecreasing.

Applying this argument to one coordinate at a time gives
\[
0\le h_i\le t_i\quad\text{for all }i
\qquad\Longrightarrow\qquad
\delta_{p,b}(h)\le\delta_{p,b}(t).
\]
Together with sign invariance,
\[
|h_i|\le |t_i|\quad\text{for all }i
\qquad\Longrightarrow\qquad
\delta_{p,b}(h)\le\delta_{p,b}(t).
\]
Thus, for every \(h\in\mathcal S_\Delta\),
\[
\delta_{p,b}(h)\le\delta_{p,b}(\Delta).
\]
Since \(\Delta\in\mathcal S_\Delta\), it follows that
\[
\sup_{h\in\mathcal S_\Delta}\delta_{p,b}(h)=\delta_{p,b}(\Delta).
\]
\end{proof}
\subsection{Feasibility of the initial upper scale}
\label{app:initial-upper-scale}

We first verify \eqref{eq:fixed-shape-sign}. For \(p>1\), the map
\(b\mapsto\delta_{p,b}(\Delta)\) is continuous and decreasing by the argument in Appendix~\ref{app:proof-scale-function-properties}, so
\[
b<b(p)
\Longrightarrow
\delta_{p,b}(\Delta)>\delta,
\qquad
b>b(p)
\Longrightarrow
\delta_{p,b}(\Delta)<\delta.
\]
For \(p=1\), let \(Z\sim f_{1,1}^{(d)}\). A change of variables gives
\[
\delta_{1,b}(\Delta)
=
\mathbb E\left[
\left(
1-\exp\left\{
\varepsilon-
\left(
\left\|Z+\frac{\Delta}{b}\right\|_1-\|Z\|_1
\right)
\right\}
\right)_+
\right].
\]
For fixed \(Z\), the map
\[
u\longmapsto \|Z+u\Delta\|_1-\|Z\|_1
\]
is convex and equals zero at \(u=0\). Hence \(b\mapsto\delta_{1,b}(\Delta)\) is nonincreasing. On \(\{Z\in[0,\infty)^d\}\), which has probability \(2^{-d}>0\),
\[
\left\|Z+\frac{\Delta}{b}\right\|_1-\|Z\|_1
=
\frac{\|\Delta\|_1}{b}.
\]
Therefore \(b\mapsto\delta_{1,b}(\Delta)\) is strictly decreasing for \(b<\|\Delta\|_1/\varepsilon\). Since
\[
\delta_{1,b}(\Delta)=0
\qquad
\text{for }
b\ge\frac{\|\Delta\|_1}{\varepsilon},
\]
and \(0<\delta<1\), \(b(1)<\|\Delta\|_1/\varepsilon\). Thus \eqref{eq:fixed-shape-sign} also holds for \(p=1\).

We now verify the feasibility of the initial upper scale. For \(p=1\), the choice
\[
b_{\mathrm{high}}^{(0)}
=
\frac{\|\Delta\|_1}{\varepsilon}
\]
is feasible because the privacy loss is bounded above by
\(\|\Delta\|_1/b_{\mathrm{high}}^{(0)}=\varepsilon\). Hence the corresponding hockey-stick divergence is zero.

Assume now that \(p>1\), and let \(Z\sim f^{(d)}_{p,b}\). Then
\[
\frac{\|Z\|_p^p}{b^p}
\sim
\mathrm{Gamma}\left(\frac{d}{p},1\right).
\]
The standard gamma-tail bound gives, with
\[
\kappa
=
\frac{d}{p}
+
\sqrt{\frac{2d}{p}\log\frac1\delta}
+
\log\frac1\delta,
\]
that
\[
\mathbb P\left\{
\frac{\|Z\|_p^p}{b^p}>\kappa
\right\}
\le
\delta.
\]
On the event \(\|Z\|_p^p/b^p\le\kappa\), the inequality
\[
|x+y|^p-|x|^p
\le
p2^{p-1}\left(|x|^{p-1}|y|+|y|^p\right)
\]
and Hölder's inequality imply
\[
\frac{\|Z+\Delta\|_p^p-\|Z\|_p^p}{b^p}
\le
p2^{p-1}
\left(
\frac{\|\Delta\|_p\kappa^{(p-1)/p}}{b}
+
\frac{\|\Delta\|_p^p}{b^p}
\right).
\]
Thus, for
\[
b
\ge
\max\left\{
\frac{p\,2^p}{\varepsilon}
\|\Delta\|_p\,\kappa^{(p-1)/p},
\left(
\frac{p\,2^p}{\varepsilon}
\right)^{1/p}
\|\Delta\|_p
\right\},
\]
the privacy loss
\[
\Lambda_{p,b,\Delta}(Z)
=
\frac{\|Z+\Delta\|_p^p-\|Z\|_p^p}{b^p}
\]
satisfies \(\Lambda_{p,b,\Delta}(Z)\le\varepsilon\) whenever
\(\|Z\|_p^p/b^p\le\kappa\). Hence
\[
\mathbb P\{\Lambda_{p,b,\Delta}(Z)>\varepsilon\}
\le
\mathbb P\left\{
\frac{\|Z\|_p^p}{b^p}>\kappa
\right\}
\le
\delta.
\]
Since
\[
\delta_{p,b}(\Delta)
=
\mathbb E\left[
\left(
1-\exp\{\varepsilon-\Lambda_{p,b,\Delta}(Z)\}
\right)_+
\right]
\le
\mathbb P\{\Lambda_{p,b,\Delta}(Z)>\varepsilon\},
\]
we obtain
\[
\delta_{p,b}(\Delta)\le\delta.
\]
In particular, the endpoint \(b_{\mathrm{high}}^{(0)}\) used in Section~\ref{subsec:fixed-shape-upper-approximation} is privacy-feasible.
\section{Proofs for Section~\ref{sec:optimization}}

\subsection{Differentiability and limiting behaviour of \(b(p)\)}
\label{app:proof-scale-function-properties}

\begin{proof}[Proof of Theorem~\ref{thm:scale-function-properties}]
For \(p>1\), set \(\phi_p(z):=\|z\|_p^p\). Then \(\phi_p\in C^1(\mathbb R^d)\) is strictly convex and
\[
\nabla\phi_p(z)=\left(p\,\operatorname{sgn}(z_i)|z_i|^{p-1}\right)_{i=1}^d.
\]
Hence, for \(x\ne y\),
\[
\langle \nabla\phi_p(x)-\nabla\phi_p(y),x-y\rangle>0.
\]
Let
\[
g_p(z):=\left(\frac{p}{2\Gamma(1/p)}\right)^d e^{-\|z\|_p^p},
\qquad
\Phi(p,u):=\int_{\mathbb R^d}\left(g_p(z)-e^\varepsilon g_p(z+u\Delta)\right)_+\,dz .
\]
The change of variables \(x=bz\) gives
\[
\delta_\Delta(p,b)=\Phi(p,1/b).
\]
Fix \(p>1\) and \(u>0\). Choose \(V\in\mathbb R^{d\times(d-1)}\) such that \(A=[\Delta\ V]\) is invertible, and write \(z=t\Delta+Vy\). For fixed \(y\), define
\[
R_{p,u,y}(t):=\phi_p(Vy+(t+u)\Delta)-\phi_p(Vy+t\Delta)-\varepsilon .
\]
Then
\[
R'_{p,u,y}(t)=\langle \nabla\phi_p(Vy+(t+u)\Delta)-\nabla\phi_p(Vy+t\Delta),\Delta\rangle>0,
\]
and
\[
\lim_{t\to-\infty}R_{p,u,y}(t)=-\infty,\qquad
\lim_{t\to+\infty}R_{p,u,y}(t)=+\infty.
\]
Thus there is a unique \(t^\ast=t^\ast(p,u,y)\) such that \(R_{p,u,y}(t^\ast)=0\), and
\[
g_p(Vy+t\Delta)>e^\varepsilon g_p(Vy+(t+u)\Delta)
\quad\Longleftrightarrow\quad
t>t^\ast .
\]
Therefore
\[
\Phi(p,u)=|\det A|\int_{\mathbb R^{d-1}}\int_{t^\ast(p,u,y)}^\infty
\left(g_p(Vy+t\Delta)-e^\varepsilon g_p(Vy+(t+u)\Delta)\right)\,dt\,dy .
\]
Differentiating in \(u\), the boundary term vanishes, and dominated convergence gives
\[
\partial_u\Phi(p,u)=|\det A|e^\varepsilon\int_{\mathbb R^{d-1}}g_p(Vy+(t^\ast(p,u,y)+u)\Delta)\,dy>0.
\]
Hence \(u\mapsto\Phi(p,u)\) is increasing, and \(b\mapsto\delta_\Delta(p,b)=\Phi(p,1/b)\) is strictly decreasing.
Moreover,
\[
0\le\left(g_p(z)-e^\varepsilon g_p(z+u\Delta)\right)_+\le g_p(z).
\]
As \(u\downarrow0\),
\[
\left(g_p(z)-e^\varepsilon g_p(z+u\Delta)\right)_+\to0,
\]
and as \(u\to\infty\),
\[
g_p(z+u\Delta)\to0,\qquad
\left(g_p(z)-e^\varepsilon g_p(z+u\Delta)\right)_+\to g_p(z).
\]
Thus, by dominated convergence,
\[
\lim_{u\downarrow0}\Phi(p,u)=0,\qquad
\lim_{u\to\infty}\Phi(p,u)=1.
\]
Equivalently,
\[
\lim_{b\to\infty}\delta_\Delta(p,b)=0,\qquad
\lim_{b\downarrow0}\delta_\Delta(p,b)=1.
\]
Therefore, for each \(p\in(1,\infty)\), there is a unique \(b(p)\in(0,\infty)\) satisfying
\[
\delta_\Delta(p,b(p))=\delta.
\]
It remains to justify differentiability in \(p\). On compact subsets of \((1,\infty)\times(0,\infty)\), the functions
\[
g_p(z),\quad g_p(z+u\Delta),\quad \partial_p g_p(z),\quad \partial_p g_p(z+u\Delta),\quad \nabla g_p(z+u\Delta)
\]
are dominated by an integrable envelope. For each \(y\in\mathbb R^{d-1}\),
\[
\{t\in\mathbb R:g_p(Vy+t\Delta)=e^\varepsilon g_p(Vy+(t+u)\Delta)\}
\]
contains at most one point. Therefore,
\[
\mathcal L^d\{z:g_p(z)=e^\varepsilon g_p(z+u\Delta)\}=0.
\]
This implies \(\Phi\in C^1((1,\infty)\times(0,\infty))\), with
\[
\partial_p\Phi(p,u)=\int_{\{g_p(z)>e^\varepsilon g_p(z+u\Delta)\}}
\left(\partial_p g_p(z)-e^\varepsilon\partial_p g_p(z+u\Delta)\right)\,dz.
\]
Since \(\partial_u\Phi(p,u)>0\), the implicit function theorem applied to
\[
\Phi(p,u(p))=\delta,\qquad u(p)=1/b(p),
\]
gives
\[
u'(p)=-\frac{\partial_p\Phi(p,u)}{\partial_u\Phi(p,u)}\bigg|_{u=u(p)},\qquad
b'(p)=-\frac{u'(p)}{u(p)^2}.
\]
This proves the assertion on \(p\in(1,\infty)\). The behaviour when \(p=1\) and the limit when \(p\to\infty\) are proved in Propositions~\ref{prop:right-diff-p1} and~\ref{prop:bp-to-binfty}.
\end{proof}
\begin{proposition}[Continuity and right differentiability at \(p=1\)]
\label{prop:right-diff-p1}
Let \(u(p)=1/b(p)\), and define
\[
\mathcal E
:=
\left\{
u>0:
u\langle s,\Delta\rangle=\varepsilon
\text{ for some }s\in\{-1,1\}^d
\right\}.
\]
Then
\[
u(p)\longrightarrow u(1),
\qquad
b(p)\longrightarrow b(1),
\qquad p\downarrow1.
\]
If \(u(1)\notin\mathcal E\), then
\[
u'_+(1)
=
-\frac{\partial_{p+}\Phi(1,u(1))}
        {\partial_u\Phi(1,u(1))},
\qquad
b'_+(1)
=
-\frac{u'_+(1)}{u(1)^2}.
\]
\end{proposition}

\begin{proof}
We first prove continuity at \(p=1\), without the exceptional-set
condition. 
For each fixed \(u>0\),
$
g_p(z)\to g_1(z),
\
g_p(z+u\Delta)\to g_1(z+u\Delta)
$
as \(p\downarrow1\).
These densities are bounded by an integrable function for \(p\) sufficiently
close to \(1\). Hence, by dominated convergence,
$
\Phi(p,u)\to\Phi(1,u),
\ p\downarrow1.
$
Let \(u_1:=u(1)\). Since \(u_1\) is the unique solution of
\[
\Phi(1,u)=\delta,
\]
choose
\[
0<u_-<u_1<u_+
\]
such that
\[
\Phi(1,u_-)<\delta<\Phi(1,u_+).
\]
By the continuity established above, for all \(p>1\) sufficiently close
to \(1\),
\[
\Phi(p,u_-)<\delta<\Phi(p,u_+).
\]
Since \(u\mapsto\Phi(p,u)\) is strictly increasing for \(p>1\), it follows
that
\[
u_-<u(p)<u_+.
\]
If \(p_n\downarrow1\) and \(u(p_n)\to u_\ast\) along a subsequence, then
the same dominated convergence argument gives
\[
\delta
=
\lim_{n\to\infty}\Phi(p_n,u(p_n))
=
\Phi(1,u_\ast).
\]
By uniqueness of the solution at \(p=1\), \(u_\ast=u_1\). Therefore
\(u(p)\to u(1)\), and hence \(b(p)=1/u(p)\to b(1)\).
We now assume \(u(1)\notin\mathcal E\). At \(p=1\),
\[
g_1(z)=c_1e^{-\|z\|_1},
\]
and
\[
g_1(z)=e^\varepsilon g_1(z+u\Delta)
\quad\Longleftrightarrow\quad
\|z+u\Delta\|_1-\|z\|_1=\varepsilon.
\]
On each region with fixed signs of \(z\) and \(z+u\Delta\), the boundary
equation is affine in \(z\). It can have nonempty interior only if
\[
u\langle s,\Delta\rangle=\varepsilon
\]
for some \(s\in\{-1,1\}^d\). Hence \(u(1)\notin\mathcal E\) implies
\[
\mathcal L^d
\left\{
z:
g_1(z)=e^\varepsilon g_1(z+u(1)\Delta)
\right\}
=0.
\]

On compact neighborhoods of \((1,u(1))\), the relevant densities and
their right \(p\)-derivatives are dominated by an integrable envelope of
the form
\[
C\left(
1+\|z\|_1^{1+\eta}\log(2+\|z\|_1)
\right)e^{-a\|z\|_1}.
\]
Therefore differentiation under the integral sign gives
\[
\partial_{p+}\Phi(1,u(1))
=
\int_{\{g_1(z)>e^\varepsilon g_1(z+u(1)\Delta)\}}
\left[
\partial_{p+}g_1(z)
-
e^\varepsilon\partial_{p+}g_1(z+u(1)\Delta)
\right]dz.
\]
Using
\[
\Phi(p,u(p))=\Phi(1,u(1))=\delta
\]
and the mean value theorem in \(u\), we obtain
\[
\partial_{p+}\Phi(1,u(1))
+
\partial_u\Phi(1,u(1))u'_+(1)
=0.
\]
Hence
\[
u'_+(1)
=
-\frac{\partial_{p+}\Phi(1,u(1))}
        {\partial_u\Phi(1,u(1))}.
\]
The formula for \(b'_+(1)\) follows from \(b(p)=1/u(p)\).
\end{proof}

\begin{proposition}[Limit as \(p\to\infty\)]
\label{prop:bp-to-binfty}
Let \(b(\infty)\) be defined by uniform noise on \([-b,b]^d\). Then
\[
\lim_{p\to\infty}b(p)=b(\infty),
\]
and
\[
\delta_\Delta(\infty,b)=1-\prod_{i=1}^d\left(1-\frac{\Delta_i}{2b}\right)_+.
\]
\end{proposition}

\begin{proof}
For \(Q_b=[-b,b]^d\),
\[
f_{\infty,b}^{(d)}(z)=\frac1{(2b)^d}\mathbf 1_{Q_b}(z).
\]
Since \(e^\varepsilon\ge1\),
\[
\left(f_{\infty,b}^{(d)}(z)-e^\varepsilon f_{\infty,b}^{(d)}(z+\Delta)\right)_+
=
\frac1{(2b)^d}\mathbf 1_{Q_b\setminus(Q_b-\Delta)}(z).
\]
Therefore
\[
\delta_\Delta(\infty,b)=1-\frac{\operatorname{Vol}(Q_b\cap(Q_b-\Delta))}{(2b)^d}
=
1-\prod_{i=1}^d\left(1-\frac{\Delta_i}{2b}\right)_+.
\]
Thus \(b(\infty)\) is the unique solution of
\[
\delta_\Delta(\infty,b)=\delta
\]
on \((\max_i\Delta_i/2,\infty)\).
For fixed \(b>0\),
\[
f_{p,b}^{(d)}(z)\to f_{\infty,b}^{(d)}(z)
\quad\text{for a.e. }z.
\]
Fix \(p_0>1\). For \(p\ge p_0\), the normalizing constants are bounded and
\[
e^{-|x/b|^p}
\le
\mathbf 1_{\{|x|\le2b\}}
+
e^{-(|x|/b)^{p_0}}\mathbf 1_{\{|x|>2b\}}.
\]
This gives an integrable dominating function, also for the shifted density. Hence
\[
\delta_\Delta(p,b)\to\delta_\Delta(\infty,b)
\qquad(p\to\infty).
\]
Let \(\eta>0\) satisfy \(b(\infty)-\eta>\max_i\Delta_i/2\). Since \(b\mapsto\delta_\Delta(\infty,b)\) is strictly decreasing at \(b(\infty)\),
\[
\delta_\Delta(\infty,b(\infty)-\eta)>\delta>\delta_\Delta(\infty,b(\infty)+\eta).
\]
By pointwise convergence, the same inequalities hold for \(\delta_\Delta(p,\cdot)\) for all sufficiently large \(p\). Since \(b\mapsto\delta_\Delta(p,b)\) is strictly decreasing,
\[
b(\infty)-\eta<b(p)<b(\infty)+\eta.
\]
Thus \(b(p)\to b(\infty)\).
\end{proof}

\subsection{Derivative bounds for the scale equation}
\label{app:derivative-bounds-section4}

\subsubsection{Integral representation of \(\partial_p H(p,b)\)}
\label{app:derivative-representation}
Let
\[
a:=\frac{\Delta}{b},
\qquad
h_p^{(d)}(x)
:=
\left(\frac{p}{2\Gamma(1/p)}\right)^d
\exp\left(-\sum_{i=1}^d |x_i|^p\right).
\]
For \(x\in\mathbb R^d\), define
\[
\Lambda_{p,a}(x)
:=
\sum_{i=1}^d\left(|x_i+a_i|^p-|x_i|^p\right),
\qquad
w_{p,a}(x)
:=
\left(1-e^{\varepsilon-\Lambda_{p,a}(x)}\right)_+ .
\]
Then
\[
H(p,b)+\delta
=
\int_{\mathbb R^d}
w_{p,a}(x)h_p^{(d)}(x)\,dx .
\]
Let
\[
\varphi_p(x)
:=
\begin{cases}
|x|^p\log |x|, & x\ne0,\\
0, & x=0.
\end{cases}
\]
Then
\[
\partial_p\Lambda_{p,a}(x)
=
\sum_{i=1}^d
\left(\varphi_p(x_i+a_i)-\varphi_p(x_i)\right),
\]
and
\[
\partial_p h_p^{(d)}(x)
=
S_p(x)h_p^{(d)}(x),
\]
where
\[
S_p(x)
=
d\left(\frac1p+\frac{\psi(1/p)}{p^2}\right)
-
\sum_{i=1}^d\varphi_p(x_i).
\]
Set
\[
G_{p,a}(x)
:=
e^{\varepsilon-\Lambda_{p,a}(x)}
\mathbf 1_{\{\Lambda_{p,a}(x)>\varepsilon\}}
\partial_p\Lambda_{p,a}(x)
+
w_{p,a}(x)S_p(x).
\]
\begin{lemma}
\label{lem:chap4-derivative-new}
For \(p\in(1,\infty)\),
\[
\partial_pH(p,b)
=
\int_{\mathbb R^d}
G_{p,a}(x)h_p^{(d)}(x)\,dx .
\]
\end{lemma}

\begin{proof}
It remains to justify differentiating
\[
\int_{\mathbb R^d}w_{p,a}(x)h_p^{(d)}(x)\,dx
\]
with respect to \(p\). Let \(V\in\mathbb R^{d\times(d-1)}\) be such that
\(A=[a\ V]\) is invertible, and write
\[
x=Vy+ta,
\qquad
y\in\mathbb R^{d-1},
\quad
t\in\mathbb R.
\]
For fixed \(y\), define
\[
R_y(t)
:=
\Lambda_{p,a}(Vy+ta)
=
\|Vy+(t+1)a\|_p^p-\|Vy+ta\|_p^p.
\]
Since \(p>1\), \(x\mapsto\|x\|_p^p\) is strictly convex. Hence
\[
R_y'(t)
=
\left\langle
\nabla\|Vy+(t+1)a\|_p^p
-
\nabla\|Vy+ta\|_p^p,
a
\right\rangle
>0.
\]
Thus
\[
\mathcal L^1\{t\in\mathbb R:R_y(t)=\varepsilon\}=0,
\qquad y\in\mathbb R^{d-1}.
\]
Consequently,
\[
\mathcal L^d\{x:\Lambda_{p,a}(x)=\varepsilon\}
=
|\det A|
\int_{\mathbb R^{d-1}}
\mathcal L^1\{t:R_y(t)=\varepsilon\}\,dy
=
0.
\]
On compact subsets of \((1,\infty)\), the functions
\[
w_{p,a}(x)h_p^{(d)}(x),
\qquad
\partial_p\{w_{p,a}(x)h_p^{(d)}(x)\}
\]
are dominated by an integrable envelope. Therefore differentiation under the integral is valid. Since
\[
\partial_p h_p^{(d)}(x)=S_p(x)h_p^{(d)}(x),
\]
and, outside the null set \(\{\Lambda_{p,a}=\varepsilon\}\),
\[
\partial_p w_{p,a}(x)
=
e^{\varepsilon-\Lambda_{p,a}(x)}
\mathbf 1_{\{\Lambda_{p,a}(x)>\varepsilon\}}
\partial_p\Lambda_{p,a}(x),
\]
we obtain
\[
\partial_p\{w_{p,a}(x)h_p^{(d)}(x)\}
=
G_{p,a}(x)h_p^{(d)}(x).
\]
Thus
\[
\partial_pH(p,b)
=
\int_{\mathbb R^d}
G_{p,a}(x)h_p^{(d)}(x)\,dx .
\]
\end{proof}
\subsubsection{One-dimensional tight bound}
\label{app:one-dimensional-tight-bound}

Assume \(d=1\), \(\Delta>0\), and set \(a:=\Delta/b\). Then
\[
\Lambda_{p,a}(x)=|x+a|^p-|x|^p,\qquad x\in\mathbb R .
\]
For \(p>1\), \(\Lambda_{p,a}\) is increasing, and hence there is a unique
\(x^\star_{p,a}\in\mathbb R\) such that
\[
\Lambda_{p,a}(x^\star_{p,a})=\varepsilon .
\]
Set
\[
A_{p,a}:=|x^\star_{p,a}|^p,
\qquad
s_{p,a}:=\mathbf 1_{\{x^\star_{p,a}<0\}}.
\]
Then
\[
1-2s_{p,a}
=
\begin{cases}
1,&x^\star_{p,a}\ge0,\\
-1,&x^\star_{p,a}<0.
\end{cases}
\]
For \(\varepsilon>0\),
\[
x^\star_{p,a}\ge0\Longleftrightarrow a^p\le\varepsilon,
\qquad
x^\star_{p,a}<0\Longleftrightarrow a^p>\varepsilon,
\]
while for \(\varepsilon=0\), \(x^\star_{p,a}<0\).
Let
\[
Q(q,t):=\frac{\Gamma(q,t)}{\Gamma(q)}
\]
be the regularized upper incomplete gamma function. With \(q=1/p\),
\begin{equation}
\label{eq:one-dimensional-delta-gamma}
\delta_\Delta(p,b)
=
s_{p,a}
+
\frac{1-2s_{p,a}}2 Q(q,A_{p,a})
-
\frac{e^\varepsilon}{2}Q(q,A_{p,a}+\varepsilon).
\end{equation}
Define
\[
D_q(v)
:=
\int_v^\infty
\frac{q}{2\Gamma(q)}y^{q-1}e^{-y}
\left(1+q\psi(q)-y\log y\right)\,dy .
\]
Differentiating \eqref{eq:one-dimensional-delta-gamma} in \(p\) gives
\begin{equation}
\label{eq:one-d-H-derivative}
\partial_pH(p,b)
=
(1-2s_{p,a})D_{1/p}(A_{p,a})
-
e^\varepsilon D_{1/p}(A_{p,a}+\varepsilon).
\end{equation}
The boundary terms cancel because
\[
|x^\star_{p,a}+a|^p-|x^\star_{p,a}|^p=\varepsilon,
\qquad
|x^\star_{p,a}+a|^p=A_{p,a}+\varepsilon .
\]
Let \(D_q'\) denote the derivative of \(D_q\) with respect to its argument.
Then
\[
\left\|D_q'\right\|_{L^1(0,\infty)}
=
\int_0^\infty
\left|
\frac{q}{2\Gamma(q)}y^{q-1}e^{-y}
\left(1+q\psi(q)-y\log y\right)
\right|\,dy .
\]
From~\eqref{eq:one-d-H-derivative},
\[
|\partial_pH(p,b)|
\le
C_{p,a,\varepsilon}
\left\|D_{1/p}'\right\|_{L^1(0,\infty)},
\]
where
\[
C_{p,a,\varepsilon}
:=
\begin{cases}
\max\{1,e^\varepsilon-1\},&a^p\le\varepsilon,\\
1+e^\varepsilon,&a^p>\varepsilon.
\end{cases}
\]
For \(J=[p_-,p_+]\subset(1,\infty)\), put
\[
q_-:=\frac1{p_+},
\qquad
q_+:=\frac1{p_-}.
\]
Since
\[
1+q\psi(q)=q\psi(q+1),
\]
and, for \(Y\sim\Gamma(q,1)\),
\[
\mathbb E\{Y^2(\log Y)^2\}
=
q(q+1)\{\psi_1(q+2)+\psi(q+2)^2\},
\]
Cauchy--Schwarz gives
\[
\left\|D_q'\right\|_{L^1(0,\infty)}
\le
\frac{q^2}{2}|\psi(q+1)|
+
\frac q2
\left[
q(q+1)\{\psi_1(q+2)+\psi(q+2)^2\}
\right]^{1/2}.
\]
Hence, uniformly for \(p\in J\),
\[
\left\|D_{1/p}'\right\|_{L^1(0,\infty)}
\le V(J),
\]
where
\[
V(J)
:=
\frac{q_+^2}{2}
\max\{|\psi(q_-+1)|,|\psi(q_++1)|\}
+
\frac{q_+}{2}
\left[
q_+(q_++1)
\left\{
\psi_1(q_-+2)
+
\max\{|\psi(q_-+2)|,|\psi(q_++2)|\}^2
\right\}
\right]^{1/2}.
\]
Let
\[
I=[p_-,p_+]\subset[1,\infty),
\qquad
p_-<p_+.
\]
First suppose \(p_->1\). If \(\varepsilon>0\), \(a\neq1\), and
\[
p_a:=\frac{\log\varepsilon}{\log a}\in(p_-,p_+),
\]
split \(I\) at \(p_a\); otherwise keep \(I\) unsplit.
Let \(\mathcal P(I,a)\) be the resulting collection of one or two
subintervals. Define
\[
C(J):=
\begin{cases}
\max\{1,e^\varepsilon-1\},
& a^p\le\varepsilon \text{ for all }p\in J,\\[1mm]
1+e^\varepsilon,
& \text{otherwise}.
\end{cases}
\]
Then set
\[
S_{\mathrm{1D}}(I,b)
:=
\max_{J\in\mathcal P(I,a)}C(J)V(J).
\]
Now suppose \(p_-=1\). Put
\[
q_-:=\frac{1}{p_+},
\qquad
q_+:=1,
\]
and define \(V(I)\) by the same expression above with these values of
\(q_-\) and \(q_+\). Since \(q_->0\), \(V(I)\)
is finite, and the preceding bound remains valid for
\(q\in[q_-,1)\). Hence
\[
|\partial_pH(p,b)|
\le
(1+e^\varepsilon)
\|D'_{1/p}\|_{L^1(0,\infty)}
\le
(1+e^\varepsilon)V(I),
\qquad 1<p\le p_+.
\]
For intervals with \(p_-=1\), set
\[
S_{\mathrm{1D}}(I,b):=(1+e^\varepsilon)V(I).
\]
In either case,
\[
S_{\mathrm{1D}}(I,b)
\ge
\sup_{p\in I\cap(1,\infty)}
|\partial_pH(p,b)|.
\]
\subsubsection{Multidimensional upper bound}
\label{app:proof-derivative-bound}
Let
\[
I=[p_-,p_+]\subset[1,P_{\max}],
\qquad
p_-<p_+,
\]
and 
\[
a:=\frac{\Delta}{b},
\qquad
t:=\frac{R}{b},
\qquad
Q_t:=[-t,t]^d.
\]
For \(p\in I\cap(1,\infty)\),
\begin{equation}
\label{eq:bulk-tail-split-derivative}
|\partial_pH(p,b)|
\le
\int_{Q_t}|G_{p,a}(x)|h_p^{(d)}(x)\,dx
+
\int_{\mathbb R^d\setminus Q_t}
|G_{p,a}(x)|h_p^{(d)}(x)\,dx .
\end{equation}
For \(s\ge0\), define
\[
\pi_p(s)
:=
\frac{p}{\Gamma(1/p)}
\int_s^\infty e^{-x^p}\,dx,
\qquad
\pi_I(s)
:=
\sup_{p\in I}
\left\{1-(1-\pi_p(s))^d\right\},
\]
\[
I_{q,\nu}(s)
:=
\int_s^\infty x^qe^{-x^\nu}\,dx,
\qquad
J_{q,\nu}(s)
:=
\int_s^\infty x^q\log x\,e^{-x^\nu}\,dx,
\]
\[
\alpha_I
:=
\sup_{p\in I}\frac{p}{\Gamma(1/p)},
\qquad
\kappa_I
:=
\sup_{p\in I}
\left|
\frac1p+\frac{\psi(1/p)}{p^2}
\right|,
\]
and
\[
\Psi_I(s)
:=
\sup_{p\in I}
\sup_{0\le x\le s}
x^p|\log x|,
\]
where \(x^p|\log x|=0\) at \(x=0\). For \(a_i=\Delta_i/b\), set
\[
B_i(I,s)
:=
\Psi_I(s+a_i)+\Psi_I(s),
\qquad
\mathcal S_I(s)
:=
\alpha_IJ_{p_+,p_-}(s),
\]
and, for \(s>0\),
\[
\mathcal D_{I,i}(s)
:=
\alpha_I
\left(1+\frac{a_i}{s}\right)^{p_+-1}
\left[
\left(1+p_+\log2\right)I_{p_+-1,p_-}(s)
+
p_+J_{p_+-1,p_-}(s)
\right].
\]
Assume
\[
R\ge\max\{b,2\|\Delta\|_\infty\}.
\]
Then \(t\ge1\) and \(t\ge2a_i\) for all \(i\).
Since
\[
0\le w_{p,a}(x)\le1,
\qquad
e^{\varepsilon-\Lambda_{p,a}(x)}
\mathbf 1_{\{\Lambda_{p,a}(x)>\varepsilon\}}
\le1,
\]
we have
\begin{equation}
\label{eq:G-simple-envelope-section4}
|G_{p,a}(x)|
\le
|\partial_p\Lambda_{p,a}(x)|+|S_p(x)|.
\end{equation}
Moreover,
\[
|S_p(x)|
\le
d\kappa_I+\sum_{i=1}^d|\varphi_p(x_i)|,
\]
and
\[
|\partial_p\Lambda_{p,a}(x)|
\le
\sum_{i=1}^d
|\varphi_p(x_i+a_i)-\varphi_p(x_i)|.
\]
Let \(X=(X_1,\ldots,X_d)\sim h_p^{(d)}\). Then the coordinates are independent and
\[
\mathbb P_p(|X_i|>s)=\pi_p(s).
\]
Thus
\[
\mathbb P_p(X\notin Q_s)
=
1-(1-\pi_p(s))^d
\le
\pi_I(s),
\qquad p\in I.
\]
For \(s\ge1\),
\[
\int_{\{|X_i|>s\}}
|\varphi_p(X_i)|\,d\mathbb P_p
=
\frac{p}{\Gamma(1/p)}
\int_s^\infty x^p\log x\,e^{-x^p}\,dx,
\]
and therefore
\[
\sup_{p\in I}
\int_{\{|X_i|>s\}}
|\varphi_p(X_i)|\,d\mathbb P_p
\le
\alpha_IJ_{p_+,p_-}(s)
=
\mathcal S_I(s).
\]
For \(|x|>s\), \(s\ge1\), and \(s\ge2a_i\), the mean value theorem gives
\[
|\varphi_p(x+a_i)-\varphi_p(x)|
\le
a_i
\sup_{0\le r\le a_i}
|x+r|^{p-1}\left|1+p\log|x+r|\right|.
\]
Since
\[
|x+r|\le |x|+a_i\le \left(1+\frac{a_i}{s}\right)|x|,
\qquad
|x+r|\ge |x|-a_i\ge \frac{|x|}{2},
\]
we have
\[
|x+r|^{p-1}
\le
\left(1+\frac{a_i}{s}\right)^{p_+-1}|x|^{p_+-1}
\]
and
\[
\left|1+p\log|x+r|\right|
\le
1+p_+\log2+p_+\log|x|.
\]
Consequently,
\[
|\varphi_p(x+a_i)-\varphi_p(x)|
\le
a_i
\left(1+\frac{a_i}{s}\right)^{p_+-1}
|x|^{p_+-1}
\left(1+p_+\log2+p_+\log|x|\right).
\]
Hence
\[
\sup_{p\in I}
\int_{\{|X_i|>s\}}
|\varphi_p(X_i+a_i)-\varphi_p(X_i)|\,d\mathbb P_p
\le
a_i\mathcal D_{I,i}(s).
\]
On \(\mathbb R^d\setminus Q_s\), split
\[
|\varphi_p(X_i)|
=
|\varphi_p(X_i)|\mathbf 1_{\{|X_i|\le s\}}
+
|\varphi_p(X_i)|\mathbf 1_{\{|X_i|>s\}}.
\]
Thus
\[
\sup_{p\in I}
\int_{\mathbb R^d\setminus Q_s}
|\varphi_p(x_i)|h_p^{(d)}(x)\,dx
\le
\Psi_I(s)\pi_I(s)+\mathcal S_I(s).
\]
Similarly, if \(|x_i|\le s\), then
\[
|\varphi_p(x_i+a_i)-\varphi_p(x_i)|
\le
B_i(I,s),
\]
and hence
\[
\sup_{p\in I}
\int_{\mathbb R^d\setminus Q_s}
|\varphi_p(x_i+a_i)-\varphi_p(x_i)|h_p^{(d)}(x)\,dx
\le
B_i(I,s)\pi_I(s)+a_i\mathcal D_{I,i}(s).
\]
Combining these inequalities with \eqref{eq:G-simple-envelope-section4}, at \(s=t\),
\begin{equation}
\label{eq:tail-envelope-section4}
\sup_{p\in I}
\int_{\mathbb R^d\setminus Q_t}
|G_{p,a}(x)|h_p^{(d)}(x)\,dx
\le
T(I,b,R),
\end{equation}
where
\[
T(I,b,R)
:=
\left(
d\kappa_I
+
d\Psi_I(t)
+
\sum_{i=1}^dB_i(I,t)
\right)\pi_I(t)
+
d\mathcal S_I(t)
+
\sum_{i=1}^d a_i\mathcal D_{I,i}(t).
\]
It remains to bound the integral over \(Q_t\). Let \(\mathcal Q\) be a finite partition of \(Q_t\) into boxes
\[
Q=\prod_{k=1}^dJ_k.
\]
For \(Q\in\mathcal Q\), define
\[
\underline\Lambda_Q(p)
:=
\sum_{k=1}^d
\inf_{x\in J_k}
\left(|x+a_k|^p-|x|^p\right),
\]
\[
\overline\Lambda_Q(p)
:=
\sum_{k=1}^d
\sup_{x\in J_k}
\left(|x+a_k|^p-|x|^p\right),
\]
\[
L_\Lambda(Q;p)
:=
\sup_{x\in Q}
|\partial_p\Lambda_{p,a}(x)|,
\qquad
L_S(Q;p)
:=
\sup_{x\in Q}|S_p(x)|,
\]
and
\[
M_Q(p)
:=
\int_Qh_p^{(d)}(x)\,dx.
\]
Set
\[
\Xi_Q(p)
:=
\begin{cases}
0,
&
\overline\Lambda_Q(p)\le\varepsilon,
\\
e^{\varepsilon-\underline\Lambda_Q(p)},
&
\underline\Lambda_Q(p)>\varepsilon,
\\
1,
&
\underline\Lambda_Q(p)\le\varepsilon<\overline\Lambda_Q(p).
\end{cases}
\]
For \(x\in Q\),
\[
|\partial_pw_{p,a}(x)|
\le
\Xi_Q(p)L_\Lambda(Q;p),
\]
and
\[
0\le w_{p,a}(x)
\le
\left(1-e^{\varepsilon-\overline\Lambda_Q(p)}\right)_+.
\]
Indeed, if \(\overline\Lambda_Q(p)\le\varepsilon\), then \(w_{p,a}=0\) on \(Q\). If
\(\underline\Lambda_Q(p)>\varepsilon\), then
\[
|\partial_pw_{p,a}(x)|
=
e^{\varepsilon-\Lambda_{p,a}(x)}
|\partial_p\Lambda_{p,a}(x)|
\le
e^{\varepsilon-\underline\Lambda_Q(p)}L_\Lambda(Q;p).
\]
In the remaining case,
\[
e^{\varepsilon-\Lambda_{p,a}(x)}
\mathbf 1_{\{\Lambda_{p,a}(x)>\varepsilon\}}
\le1.
\]
Therefore
\[
|\partial_pw_{p,a}(x)|
\le
\Xi_Q(p)L_\Lambda(Q;p).
\]
Also, since \(\Lambda_{p,a}(x)\le\overline\Lambda_Q(p)\),
\[
w_{p,a}(x)
=
\left(1-e^{\varepsilon-\Lambda_{p,a}(x)}\right)_+
\le
\left(1-e^{\varepsilon-\overline\Lambda_Q(p)}\right)_+.
\]
Thus
\begin{equation}
\label{eq:boxwise-bulk-bound-section4}
\int_Q|G_{p,a}(x)|h_p^{(d)}(x)\,dx
\le
\left[
\Xi_Q(p)L_\Lambda(Q;p)
+
\left(1-e^{\varepsilon-\overline\Lambda_Q(p)}\right)_+L_S(Q;p)
\right]M_Q(p).
\end{equation}
Define
\[
S_{\rm bulk}(\mathcal Q;I,b,R)
:=
\sup_{p\in I}
\sum_{Q\in\mathcal Q}
\left[
\Xi_Q(p)L_\Lambda(Q;p)
+
\left(1-e^{\varepsilon-\overline\Lambda_Q(p)}\right)_+L_S(Q;p)
\right]M_Q(p).
\]
Then
\[
\sup_{p\in I}
\int_{Q_t}|G_{p,a}(x)|h_p^{(d)}(x)\,dx
\le
S_{\rm bulk}(\mathcal Q;I,b,R).
\]
Combining \eqref{eq:bulk-tail-split-derivative},
\eqref{eq:tail-envelope-section4}, and
\eqref{eq:boxwise-bulk-bound-section4}, define
\[
S_{\rm multi}(I,b;\mathcal Q,R)
:=
S_{\rm bulk}(\mathcal Q;I,b,R)+T(I,b,R).
\]
Then
\[
S_{\mathrm{multi}}(I,b;\mathcal Q,R)
\ge
\sup_{p\in I\cap(1,\infty)}|\partial_pH(p,b)|.
\]
For any finite admissible family
\(\mathcal A\subset\{(\mathcal Q,R):\mathcal Q \text{ is a finite partition of }Q_{R/b},\ R\ge\max\{b,2\|\Delta\|_\infty\}\}\), define
\[
S_{\mathrm{multi},\mathcal A}(I,b)
:=\min_{(\mathcal Q,R)\in\mathcal A}S_{\mathrm{multi}}(I,b;\mathcal Q,R).
\]
Then \(S_{\mathrm{multi},\mathcal A}(I,b)\ge\sup_{p\in I\cap(1,\infty)}|\partial_pH(p,b)|\). Hence, in Section~4, one may take \(S(I,b)=S_{\mathrm{1D}}(I,b)\) for \(d=1\) and \(S(I,b)=S_{\mathrm{multi},\mathcal A}(I,b)\) for \(d\ge2\).
\section{Proofs for Section~\ref{sec:optimal-shape-properties}}
\label{app:optimal-shape-properties-proofs}
\subsection{Proof of Proposition~\ref{prop:sensitivity-scaling-active-coordinates}}
\label{app:sensitivity-scaling-active-proof}
For \(k\ge1\), let \(f_{p,b}^{(k)}\) denote the \(k\)-dimensional generalized Gaussian density. For every \(p\in[1,\infty]\), \(b>0\), and
\(z\in\mathbb R^k\),
\[
f_{p,b}^{(k)}(z)
=
b^{-k}f_{p,1}^{(k)}(z/b).
\]
First let \(c>0\). For \(z=cx\), \(dz=c^d\,dx\), and
\[
f_{p,cb}^{(d)}(cx)=c^{-d}f_{p,b}^{(d)}(x),
\qquad
f_{p,cb}^{(d)}(c(x+\Delta))=c^{-d}f_{p,b}^{(d)}(x+\Delta).
\]
Hence
\[
\begin{aligned}
\delta_{p,cb}^{(d)}(c\Delta)
&=
\int_{\mathbb R^d}
\left(
f_{p,cb}^{(d)}(z)
-
e^\varepsilon f_{p,cb}^{(d)}(z+c\Delta)
\right)_+\,dz \\
&=
\int_{\mathbb R^d}
\left(
c^{-d}f_{p,b}^{(d)}(x)
-
e^\varepsilon c^{-d}f_{p,b}^{(d)}(x+\Delta)
\right)_+
c^d\,dx \\
&=
\int_{\mathbb R^d}
\left(
f_{p,b}^{(d)}(x)
-
e^\varepsilon f_{p,b}^{(d)}(x+\Delta)
\right)_+\,dx \\
&=
\delta_{p,b}^{(d)}(\Delta).
\end{aligned}
\]
Thus, for every \(b>0\),
\[
\delta_{p,cb}^{(d)}(c\Delta)
=
\delta_{p,b}^{(d)}(\Delta).
\]
Hence
\[
\begin{aligned}
b^{(d)}(p;c\Delta)
&=
\inf\left\{s>0:\delta_{p,s}^{(d)}(c\Delta)\le\delta\right\} \\
&=
\inf\left\{cb>0:\delta_{p,cb}^{(d)}(c\Delta)\le\delta\right\} \\
&=
c\inf\left\{b>0:\delta_{p,b}^{(d)}(\Delta)\le\delta\right\} \\
&=
c\,b^{(d)}(p;\Delta).
\end{aligned}
\]
Since \(\mathcal L(p,b)=b^r\nu(p)\),
\[
\mathcal L\bigl(p,b^{(d)}(p;c\Delta)\bigr)
=
c^r
\mathcal L\bigl(p,b^{(d)}(p;\Delta)\bigr).
\]
The factor \(c^r\) is positive and independent of \(p\). Hence, for every
\(I\subseteq[1,\infty]\),
\[
\operatorname*{arg\,min}_{p\in I}
\mathcal L\bigl(p,b^{(d)}(p;c\Delta)\bigr)
=
\operatorname*{arg\,min}_{p\in I}
\mathcal L\bigl(p,b^{(d)}(p;\Delta)\bigr).
\]
It remains to prove that zero-sensitivity coordinates do not affect the
hockey-stick divergence. Let
$
A:=\{i:\Delta_i>0\}, 
d':=|A|,
$
and let \(\Delta_A\in(0,\infty)^{d'}\) be the vector of nonzero coordinates.
After a permutation of coordinates, write
$
\Delta=(\Delta_A,0),
z=(z_A,z_0)\in\mathbb R^{d'}\times\mathbb R^{d-d'}.
$
The product density factorizes as
\[
f_{p,b}^{(d)}(z)
=
f_{p,b}^{(d')}(z_A)f_{p,b}^{(d-d')}(z_0),
\]
and, because the inactive coordinates have zero shift,
\[
f_{p,b}^{(d)}(z+\Delta)
=
f_{p,b}^{(d')}(z_A+\Delta_A)f_{p,b}^{(d-d')}(z_0).
\]
Thus
\[
\begin{aligned}
\delta_{p,b}^{(d)}(\Delta)
&=
\int_{\mathbb R^{d'}}
\int_{\mathbb R^{d-d'}}
\Big[
\bigl(
f_{p,b}^{(d')}(z_A)
-
e^\varepsilon f_{p,b}^{(d')}(z_A+\Delta_A)
\bigr)
f_{p,b}^{(d-d')}(z_0)
\Big]_+
\,dz_0\,dz_A .
\end{aligned}
\]
Since \(f_{p,b}^{(d-d')}(z_0)\ge0\),
\[
\Big[
\bigl(
f_{p,b}^{(d')}(z_A)
-
e^\varepsilon f_{p,b}^{(d')}(z_A+\Delta_A)
\bigr)
f_{p,b}^{(d-d')}(z_0)
\Big]_+
=
\bigl(
f_{p,b}^{(d')}(z_A)
-
e^\varepsilon f_{p,b}^{(d')}(z_A+\Delta_A)
\bigr)_+
f_{p,b}^{(d-d')}(z_0).
\]
Therefore,
\[
\begin{aligned}
\delta_{p,b}^{(d)}(\Delta)
&=
\int_{\mathbb R^{d'}}
\bigl(
f_{p,b}^{(d')}(z_A)
-
e^\varepsilon f_{p,b}^{(d')}(z_A+\Delta_A)
\bigr)_+
\left[
\int_{\mathbb R^{d-d'}}
f_{p,b}^{(d-d')}(z_0)\,dz_0
\right]dz_A \\
&=
\int_{\mathbb R^{d'}}
\bigl(
f_{p,b}^{(d')}(z_A)
-
e^\varepsilon f_{p,b}^{(d')}(z_A+\Delta_A)
\bigr)_+
\,dz_A \\
&=
\delta_{p,b}^{(d')}(\Delta_A).
\end{aligned}
\]
Hence
\[
\delta_{p,b}^{(d)}(\Delta)\le\delta
\quad\Longleftrightarrow\quad
\delta_{p,b}^{(d')}(\Delta_A)\le\delta,
\]
Taking infima over \(b>0\) yields
\[
b^{(d)}(p;\Delta)=b^{(d')}(p;\Delta_A).
\]
Consequently, for every \(I\subseteq[1,\infty]\),
\[
\operatorname*{arg\,min}_{p\in I}
\mathcal L\bigl(p,b^{(d)}(p;\Delta)\bigr)
=
\operatorname*{arg\,min}_{p\in I}
\mathcal L\bigl(p,b^{(d')}(p;\Delta_A)\bigr).
\]
This proves the proposition.
\subsection{Proof of Theorem~\ref{thm:small-delta-p-to-one}}
\label{app:small-delta-proof}
For any \(\rho>0\), it suffices to prove that \(\inf_{p\in[1+\rho,\infty]} L_m\bigl(p,b_\delta(p;\Delta)\bigr)\to\infty\) as \(\delta\downarrow0\), while \(L_m\bigl(1,b_\delta(1;\Delta)\bigr)\) remains bounded.

We first consider the one-dimensional case. For \(\Delta>0\), \(1\le p<\infty\), let
\[
f_{p,b}^{(1)}(z)=\frac{p}{2b\Gamma(1/p)}\exp\left\{-\left|\frac{z}{b}\right|^p\right\},
\]
and \(f_{\infty,b}^{(1)}\) be the uniform density on \([-b,b]\). Set
\[
\delta_{p,b}^{(1)}(\Delta)=\int_{\mathbb R}\left(f_{p,b}^{(1)}(z)-e^\varepsilon f_{p,b}^{(1)}(z+\Delta)\right)_+\,dz,
\qquad
b_\delta^{(1)}(p;\Delta)=\inf\{b>0:\delta_{p,b}^{(1)}(\Delta)\le\delta\}.
\]
We claim that
\begin{equation}
\label{eq:small-delta-1d-scale-diverges}
\inf_{p\in[1+\rho,\infty]} b_\delta^{(1)}(p;\Delta)\to\infty
\qquad
(\delta\downarrow0).
\end{equation}
Assume that \eqref{eq:small-delta-1d-scale-diverges} is false. Then there exist \(B<\infty\), \(\delta_n\downarrow0\), \(p_n\in[1+\rho,\infty]\), and \(b_n\le B\) such that \(\delta_{p_n,b_n}^{(1)}(\Delta)\le\delta_n\). Define
\[
A_n=\left\{z:f_{p_n,b_n}^{(1)}(z)\ge e^{\varepsilon+1}f_{p_n,b_n}^{(1)}(z+\Delta)\right\}.
\]
On \(A_n\),
\[
e^\varepsilon f_{p_n,b_n}^{(1)}(z+\Delta)\le e^{-1}f_{p_n,b_n}^{(1)}(z).
\]
Thus, with \(Z_n\sim f_{p_n,b_n}^{(1)}\),
\[
\begin{aligned}
\delta_n
&\ge
\delta_{p_n,b_n}^{(1)}(\Delta) \\
&\ge
\int_{A_n}
\left(f_{p_n,b_n}^{(1)}(z)-e^\varepsilon f_{p_n,b_n}^{(1)}(z+\Delta)\right)\,dz \\
&\ge
(1-e^{-1})\int_{A_n}f_{p_n,b_n}^{(1)}(z)\,dz \\
&=
(1-e^{-1})\mathbb P(Z_n\in A_n).
\end{aligned}
\]
Hence \(\mathbb P(Z_n\in A_n)\to0\).
Passing to a subsequence, either
\[
\text{(I)}\quad p_n\to\bar p\in[1+\rho,\infty),
\qquad\text{or}\qquad
\text{(II)}\quad p_n\to\infty .
\]
In case (I), for \(z\ge0\) and \(p<\infty\),
\[
\log\frac{f_{p,b}^{(1)}(z)}{f_{p,b}^{(1)}(z+\Delta)}
=
\frac{(z+\Delta)^p-z^p}{b^p}
\ge
\frac{p\Delta z^{p-1}}{b^p}.
\]
Hence \([t_n,\infty)\subseteq A_n\), where
\[
t_n=\left(\frac{(\varepsilon+1)b_n^{p_n}}{p_n\Delta}\right)^{1/(p_n-1)}.
\]
Since \(p_n\ge1+\rho\) and \(b_n\le B\), the ratios \(t_n/b_n\) are bounded. Therefore, for some \(T<\infty\),
\[
\mathbb P(Z_n\in A_n)\ge \mathbb P(Y_n\ge T),
\qquad
Y_n\sim f_{p_n,1}^{(1)}.
\]
By compactness and positivity of \(f_{p,1}^{(1)}\) near \(\bar p\), the right-hand side is bounded below by a positive constant, contradicting \(\mathbb P(Z_n\in A_n)\to0\).\\
In case (II), let
\[
J_n=\left[\max\{0,b_n-\Delta/4\},\,b_n\right].
\]
For finite \(p_n\) and \(z\in J_n\),
\[
\log\frac{f_{p_n,b_n}^{(1)}(z)}{f_{p_n,b_n}^{(1)}(z+\Delta)}
=
\frac{(z+\Delta)^{p_n}-z^{p_n}}{b_n^{p_n}}.
\]
If \(b_n\ge\Delta/4\), then
\[
\frac{(z+\Delta)^{p_n}-z^{p_n}}{b_n^{p_n}}
\ge
\left(1+\frac{3\Delta}{4b_n}\right)^{p_n}-1
\ge
\left(1+\frac{3\Delta}{4B}\right)^{p_n}-1\to\infty.
\]
If \(b_n<\Delta/4\), then
\[
\frac{(z+\Delta)^{p_n}-z^{p_n}}{b_n^{p_n}}
\ge
\left(\frac{\Delta}{b_n}\right)^{p_n}-1\to\infty.
\]
Thus \(J_n\subseteq A_n\) for all large \(n\); for \(p_n=\infty\), the inclusion is immediate. Let
\[
m_0:=\inf_{p\ge1}\frac{p}{\Gamma(1/p)}>0.
\]
For \(0\le z\le b\),
\[
f_{p,b}^{(1)}(z)\ge\frac{m_0e^{-1}}{2b},
\qquad 1\le p\le\infty.
\]
Therefore
\[
\mathbb P(Z_n\in A_n)
\ge
\mathbb P(Z_n\in J_n)
\ge
\frac{m_0e^{-1}}{2}\frac{|J_n|}{b_n}
\ge
\frac{m_0e^{-1}}{2}\min\left\{\frac{\Delta}{4B},1\right\}>0,
\]
again contradicting \(\mathbb P(Z_n\in A_n)\to0\). This proves \eqref{eq:small-delta-1d-scale-diverges}.

Now let \(\Delta\in[0,\infty)^d\setminus\{0\}\). Choose \(j\) with \(\Delta_j>0\). The one-dimensional divergence is obtained by restricting the \(d\)-dimensional supremum to sets depending only on coordinate \(j\). Hence
\[
\delta_{p,b}^{(1)}(\Delta_j)\le \delta_{p,b}^{(d)}(\Delta).
\]
Thus
\[
b_\delta(p;\Delta)\ge b_\delta^{(1)}(p;\Delta_j).
\]
Equation \eqref{eq:small-delta-1d-scale-diverges} then gives
\[
\inf_{p\in[1+\rho,\infty]} b_\delta(p;\Delta)\to\infty.
\]
Since \(M_m\) is continuous and positive on \([1+\rho,\infty]\),
\[
\inf_{p\in[1+\rho,\infty]} L_m\bigl(p,b_\delta(p;\Delta)\bigr)
=
\inf_{p\in[1+\rho,\infty]}M_m(p)b_\delta(p;\Delta)^m
\to\infty.
\]
For \(p=1\), the mechanism with scale \(\|\Delta\|_1/\varepsilon\) satisfies pure \(\varepsilon\)-DP. Hence
\[
b_\delta(1;\Delta)\le\frac{\|\Delta\|_1}{\varepsilon},
\]
and
\[
L_m\bigl(1,b_\delta(1;\Delta)\bigr)
=
M_m(1)b_\delta(1;\Delta)^m
\le
\Gamma(m+1)\left(\frac{\|\Delta\|_1}{\varepsilon}\right)^m.
\]
Thus, for sufficiently small \(\delta\), no minimizer lies in \([1+\rho,\infty]\). Since \(\rho>0\) is arbitrary, \(p^\star(\delta)\to1\) as \(\delta\downarrow0\).
\subsection{Proof of Theorem~\ref{thm:small-eps-p-to-infty}}
\label{app:small-eps-proof}

We first consider the \(\varepsilon=0\) variational bound for the \(m\)-th absolute moment.

\begin{proposition}
\label{prop:eps0-uniform-moment}
Fix \(m>0\), \(\Delta>0\), and
\(\delta\in(0,m/(m+1)]\). Let \(Z\) have an even density \(f\) on
\(\mathbb R\), nonincreasing on \([0,\infty)\), with
\(\mathbb E|Z|^m<\infty\). Let \(P\) be the law of \(Z\), and let
\(P_\Delta\) be the law of \(Z+\Delta\). Among all such densities satisfying
$
\|P-P_\Delta\|_{\mathrm{TV}}\le\delta,
$
the quantity \(\mathbb E|Z|^m\) is uniquely minimized by
\[
Z\sim
\mathrm{Unif}\left[
-\frac{\Delta}{2\delta},
\frac{\Delta}{2\delta}
\right].
\]
The minimum value is
\[
\frac{1}{m+1}
\left(
\frac{\Delta}{2\delta}
\right)^m.
\]
\end{proposition}

\begin{proof}
For an even density nonincreasing on \([0,\infty)\),
\[
\|P-P_\Delta\|_{\mathrm{TV}}
=
\int_{-\Delta/2}^{\Delta/2}f(z)\,dz
=
\mathbb P(|Z|\le\Delta/2).
\]
Take a right-continuous version of \(f\) on \([0,\infty)\), and define a measure \(G\) on \((0,\infty)\) by
\[
G(dt)=-2t\,df(t).
\]
Since \(tf(t)\to0\) as \(t\downarrow0\) and \(t\to\infty\),
\[
G((0,\infty))
=
-2\int_{(0,\infty)}t\,df(t)
=
2\int_0^\infty f(t)\,dt
=
1.
\]
Thus \(G\) is a probability measure. Let \(T\sim G\), and conditionally on \(T=t\), let
\[
Z\mid T=t\sim\mathrm{Unif}[-t,t].
\]
For almost every \(z\),
\[
\int_{[|z|,\infty)}
\frac{1}{2t}\,G(dt)
=
-\int_{[|z|,\infty)}df(t)
=
f(z),
\]
so this mixture has density \(f\). Hence
\[
\mathbb P(|Z|\le\Delta/2)
=
\mathbb E
\min\left\{
1,\frac{\Delta}{2T}
\right\},
\qquad
\mathbb E|Z|^m
=
\frac{1}{m+1}\mathbb ET^m.
\]
Then it is enough to consider
\[
\|P-P_\Delta\|_{\mathrm{TV}}=\delta.
\]
Indeed, if the inequality were strict, then for \(Z_c=cZ\), \(0<c<1\),
\[
\|P_c-(P_c)_\Delta\|_{\mathrm{TV}}
=
\mathbb P\left(
|Z|\le\frac{\Delta}{2c}
\right)
\longrightarrow1
\qquad(c\downarrow0),
\]
while
\[
\mathbb E|Z_c|^m=c^m\mathbb E|Z|^m.
\]
By continuity, some \(c<1\) would satisfy the constraint with equality and have a smaller \(m\)-th absolute moment.
Let
\[
\theta:=\mathbb P(T\le\Delta/2).
\]
Then \(0\le\theta<\delta\) and
\[
\delta
=
\theta
+
\frac{\Delta}{2}
\mathbb E\left[
T^{-1}\mathbf 1_{\{T>\Delta/2\}}
\right],
\]
so
\[
\mathbb E\left[
T^{-1}\mid T>\Delta/2
\right]
=
\frac{\delta-\theta}
{(\Delta/2)(1-\theta)}.
\]
By Jensen's inequality,
\[
\mathbb E\left[
T^m\mid T>\Delta/2
\right]
=
\mathbb E\left[
(T^{-1})^{-m}\mid T>\Delta/2
\right]
\ge
\left(
\mathbb E[T^{-1}\mid T>\Delta/2]
\right)^{-m}.
\]
Therefore
\[
\mathbb ET^m
\ge
\left(\frac{\Delta}{2}\right)^m
\frac{(1-\theta)^{m+1}}{(\delta-\theta)^m}.
\]
Moreover,
\[
\frac{d}{d\theta}
\log
\frac{(1-\theta)^{m+1}}{(\delta-\theta)^m}
=
\frac{m-(m+1)\delta+\theta}
{(1-\theta)(\delta-\theta)}
\ge0,
\]
because \(\delta\le m/(m+1)\). Hence
\[
\mathbb ET^m
\ge
\left(
\frac{\Delta}{2\delta}
\right)^m,
\]
and therefore
\[
\mathbb E|Z|^m
\ge
\frac{1}{m+1}
\left(
\frac{\Delta}{2\delta}
\right)^m.
\]
Equality requires \(\theta=0\) and equality in Jensen's inequality. Thus
\[
T=\frac{\Delta}{2\delta}
\qquad\text{a.s.},
\]
and hence
\[
Z\sim
\mathrm{Unif}\left[
-\frac{\Delta}{2\delta},
\frac{\Delta}{2\delta}
\right].
\]
\end{proof}
For \(1\le p\le\infty\), let \(P_{p,b}\) be the one-dimensional generalized Gaussian law with scale \(b\), and let \(P_{p,b,\Delta}\) be its shift by \(\Delta\). At \(\varepsilon=0\),
\[
D_0(P_{p,b}\|P_{p,b,\Delta})
=
\|P_{p,b}-P_{p,b,\Delta}\|_{\mathrm{TV}}.
\]
By Proposition~\ref{prop:eps0-uniform-moment}, for every finite \(p\),
\[
L_m\bigl(p,b_0(p;\Delta)\bigr)
>
L_m\bigl(\infty,b_0(\infty;\Delta)\bigr)
=
\frac{1}{m+1}
\left(
\frac{\Delta}{2\delta}
\right)^m.
\]
Fix \(P>1\). The map
\[
(p,b,\varepsilon)
\longmapsto
D_\varepsilon(P_{p,b}\|P_{p,b,\Delta})
\]
is continuous on compact subsets of
\([1,P]\times(0,\infty)\times[0,\infty)\). For fixed \(p\) and \(\varepsilon\), it is nonincreasing in \(b\), with
\[
\lim_{b\downarrow0}
D_\varepsilon(P_{p,b}\|P_{p,b,\Delta})
=
1,
\qquad
\lim_{b\to\infty}
D_\varepsilon(P_{p,b}\|P_{p,b,\Delta})
=
0.
\]
Thus
\[
D_\varepsilon
\left(
P_{p,b_\varepsilon(p;\Delta)}
\middle\|
P_{p,b_\varepsilon(p;\Delta),\Delta}
\right)
=
\delta.
\]
At \(\varepsilon=0\), the map is strictly decreasing, so \(b_0(p;\Delta)\) is the unique solution of
\[
D_0(P_{p,b}\|P_{p,b,\Delta})=\delta.
\]
We next show
\[
\sup_{p\in[1,P]}
\left|
b_\varepsilon(p;\Delta)-b_0(p;\Delta)
\right|
\longrightarrow0
\qquad
(\varepsilon\downarrow0).
\]
By continuity of \(p\mapsto b_0(p;\Delta)\), choose \(0<b_-<b_+\) such that
\[
b_-<b_0(p;\Delta)<b_+,
\qquad
p\in[1,P].
\]
Then
\[
D_0(P_{p,b_-}\|P_{p,b_-,\Delta})>\delta,
\qquad
D_0(P_{p,b_+}\|P_{p,b_+,\Delta})<\delta,
\]
uniformly for \(p\in[1,P]\). By continuity, the same inequalities hold for all sufficiently small \(\varepsilon>0\). Hence
\[
b_\varepsilon(p;\Delta)\in[b_-,b_+],
\qquad
p\in[1,P].
\]
Suppose the uniform convergence fails. Then, for some \(\gamma>0\), there exist
\(\varepsilon_n\downarrow0\) and \(p_n\in[1,P]\) such that
\[
\left|
b_{\varepsilon_n}(p_n;\Delta)
-
b_0(p_n;\Delta)
\right|
\ge\gamma.
\]
Passing to a subsequence,
\[
p_n\to p_0\in[1,P],
\qquad
b_{\varepsilon_n}(p_n;\Delta)\to b^\ast\in[b_-,b_+].
\]
Since
\[
D_{\varepsilon_n}
\left(
P_{p_n,b_{\varepsilon_n}(p_n;\Delta)}
\middle\|
P_{p_n,b_{\varepsilon_n}(p_n;\Delta),\Delta}
\right)
=
\delta,
\]
continuity gives
\[
D_0(P_{p_0,b^\ast}\|P_{p_0,b^\ast,\Delta})
=
\delta.
\]
By uniqueness,
\[
b^\ast=b_0(p_0;\Delta).
\]
Also
\[
b_0(p_n;\Delta)\to b_0(p_0;\Delta),
\]
a contradiction. Therefore
\[
\sup_{p\in[1,P]}
\left|
b_\varepsilon(p;\Delta)-b_0(p;\Delta)
\right|
\to0.
\]
Since
\[
L_m(p,b)=M_m(p)b^m
\]
and \(M_m\) is continuous on \([1,P]\),
\[
\sup_{p\in[1,P]}
\left|
L_m\bigl(p,b_\varepsilon(p;\Delta)\bigr)
-
L_m\bigl(p,b_0(p;\Delta)\bigr)
\right|
\longrightarrow0.
\]
For \(p=\infty\),
\[
D_\varepsilon(P_{\infty,b}\|P_{\infty,b,\Delta})
=
\min\left\{1,\frac{\Delta}{2b}\right\},
\]
so
\[
b_\varepsilon(\infty;\Delta)
=
\frac{\Delta}{2\delta},
\qquad
L_m\bigl(\infty,b_\varepsilon(\infty;\Delta)\bigr)
=
L_m\bigl(\infty,b_0(\infty;\Delta)\bigr).
\]
Define
\[
g_P
:=
\min_{p\in[1,P]}
\left\{
L_m\bigl(p,b_0(p;\Delta)\bigr)
-
L_m\bigl(\infty,b_0(\infty;\Delta)\bigr)
\right\}.
\]
By Proposition~\ref{prop:eps0-uniform-moment} and continuity,
\[
g_P>0.
\]
Choose \(\varepsilon_P>0\) such that, for \(0<\varepsilon<\varepsilon_P\),
\[
\sup_{p\in[1,P]}
\left|
L_m\bigl(p,b_\varepsilon(p;\Delta)\bigr)
-
L_m\bigl(p,b_0(p;\Delta)\bigr)
\right|
\le
\frac{g_P}{2}.
\]
Then, for every \(p\in[1,P]\),
\[
\begin{aligned}
L_m\bigl(p,b_\varepsilon(p;\Delta)\bigr)
-
L_m\bigl(\infty,b_\varepsilon(\infty;\Delta)\bigr)
&\ge
L_m\bigl(p,b_0(p;\Delta)\bigr)
-
L_m\bigl(\infty,b_0(\infty;\Delta)\bigr)
-
\frac{g_P}{2}
\ge
\frac{g_P}{2}
>
0.
\end{aligned}
\]
Thus
\[
p^\star(\varepsilon)>P,
\qquad
0<\varepsilon<\varepsilon_P.
\]
Since \(P>1\) is arbitrary,
\[
p^\star(\varepsilon)\longrightarrow\infty
\qquad
(\varepsilon\downarrow0).
\]

\subsection{Proof of Theorem~\ref{thm:joint-high-privacy}}
\label{app:joint-high-privacy-proof}
For \(1\le p<\infty\), let
\[
g_p(x):=\frac{p}{2\Gamma(1/p)}e^{-|x|^p},
\qquad
f_{p,b}(z)=\frac{1}{b}g_p\left(\frac{z}{b}\right),
\]
and let
\[
g_\infty(x):=\frac12\mathbf 1_{\{|x|\le1\}}.
\]
Recall that
\[
L_m(p,b)=M_m(p)b^m,
\]
where
\[
M_m(p)=\frac{\Gamma((m+1)/p)}{\Gamma(1/p)},
\qquad 1\le p<\infty,
\]
and
\[
M_m(\infty)=\frac{1}{m+1}.
\]
In particular, \(M_m\) extends to a positive continuous function on
the interval \([1,\infty]\).
For \(s\ge0\), define
\[
H_{\varepsilon,p}(s)
:=
\int_{\mathbb R}
\bigl(g_p(x)-e^\varepsilon g_p(x+s)\bigr)_+\,dx.
\]
A change of variables gives
\begin{equation}
\label{eq:standardized-hockey-stick}
\delta_{\varepsilon,p,b}^{(1)}(\Delta)
=
H_{\varepsilon,p}\left(\frac{\Delta}{b}\right).
\end{equation}
We first establish the joint first-order behaviour of
\(H_{\varepsilon,p}\) as \(\varepsilon\downarrow0\).

\begin{lemma}
\label{lem:joint-first-order-hockey-stick}
Suppose that
\[
\varepsilon_n\downarrow0,
\qquad
p_n\to p_0\in[1,\infty],
\qquad
\rho_n\to\rho\in(0,\infty).
\]
Then
\begin{equation}
\label{eq:joint-first-order-hockey-stick}
\frac{
H_{\varepsilon_n,p_n}(\varepsilon_n/\rho_n)
}{
\varepsilon_n
}
\longrightarrow h_{p_0}(\rho),
\end{equation}
where
\[
h_1(\rho)=\frac{(1-\rho)_+}{2\rho},
\qquad
h_\infty(\rho)=\frac{1}{2\rho},
\]
and, for \(1<p<\infty\),
\[
h_p(\rho)=\frac{A_p(\rho)}{\rho},
\]
with
\[
A_p(\rho)
=
\int_{(\rho/p)^{1/(p-1)}}^\infty
\left(px^{p-1}-\rho\right)
\frac{p}{2\Gamma(1/p)}e^{-x^p}\,dx.
\]
\end{lemma}

\begin{proof}
We first consider \(1<p_n<\infty\). Put
\[
s_n:=\frac{\varepsilon_n}{\rho_n}.
\]
The function
\[
x\longmapsto |x+s_n|^{p_n}-|x|^{p_n}
\]
is increasing from \(-\infty\) to \(\infty\). Hence there is
a unique \(a_n\in\mathbb R\) such that
\begin{equation}
\label{eq:crossing-point}
|a_n+s_n|^{p_n}-|a_n|^{p_n}=\varepsilon_n.
\end{equation}
The integrand defining \(H_{\varepsilon_n,p_n}(s_n)\) is positive
exactly on \((a_n,\infty)\). Writing
\[
\overline G_p(t):=\int_t^\infty g_p(x)\,dx,
\]
we obtain
\begin{align}
H_{\varepsilon_n,p_n}(s_n)
&=
\overline G_{p_n}(a_n)
-
e^{\varepsilon_n}\overline G_{p_n}(a_n+s_n)
\nonumber\\
&=
\int_{a_n}^{a_n+s_n}g_{p_n}(x)\,dx
-
(e^{\varepsilon_n}-1)
\overline G_{p_n}(a_n+s_n).
\label{eq:hockey-stick-crossing-representation}
\end{align}
Suppose first that \(1<p_0<\infty\). For all sufficiently large \(n\),
the crossing point is positive. By the mean value theorem, there exists
\(\xi_n\in(a_n,a_n+s_n)\) such that
\[
p_n\xi_n^{p_n-1}
=
\frac{(a_n+s_n)^{p_n}-a_n^{p_n}}{s_n}
=
\rho_n.
\]
Therefore
\[
\xi_n=
\left(\frac{\rho_n}{p_n}\right)^{1/(p_n-1)}
\longrightarrow
r_{p_0}(\rho)
:=
\left(\frac{\rho}{p_0}\right)^{1/(p_0-1)}.
\]
Since \(s_n\to0\),
\[
a_n\to r_{p_0}(\rho),
\qquad
a_n+s_n\to r_{p_0}(\rho).
\]
Dividing \eqref{eq:hockey-stick-crossing-representation} by
\(\varepsilon_n=\rho_ns_n\) gives
\[
\frac{H_{\varepsilon_n,p_n}(s_n)}{\varepsilon_n}
=
\frac{1}{\rho_n}
\frac{1}{s_n}
\int_{a_n}^{a_n+s_n}g_{p_n}(x)\,dx
-
\frac{e^{\varepsilon_n}-1}{\varepsilon_n}
\overline G_{p_n}(a_n+s_n).
\]
It follows that
\begin{align*}
\frac{H_{\varepsilon_n,p_n}(s_n)}{\varepsilon_n}
&\longrightarrow
\frac{g_{p_0}(r_{p_0}(\rho))}{\rho}
-
\overline G_{p_0}(r_{p_0}(\rho)).
\end{align*}
Since
\[
g_p'(x)=-px^{p-1}g_p(x),
\qquad x>0,
\]
integration by parts yields
\[
A_p(\rho)
=
g_p(r_p(\rho))
-
\rho\,\overline G_p(r_p(\rho)).
\]
Thus the preceding limit is \(A_{p_0}(\rho)/\rho\).
Suppose next that \(p_n\to\infty\). The mean value argument again gives
\[
\xi_n
=
\left(\frac{\rho_n}{p_n}\right)^{1/(p_n-1)}
\longrightarrow1.
\]
Moreover,
\[
\xi_n^{p_n}
=
\frac{\rho_n\xi_n}{p_n}
\longrightarrow0.
\]
Because \(a_n\le\xi_n\le a_n+s_n\), the crossing equation implies
\[
\sup_{x\in[a_n,a_n+s_n]}x^{p_n}\longrightarrow0.
\]
Also,
\[
\frac{p_n}{2\Gamma(1/p_n)}
=
\frac{1}{2\Gamma(1+1/p_n)}
\longrightarrow\frac12.
\]
Consequently,
\[
\frac{1}{s_n}
\int_{a_n}^{a_n+s_n}g_{p_n}(x)\,dx
\longrightarrow\frac12.
\]
Furthermore, \(a_n+s_n\to1\), and
\[
\overline G_{p_n}(a_n+s_n)\longrightarrow0.
\]
Indeed,
\[
\overline G_{p_n}(1)
\le
\frac{p_n}{2\Gamma(1/p_n)}
\int_1^\infty e^{-1-p_n(x-1)}\,dx
=
\frac{e^{-1}}{2\Gamma(1/p_n)}
\longrightarrow0,
\]
while the contribution between \(a_n+s_n\) and \(1\), when
\(a_n+s_n<1\), also tends to zero. Hence
\[
\frac{H_{\varepsilon_n,p_n}(s_n)}{\varepsilon_n}
\longrightarrow\frac{1}{2\rho}.
\]
It remains to consider \(p_n\to1\). For \(p=1\), direct calculation
gives
\begin{equation}
\label{eq:laplace-standardized-hockey-stick}
H_{\varepsilon,1}(s)
=
\begin{cases}
0, & s\le\varepsilon,\\[1mm]
1-\exp\bigl((\varepsilon-s)/2\bigr),
& s>\varepsilon.
\end{cases}
\end{equation}
Thus
\[
\frac{H_{\varepsilon_n,1}(\varepsilon_n/\rho_n)}
{\varepsilon_n}
\longrightarrow
\frac{(1-\rho)_+}{2\rho}.
\]
We therefore only need to consider \(p_n>1\). Since
\[
|{-s_n/2+s_n}|^{p_n}-|{-s_n/2}|^{p_n}=0,
\]
the crossing point satisfies
\[
a_n>-\frac{s_n}{2}.
\]
Hence every finite subsequential limit of \(a_n\) is nonnegative.

If \(0<\rho<1\), then \(a_n\to0\). Indeed, if \(a_n\ge\eta>0\)
along a subsequence, the mean value theorem would give
\[
\rho_n=p_n\xi_n^{p_n-1}
\ge p_n\eta^{p_n-1}\longrightarrow1,
\]
a contradiction. Therefore
\[
a_n\to0,
\qquad
a_n+s_n\to0.
\]
Using \eqref{eq:hockey-stick-crossing-representation} and the local
convergence \(g_{p_n}\to g_1\),
\[
\frac{H_{\varepsilon_n,p_n}(s_n)}{\varepsilon_n}
\longrightarrow
\frac{g_1(0)}{\rho}-\overline G_1(0)
=
\frac{1-\rho}{2\rho}.
\]
If \(\rho=1\), every subsequence of \(a_n\) has a further subsequence
that either converges to some \(a\in[0,\infty)\) or tends to infinity.
In the first case,
\[
\frac{H_{\varepsilon_n,p_n}(s_n)}{\varepsilon_n}
\longrightarrow
g_1(a)-\overline G_1(a)=0,
\]
because
\[
g_1(a)=\overline G_1(a)=\frac12e^{-a},
\qquad a\ge0.
\]
In the second case both terms in
\eqref{eq:hockey-stick-crossing-representation}, after division by
\(\varepsilon_n\), tend to zero. Thus the limit is again zero.
Finally, suppose \(\rho>1\). If \(a_n\) had a bounded subsequence
converging to \(a\ge0\), then
\[
\frac{H_{\varepsilon_n,p_n}(s_n)}{\varepsilon_n}
\longrightarrow
\frac{g_1(a)}{\rho}-\overline G_1(a)
=
\left(\frac1\rho-1\right)g_1(a)<0,
\]
which is impossible. Hence \(a_n\to\infty\), and the limit is zero.
This proves \eqref{eq:joint-first-order-hockey-stick}.
\end{proof}
For \(1<p<\infty\), the function \(A_p\) may equivalently be written as
\[
A_p(\rho)
=
\int_0^\infty
\left(px^{p-1}-\rho\right)_+g_p(x)\,dx.
\]
It is continuous and strictly decreasing in \(\rho\), with
\[
A_p(0)>0,
\qquad
A_p(\rho)\longrightarrow0
\quad\text{as }\rho\to\infty.
\]
Consequently, for every \(\lambda>0\), the equation
\[
A_p(\rho)=\lambda\rho
\]
has a unique solution, denoted by \(\rho_{p,\lambda}\). We also set
\[
\rho_{1,\lambda}:=\frac{1}{1+2\lambda},
\qquad
\rho_{\infty,\lambda}:=\frac{1}{2\lambda}.
\]
The preceding lemma implies the following joint convergence. If
\[
\frac{\delta_n}{\varepsilon_n}\longrightarrow\lambda\in(0,\infty),
\qquad
p_n\longrightarrow p_0\in[1,\infty],
\]
then
\begin{equation}
\label{eq:joint-normalized-scale-limit}
\frac{
\varepsilon_n
b_{\varepsilon_n,\delta_n}^{(1)}(p_n;\Delta)
}{\Delta}
\longrightarrow
\rho_{p_0,\lambda}.
\end{equation}
To verify this, choose
\[
\rho_-<\rho_{p_0,\lambda}<\rho_+
\]
such that
\[
h_{p_0}(\rho_-)>\lambda>h_{p_0}(\rho_+).
\]
Lemma~\ref{lem:joint-first-order-hockey-stick} and the monotonicity of
\(H_{\varepsilon,p}(s)\) in \(s\) imply, for all sufficiently large
\(n\),
\[
\rho_-
<
\frac{
\varepsilon_n
b_{\varepsilon_n,\delta_n}^{(1)}(p_n;\Delta)
}{\Delta}
<
\rho_+.
\]
Letting \(\rho_-\uparrow\rho_{p_0,\lambda}\) and
\(\rho_+\downarrow\rho_{p_0,\lambda}\) gives
\eqref{eq:joint-normalized-scale-limit}. Therefore,
\begin{equation}
\label{eq:joint-normalized-objective-limit}
\frac{\varepsilon_n^m}{\Delta^m}
L_m\left(
p_n,
b_{\varepsilon_n,\delta_n}^{(1)}(p_n;\Delta)
\right)
\longrightarrow
M_m(p_0)\rho_{p_0,\lambda}^m.
\end{equation}

\medskip
\noindent
\emph{Proof of (i).}
First suppose \(d=1\) and \(\Delta>0\). At \(p=1\),
\[
b_{\varepsilon,\delta}^{(1)}(1;\Delta)
=
\frac{\Delta}{\varepsilon-2\log(1-\delta)}.
\]
Since \(\delta_n/\varepsilon_n\to0\),
\[
-2\log(1-\delta_n)
=
2\delta_n+o(\delta_n)
=
o(\varepsilon_n),
\]
and hence
\begin{equation}
\label{eq:part-i-laplace-scale}
\frac{
\varepsilon_n
b_{\varepsilon_n,\delta_n}^{(1)}(1;\Delta)
}{\Delta}
\longrightarrow1.
\end{equation}
In particular,
\begin{equation}
\label{eq:part-i-laplace-objective}
\varepsilon_n^m
L_m\left(
1,b_{\varepsilon_n,\delta_n}^{(1)}(1;\Delta)
\right)
\longrightarrow
\Gamma(m+1)\Delta^m.
\end{equation}
We claim that, for every \(\kappa>0\),
\begin{equation}
\label{eq:part-i-uniform-scale-divergence}
\inf_{p\in[1+\kappa,\infty]}
\frac{
\varepsilon_n
b_{\varepsilon_n,\delta_n}^{(1)}(p;\Delta)
}{\Delta}
\longrightarrow\infty.
\end{equation}
Suppose otherwise. Then there exist \(M<\infty\), a subsequence, and
\(p_n\in[1+\kappa,\infty]\) such that
\[
\rho_n
:=
\frac{
\varepsilon_n
b_{\varepsilon_n,\delta_n}^{(1)}(p_n;\Delta)
}{\Delta}
\le M.
\]
Passing to a further subsequence, assume that
\[
p_n\to p_0\in[1+\kappa,\infty],
\qquad
\rho_n\to\rho_0\in[0,M].
\]
If \(\rho_0>0\), Lemma~\ref{lem:joint-first-order-hockey-stick}
gives
\[
\frac{
H_{\varepsilon_n,p_n}(\varepsilon_n/\rho_n)
}{\varepsilon_n}
\longrightarrow
h_{p_0}(\rho_0)>0.
\]
This contradicts feasibility, because
\[
H_{\varepsilon_n,p_n}(\varepsilon_n/\rho_n)
\le\delta_n
\]
and \(\delta_n/\varepsilon_n\to0\).
If \(\rho_0=0\), fix any \(r>0\). For all sufficiently large \(n\),
\(\rho_n<r\), and hence
\[
H_{\varepsilon_n,p_n}(\varepsilon_n/\rho_n)
\ge
H_{\varepsilon_n,p_n}(\varepsilon_n/r).
\]
Lemma~\ref{lem:joint-first-order-hockey-stick} then gives
\[
\liminf_{n\to\infty}
\frac{
H_{\varepsilon_n,p_n}(\varepsilon_n/\rho_n)
}{\varepsilon_n}
\ge h_{p_0}(r)>0,
\]
which is again incompatible with
\(\delta_n/\varepsilon_n\to0\). This proves
\eqref{eq:part-i-uniform-scale-divergence}.
Since \(M_m\) is positive on \([1+\kappa,\infty]\),
\[
\inf_{p\in[1+\kappa,\infty]}
\varepsilon_n^m
L_m\left(
p,b_{\varepsilon_n,\delta_n}^{(1)}(p;\Delta)
\right)
\longrightarrow\infty.
\]
Together with \eqref{eq:part-i-laplace-objective}, this shows that no
minimizer lies in \([1+\kappa,\infty]\) for all sufficiently large
\(n\). Since \(\kappa>0\) is arbitrary,
\[
p^\star(\varepsilon_n,\delta_n)\longrightarrow1.
\]
Now let \(d\ge1\), and choose \(j\) such that \(\Delta_j\neq0\).
Restricting the \(d\)-dimensional hockey-stick supremum to measurable
sets depending only on coordinate \(j\) gives
\[
\delta_{\varepsilon,p,b}^{(1)}(|\Delta_j|)
\le
\delta_{\varepsilon,p,b}^{(d)}(\Delta).
\]
Consequently,
\[
b_{\varepsilon,\delta}^{(d)}(p;\Delta)
\ge
b_{\varepsilon,\delta}^{(1)}(p;|\Delta_j|).
\]
It follows from the one-dimensional result that, for every
\(\kappa>0\),
\[
\inf_{p\in[1+\kappa,\infty]}
\varepsilon_n
b_{\varepsilon_n,\delta_n}^{(d)}(p;\Delta)
\longrightarrow\infty.
\]
On the other hand, the product Laplace mechanism with
\[
b=\frac{\|\Delta\|_1}{\varepsilon_n}
\]
satisfies \(\varepsilon_n\)-DP, and therefore
\[
b_{\varepsilon_n,\delta_n}^{(d)}(1;\Delta)
\le
\frac{\|\Delta\|_1}{\varepsilon_n}.
\]
Thus the objective at \(p=1\), after multiplication by
\(\varepsilon_n^m\), remains bounded, whereas the corresponding
infimum over \([1+\kappa,\infty]\) diverges. Hence
\[
p^\star(\varepsilon_n,\delta_n)\longrightarrow1.
\]

\medskip
\noindent
\emph{Proof of (ii).}
Suppose \(d=1\), \(\Delta>0\), and
\[
\frac{\delta_n}{\varepsilon_n}\longrightarrow
\lambda\in(0,\infty).
\]
Define
\[
F_n(p)
:=
\frac{\varepsilon_n^m}{\Delta^m}
L_m\left(
p,b_{\varepsilon_n,\delta_n}^{(1)}(p;\Delta)
\right).
\]
By \eqref{eq:joint-normalized-objective-limit}, for every sequence
\(p_n\to p_0\in[1,\infty]\),
\begin{equation}
\label{eq:sequential-objective-convergence}
F_n(p_n)\longrightarrow K_{m,\lambda}(p_0),
\end{equation}
where, for \(1<p<\infty\),
\[
K_{m,\lambda}(p)
=
M_m(p)\rho_{p,\lambda}^m,
\]
and
\[
K_{m,\lambda}(1)
=
\frac{\Gamma(m+1)}{(1+2\lambda)^m},
\qquad
K_{m,\lambda}(\infty)
=
\frac{1}{(m+1)(2\lambda)^m}.
\]
Let
\[
p_n^\star:=p^\star(\varepsilon_n,\delta_n)
\]
and consider any subsequence such that
\[
p_{n_k}^\star\longrightarrow p_0\in[1,\infty].
\]
For every fixed \(q\in[1,\infty]\), optimality gives
\[
F_{n_k}(p_{n_k}^\star)\le F_{n_k}(q).
\]
Passing to the limit by
\eqref{eq:sequential-objective-convergence} yields
\[
K_{m,\lambda}(p_0)
\le
K_{m,\lambda}(q).
\]
Since \(q\) is arbitrary,
\[
p_0\in
\operatorname*{arg\,min}_{p\in[1,\infty]}
K_{m,\lambda}(p).
\]
Thus every subsequential limit of
\(p^\star(\varepsilon_n,\delta_n)\) belongs to the minimizer set of
\(K_{m,\lambda}\).
If \(K_{m,\lambda}\) has the unique minimizer \(p_\lambda\), then every
convergent subsequence of \(p_n^\star\) has limit \(p_\lambda\).
Compactness of \([1,\infty]\) therefore implies
\[
p^\star(\varepsilon_n,\delta_n)\longrightarrow p_\lambda.
\]

\medskip
\noindent
\emph{Proof of (iii).}
Suppose \(d=1\), \(\Delta>0\), and
\[
\frac{\delta_n}{\varepsilon_n}\longrightarrow\infty.
\]
Then
\[
\varepsilon_n=o(\delta_n).
\]
For any two densities \(f\) and \(g\),
\begin{equation}
\label{eq:hockey-stick-tv-comparison}
\operatorname{TV}(f,g)-(e^\varepsilon-1)
\le
\int_{\mathbb R}(f-e^\varepsilon g)_+
\le
\operatorname{TV}(f,g),
\end{equation}
where
\[
\operatorname{TV}(f,g)
=
\int_{\mathbb R}(f-g)_+.
\]
Since
\[
e^{\varepsilon_n}-1
=
\varepsilon_n+o(\varepsilon_n)
=
o(\delta_n),
\]
the hockey-stick constraint is asymptotically equivalent to the
total-variation constraint at level \(\delta_n\).
For \(1\le p<\infty\), define
\[
T_p(s)
:=
\operatorname{TV}\bigl(g_p(\cdot),g_p(\cdot+s)\bigr).
\]
By symmetry and unimodality of \(g_p\),
\[
T_p(s)
=
\int_{-s/2}^{s/2}g_p(x)\,dx.
\]
Let
\[
c_p:=g_p(0)
=
\frac{p}{2\Gamma(1/p)}
=
\frac{1}{2\Gamma(1+1/p)}.
\]
For every fixed \(P<\infty\),
\begin{equation}
\label{eq:uniform-tv-expansion}
T_p(s)=c_ps+o(s),
\qquad s\downarrow0,
\end{equation}
uniformly for \(p\in[1,P]\).
Let
\[
s_{n,p}
:=
\frac{\Delta}
{b_{\varepsilon_n,\delta_n}^{(1)}(p;\Delta)}.
\]
The comparison \eqref{eq:hockey-stick-tv-comparison} and
\eqref{eq:uniform-tv-expansion} imply
\[
s_{n,p}
=
\frac{\delta_n}{c_p}\bigl(1+o(1)\bigr)
\]
uniformly for \(p\in[1,P]\). Equivalently,
\begin{equation}
\label{eq:part-iii-scale-limit}
\frac{
\delta_n
b_{\varepsilon_n,\delta_n}^{(1)}(p;\Delta)
}{\Delta}
\longrightarrow c_p
\end{equation}
uniformly for \(p\in[1,P]\). Hence
\begin{equation}
\label{eq:part-iii-objective-limit}
\frac{\delta_n^m}{\Delta^m}
L_m\left(
p,b_{\varepsilon_n,\delta_n}^{(1)}(p;\Delta)
\right)
\longrightarrow
J_m(p):=M_m(p)c_p^m
\end{equation}
uniformly on \([1,P]\).

We claim that
\begin{equation}
\label{eq:finite-p-strict-tv-bound}
J_m(p)>
\frac{1}{2^m(m+1)}
\qquad
\text{for every }p<\infty.
\end{equation}
Indeed, if \(X_p\) has density \(g_p\), then
\[
g_p(x)\le c_p
\]
and therefore
\[
\mathbb P(|X_p|\le t)\le2c_pt.
\]
It follows that
\[
\mathbb P(|X_p|>t)\ge(1-2c_pt)_+.
\]
Using the tail representation of the \(m\)-th moment,
\begin{align*}
M_m(p)
&=
m\int_0^\infty
t^{m-1}\mathbb P(|X_p|>t)\,dt\\
&>
m\int_0^{1/(2c_p)}
t^{m-1}(1-2c_pt)\,dt\\
&=
\frac{1}{(m+1)(2c_p)^m}.
\end{align*}
The inequality is strict because, for finite \(p\), \(g_p\) is not
constant on any interval of positive length. Multiplying by \(c_p^m\)
proves \eqref{eq:finite-p-strict-tv-bound}.

Since \(J_m\) is continuous, for every \(P<\infty\),
\[
\eta_P
:=
\min_{p\in[1,P]}
\left\{
J_m(p)-\frac{1}{2^m(m+1)}
\right\}
>0.
\]
The uniform convergence in
\eqref{eq:part-iii-objective-limit} therefore implies that, for all
sufficiently large \(n\),
\[
\frac{\delta_n^m}{\Delta^m}
L_m\left(
p,b_{\varepsilon_n,\delta_n}^{(1)}(p;\Delta)
\right)
\ge
\frac{1}{2^m(m+1)}+\frac{\eta_P}{2}
\]
for every \(p\in[1,P]\).

For \(p=\infty\), the positive part arises only on the nonoverlapping
portion of the two uniform supports. Thus, for sufficiently small
\(\delta\),
\[
H_{\varepsilon,\infty}(s)=\frac{s}{2},
\]
and hence
\[
b_{\varepsilon_n,\delta_n}^{(1)}(\infty;\Delta)
=
\frac{\Delta}{2\delta_n}.
\]
Consequently,
\[
\frac{\delta_n^m}{\Delta^m}
L_m\left(
\infty,
b_{\varepsilon_n,\delta_n}^{(1)}(\infty;\Delta)
\right)
=
\frac{1}{2^m(m+1)}.
\]
It follows that no minimizer belongs to \([1,P]\) for all sufficiently
large \(n\). Since \(P<\infty\) is arbitrary,
\[
p^\star(\varepsilon_n,\delta_n)\longrightarrow\infty.
\]
This completes the proof.
\subsection{A two-dimensional counterexample to the uniform-shape limit}
\label{app:multidim-uniform-counterexample}
The conclusions involving convergence to the uniform shape in
Theorems~\ref{thm:small-eps-p-to-infty} and
\ref{thm:joint-high-privacy}(iii) are specific to one dimension.
Consider
\[
d=2,\qquad
\Delta=(1,1),\qquad
m=2.
\]
Let \(Z\sim\operatorname{Unif}([-b,b]^2)\), where \(b>1/2\).
For every \(\varepsilon\ge0\), the integrand in the hockey-stick
divergence is nonpositive on the overlap of the two squares. Therefore,
\[
D_\varepsilon(Z,Z+\Delta)
=
1-\frac{(2b-1)^2}{(2b)^2}
=
\frac1b-\frac{1}{4b^2}.
\]
Thus \(D_\varepsilon(Z,Z+\Delta)\le\delta\) implies
\[
b\ge
\frac{1+\sqrt{1-\delta}}{2\delta},
\]
and every feasible uniform mechanism satisfies
\[
\mathbb E Z_1^2
=
\frac{b^2}{3}
\ge
\frac{(1+\sqrt{1-\delta})^2}{12\delta^2}.
\]
Now let \(Z\sim N(0,\sigma^2I_2)\). Since
\[
D_\varepsilon(Z,Z+\Delta)
\le
D_0(Z,Z+\Delta)
=
2\Phi\left(\frac{\|\Delta\|_2}{2\sigma}\right)-1
\le
\frac{1}{\sigma\sqrt{\pi}},
\]
the choice
\[
\sigma=\frac{1}{\delta\sqrt{\pi}}
\]
is privacy feasible and gives
\[
\mathbb E Z_1^2
=
\frac{1}{\pi\delta^2}.
\]
Hence the Gaussian mechanism has a smaller per-coordinate second moment
than every feasible uniform mechanism whenever
\[
\frac1\pi
<
\frac{(1+\sqrt{1-\delta})^2}{12},
\]
or equivalently,
\[
0<\delta<
\delta_0
:=
1-\left(\sqrt{\frac{12}{\pi}}-1\right)^2
\approx0.0891.
\]
Thus, for \(d=2\) and \(0<\delta<\delta_0\), the uniform mechanism is
not optimal, even as \(\varepsilon\downarrow0\). The same comparison
applies along any sequence satisfying
\[
\delta_n\downarrow0,
\qquad
\varepsilon_n=o(\delta_n),
\]
and therefore the uniform-shape conclusions in
Theorems~\ref{thm:small-eps-p-to-infty} and
\ref{thm:joint-high-privacy}(iii) do not extend directly to
multidimensional queries.
\section{Additional Results for Section 6}
\label{app:section6}
\subsection{Additional numerical results}
\label{app:additional-numerical-results}
Figures~\ref{fig:appendix-gap-grid} and~\ref{fig:appendix-bestp-grid} collect the Phase II trajectories over the twelve \((\varepsilon,\delta)\) pairs. Rows correspond to \(\varepsilon\in\{0.25,0.5,1,2\}\), and columns correspond to \(\delta\in\{0.1,0.01,0.001\}\).
\begin{figure}[!htbp]
\centering
\small
\setlength{\tabcolsep}{1pt}
\renewcommand{\arraystretch}{1.00}
\begin{tabular}{@{}>{\centering\arraybackslash}m{0.04\textwidth}|>{\centering\arraybackslash}m{0.35\textwidth}>{\centering\arraybackslash}m{0.35\textwidth}>{\centering\arraybackslash}m{0.35\textwidth}@{}}
\(\varepsilon\backslash\delta\) & \(0.1\) & \(0.01\) & \(0.001\) \\
\hline
\(0.25\) &
\includegraphics[width=\linewidth,draft=false]{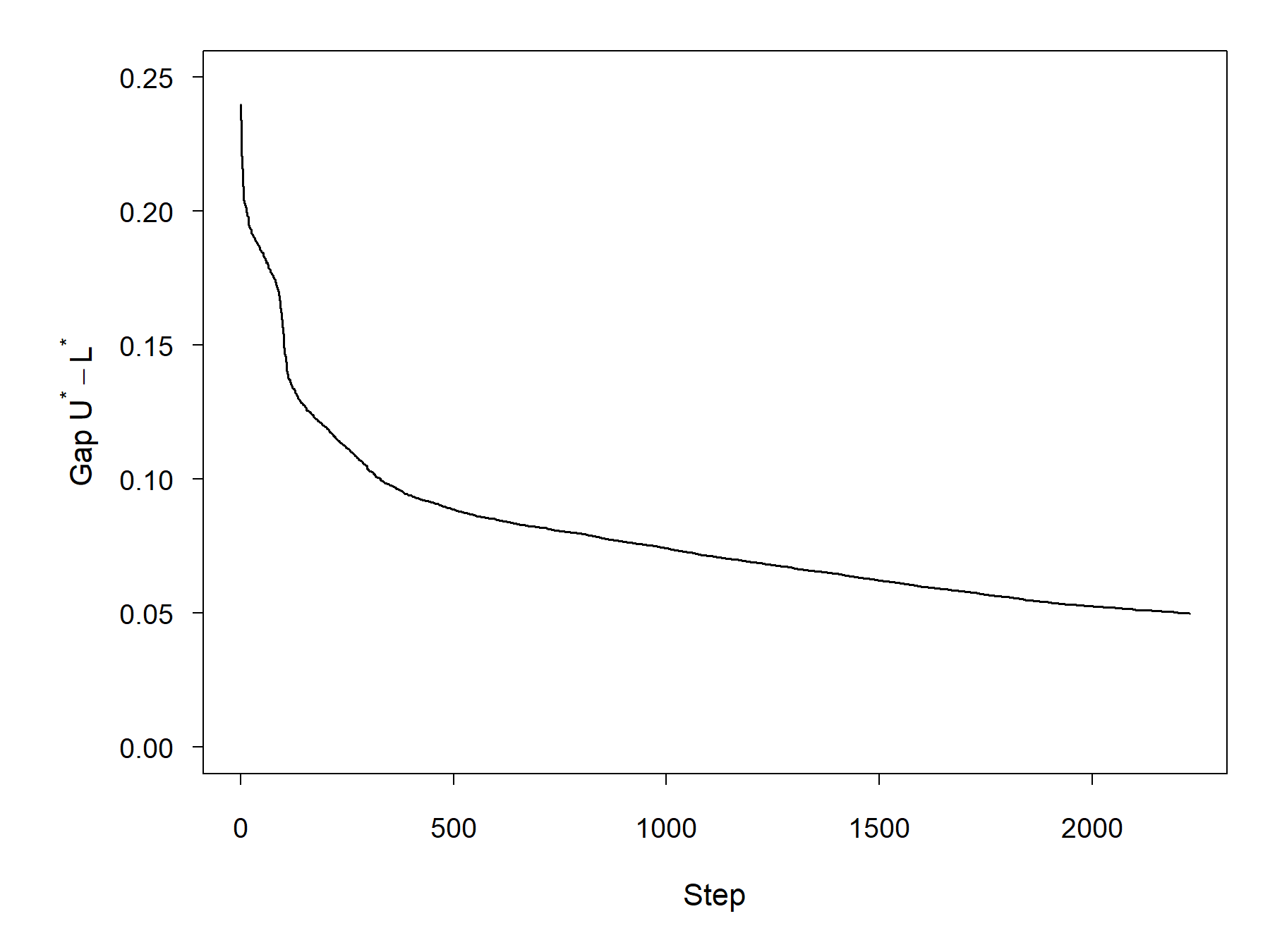} &
\includegraphics[width=\linewidth,draft=false]{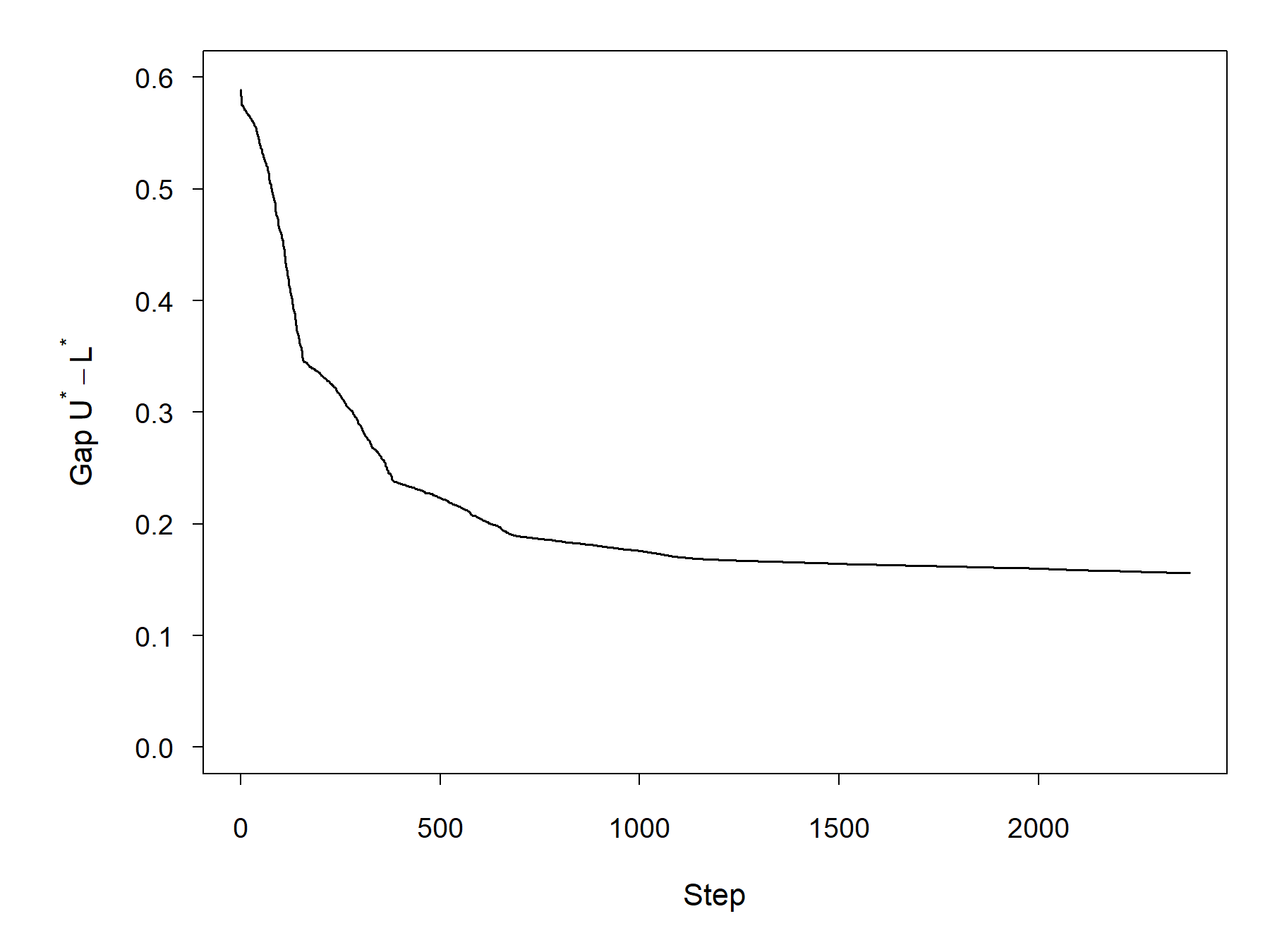} &
\includegraphics[width=\linewidth,draft=false]{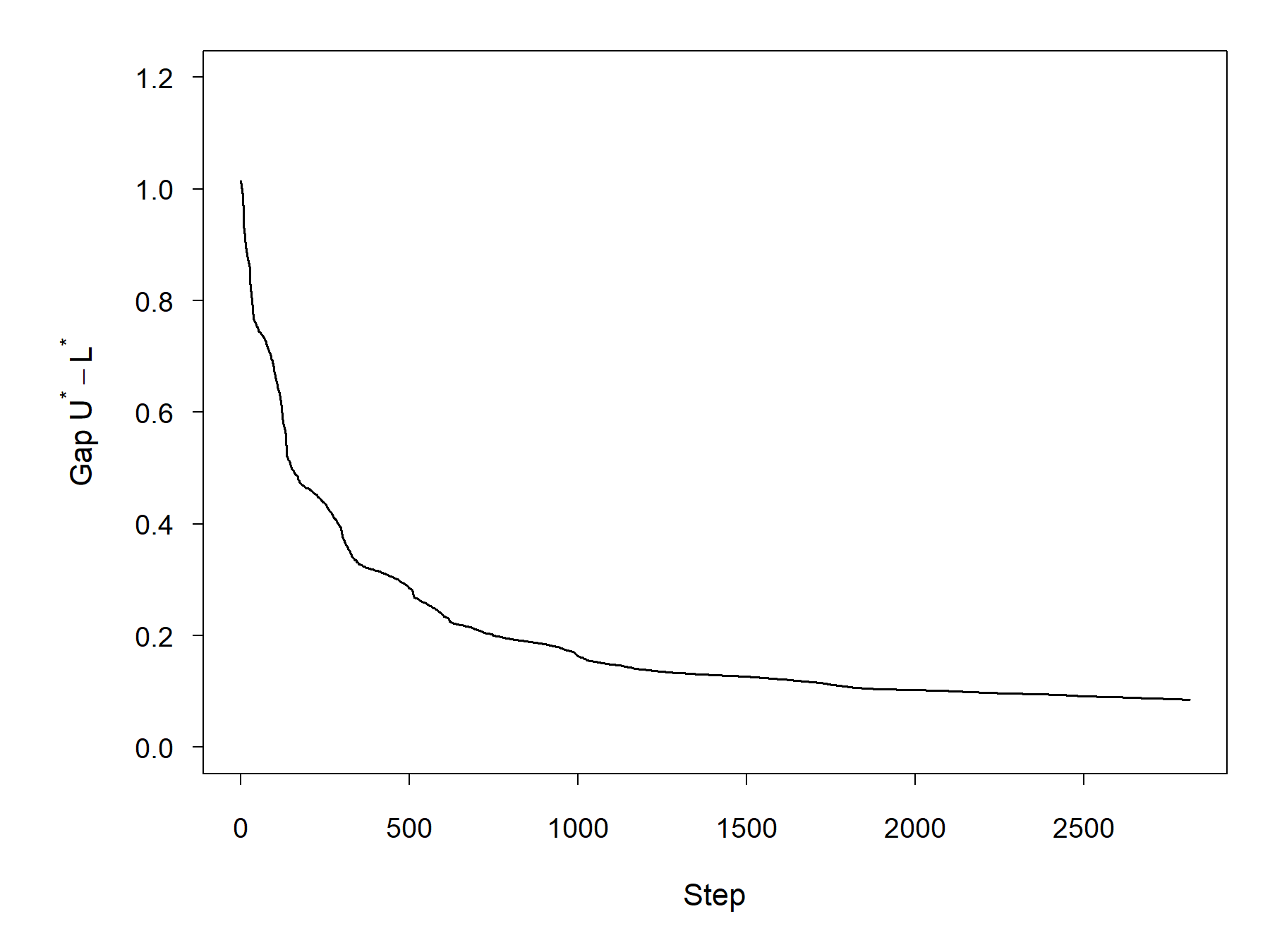} \\
\(0.5\) &
\includegraphics[width=\linewidth,draft=false]{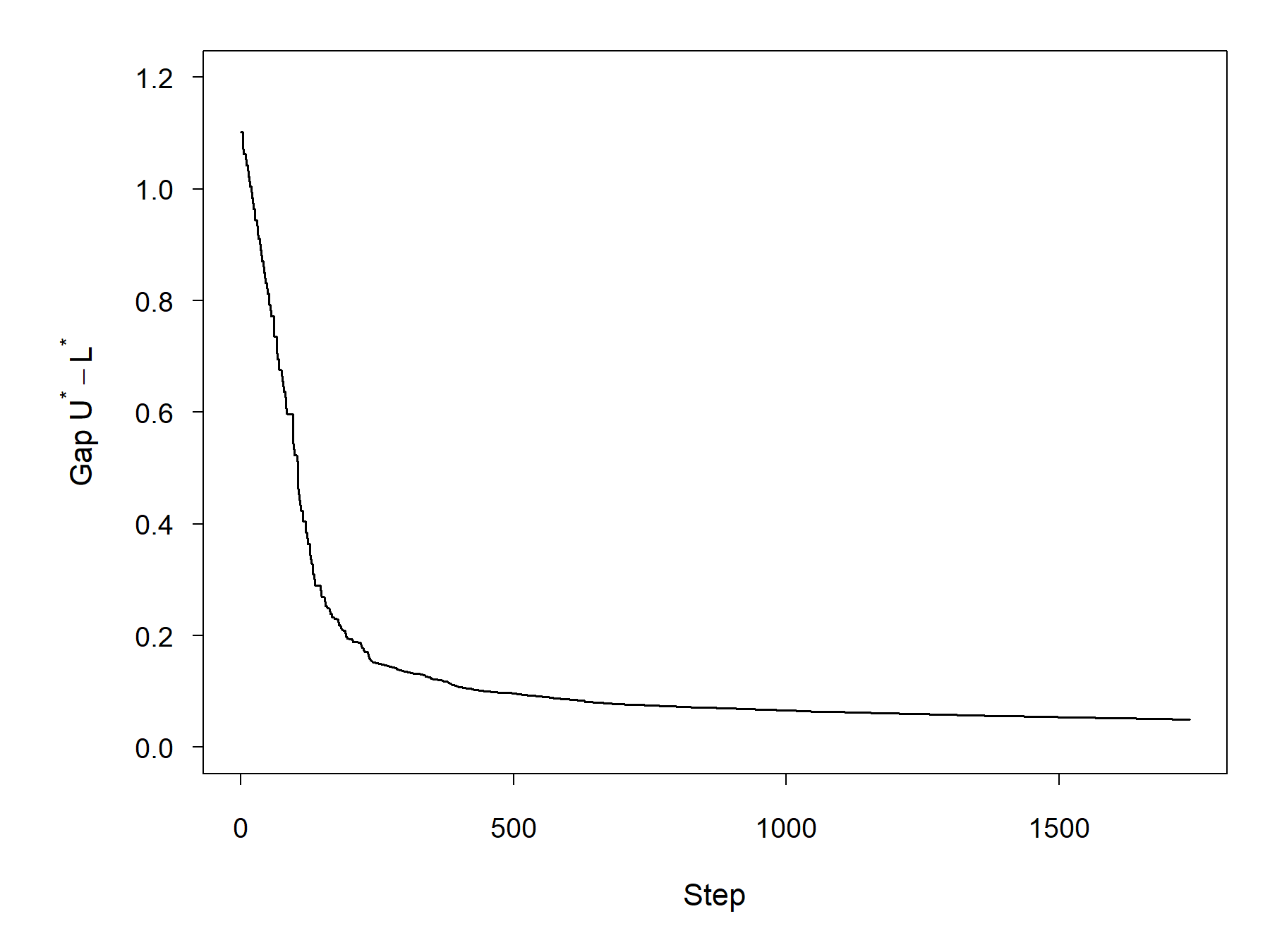} &
\includegraphics[width=\linewidth,draft=false]{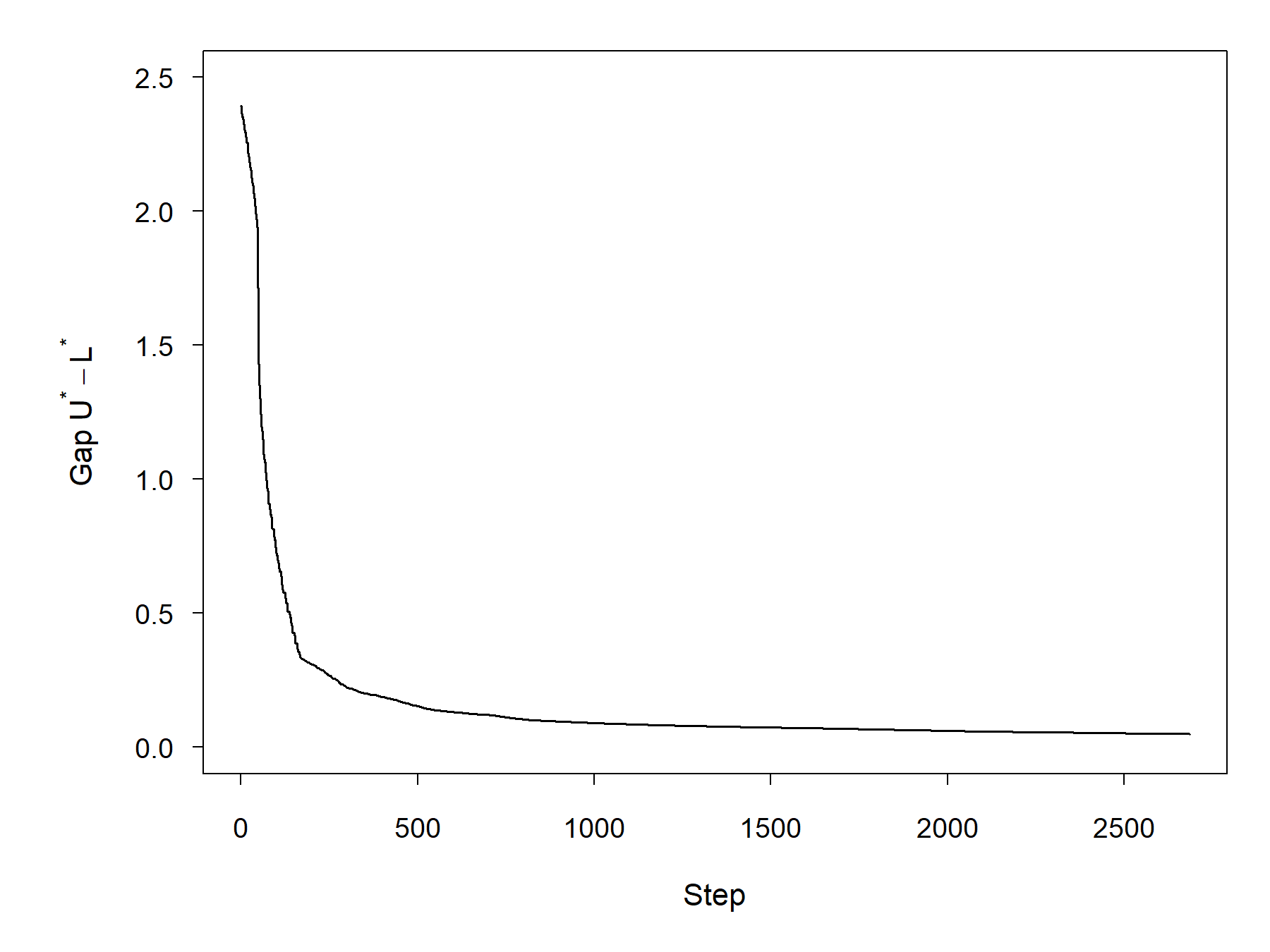} &
\includegraphics[width=\linewidth,draft=false]{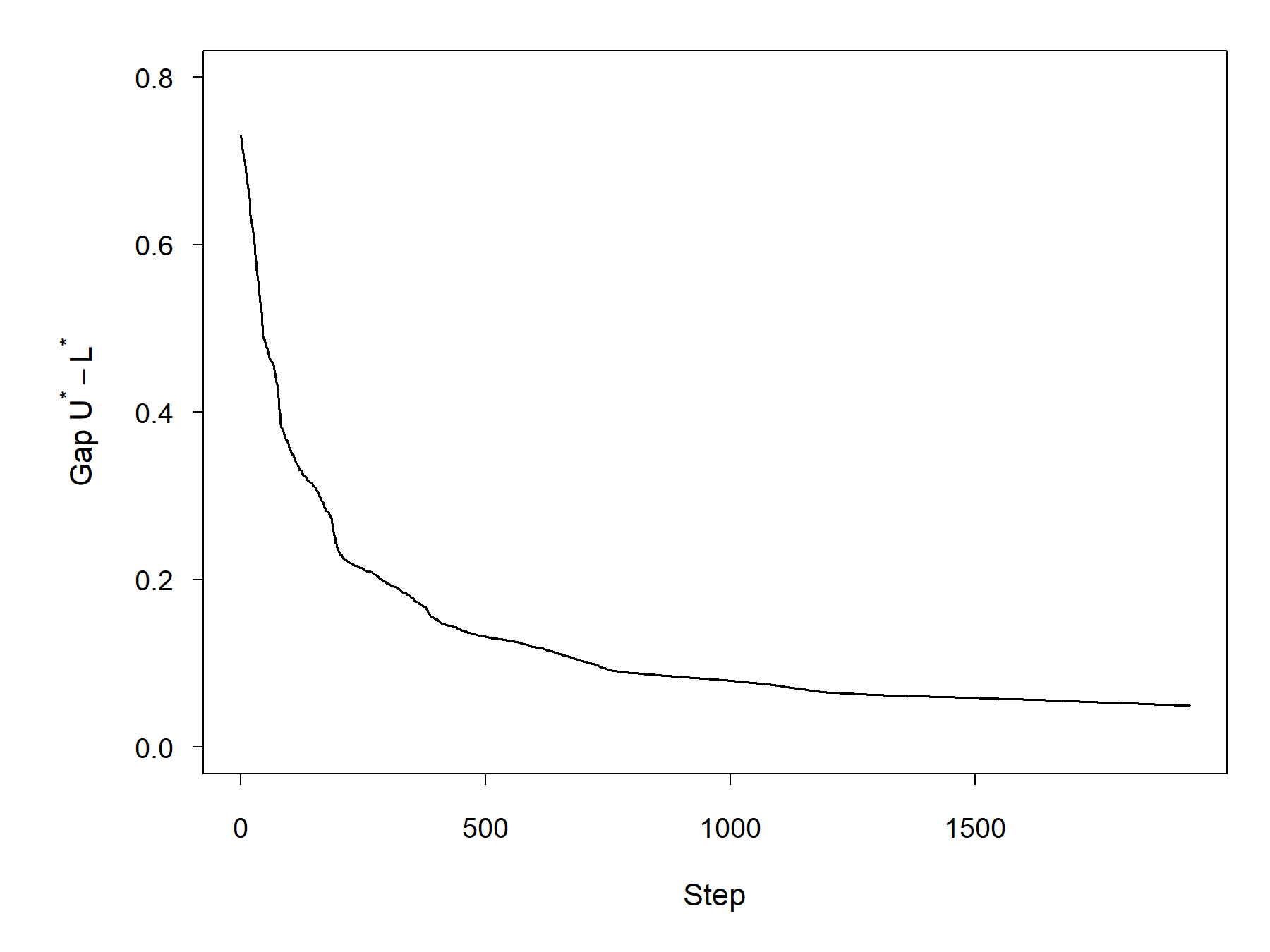} \\
\(1\) &
\includegraphics[width=\linewidth,draft=false]{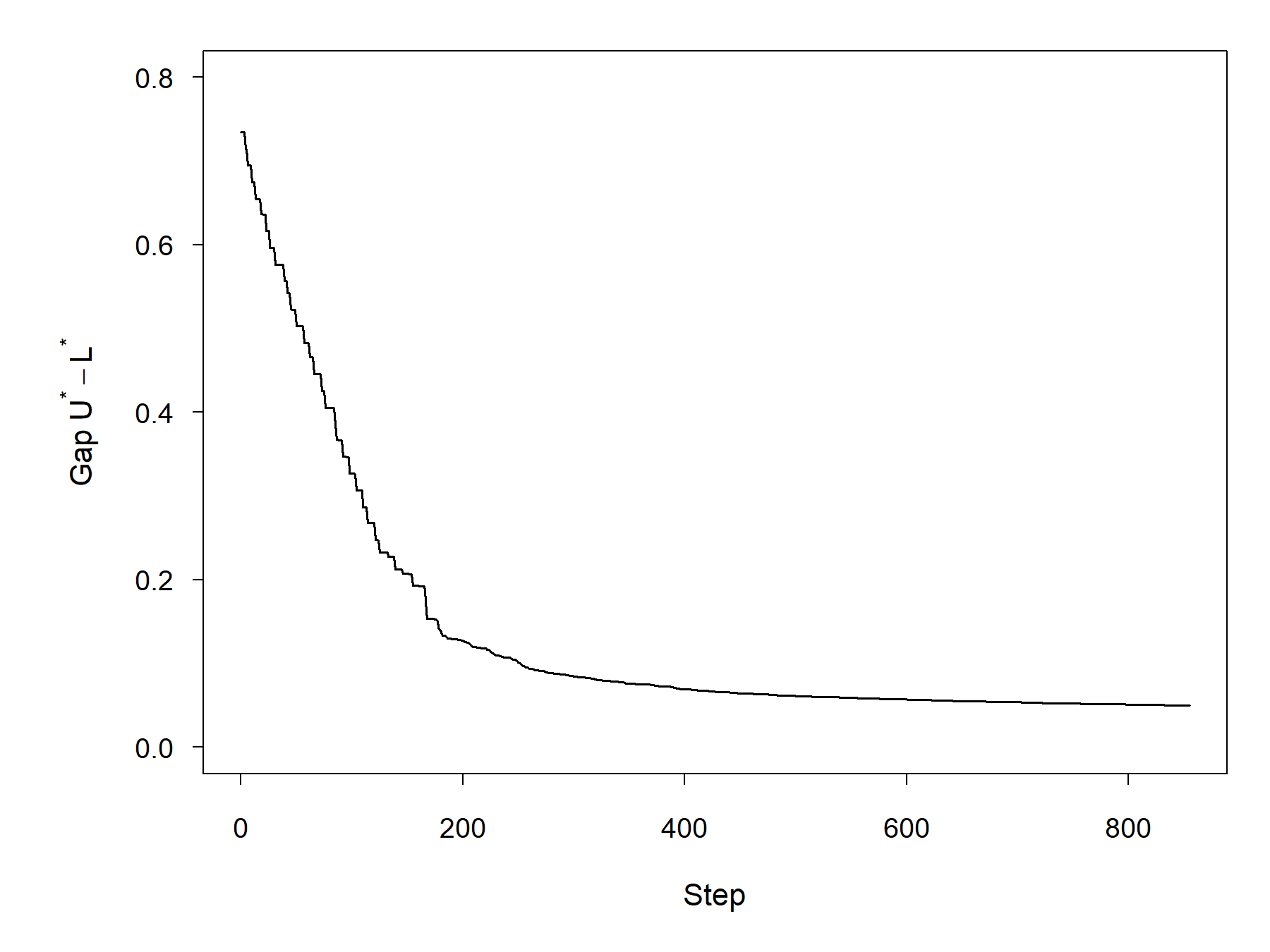} &
\includegraphics[width=\linewidth,draft=false]{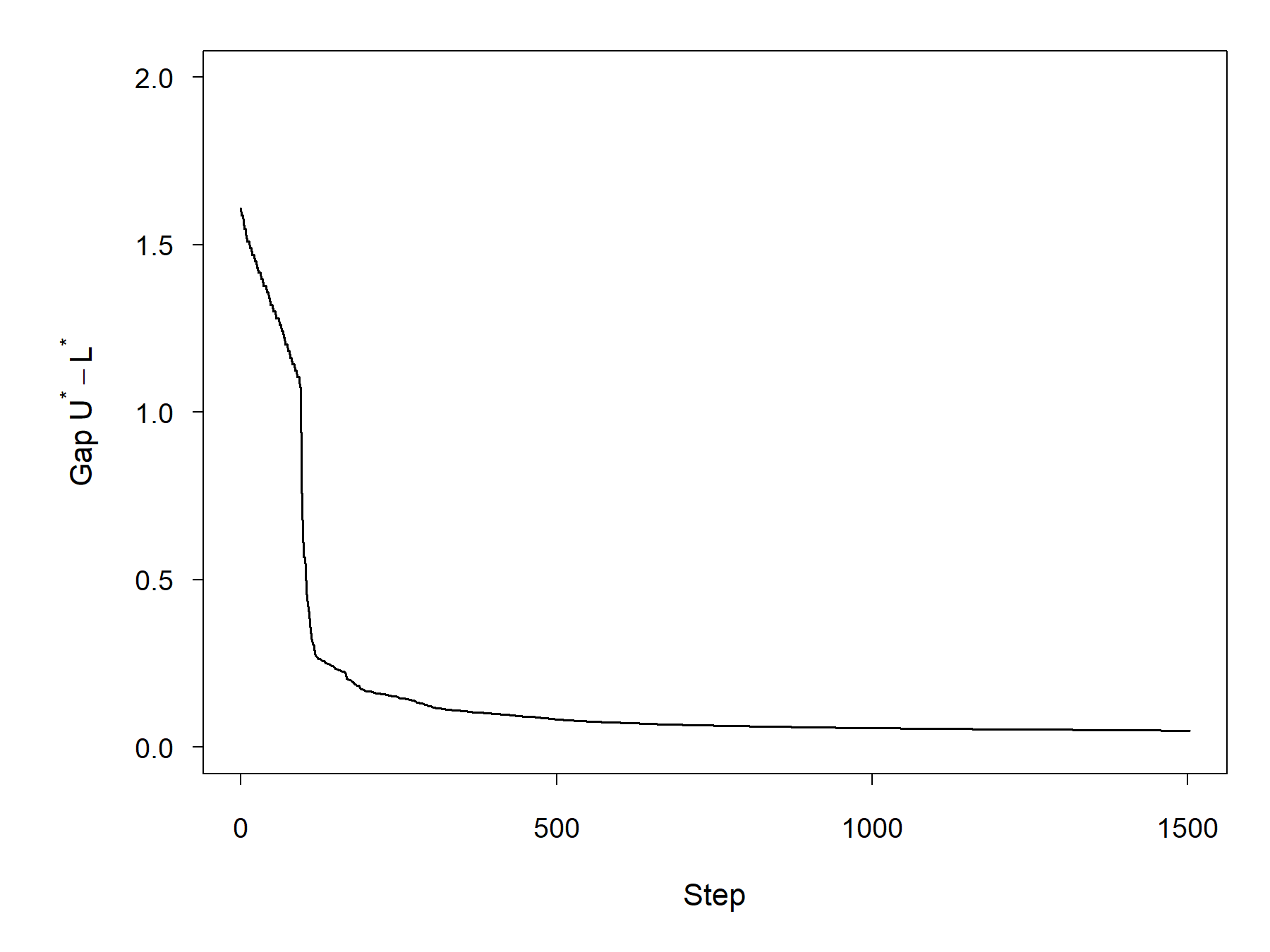} &
\includegraphics[width=\linewidth,draft=false]{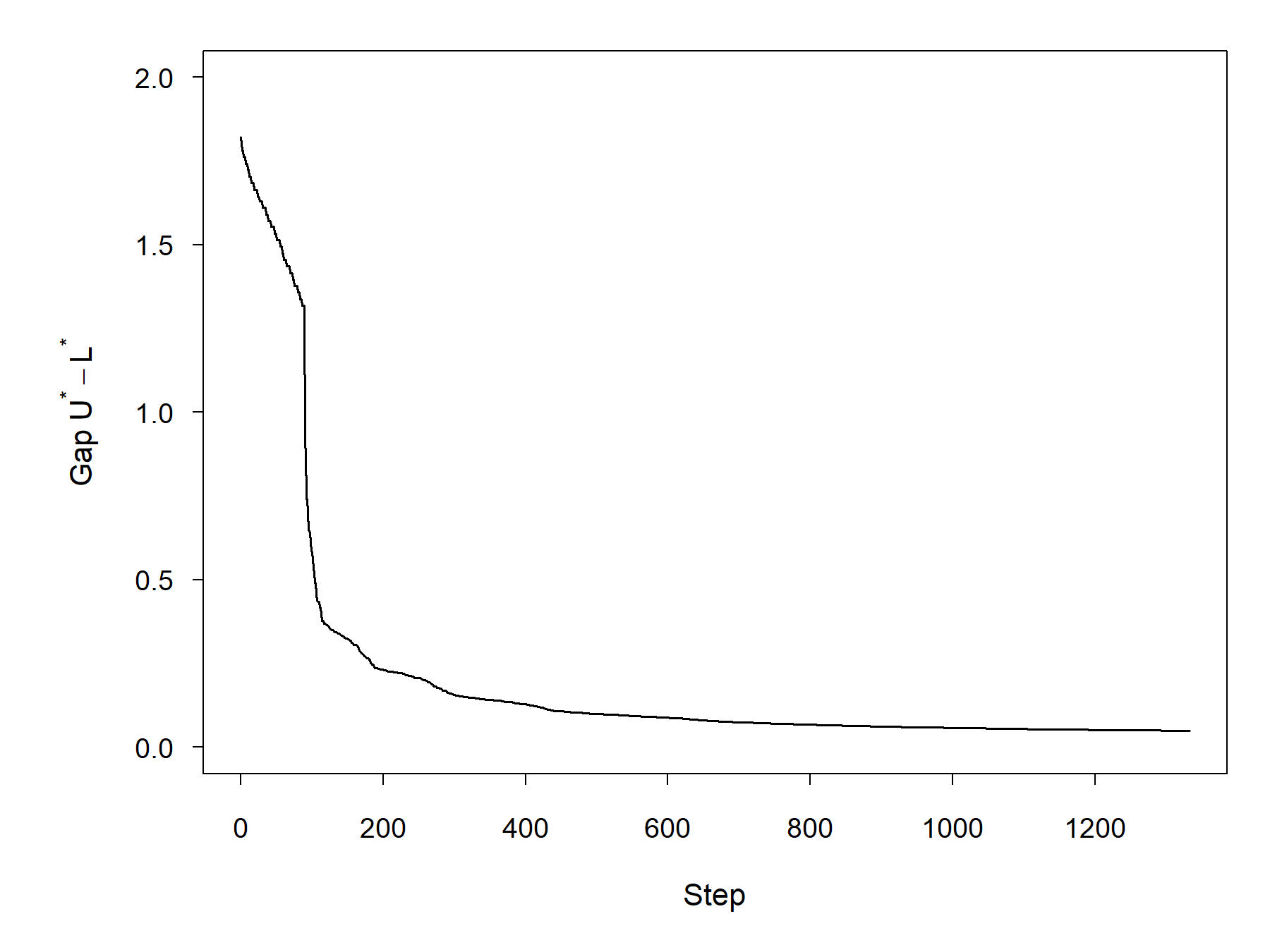} \\
\(2\) &
\includegraphics[width=\linewidth,draft=false]{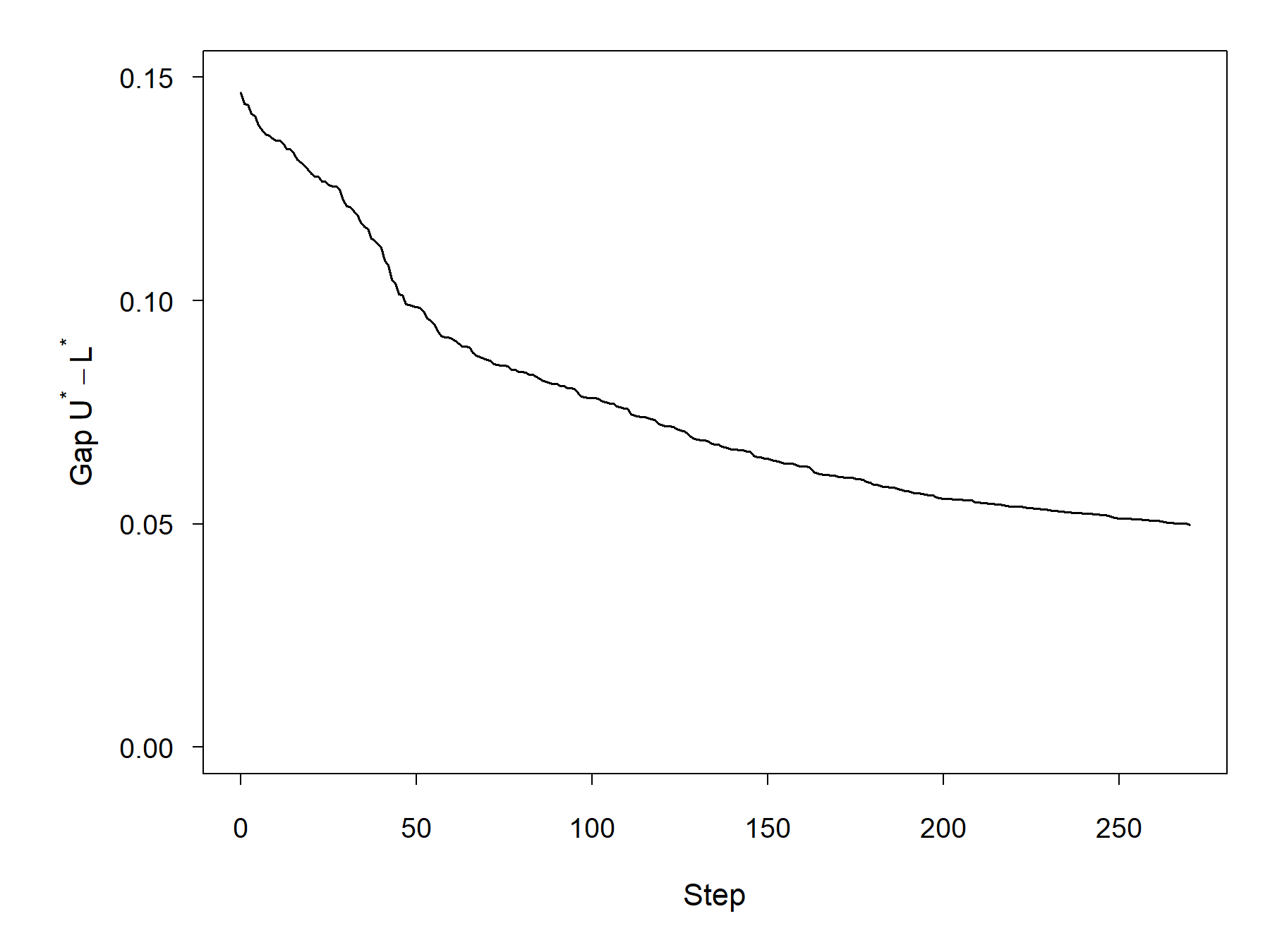} &
\includegraphics[width=\linewidth,draft=false]{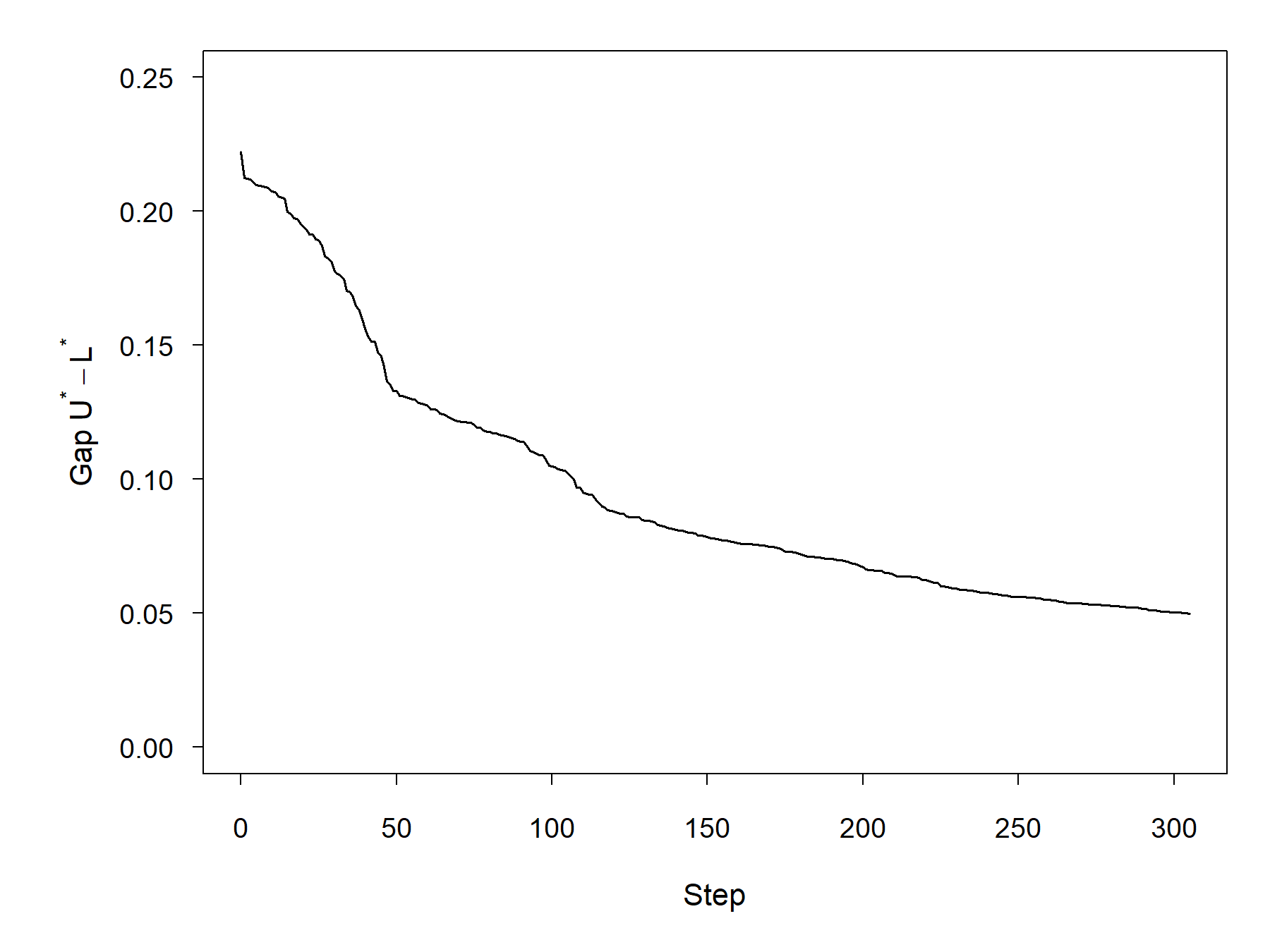} &
\includegraphics[width=\linewidth,draft=false]{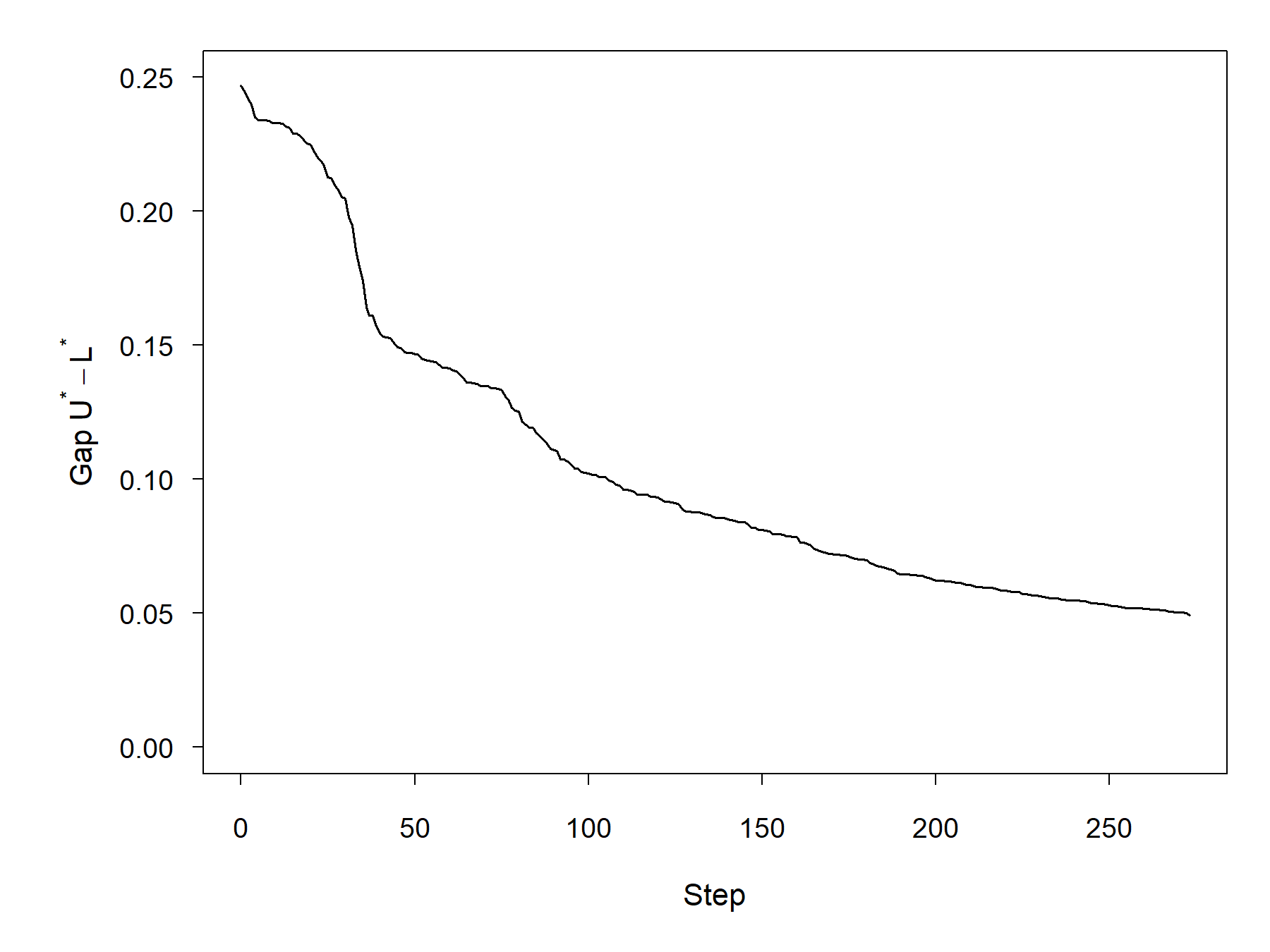} \\
\end{tabular}
\caption{Gap \(U^*-L^*\) during the certified interval-wise search for the twelve \((\varepsilon,\delta)\) pairs.}
\label{fig:appendix-gap-grid}
\end{figure}

\clearpage

\begin{figure}[!htbp]
\centering
\small
\setlength{\tabcolsep}{1pt}
\renewcommand{\arraystretch}{1.00}
\begin{tabular}{@{}>{\centering\arraybackslash}m{0.04\textwidth}|>{\centering\arraybackslash}m{0.35\textwidth}>{\centering\arraybackslash}m{0.35\textwidth}>{\centering\arraybackslash}m{0.35\textwidth}@{}}
\(\varepsilon\backslash\delta\) & \(0.1\) & \(0.01\) & \(0.001\) \\
\hline
\(0.25\) &
\includegraphics[width=\linewidth,draft=false]{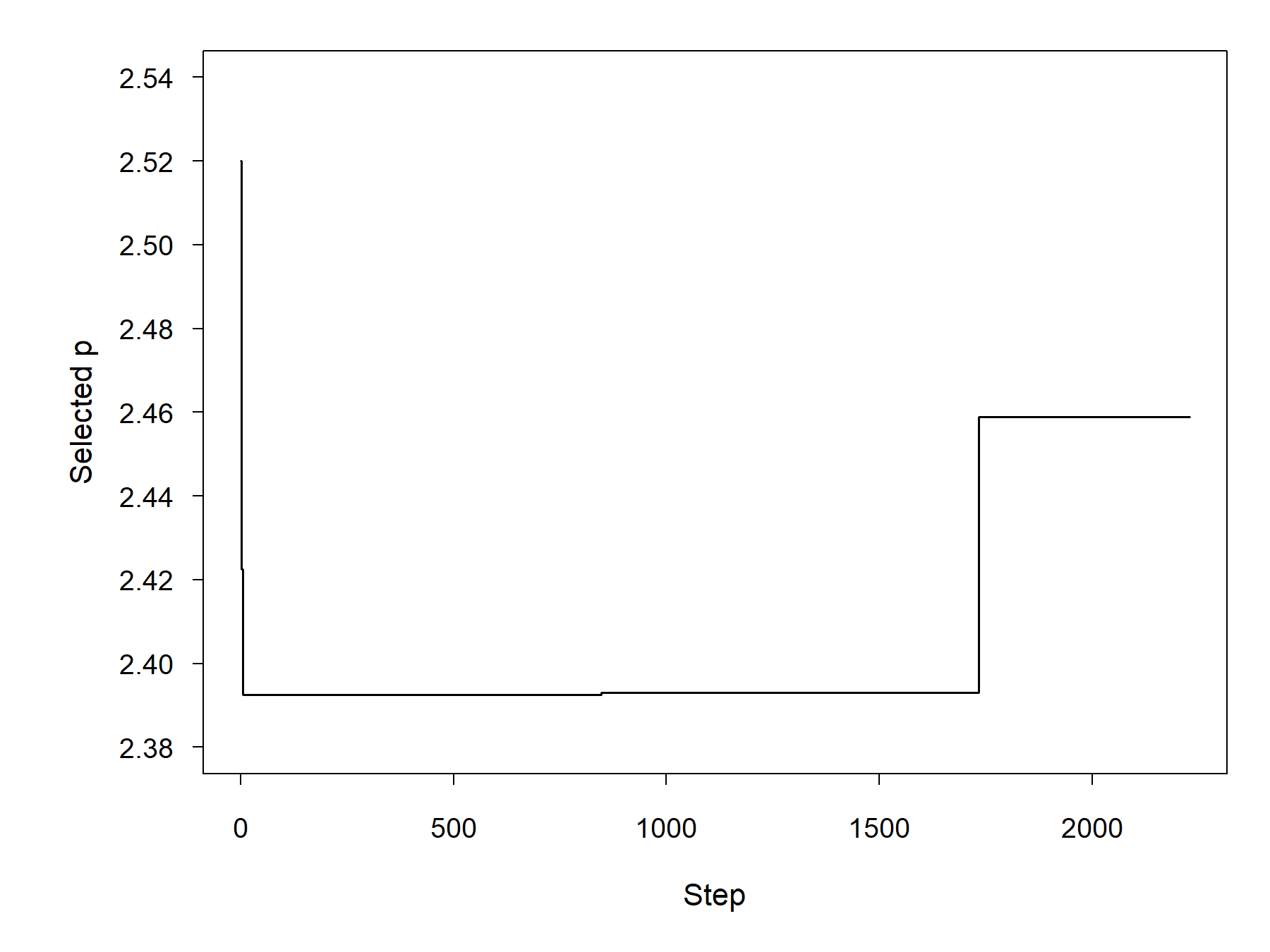} &
\includegraphics[width=\linewidth,draft=false]{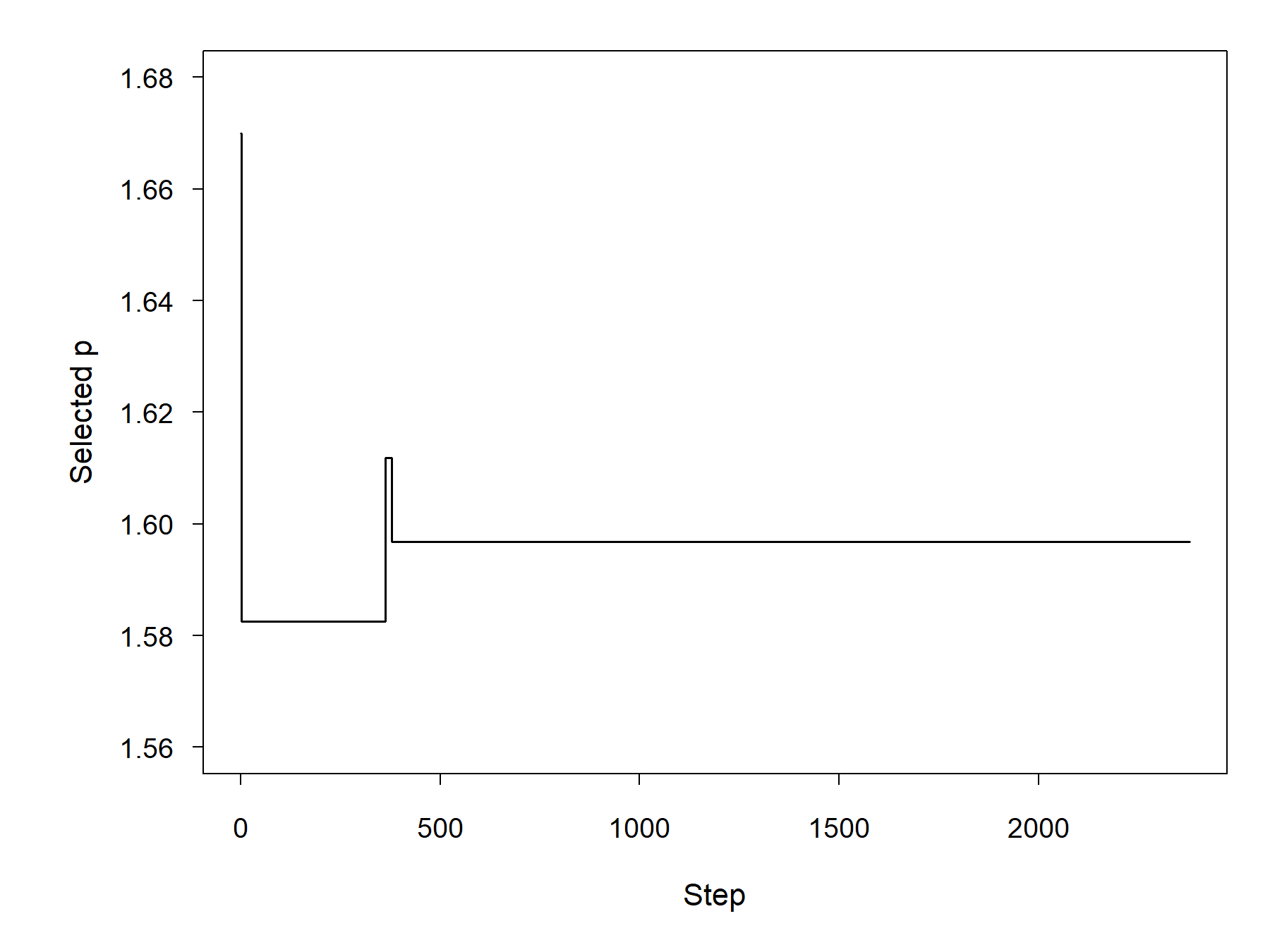} &
\includegraphics[width=\linewidth,draft=false]{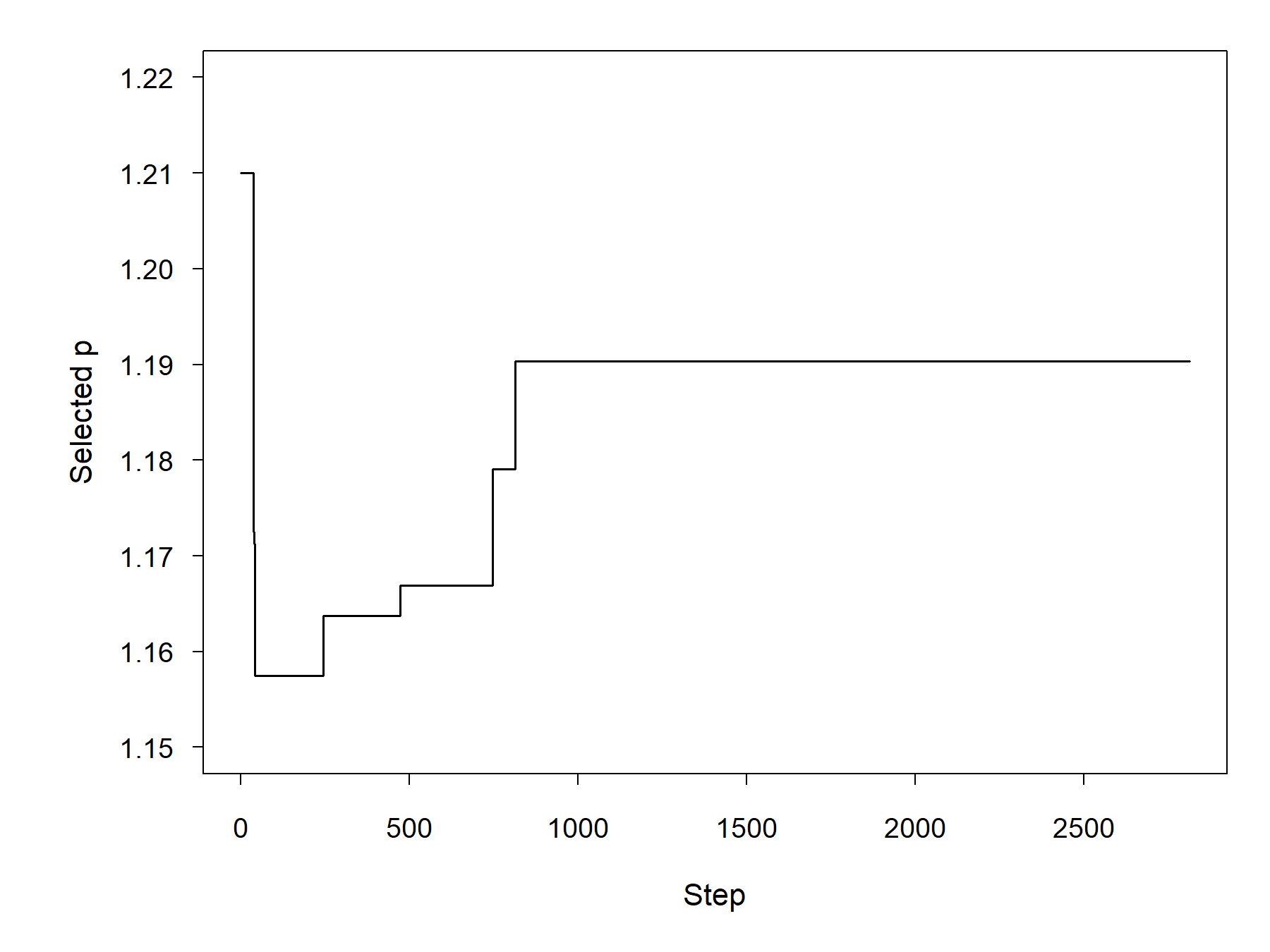} \\
\(0.5\) &
\includegraphics[width=\linewidth,draft=false]{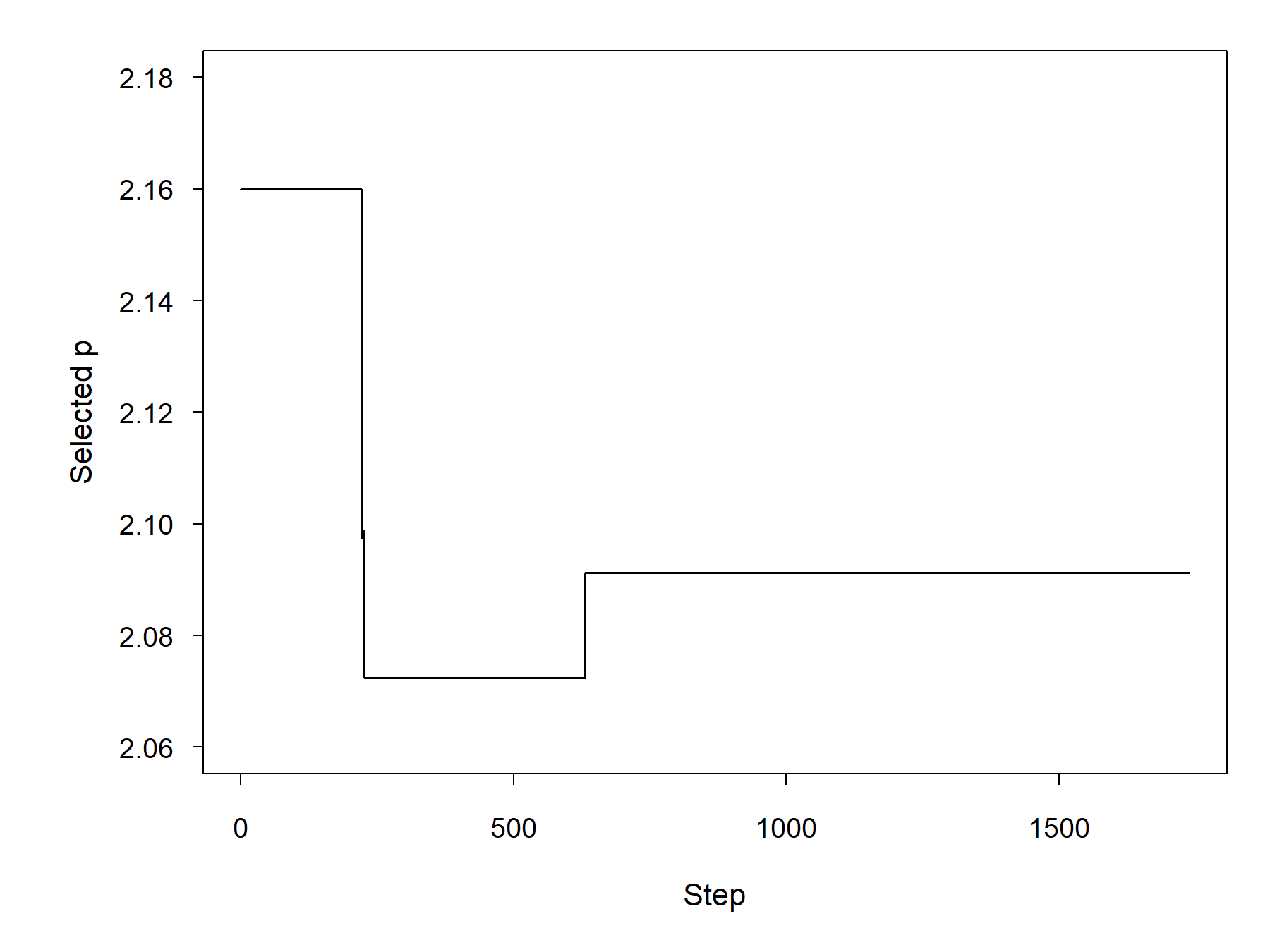} &
\includegraphics[width=\linewidth,draft=false]{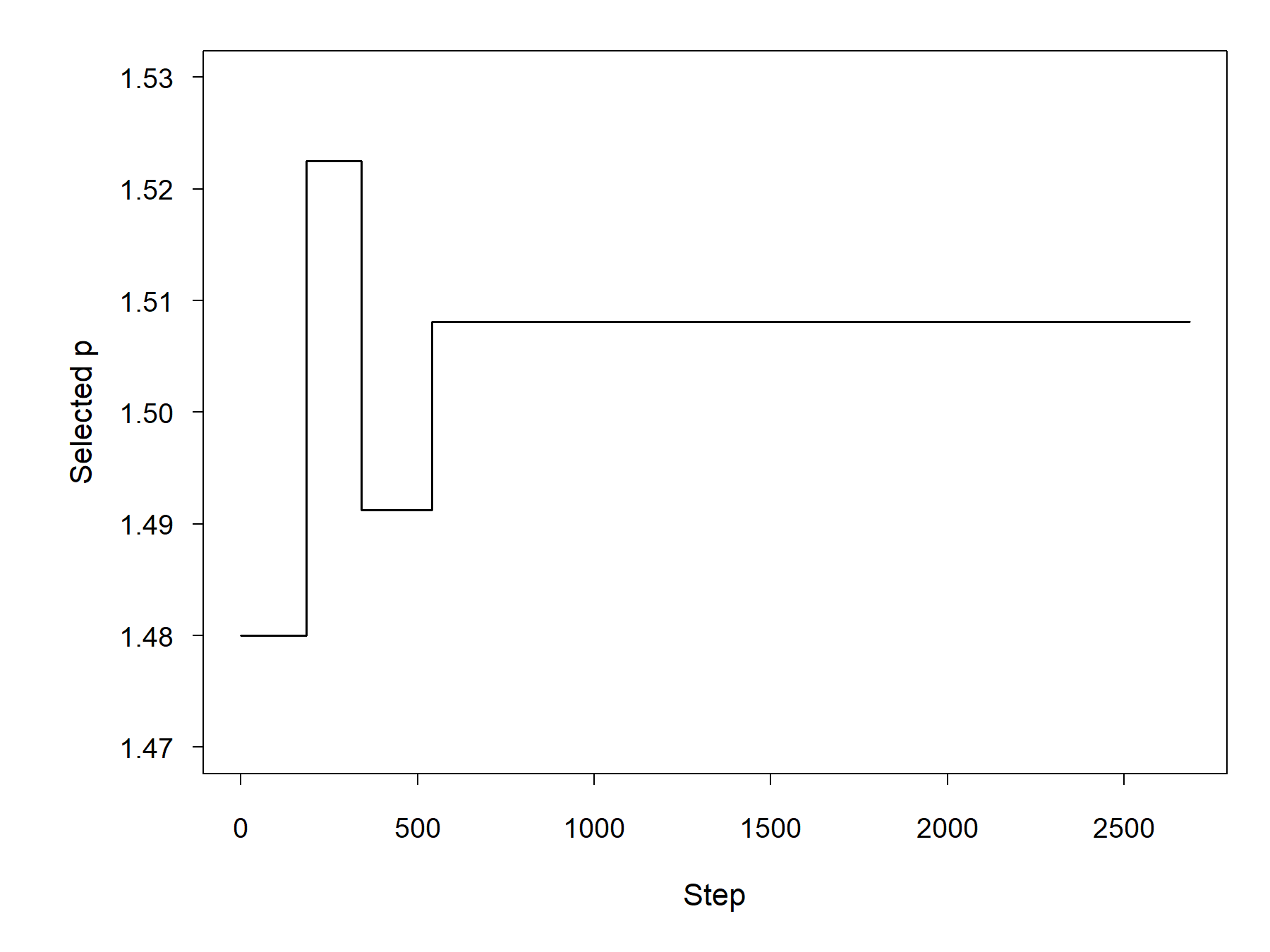} &
\includegraphics[width=\linewidth,draft=false]{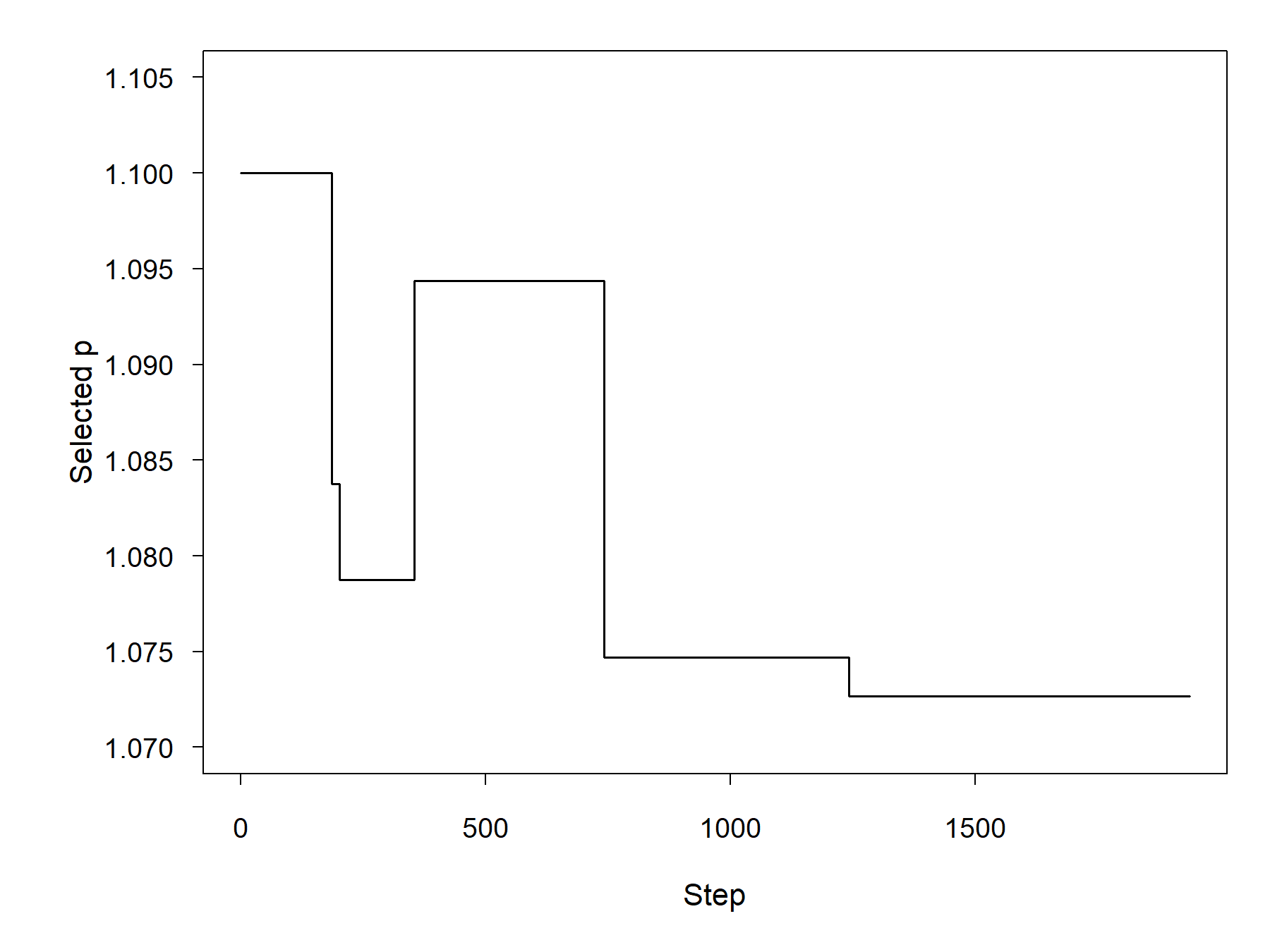} \\
\(1\) &
\includegraphics[width=\linewidth,draft=false]{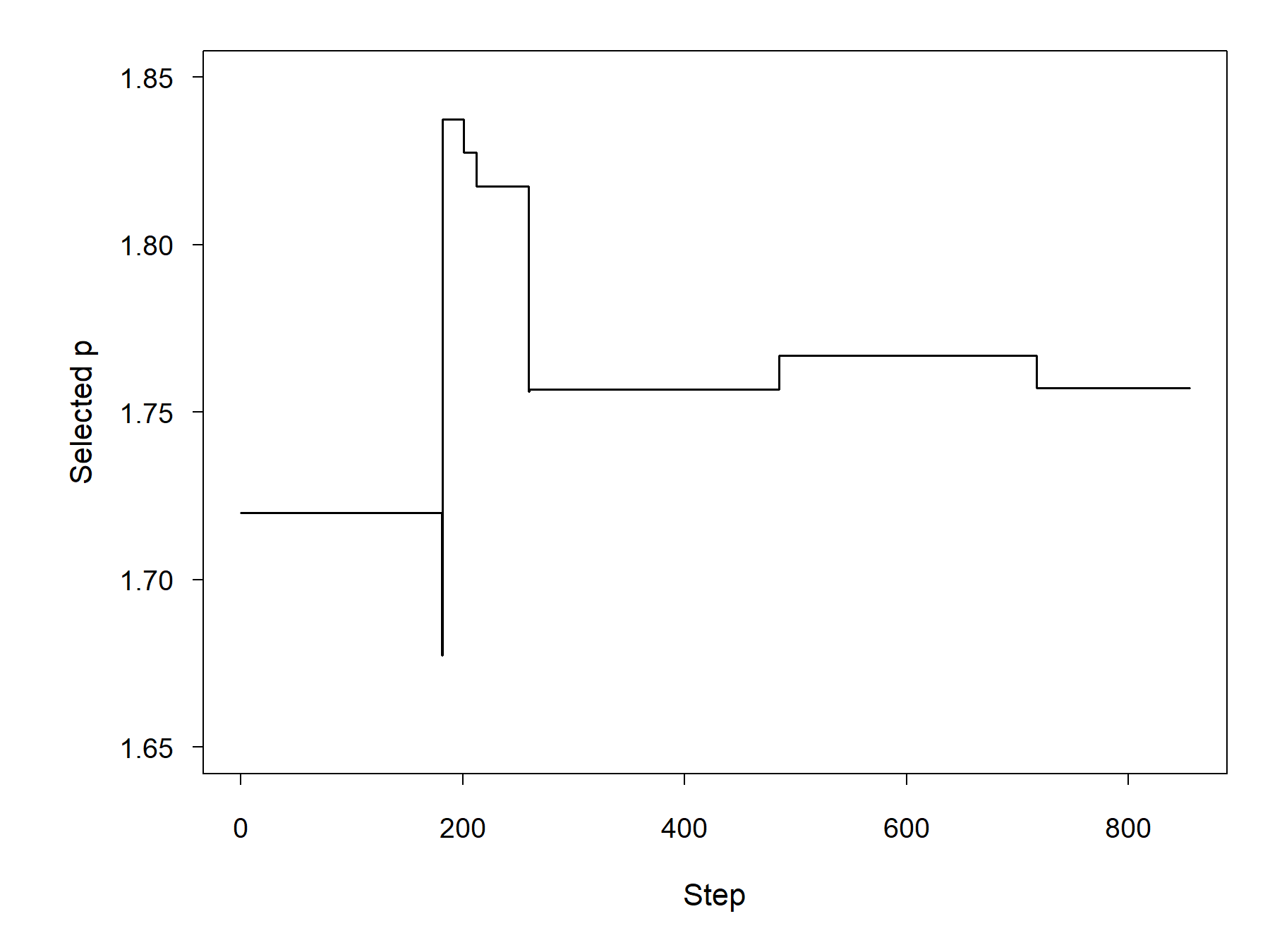} &
\includegraphics[width=\linewidth,draft=false]{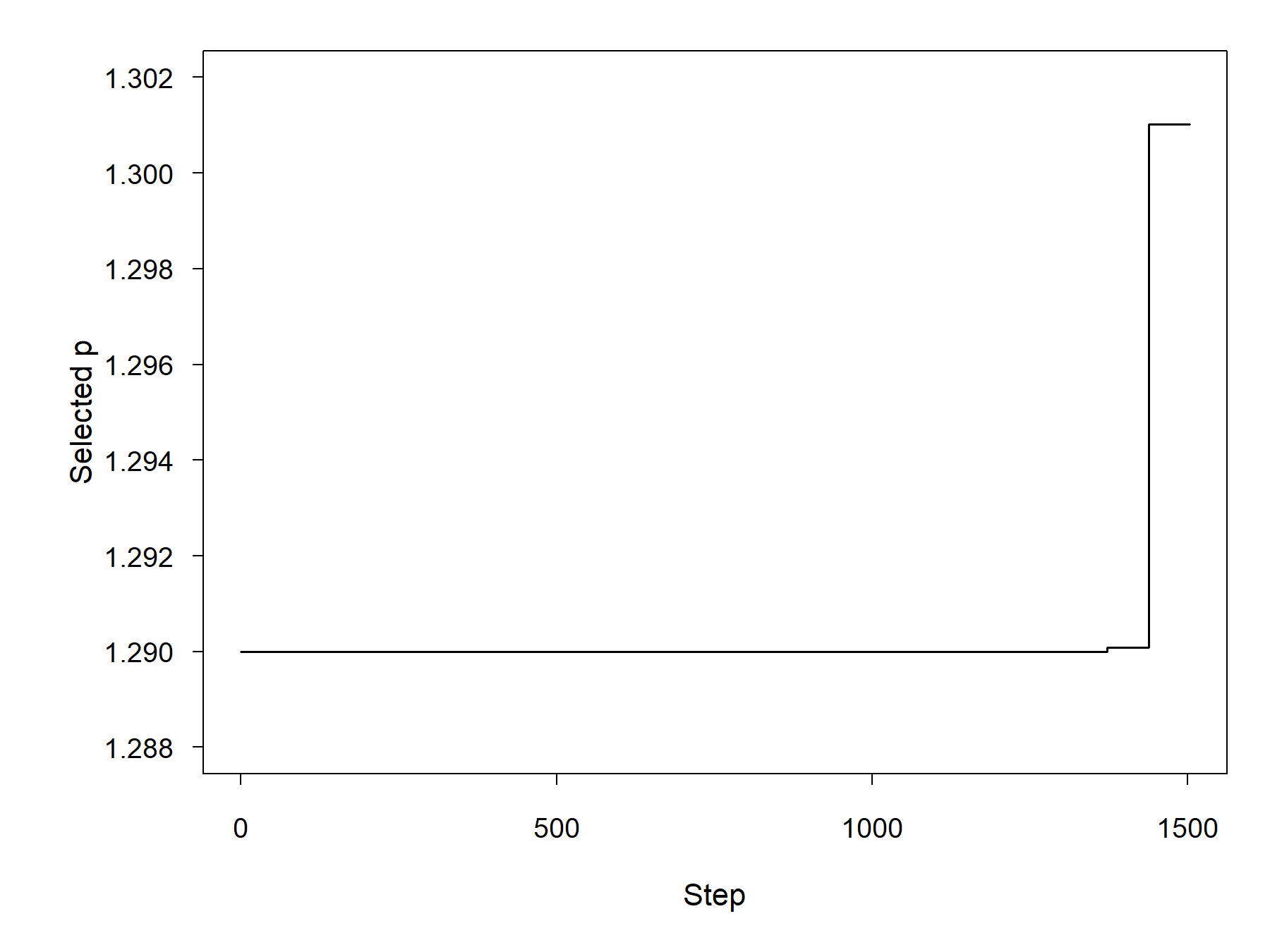} &
\includegraphics[width=\linewidth,draft=false]{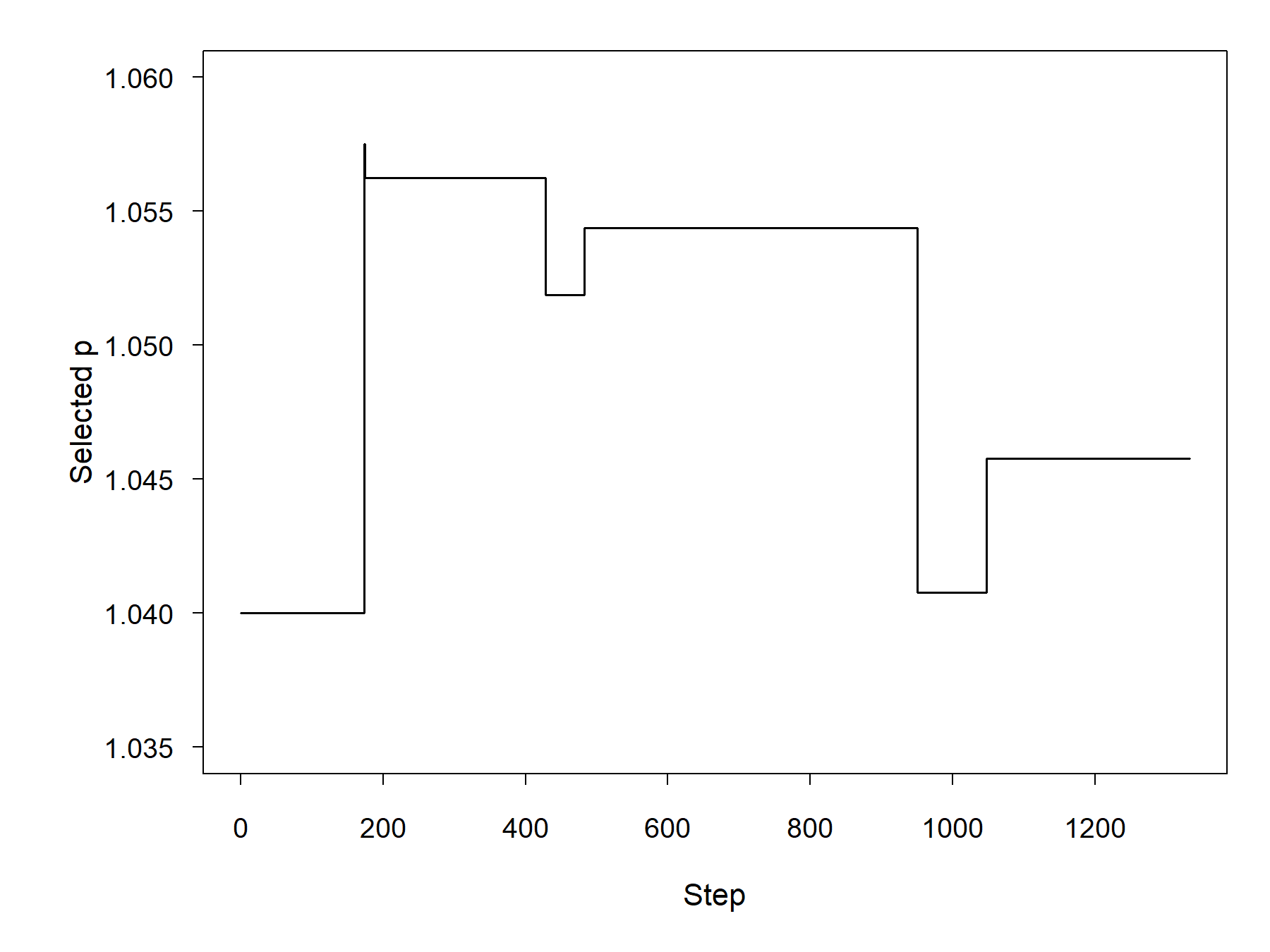} \\
\(2\) &
\includegraphics[width=\linewidth,draft=false]{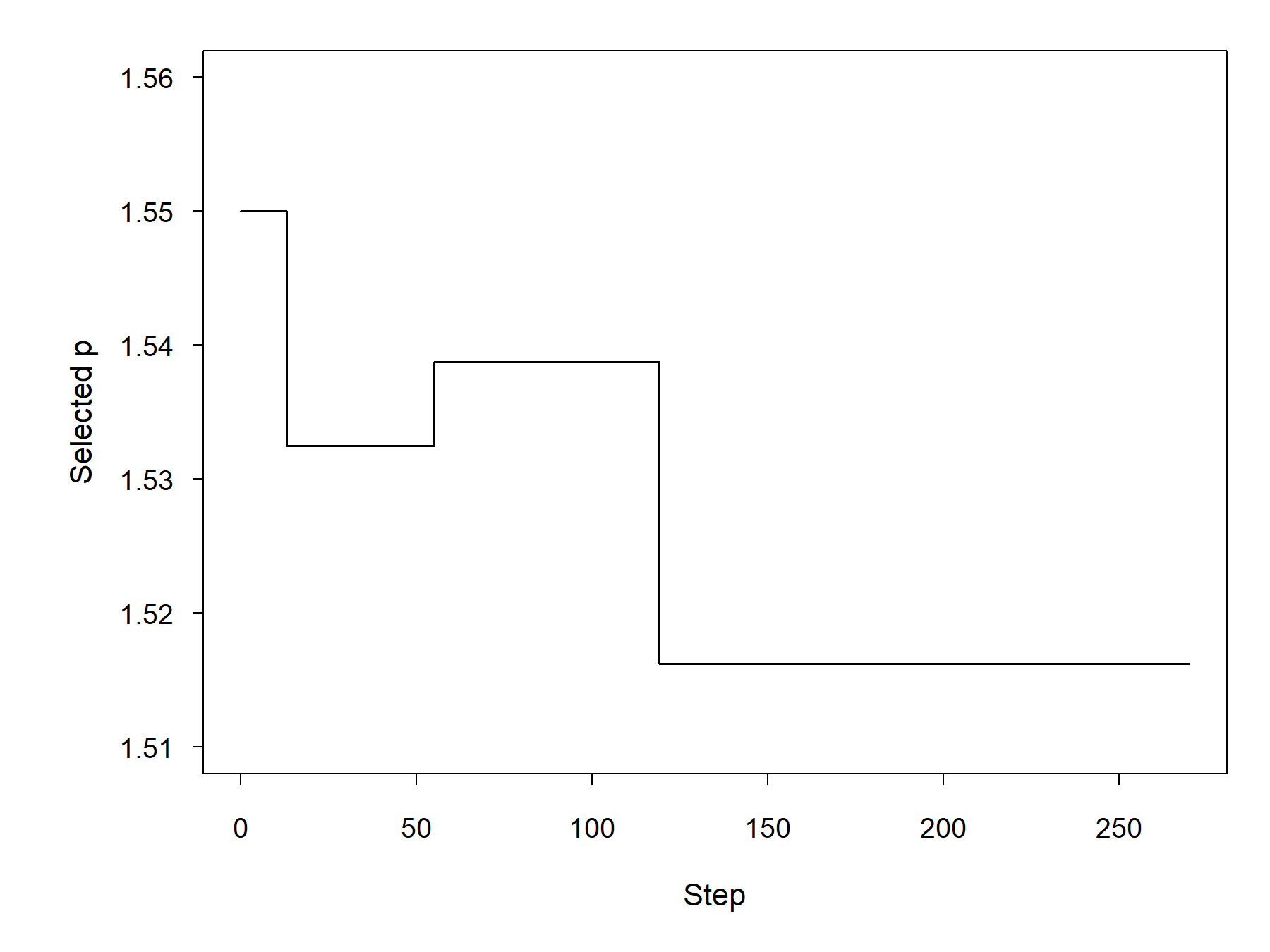} &
\includegraphics[width=\linewidth,draft=false]{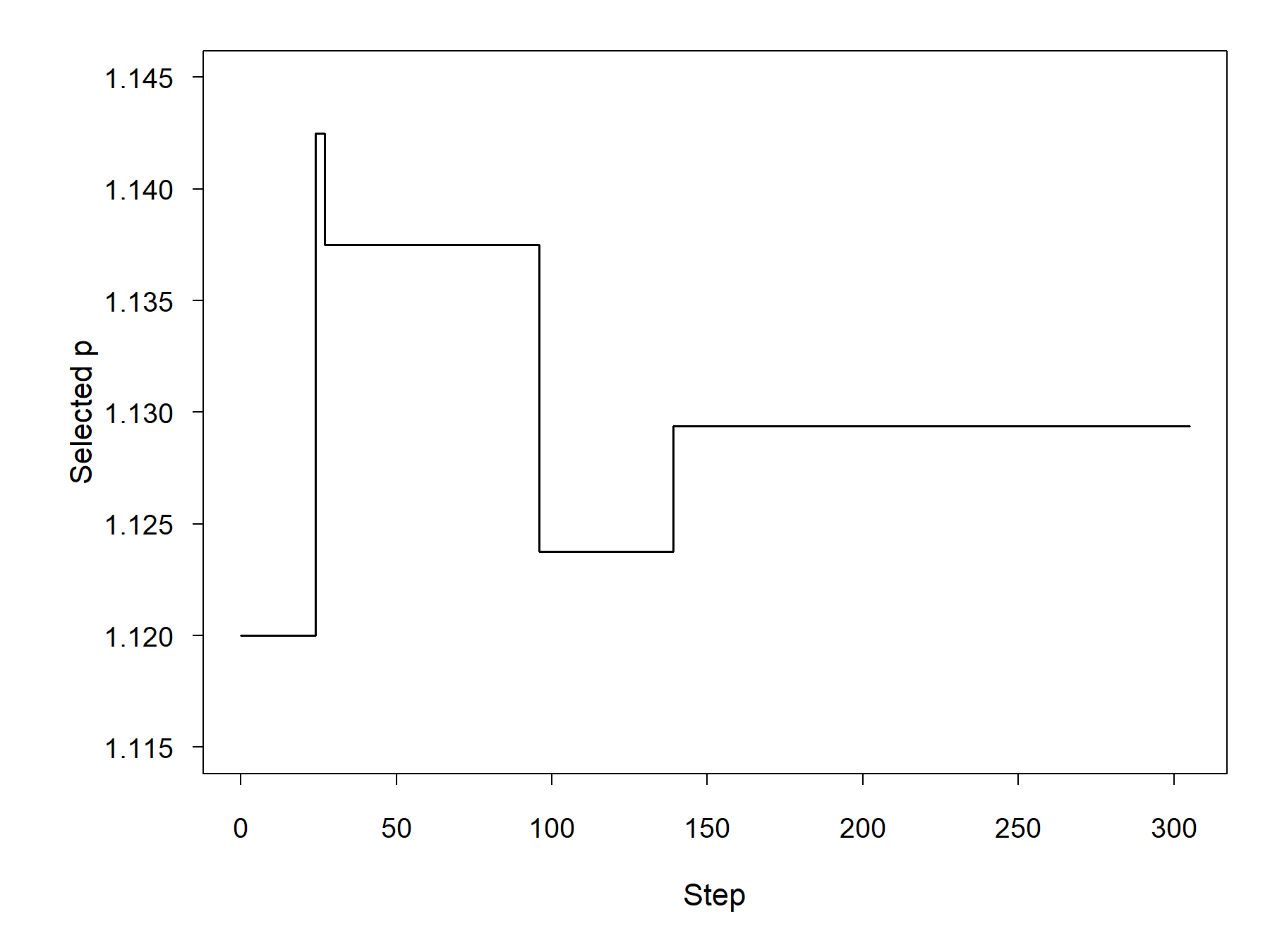} &
\includegraphics[width=\linewidth,draft=false]{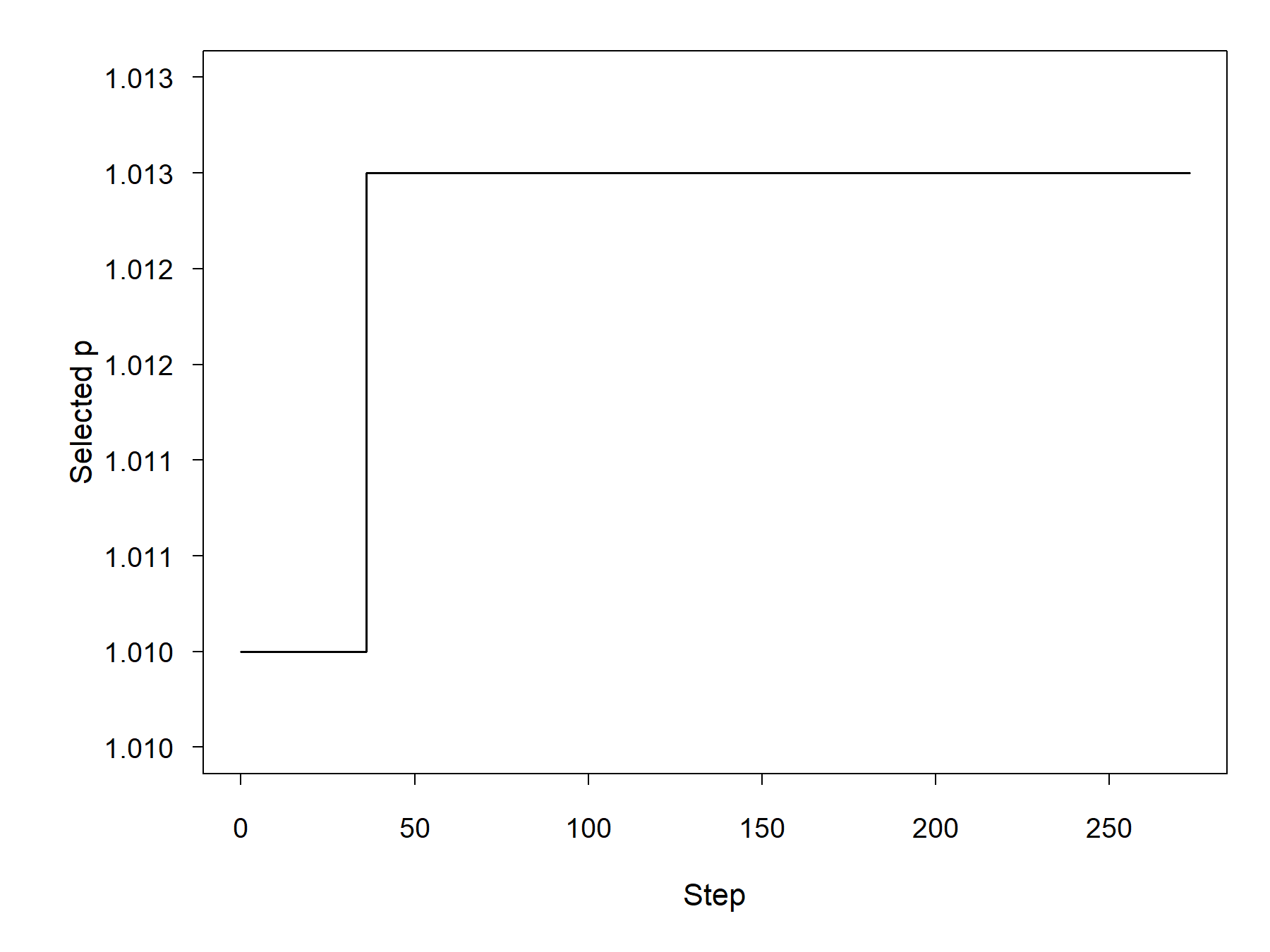} \\
\end{tabular}
\caption{Best shape parameter during the certified interval-wise search for the twelve \((\varepsilon,\delta)\) pairs.}
\label{fig:appendix-bestp-grid}
\end{figure}
\subsection{Uniform shape-derivative bound}
\label{app:task-specific-bound}

The following bound is used in Section~\ref{sec:task-specific-example}.

\begin{lemma}[Uniform shape-derivative bound]
\label{lem:uniform-task-derivative}
Let \(I=[p_L,p_R]\subset(1,\infty)\), and set
\(a_-=1/p_R\) and \(a_+=1/p_L\). Define
\[
\begin{aligned}
V(I):={}&
\frac{a_+^2}{2}
\max\{|\psi(a_-+1)|,|\psi(a_++1)|\}
\\
&+
\frac{a_+}{2}
\sqrt{
a_+(a_++1)
\left[
\psi_1(a_-+2)
+
\max\{\psi(a_-+2)^2,\psi(a_++2)^2\}
\right]
}.
\end{aligned}
\]
Then, for every \(c>0\),
\[
\sup_{p\in I}
\left|
\frac{d}{dp}
Q\left(\frac1p,c^p\right)
\right|
\le V(I).
\]
Consequently, if \(p_I=(p_L+p_R)/2\) and
\(h_I=(p_R-p_L)/2\), then
\[
\left|
Q\left(\frac1p,c^p\right)
-
Q\left(\frac1{p_I},c^{p_I}\right)
\right|
\le h_I V(I),
\qquad p\in I.
\]
\end{lemma}

\begin{proof}
Writing
\[
Q\left(\frac1p,c^p\right)
=
2\int_c^\infty f_p(x)\,dx,
\]
differentiation under the integral sign, justified by dominated convergence
on compact subintervals of \((1,\infty)\), followed by \(y=x^p\), gives,
with \(a=1/p\),
\[
\frac{d}{dp}Q\left(\frac1p,c^p\right)
=
2\int_{c^p}^{\infty}
\frac{a}{2\Gamma(a)}
y^{a-1}e^{-y}
\{1+a\psi(a)-y\log y\}\,dy.
\]
Moreover,
\[
\int_0^\infty
\frac{a}{2\Gamma(a)}
y^{a-1}e^{-y}
\{1+a\psi(a)-y\log y\}\,dy
=
\frac a2
\{1+a\psi(a)-a\psi(a+1)\}
=0.
\]
Hence the zero-mass property implies
\[
\left|
\frac{d}{dp}Q\left(\frac1p,c^p\right)
\right|
\le
\int_0^\infty
\left|
\frac{a}{2\Gamma(a)}
y^{a-1}e^{-y}
\{1+a\psi(a)-y\log y\}
\right|dy.
\]
Using \(1+a\psi(a)=a\psi(a+1)\), the triangle inequality, and
Cauchy--Schwarz,
\[
\begin{aligned}
\left|
\frac{d}{dp}Q\left(\frac1p,c^p\right)
\right|
\le{}&
\frac{a^2}{2}|\psi(a+1)|
+
\frac a2
\sqrt{
a(a+1)
\{\psi_1(a+2)+\psi(a+2)^2\}
}.
\end{aligned}
\]
Since \(a\in[a_-,a_+]\), the monotonicity of \(\psi\) and \(\psi_1\)
gives the bound \(V(I)\). The final inequality follows from the
mean value theorem.
\end{proof}

\end{document}